\documentclass[10pt,openany,leqno]{amsart}
\usepackage[backref=page,colorlinks=true]{hyperref}
\usepackage[capitalize]{cleveref}
\usepackage[english, francais]{babel}
\usepackage{amsmath,amsthm,amsfonts,amssymb,amscd,url}
\usepackage{graphics}
\numberwithin{equation}{section}
\input{xy} \xyoption{all}
\usepackage[latin1]{inputenc}

\usepackage{enumerate,xcolor}

\usepackage{hyperref, mathrsfs,xcolor}
\usepackage{epsfig}
\input{xy} \xyoption{all}

\usepackage[scr=boondoxo,scrscaled=1.05]{mathalfa}
\makeindex
\begin{document}
\def \mathfrak{\cg
 \mathsf}
\def\blacklozenge{{\mathcal E}}
\def\Card{\mathrm{Card\, }}
\def\exp{\mathrm{exp}}
\def\S{\mathbb S}
\def \mC{{i\rho}}
\def\inv{^{-1}}
\def\Max{\text{max}}
\def\cal{\mathcal}
\def\Graph{\mathrm{Graph\, }}
\def \sfw{{\mathscr w}}
\def \sfh{{\mathscr h}}
\def\R{\mathbb R}
\def\N{\mathbb N}
\def\Z{\mathbb Z}
\def\D{\mathbb D}
\def\A{\mathbb A}
\def\B{\mathbb B}
\def\C{\mathbb C}
\def\O{\mathbb O}
\def\cC{\mathcal C}
\def\T{\mathbb T}
\def\a{{\underline a}}
\def\b{{\underline b}}
\def\c{{\underline c}}
\def\Log{\mathrm{log}}
\def\loc{\mathrm{loc}}
\def\inta{\mathrm{int }}
\def é{\'e}
\def \Diff{\mathrm{Diff}}
\def\det{\mathrm{det}}
\def\Re{\mathrm{Re}}
\def\lip{\mathrm{Lip}}
\def\leb{\mathrm{Leb}}
\def\dom{\mathrm{Dom}}
\def\diam{\mathrm{diam}\:}
\def\supp{\mathrm{supp}\:}
\newcommand{\ovfork}{{\overline{\pitchfork}}}
\newcommand{\ovforki}{{\overline{\pitchfork}_{I}}}
\newcommand{\Tfork}{{\cap\!\!\!\!^\mathrm{T}}}
\newcommand{\whforki}{{\widehat{\pitchfork}_{I}}}
\newcommand{\marginal}[1]{\marginpar{{\scriptsize {#1}}}}
\def \np{{\color{red} 4}}
\def \npm{{\color{red} 3}}
\def \npmm{{\color{red} 2}}
\def \det{\mathrm{det}\, }
\def\cU{{\mathscr U}}
\def\cP{{\mathcal P}}
\def\cF{{\mathscr F}}
\def\cH{{\mathcal H}}
\def\cV{{\mathcal V}}
\def\cW{{\mathcal W}}
\def\cX{{\mathcal X}}
\def\cY{{\mathcal Y}}
\def\cZ{{\mathcal Z}}
\def\qand{\quad \text{and} \quad}

\def\boxdot{{\boldsymbol{\mathfrak c}}}
\def\sA{{\boldsymbol{\mathfrak A}}}
\def\sB{{\boldsymbol{\mathfrak B}}}
\def\sC{{\boldsymbol{\mathfrak C}}}
\def\sD{{\boldsymbol{\mathfrak D}}}
\def\sE{{\boldsymbol{\mathfrak E}}}
\def\sG{{\boldsymbol{\mathfrak G}}}
\def\sL{{\boldsymbol{\mathfrak L}}}
\def\sM{{\boldsymbol{\mathfrak M}}}
\def\sP{{\boldsymbol{\mathfrak P}}}
\def\sR{{\boldsymbol{\mathfrak R}}}
\def\sS{{\boldsymbol{\mathfrak S}}}
\def \sT{{\boldsymbol{\mathfrak T}}}
\def\sU{{\boldsymbol{\mathfrak U}}}
\def\sV{{\boldsymbol{\mathfrak V}}}
\def\sW{{\boldsymbol{\mathfrak W}}}
\def\sX{{\boldsymbol{\mathfrak X}}}
\def\sY{{\boldsymbol{\mathfrak Y}}}
\def\sZ{{\boldsymbol{\mathfrak Z}}}

\def\sa{{\boldsymbol{\mathfrak a}}}
\def\sb{{\boldsymbol{\mathfrak b}}}
\def\sd{{\boldsymbol{\mathfrak d}}}
\def\se{{\boldsymbol{\mathfrak e}}}
\def\sf{{\boldsymbol{\mathfrak f}}}
\def\st{{\boldsymbol{\mathfrak t}}}
\def\sg{{\boldsymbol{\mathfrak g}}}
\def\sh{{\boldsymbol{\mathfrak h}}}
\def\si{{\boldsymbol{\mathfrak o}}}
\def\sm{{\boldsymbol{\mathfrak m}}}
\def\sn{{\boldsymbol{\mathfrak n}}}
\def\sq{{\boldsymbol{\mathfrak q}}}

\def\so{\boldsymbol{\diamond}}
\def\sp{{\boldsymbol{\mathfrak p}}}
\def\sl{{\boldsymbol{\mathfrak {l}}}}
\def\sr{{\boldsymbol{\mathfrak {r}}}}
\def\ss{{\boldsymbol{\mathfrak {s}}}}
\def\st{{\boldsymbol{\mathfrak {t}}}}
\def\su{{\boldsymbol{\mathfrak {u}}}}
\def\sv{{\boldsymbol{\mathfrak {v}}}}
\def\spp{{\boldsymbol{\mathfrak {p}}}}
\def\sw{{\boldsymbol{\mathfrak {w}}}}

\def \sc{\sw}

\def\cg{\color{black} }
\def\cb{\color{black}}
\def\arr{\overleftarrow}
\def\avv{\overrightarrow}
\theoremstyle{plain}

\def\Cr#1{\overset{\tilde Y}{#1}}

 	\definecolor{tropicalrainforest}{rgb}{0.0, 0.46, 0.37}

\newtheorem{thm}{\bf Theorem}[section]
\newtheorem{theorem}[thm]{\bf Theorem}

\newtheorem*{conjecture*}{\bf Conjecture}

\newtheorem{conj}[thm]{\bf Conjecture}
\newtheorem{claim}[thm]{\bf Claim}
\newtheorem{assumption}[thm]{\bf Assumption}

\newtheorem{problem}[thm]{\bf Problem}
\newtheorem{question}[thm]{\bf Question}
\newtheorem{proposition}[thm]{\bf Proposition}
\newtheorem{corollary}[thm]{\bf Corollary} 
\newtheorem{lemma}[thm]{\bf Lemma}

\newtheorem{sublemma}[thm]{\bf Sublemma}
\newtheorem*{Takens prbm}{Takens' Last Problem} 
\newtheorem{remark}[thm]{\bf Remark}
\newtheorem{fact}[thm]{\bf Fact}
\newtheorem{exem}[thm]{\bf Example}
\newtheorem{definition}[thm]{\bf Definition}
\newtheorem*{definition*}{\bf Definition}
\newtheorem{defi}[thm]{\bf Definition}

\newtheorem{theo}{\bf Theorem}[section]
\renewcommand{\thetheo}{\Alph{theo}}
\newtheorem{coro}{\bf Corollary}[section]
\renewcommand{\thecoro}{\Alph{coro}}

\newtheorem{example}[thm]{\bf Example}

\renewcommand*{\backref}[1]{}
\renewcommand*{\backrefalt}[4]{\quad \tiny 
  \ifcase #1 (\textbf{NOT CITED.})%
  \or    (Cited on page~#2.)%
  \else   (Cited on pages~#2.)%
  \fi}
  \begin{otherlanguage}{english}

\title{Emergence of wandering stable components}

\author{Pierre Berger and S\'ebastien Biebler$^*$}

\date{\today.\\
*{The authors  were partially supported by the ERC project 818737 
Emergence of wild differentiable dynamical systems.}}

\maketitle
\begin{flushright}
\emph{To Mikhail Lyubich on his 60th birthday.}
\end{flushright}

\begin{abstract}
We prove the existence of a locally dense set of real polynomial automorphisms of $\mathbb C^2$ displaying a wandering Fatou component; in particular this solves the problem of their existence, reported by Bedford and Smillie in 1991. These Fatou components have nonempty real trace and their statistical behavior is historic with high emergence. 
The proof is based on a geometric model for parameter families of surface real mappings. 
At a dense set of parameters, we show that the dynamics of the model displays  a historic, high emergent, stable domain. We show that this model can be embedded into families of Hénon maps of explicit degree and  also in 
an open and dense set of $5$-parameter $C^r$-families of surface diffeomorphisms in the Newhouse domain, for every $2\le r\le \infty$ and $r=\omega$. 
This implies a complement of the work of Kiriki and Soma (2017), a proof of the last Taken's problem in the $C^{\infty}$ and $C^\omega$-case.  The main difficulty is that here perturbations are done only along finite-dimensional parameter families. The proof is based on the multi-renormalization introduced in \cite{berger2018zoology}.
\end{abstract}
\tableofcontents

\section*{Introduction: State of the art and main results}
The aim of this paper is to reveal two new phenomena in the dynamics of analytic surface diffeomorphisms.  In \cref{intro.1,intro.2}, we will describe these two phenomena in the iconic Hénon family of diffeomorphisms of the (real and complex) plane. Then in \cref{intro.3}, we will present the general set up of our study. Finally an outline of the proof and of the manuscript is given in \cref{intro.4}.   The proof will occupy most of the manuscript and lies  to  combinatorics, geometry and 
real or complex  analysis.

\subsection{Wandering Fatou components}\label{intro.1}
Given a holomorphic endomorphism $f$ of a complex manifold $X$, the \emph{Fatou set} consists of the set of points $x\in X$ which have a neighborhood $U$ such that $(f^n|U)_n$ is normal. By definition, it is open and invariant by the dynamics. In particular the connected components of the Fatou set, called \emph{Fatou components}, are mapped to each other under the dynamics. A main question is:
\begin{question}\label{funda ques} Is the dynamics restricted to the Fatou set always ``simple''?\end{question}
To answer this ``philosophical" question it is fundamental to know whether a dynamics may admit a \emph{wandering Fatou component}, that is a component which is not sent by an iterate of the dynamics to a periodic one. \medskip

When $X$ is the Riemann sphere $\mathbb{P}^{1}(\mathbb{C})$, in a seminal work \cite{sullivan}, Sullivan showed that a rational mapping does not have any wandering Fatou component. This result, together with the classification of Fatou components due to Fatou himself, Siegel and Herman, allows for a complete description of the dynamics restricted to the Fatou set: the orbit of every point in the Fatou set eventually lands in an attracting basin,   a parabolic basin or a rotation domain (that is, a Siegel disk or a Herman ring). Hence, in this context the answer to Question \ref{funda ques} is positive.  
This was later generalized by Eremenko and Lyubich \cite{eremlyubi} and Goldberg and Keen \cite{goldberg_keen_1986} who showed that entire mappings with finitely many singular values have no wandering Fatou component. The absence of wandering Fatou component for another class of entire maps has been studied in \cite{rempe}.

On the other hand, when $X=\C$, Baker \cite{baker} gave (prior to Sullivan's result) the first example of an entire map displaying a wandering Fatou component, see also\cite{eremlyubi}, \cite{sullivan}, \cite{hermn}. In all these examples, the singular set is unbounded. Bishop introduced in \cite{bishop2015} another example displaying this time a bounded singular set. The rich history of wandering Fatou components for entire maps includes the more recent examples \cite{shishimarti,high3,2014arXiv1410.3221F,kisaka_shishikura_2008,
2007arXiv0708.0941B}. \medskip

In higher dimension, the fundamental problem of the existence of  wandering Fatou components was first studied in 1991 in the work of Bedford and Smillie \cite[Theorem 5.6]{BS1}, in the context of polynomial automorphisms of $\mathbb{C}^{2}$ (see also \cite[page 275]{FS99}). These are the polynomial maps of two complex variables which are invertible with a polynomial inverse.
In this setting, a Fatou component is either the open set of points with unbounded forward orbit (which is invariant by the dynamics) or a connected component of the interior of the set $K^+$ of points with bounded forward orbit.  
Later the question of the existence of wandering Fatou components was generalized by Fornaess and Sibony in \cite[Question 2.2]{fornaesssibny} to a holomorphic endomorphism of $\mathbb{P}^k(\C)$, for $k\ge 2$. This question was completely open until the recent breakthrough of Astorg, Buff, Dujardin, Peters and Raissy \cite{ABDPR16} who proved, following an idea of Lyubich, the existence of a wandering Fatou component for holomorphic endomorphisms of $\mathbb P^2 (\mathbb{C})$. Their example is a skew product of real polynomial maps and possesses a parabolic point (two eigenvalues equal to 1). Their proof uses parabolic implosion techniques. 
Developments of this construction allowed recently Hahn and Peters to build examples of polynomial automorphisms of $\mathbb{C}^{4}$ displaying a wandering Fatou component in \cite{hahnpeters} and  then Astorg,  Boc Thaler and Peters to replace 
 the parabolic fixed point by an elliptic point for two-dimensional dynamics in \cite{astorg2019wandering}.
 However it does not seem possible to adapt these techniques to the long standing problem of the existence of such components for polynomial automorphisms of $\C^2$, for two reasons. First by \cite{MF89}, there is no polynomial automorphism of $\C^2$ which is a skew product and has a wandering Fatou component. Secondly, the Jacobian of an automorphism is constant and so the existence of a parabolic fixed point implies that the Jacobian is $1$, which is incompatible with their construction (a real disk is shrunk along its orbit). In the same topic, let us mention the example of a transcendental biholomorphic map in $\mathbb{C}^{2}$ with a wandering Fatou component oscillating to infinity by Forn\ae ss and Sibony in \cite{fswandering} and other examples of wandering domains for transcendental mappings in higher dimension in \cite{high2,high1}. 
About the non-existence of wandering Fatou components in higher dimension, a few cases study have been done proved \cite{lilo,ji} in some particular cases (skew products with a super-attracting invariant fiber). 

The problem of the existence of a wandering Fatou component for polynomial automorphisms of $\C^2$ increased its interest with the recent conjecture of Bedford \cite{bedfordquestion} stating the existence of wandering Fatou components for polynomial automorphisms of $\C^2$ and the spectacular development \cite{LP14} giving a complete classification of the periodic Fatou components for some polynomial automorphisms\footnote{called moderately dissipative, but in a different meaning as we will use.}, see also \cite{We03}.

The (first) main result of the present work is an answer to the original problem raised by Bedford and Smillie \cite{BS2}:
\begin{theo}\label{main wandering}
 For any sufficiently small real number $b\neq 0$, there exists {a nonempty  open subset $\cal P_b\subset \R^5$ such that for a dense subset of parameter  $p=(p_i)_{0\le i\le 4}\in \cal P_b$,}  
  the following polynomial automorphism displays a  wandering  Fatou component $\cal C$ in $\C^2$:
\[f_p: (z,w)\in \C^2\mapsto  (z^6+ \sum_{i=0}^4 p_i \cdot z^i-  w, b\cdot z)\in \C^2 \; .\]
   Moreover the Fatou component  $\cal C$  satisfies the following properties:
\begin{enumerate}[(1)]
\item the real trace $\cal C\cap \R^2$ of $\cal C$ is nonempty,
\item for any compact set $\cal C' \Subset \cal C$, it holds  $\lim_{n\to \infty} \diam f_p^n(\cal C')=0$,
\item for every $(z,w)\in \cal C$, the limit set of the orbit of $(z,w)$ contains a (real)  horseshoe $K$, 
\item for every $(z,w)\in \cal C$, the sequence $(\mathscr e_n)_n$ of empirical measures $\mathscr e_n:= \frac{1}{n} \sum_{i=0}^{n-1} \delta_{f^i_p(z,w)}$ diverges. 
Furthermore, there exists $\mu$ in the set of invariant probability measures $\cal M_p(K)$ of $f_p|K$ such that the limit set of $(\mathscr e_n)_n$ contains $t \cdot \mu+(1-t)\cdot \cal M_p(K)$ for some $0<t<1$. 
 \index{$\mathscr e_n$}
\end{enumerate}
\end{theo}
     We recall that a property $\cP(a)$ depending on $a \in \R$ is said to be true for any (sufficiently) small $a$ if there exists $a_0>0$ such that  $\cP(a)$ is true when $|a|<a_0$. 
We recall that a \emph{horseshoe} is both a Cantor set and   a hyperbolic  basic set. We will see in the sequel (\cref{horseshoe emerge}) that its set  of invariant probability measures $\cal M_p(K)$ is huge.  

\begin{remark}\label{localy dense}
We will see in \cref{Newhouse model2} that for every $b\neq 0$, there is a neighborhood of $V_b$ of $b$ such that $\bigcap_{b'\in V_b}\cP_{b'}$ has nonempty interior.   
This implies   that 
the set $\cal P$  of parameters $( (p_i)_{0\le i\le 4}, b)\in \R^5 \times \R$  for which there is a wandering Fatou component is \emph{locally dense} in $\R^6$: its closure has nonempty interior. 
 
\end{remark}

\begin{remark}\label{rk main wandering}
Our techniques give also new examples of endomorphisms of $\mathbb P^2(\C)$. For instance, the conclusion of Theorem \ref{main wandering} holds true also for the family of  maps:
\[f_p: [z,w, t]\in \mathbb P^2(\C) \mapsto  [z^6+ \sum_{i=0}^4 p_i \cdot z^i\cdot t^{6-i} - w^6, b\cdot z^6, t^6] \; .\]
\end{remark} 

\subsection{Historical behavior and wandering stable components in the Newhouse domain}\label{intro.3}
Given a (general) differentiable dynamical system $f$ of a compact manifold $M$, since the seminal works of Boltzmann and Birkhoff, one is interested in describing the statistical behavior of most (Lebesgue) of the points. The statistical behavior of a point $x$ is described by the sequence $(\mathscr e_n(x))_n$ of empirical measures \[\mathscr e_n(x):= \frac{1}{n}\sum_{i=0}^{n-1} \delta_{f^i(x)}\; .\]
The point $x$ has a \emph{historic behavior} if its sequence of empirical measures $(\mathscr e_n(x))_n$ does not converge   \cite{Ruelle}. The method used to prove Theorem \ref{main wandering}  basically comes from real, smooth dynamical systems and is related to the following well-known:
\begin{Takens prbm}[\cite{T08}] Are there persistent classes of smooth dynamical systems for which the set of initial states which give rise to orbits with historic behavior has positive Lebesgue measure?
\end{Takens prbm}
   We recall that a \emph{surface diffeomorphism} is a smooth diffeomorphism of a real surface.   A strategy to solve this problem is deduced  from a scenario introduced by Colli and Vargas \cite{colli_vargas_2001}, and takes place in the dissipative Newhouse domain.  The \emph{dissipative $C^r$-Newhouse domain} $\cal N^r$ \index{$\cal N^r$} is the open set of surface diffeomorphisms leaving invariant a hyperbolic basic set displaying a $C^r$-robust quadratic homoclinic tangency and an area contracting periodic point (see Definition \ref{def newhouse domain} for more details). This strategy was recently implemented by  Kiriki and Soma \cite{KS17} to solve this problem in the case of finitely differentiable diffeomorphisms of surface.  For every $2\le r< \infty$,  they proved the existence of a dense set in $\cal N^r$ of diffeomorphisms $f$ displaying a wandering open set $U$ whose points are asymptotic,  but do not lie in the basin of a periodic orbit.
Let us formalize this by introducing the following terminology:
\begin{defi}\label{def stable domaine}\index{Stable domain} \index{Stable component} \index{Wandering stable component}
A point $x$ is asymptotically \emph{stable} if it has a neighborhood $U$ formed by asymptotic  points: for every $y\in U$, it holds $d(f^n(x), f^n(y))\to 0$ as $n\to +\infty$. 
A \emph{stable domain} is a connected open subset formed by stable points.  A \emph{stable component} is a stable domain which is maximal. In other words, a \emph{stable component} is a component of the set of asymptotically stable points. A stable component is \emph{wandering} if it does not intersect its iterates.
\end{defi}
Note that the first example of a wandering stable component goes back to Denjoy \cite{De32}. 
\begin{question} Is a stable component either wandering or the basin of a  periodic point?\end{question}
It can be shown that the construction of
Colli, Vargas, Kiriki and Soma yields a wandering stable component\footnote{The argument is contained in the proof of \cref{Main theorem}.}. Moreover they showed that the sequence of empirical measures of points in their stable domain does not converge.
\medskip

We will introduce a geometric model for parameter families of dynamics, and we show that the scenario of \cite{colli_vargas_2001,KS17} occurs at a dense set of parameters of such families, using a new perturbation method. The proof is based on a technique of composed Hénon-like renormalizations  \cite[Theorem D and Remark 3.5]{berger2018zoology} originally devoted to such pathological behaviors\footnote{The renormalizations we will deal with have unbounded combinatorics, in the case of bounded combinatorics, Ou \cite{Ou2019} showed recently the nonexistence of wandering stable components.}. 
 Also, we prove that this geometric model appears both in the setting of Theorem \ref{main wandering} and densely inside the Newhouse domain. This implies the second main result of this paper, which gives a positive answer to Takens' last problem in the smooth and even the real analytic category:
\begin{theo}\label{theorem B} For every $r \in [2,\infty] \cup \{\omega\}$, 
every $f\in {\cal N}^{r}$ can be $C^r$-approximated by $\tilde f\in {\cal N}^{r}$ such that for each $C^r$-family $(f_p)_{p\in \R^5}$ in an open dense set of  $\tilde f$'s unfolding  (i.e. $f_0=\tilde f$), there exists a subset of parameters $p\in \R^5$ which is dense in a neighborhood of $0$ for which $f_p$ displays a wandering stable component $\cal C$ satisfying the following:
\begin{enumerate}[$(1)$]
\item for every $x\in \cal  C$, the limit set of the orbit of $x$ contains a horseshoe $K$,
\item for every $x\in \cal C$, the sequence $(\mathscr e_n)_n$ of empirical measures $\mathscr e_n:= \frac{1}{n} \sum_{i=0}^{n-1} \delta_{f^i_p(x)}$ diverges. 
Furthermore, there exists $\mu$ in the set of invariant probability measures $\cal M_p(K)$ of $f_p|K$ such that the limit set of $(\mathscr e_n)_n$ contains $t \cdot \mu+(1-t)\cdot \cal M_p(K)$ for some $0<t<1$.
\end{enumerate}
\end{theo}
This theorem extends Kiriki-Soma's work \cite{KS17}. First it solves the $C^\infty$-case and the real analytic case (among analytic mappings displaying a complex extension on a uniform complex strip) among surface diffeomorphisms. Secondly, it does not only show that the set of dynamics displaying a wandering domain with historic behavior is dense in $\cal N^{r}$, but also that this phenomenon is of codimension $\le 5$. Thirdly, it implies a stronger description of the statistical behavior of the wandering domain, as we will see in 
Corollary \ref{coro B}  via the notion of emergence (see \textsection \ref{intro.2}).  

Theorem \ref{main wandering}, \cref{rk main wandering} and Theorem \ref{theorem B} are basically consequences of the following result: 
\begin{theo}[Main] \label{Thepremain4application}
Let $(f_p)_{p\in \R^5}$ be a   $C^2$-family of surface diffeomorphisms.
Assume that $f_{0}$ has an area contracting periodic saddle point  displaying $5$ different quadratic homoclinic tangencies.

 If these quadratic tangencies unfold\footnote{See \cref{Non-degenerate unfolding}.} non-degenerately with $(f_p)_p$, then 
 there is a parameter subset $\cal D\subset \R^5 $ such that $cl(\cal D)\ni 0$, $\mathrm {int (cl (}\cal D))\neq \emptyset$ and  for every $p\in \cal D$ the map $f_p$ displays a wandering stable component $\cal C$ satisfying:
\begin{enumerate}[$(1)$]
\item for every $x\in \cal C$, the limit set of the orbit of $x$ contains a horseshoe $K$,
\item every $x\in \cal C$ has its sequence $(\mathscr e_n(x))_{n\ge 0}$ of empirical measures which diverges. Furthermore, the limit set of $(\mathscr e_n(x))_n$ contains $t \cdot \mu+(1-t)\cdot \cal M_{p}(K)$ for some $0<t<1$ and $\mu\in \cal M_{p}(K)$. 
\item If moreover $f_p$ is real analytic, then $\cal C$ is the real trace of a wandering (complex) Fatou component which is a stable component of the complex extension of $f_p$. 
\end{enumerate}
\end{theo}
All the holomorphic mappings we found with a wandering Fatou component are actually real (even if the Fatou component is an open subset of  $\C^2$ or $ \mathbb{P}^2(\C)$). Nevertheless, it seems to us possible to develop the techniques presented in this work to show the following:
 \begin{conj}\label{conj fatour wandering complex}  There is a locally dense set   in  the space of polynomial automorphisms   $\mathrm{Aut}(\C^2)$ of $\C^2$ formed by dynamics displaying a wandering Fatou component.
 \end{conj}
Even if Theorem \ref{main wandering} is about  H\'enon maps of degree 6, we know that there exists  a family of polynomial automorphisms of degree $5$ satisfying the assumptions of Theorem \ref{Thepremain4application}, hence possessing a wandering Fatou component. We do not know if it is possible for lower degree:
 \begin{question} Does there exist a polynomial automorphism of $\C^2$ with degree $\le 4$ which displays a wandering Fatou component?
  \end{question}
  To answer to this question, one might use a  mix of  the parabolic renormalization techniques of \cite{ABDPR16} together with those of the present paper. A reasonable problem to develop these techniques is: 
   \begin{problem} Give an explicit lower bound on the Hausdorff dimension of the set of parameters having a wandering Fatou component in the space of polynomial automorphisms of $\C^2$ of degree 5. 
\end{problem}

The last item in Theorem \ref{main wandering} indicates that the statistical complexity of the dynamics is high. More specifically this implies that its \emph{emergence is high}, and even that its order is positive and at least the order of the unstable dimension of $K$. 
 We explain this in the sequel.
\subsection{Statistical complexity of the dynamics: Emergence}\label{intro.2}

The recent notion of emergence is a natural way to quantify how far from being ergodic  a system is. In this subsection we deduce from  Theorems \ref{main wandering},  \ref{theorem B} and \ref{Thepremain4application} that the emergence of some dense set of dynamics is more than high.

Roughly speaking, the emergence of a dynamical system  is the growth rate as $\epsilon\to  0$ of the number of $\epsilon$-balls necessary to describe the statistical behavior of the system -- in mean -- up to precision $\epsilon$. The statistical behavior is given by the empirical measures; a natural distance between them is the 
 Wasserstein distance.\label{def Wasserstein} Recall that by the Kantorovich-Rubinstein's theorem, the ($1^{st}$) Wasserstein distance between two probability measures $\mu$, $\nu$ on $M$ is equal to:\index{$d_{W}$} \label{definitiondelemergenceP4}
\[d_W(\mu, \nu)= \max_{\phi \in \mathrm{Lip}^1} \int \phi d(\mu-\nu)\; ,\] 
where $\mathrm{Lip}^1$ \index{$\mathrm{Lip}^1$} stands for the space of $1$-Lipschitz real functions. Let us also recall that on the set of probability measures on any compact metric space, the Wasserstein distance induces the weak $*$ topology. To quantify the complexity of the statistical behavior of typical orbits for general dynamical systems, the notion of emergence has been introduced in \cite{beremer}.
\begin{definition*}
The \emph{emergence} $\cal E_f(\epsilon)$ of $f$ at scale $\epsilon$ is the minimal number $N$ of probability measures $(\mu_i)_{1\le i\le N} $ satisfying: \index{$\mathcal{E}_f(\epsilon)$}
\[ \limsup_{n\to \infty} \int \min_{1\le i\le N} d_W ( \mathscr e_n(x), \mu_i)d\mathrm{Leb}<\epsilon\; .\]
\end{definition*}
We are interested in understanding the asymptotic behavior of $\cal E_f(\epsilon)$ when $\epsilon\to 0$. For example, the emergence of a uniformly hyperbolic system is finite (since there are finitely many physical probability measures whose basins cover a full measure subset of $M$). On the other hand,  the emergence of the geodesic flow on the unit tangent bundle of the flat torus $\R^n/\Z^n$ is of the order of $\epsilon^{-n}$, hence polynomial. 
When the emergence is not polynomial:
\begin{equation}\tag{$\star$} \limsup_{\epsilon \to 0} \frac{\log \cal E_f(\epsilon)}{- \log \epsilon }=\infty\; ,\end{equation}
then the global statistical behavior of the system is deemed very complex.  Let us recall:
\begin{conjecture*}[\cite{beremer}]
Super polynomial emergence is typical in many senses and in many
categories of dynamical systems.
\end{conjecture*}
In the recent work \cite{bergerbochi2019}, it has been shown that the order of the emergence $\cal {OE}_f$:
\[ \cal {OE}_f:= \limsup_{\epsilon\to 0}\frac{\log \log \cal E_f(\epsilon)}{-\log \epsilon}\;  \]
 of a system $f$ on a compact manifold of dimension $d$ is at most $d$.  Furthermore, it has been shown that this upper-bound is generically attained among conservative, surface $C^\infty$-diffeomorphisms displaying an elliptic periodic point.  Note  that if a dynamics has emergence of positive order, then its emergence is not polynomial. Thus this confirmed the main Conjecture of \cite{beremer} in the category of surface conservative diffeomorphisms  with local genericity as a version of typicality.

 In the situation of Theorem \ref{main wandering}, we are able to estimate the order of emergence using   Theorem \ref{main wandering}.(4) and the following consequence of \cite[Theorem A]{bergerbochi2019}.  
Let $f$ be a {surface diffeomorphism}.    Given an invariant set $K$,  
  we recall that  $\cal M_f(K)$ denotes the set of invariant probability  measures of $f|K$.
 
 \begin{theorem}\label{horseshoe emerge}
If $K$ is a horseshoe of a $C^{1+\alpha}$-surface diffeomorphism $f$, the covering number $\cal N(\epsilon )$ at scale $\epsilon$ of  $\cal M_f(K)$ for the Wasserstein distance has  order at least  the unstable dimension\footnote{This is the Hausdorff dimension of $K\cap W^u_{loc}(x)$ for any $x\in K$.}  $d_u>0$ of $K$: 
\[\liminf_{\epsilon\to 0} \frac{\log \log \cal N(\epsilon)}{-\log \epsilon}\ge  d_u\; .\]
\end{theorem}
This together with Theorem \ref{main wandering} and \cref{localy dense} imply: 
\begin{coro}\label{coro A}
There exists a locally dense set in the space of real polynomial automorphisms of $\C^2$ whose emergence has positive order $\cal {OE}_f$ (for the canonical normalized, volume form on $\mathbb P^2(\C))$.
\end{coro}
  To this extent, Corollary~\ref{coro A} gives a negative answer to Question \ref{funda ques} in the case of polynomial automorphisms of $\C^2$ and confirms the main conjecture of \cite{beremer}, in the category of real polynomial automorphisms with local density as a (weak) version of typicality.

Note that in contrast to the nowhere dense example of \cite{ABDPR16} whose emergence of the wandering domain is linear (as for the well-known Bowen eye), Theorem~\ref{main wandering} and Corollary~\ref{coro A} exhibit a locally dense set of polynomial automorphisms of $\R^2$ whose emergence has positive order.

Let us communicate the following question raised by Ledrappier:
\begin{question}
Are there typical examples of transcendental maps of $\C$ with super-polynomial emergence?
\end{question}
In the same spirit, a positive answer to this question would lead to a negative answer to Question~\ref{funda ques} in the category of transcendental maps of $\C$.
 Also the concept of emergence enables to strengthen Conjecture  \ref{conj fatour wandering complex}:
 \begin{conj}\label{conj fatour wandering complex2}  There is a locally dense set in  the space of polynomial automorphisms   $\mathrm{Aut}(\C^2)$ of $\C^2$ formed by dynamics displaying a wandering Fatou component with high emergence.
 \end{conj}
In the real setting, we notice that Theorems \ref{theorem B} and  \ref{horseshoe emerge} imply the   following contribution to the main conjecture of \cite{beremer}: 
\begin{coro}\label{coro B}  In the setting of Theorem \ref{theorem B}, the emergence of $f_p| \bigcup_{n\ge 0}  f_p^n(\cal C)$ has order $\ge d_u (K)$.  
\end{coro}
This corollary extends the very recent work \cite{nakano} developing \cite{KS17},
where  super-polynomial emergence was shown to occur locally densely in the $C^r$-Newhouse domain for $r<\infty$.  We obtained that the emergence of this wandering component $\cal C$ has positive order (and so super polynomial), and this locally densely (with codimension $\le 5$)  in any topology $2\le r\le \infty$ and $r=\omega$. 

In this sense, Theorem \ref{theorem B} and its Corollary \ref{coro B}  confirm the main conjecture of \cite{beremer} on typicality of high emergence in the category of $C^r$-surface diffeomorphisms,  $2\le r\le \infty$, $r=\omega$ and even in the category of real polynomial automorphisms of $\R^2$, with density of codimension $\le 5$ as a weak version of typicality.

\subsection{Outline of the proof and organization of the manuscript}\label{intro.4}
Theorem \ref{Thepremain4application} follows from a general result (\cref{main them Fp}) on a geometric model. We are going to briefly describe the mechanism behind this result. 
The proofs that \cref{main them Fp} implies  Theorem \ref{Thepremain4application} which  implies Theorems \ref{main wandering} and \ref{theorem B}, are done in \cref{sec: examples}.

Let us first fix a few notations. Given two sequences $(a_j)_{j}$ and $(b_j)_j$ of non-zero real numbers:
\begin{itemize} 
\item we say that  $a_j$ is small when $j$ is large and write $a_j = o(1)$  if $\lim_{j\to \infty} a_j= 0$,
\item we say that  $a_j$ is small compared to $b_j$ when $j$ is large and write $a_j = o(b_j)$  if $\lim_{j\to \infty} \frac{a_j}{b_j}= 0$,
 \item we say that    $a_j$ and $b_j$ are equivalent when $j$ is large and write $a_j \sim b_j$   if  $\lim_{j\to \infty}\frac{a_j}{b_j}=1$.
 \end{itemize}

The starting point of the proof is to work with a wild hyperbolic basic set $\Lambda$. This is a hyperbolic set such that for any perturbation of the dynamics, there is a local unstable manifold which displays a quadratic tangency with a stable manifold.
For such dynamics there are sequences of points $(P_i)_{i\ge 0}$ in $\Lambda^\N$ which define   chains of heteroclinic quadratic tangencies: the local unstable manifold $W^u_{loc} (P_i) $ of $P_i$ is sent by $f$  tangent to the local stable manifold $W^s_{loc}(P_{i+1})$ and the tangency is quadratic. For the sake of simplicity, let us begin by assuming that each $P_i$ is periodic of some period $q_i$ and the dynamics $f^{q_i}$ near  $P_i$  is affine with diagonal linear part:
\[L_i: P_i+(x, y)\mapsto P_i+ \left(\frac x{\sigma_i}, \lambda_i\cdot  y\right)\quad \text{with } 0< \sigma_i ,  \lambda_i <1 \, .\]

In particular the local stable manifold of $P_i$  is vertical and the local unstable manifold of $P_i$ is horizontal.  Now we consider the tangencies  slightly unfolded such that $f$ folds the local unstable manifold  $W^u_{loc} (P_i) $ close to a point $T_i$ with $L_{i+1} \circ f(T_i)$  close to  $T_{i+1}$ (see \cref{figure1}). A toy model of the dynamics nearby a homoclinic tangency is given by a Hénon map:
\[f: T_i+(x, y) \mapsto  {L_{i+1}^{-1}(T_{i+1}) + }  (x^2-y/2 +a_i,bx+c_i) \; , \; \text{with }\quad (a_i, c_i)= {f(T_i) - L_{i+1}^{-1}(T_{i+1})} \, .  \]
\begin{figure}[h]
\centering
\includegraphics[width=13cm]{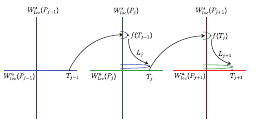}
\caption {Toy model for the heteroclinic tangencies chain leading to a wandering domain.}\label{figure1}
\end{figure}

Then we have:
\[L_{i+1}\circ f: T_i+ (x, y)\mapsto  T_{i+1}+ \left( \frac{x^2-{y/2} +a_i}{\sigma_{i+1} },\lambda_{i+1} (bx+c_i)\right)\, . \] 
 In this toy model, if there exists a sequence of neighborhoods $(B_i)_i$ of $(T_i)_i$  such that $\diam B_i\to 0$ and 
 {$L_{i+1} \circ f(B_i)\subset B_{i+1}$}, then  the point $T_0$ belongs to a stable domain. 
To find the necessary conditions on $a_i, c_i, \sigma_i, \lambda_i$ to achieve this, we  use a variation of the rescaling formula of \cite[Theorem D]{berger2018zoology}.  We define the rescaling factors and maps by:
\[\gamma_i= \prod_{j=1}^\infty \sigma_{i+j}^{2^{-j}}\qand \phi_i:(X,Y)\mapsto T_i+ (\gamma_i X, \gamma_{i}^2 Y) \, . \]
We assume that $\gamma_i>0$. Observe that $\gamma_i^2= \sigma_{i+1} \gamma_{i+1}$.  We now consider the renormalized map $\cal R_i f(X,Y):= \phi_{i+1}^{-1}  \circ L_{i+1}\circ f\circ \phi_i(X,Y)$. A simple computation gives:
 \[\cal R_i f(X,Y)=  \left(X^2 -\frac Y2+ \frac{a_i}{\gamma_i^2}, \frac{\lambda_{i+1}}{\gamma_{i+1}^2} (b  \gamma_i  X+c_i)\right)\; .\]
This is a quadratic Hénon map. Thus the idea is to state conditions implying that its  coefficients are small, so that the image by $\phi_i$ of the ball of radius $1/3$ centered at $0$ is contracted and sent into {the image by $\phi_{i+1}$ of the same ball}  by these iterations (since $1/3^2+1/(2\times 3)<1/3$). Namely our conditions express as:
\begin{enumerate}[$(C_1)$]
\item $\prod_{j\ge 0} \sigma_j^{2^{-j}}>0$ or equivalently 
$\sum_{j\ge 0}2^{-j} \log \sigma_j>-\infty$.
\item  $a_i= o(\gamma_i^2)$ when $i$ is large; this implies that the first  coordinate of $\cal R_i f$ is close to $X^2-Y/2$.
\item $\gamma_i\to 0$ and $\lambda_{i+1}=o(\gamma_{i+1}^2)$;  this ensures that the second coordinate of $\cal R_i f$ is small.
\end{enumerate}
To go from this toy model to actual surface diffeomorphisms (including polynomial automorphisms), we first work with a more general model where (in particular) the mapping is not   necessarily linearizable at its periodic points. This model of geometrical and combinatorial type is a variation of the one used in \cite{berger2018zoology}. The precise definition is given in \cref{section System of type AC} under the name of systems of type $(\sA, \sC)$. A basic and simple  example of such system is depicted \cref{example geo model} \cpageref{example geo model}.  Then an analogue of the above conditions $(C_1)$-$(C_2)$-$(C_3)$  is shown to imply the conclusions of Theorem \ref{Thepremain4application} in \cref{propreel,} (real case) and \cref{propcomplex} (complex case). Their proofs occupy  the whole \cref{section Implicit representations of hyperbolic transformation}. In \cref{sect: Implicit representations and initial bounds} we recall the notion of the implicit representations for hyperbolic transformations of Palis-Yoccoz \cite{PY01} and introduce a real-analytic generalization. These representations allow us to obtain sharp uniform  distortion bounds, without assuming any linearization (i.e. non-resonance) conditions, which is very crucial for our problem. In \cref{sectiondefBj}, we show that the assumptions of \cref{propreel,propcomplex} imply a normal form for the folding maps which is 'uniformly'  close to be Hénon-like  (and more precisely close to the assumptions of \cite[Theorem D and Remark 3.1]{berger2018zoology}). In \cref{sectiontildeBj}, we prove  \cref{propreel,propcomplex}. The main  point is to show that the domain $B_j$ is well included in the boxes where the implicit representations are defined and where the folding maps are defined. To handle the complex analytic case which is presently not included in \cite[Theorem D]{berger2018zoology},  we use the Cauchy inequality as a shortcut to obtain the complex distortion bounds. Details are given  in \cref{Computational proof of uniform bound on implicit representation}.\medskip

At this point, it is not straightforward at all to see that the assumptions of \cref{propreel,propcomplex} are satisfied for infinitely smooth surface diffeomorphisms, let alone  for polynomial automorphisms of $\R^2$. This context requires to   work with a fixed $d$-parameter family of diffeomorphisms and consider perturbations only along such a family (contrarily to \cite{KS17} where bump functions are used to perturb the dynamics). Thus in \cref{sec Unfolding of wild type}, we consider families $(F_p)_{p\in \R^d}$ of systems of type 
$(\sA, \sC)$. We ask that at every parameter $p$, there are $d$ 
pairs of points $(Q_i^u,Q^s_{i+1} )_{i}$ in $\Lambda$ such that 
$F_p$ maps 
the local unstable manifold $W^u_{loc} (Q^u_i) $ of $Q^u_i$   tangent to the local stable manifold $W^s_{loc}(Q^s_{i+1})$ and the tangency is quadratic. Then to apply \cref{propreel,propcomplex} as above, we shadow each  $(Q^s_i, Q_i^u)_i$ by a periodic orbit $P_i$ such that if the eigenvalues of $P_i$ are $(1/ \sigma_i,  \lambda_i)$, the curves $W^s_{loc}(Q^s_{i})$ and $W^s_{loc}(P_{i})$ are $O(\sigma_i)$-close whereas 
$W^u_{loc}(Q^u_{i})$ and $W^u_{loc}(P_{i})$ are $O(\lambda_i)$-close. To satisfy $(C_1)$, we take:
\[\sigma_i\asymp  \delta^{\beta^i}\; ,\quad \text{ with }\beta \in (1,2)\text{ and }\delta\text{  small.}\]
This implies that 
\[\gamma_i\asymp  \prod_{j\ge 1} \sigma_{i}^{  (\beta/2)^j}=\sigma_{i}^{\beta/(2-\beta )}\; .\]

 So if we do not perturb the parameter,  a priori $a_i$ is of the order of $\sigma_{i+1}+\lambda_i$, which is  often large compared to $\gamma_{i}^2=\sigma_{i}^{2\beta/(2-\beta )} $ (since $\gamma_i\to 0$). So we have to perturb the parameter $p$ to obtain a wandering stable component (otherwise we would have found an open set of mappings with a wandering Fatou component!). But then a new difficulty  appears:  when we perturb the parameter $p$ so that the $i^{th}$- homoclinic tangency is neatly unfolded, the others are possibly badly unfolded by a length $\sigma_{i+1}+\lambda_i>\sigma_{i+1}$. We recall that all $a_j$ must be small compared to $\gamma_j^2$.  So if the points $(P_j)_j$ are inductively constructed,  for $j<i$ we must have at step $i$, for every $j<i$:
 \[ \sigma_{i+1}=o(\gamma_j^2)\; .\]
 Note that $\gamma_j^2=\gamma_{j+1}\cdot \sigma_{j+1}$. So this is equivalent to ask for:
   \[ \sigma_{i+1}=o(\gamma_{j+1}\cdot \sigma_{j+1})=o(\sigma_{j+1}^{2/(2-\beta)})\; .\]
And with $k=i-j$, we have $\sigma_{i+1}=\sigma^{\beta^{k}}_{j+1}$. So this is equivalent to ask:
\[\beta^{k}>2/(2-\beta)\]
This inequality does not have any solution $\beta\in (1,2)$  for $k=0,1, 2, 3$, but $\beta=3/2$ is a solution of this inequality for every $k\ge 4$. One \emph{main new trick of this work} is to work 
with a several parameters family: while we unfold the $i^{th}$ tangency to plug it at a neat position using one parameter, we use 3 other parameters to maintain the $(i-1)^{th}, (i-2)^{th}, (i-3)^{th}$ folds in their neat position (see \cref{figure2}).  But there is an extra difficulty. Once we have perturbed the parameter in this way, $W^u_{loc} (Q_{i+1})$ might be not anymore tangent to the stable local manifold  $ W^s_{loc} (Q_{i+2})$. Even worst, $W^u_{loc} (Q_{i+1})$ might be $\sigma_{i+1}$-far from being tangent to any stable manifold. So we need to work with  a $5$-parameter family. The fifth  parameter is  used to restore this heteroclinic tangency. 
This explains why all our main results on wandering stable components are stated for $5$-parameters families.  To handle such an argument we will need five groups of persistent homoclinic tangencies which unfold non-degenerately at every parameter. These are assumptions $(\bf H_1)$ and $(\bf H_2)$ of \cref{nice unfolding} of \emph{unfolding of wild type $(\sA, \sC)$}, with $\Card \sC=5$. 
\begin{figure}[h]
\centering
\includegraphics[width=15cm]{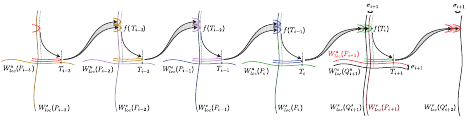}
\caption {This figure depicts the parameter selection; the motion blur reflects how the dynamics varies when the parameter is changed at each inductive step. 
We assume by induction the existence of a chain of periodic points $(P_j)_{j<i}$ such that   $W^u_{loc} (P_j)$ is in nice position w.r.t. $W^s_{loc}(P_{j+1})$   (i.e. satisfying $(C_2)$),  
 $W^u_{loc} (P_i)$ is tangent to $W^s_{loc}(Q^s_{i+1})$ and   $W^u_{loc}(Q^u_{i+1})$ is tangent to the stable lamination of $\Lambda$.  
We consider a periodic point $P_{i+1}$ such that $W^s_{loc}(P_{i+1})$  and $W^u_{loc}(P_{i+1})$ are $\sigma_{i+1}$-close to $W^s_{loc}(Q^s_{i+1})$ and $W^u_{loc}(Q^u_{i+1})$. We handle a $\sigma_{i+1}$-small perturbation of two parameters such that    $W^u_{loc} (P_i)$ is in nice position w.r.t. $W^s_{loc}(P_{i+1})$   (i.e. satisfying $(C_2)$) and   $W^u_{loc} (P_{i+1})$ is tangent to the stable lamination of $\Lambda$.
Then we use three others parameters in order that  $W^u_{loc} (P_j)$ is in nice position w.r.t. $W^s_{loc}(P_{j+1})$ for every $i-3\le j\le i-1$.  We need a fifth parameter to make $W^u_{loc} (P_{i+1})$ tangent to a leaf $W^s(Q^s_{i+2})$ of the stable lamination of $\Lambda$. 
By construction, the curve $W^u_{loc} (P_j)$ is in nice position w.r.t. $W^s_{loc}(P_{j+1})$  for every $j\le i-4$ since  $\sigma_{i+1}=o(\gamma_j ^2)$.}  \label{figure2} \end{figure}

To satisfy $(C_3)$ we must also have that $\lambda_j=o(\gamma_j^2)$. Using that $\gamma_j\asymp \sigma_{j}^{\beta/(2-\beta)}=\sigma_{j}^{3}$, this condition is equivalent to  $(\lambda_j/\sigma_j)\cdot \sigma_{j}^{-5} =o(1)$. Actually $\lambda_j/\sigma_j$ is the Jacobian determinant along the orbit of $P_j$  while $\sigma_{j}^{-5}$ is approximately the fifth power of the differential of the dynamics along the orbit of $P_j$. Hence in order to satisfy this equation, we will require the following moderate dissipativeness hypothesis in \cref{sec Unfolding of wild type}: 
\[\|Df\| \cdot  \|\mathrm{det}\, Df\|^\epsilon < 1\quad \text{with } \epsilon =\frac1{10}.\]
Let us now explain a tricky point of the parameter selection. The number $a_j$ is a function of $T_{j+1}$ which is itself a function of the manifold $W^s(P_{j+2})$. However this manifold is not yet defined by the induction, and depends on the parameter which is moving at the step in progress. So we will need to define the point $T_{j+1}$ independently of $W^s(P_{j+2})$ in a very sharp way. This will be done using \emph{an adapted  family $(\pi_p)_p$ of projections}, which is basically a $C^1$-extension of the stable lamination which is equivariant by the dynamics nearby the set $\Lambda$.  Such a family of projections is actually part of the structure of unfolding of wild type $(\sA, \sC)$. In \cref{sec:Standard results from hyperbolic theory}, we show the existence of such a family of projections given any hyperbolic set for a  $C^2$-surface diffeomorphism, following   classical techniques.

The main technical theorem in the paper  is \cref{main them Fp}. It asserts that given a moderately dissipative,  unfolding $(F_p, \pi_p)$ of wild type $(\sA,\sC)$ with $\Card \sC=5$, there is a dense set of parameters at which there is a stable component $\cal C$ satisfying the conclusions of \cref{propreel,propcomplex} (wandering combinatorics) and moreover the points in $\cal C$ satisfy conclusions (1) and (2) of Theorem \ref{Thepremain4application} ($\omega$-limit set of $x$ containing $\Lambda$ and large emergence).  
Its proof occupies  \cref{parameterselection}. We start in \cref{estimees} by showing some uniform bounds  on the parameter family  for the iterations near a hyperbolic set. Then in \cref{sectionfolding}, we give uniform bounds  on the parameter family  for the normal form nearby the folding map. In \cref{Infinite chain of nearly heteroclinic tangencies}, we proceed to the parameter selection using the argument explained above. {In \cref{sectionverifie}, we prove the first part of \cref{main them Fp} by showing the existence of a wandering stable component  by applying \cref{propreel,propcomplex}.} Finally in \cref{section emergence}, we will use some room in our estimates to make the stable component having a large emergence  and accumulate onto all the hyperbolic set $\Lambda$. \medskip

At this stage what remains to do is to show that the families of Theorems \ref{main wandering}, \ref{theorem B} and \ref{Thepremain4application} can be put in the framework of \cref{main them Fp}. We say that a dynamics satisfies the \emph{geometric model} if an iterate of it  leaves invariant a moderately   dissipative wild unfolding with $\Card \sC=5$.  \cref{consequence} of  \cref{main them Fp} then asserts  that such a  system  displays a wandering stable component which accumulates onto the embedding of $\Lambda$ and has stretched exponential emergence. Moreover if the embedding and the dynamics are real analytic this wandering domain is the real trace of a wandering Fatou component. 

In \cref{sec: examples}, we show that the families of Theorems \ref{main wandering}, \ref{theorem B} and \ref{Thepremain4application} satisfy the geometric model. We start in \cref{example for the model} with a proof that the basic and simple example depicted in 
\cref{example geo model} defines a moderately dissipative wild unfolding. 
We remark that this simple example already enables to find a family of polynomial automorphisms which display a wandering Fatou component, of some controlled degree.  In \cref{section Density of wandering domain in the Newhouse domain} we show that Theorem \ref{Thepremain4application} implies Theorems \ref{main wandering} and \ref{theorem B}. In order to do so we recall some celebrated results of Newhouse \cite{Ne79}, and its extension  by \cite{Kr92,HKY93}. These results imply that in the unfolding of a homoclinic tangency there is a wild hyperbolic set with uncountably many tangencies. In \cref{section:premain4application} we show that under the assumptions of Theorem \ref{Thepremain4application} the  geometric model occurs. The techniques of the proof are classical but rely on rather sophisticated  techniques from real bifurcations theory and uniformly hyperbolic dynamics.

Let us finally emphasize that \cref{sec: examples,sec:Standard results from hyperbolic theory} using classical tools from real uniform hyperbolicity theory   are completely   independent 
from   \cref{section Implicit representations of hyperbolic transformation,parameterselection,proof prop_compo_starc,Computational proof of uniform bound on implicit representation}  using tools  from both real and complex analysis. All these are independent of \cref{sec:emergence} which is devoted to the proof of \cref{horseshoe emerge}. \medskip

\thanks{\emph{The authors are thankful to Romain Dujardin for his comments on the introduction of this paper. We are grateful to the referees for their deep reading and their valuable suggestions.}}

\section{The geometric model}
The geometric model is devoted to offer a nice framework for showing the main theorem on the existence of a wandering stable component and which applies to the following situation: a wild horseshoe which is strongly dissipative and displays five robust homoclinic tangencies which unfold non-degenerately for a 5-parameter family, as we will see in \cref{main4application}.
Actually, we will show that the geometric model implies the density of the parameters at which there is a wandering stable component, in \cref{Main theorem}.  

Let us define the functional spaces involved in this model. Let $K\subset \R^n$ and $K' \subset \R^m$ be two compact subsets, with $n, m\ge 1$. For $r\in [1, \infty]\cup\{\omega\}$,  a map $f: K\to K'$ is of class $C^r$ if it can be  extended to a $C^r$-map from an open neighborhood of $K$ to $\R^m$. It is a $C^r$-diffeomorphism if $f$ is bijective and $f^{-1}$ is of class $C^r$. Then $n=m$. We recall that $C^\omega$ is the class of real analytic maps. Likewise, given two compact  subsets  $K\subset \C^n$ and $K' \subset \C^m$, a map $f: K\to K'$ is holomorphic  if it can be  extended to a holomorphic  map from an open neighborhood of $K$ to $\C^m$. It is a biholomorphism if $f$ is bijective and $f^{-1}$ is holomorphic (and again $n=m$).\label{def cr}

\subsection{System of type $(\sA,\sC)$}\label{section System of type AC}
In this subsection, we introduce the notion of hyperbolic map of type $\sA$ through the definition of hyperbolic transformations. One should think about a hyperbolic map of type $\sA$ as a 
horseshoe $\Lambda$  and about hyperbolic transformations as the restriction of the dynamics to a neighborhood of a rectangle of a Markov partition (see \cref{Horseshoe2model,horseshoeHp}). 

\medskip

Let $I:=[-1,1]$ with topological boundary $\partial I:=\{-1,1\}$. 
We fix for this subsection $\rho>0$ and we put $\tilde I:= I + i \cdot [-\rho,\rho] $ with topological boundary $\partial \tilde I\subset \C$. \index{$I$, $\tilde I$}
Put $Y^\se:= I^2$, $\tilde Y^\se:= \tilde I^2\subset \C^2$ and:\index{$Y^\se$, $\partial^uY^\se$, $\partial^s Y^\se$, $\tilde Y^\se$}
\[\partial^sY^\se:=(\partial I)\times I\quad , \quad \partial^uY^\se:=I\times \partial I\quad , \quad \partial^s\tilde Y^\se:=(\partial \tilde  I)\times \tilde I
\quad , \quad \partial^u\tilde  Y^\se:=\tilde I\times  \partial \tilde I
\; .\]
The sets $Y^\se$ and $\tilde Y^\se$ will be associated to the transformations $F^\se$ equal to the identity. The definition of hyperbolic transformations  depends on fixed $\theta>0$ and $\lambda>1$. For convenience, fix:\index{$\theta=1/2$}
\index{$\lambda=2$}
\[ \theta= \frac12\qand \lambda =2\; .\]
\begin{definition}[Box] \label{def box}\index{Box}A \emph{box} is a subset $Y$ of $\R^2$ which is diffeomorphic to $Y^\se$ and of the form:
$$Y= \{(x,y)\in Y^\se: \phi^{-}(y)\le x\le \phi^+(y)\}\; ,$$
where $\phi^{-}$ and $\phi^{+}$ are $C^1$-functions on $I$ such that $\phi^{-}<\phi^{+}$ with
$\| \phi^\pm\|_{C^0}\le 1$ and $\|D \phi^\pm\|_{C^0}<\theta$, for $\pm\in \{-,+\}$. We put $\partial^u Y =\partial^u Y^\se \cap Y$ and $\partial^s Y= cl( \partial Y\setminus \partial^u Y) $.
\end{definition}
We notice that the set $Y^\se$ is a box. Also in general, a box $Y$ satisfies:
\[\partial^u Y =\{(x,y)\in I\times \partial I: \phi^{-}(y)\le x\le \phi^+(y)\}\qand \partial^s Y= \{(x,y)\in I^2: x\in \{\phi^{-}(y), \phi^+(y)\}\}\; .\]
For analytic tranformations we will consider the following complex extensions of boxes.

\begin{definition}[$C^\omega_\rho$-box] \label{defcomplexbox}\index{$C^\omega_\rho$-box}
A \emph{$C^\omega_\rho$-box} $\tilde Y$ is a subset of $\tilde Y^\se$ such that  $\tilde Y\cap Y^\se$ is a box and $\tilde Y$ is the range of a biholomorphism of the form $\zeta : (z,w)\in \tilde I^2 \mapsto (\cal Z(z, w), w)\in \tilde Y$ satisfying $|\partial_w \cal Z|<\theta$. Let $\partial^u \tilde Y =\partial^u \tilde Y^\se \cap \tilde Y$ and $\partial^s \tilde Y= cl( \tilde Y\setminus \partial^u \tilde Y) $.
\end{definition}
Similarly we notice that $\tilde Y^\se$ is a $C^\omega_\rho$-box. Also in general, a $C^\omega_\rho$-box $\tilde Y$ satisfies:
\[\partial^ u \tilde Y:= \zeta(\tilde I\times \partial \tilde I)\qand \partial^s \tilde Y:= \zeta( (\partial \tilde I)\times \tilde I)\; .\]
The definition of hyperbolic transformation will involve cones.
\begin{definition}[Cones] We define the two following cones $\chi_{h}$ and $\chi_{v}$: \index{$\chi_{h}$ and $\chi_{v}$}
$$\chi_h:=\{(u_x,u_y)\in \R^2: | u_y|< \theta\cdot |u_x| \}\qand \chi_v:=\{(u_x,u_y)\in \R^2: |u_x| < \theta \cdot | u_y|\}\; .$$

The real cones $\chi_{h}$ and $\chi_{v}$ admit a canonical extension to $\mathbb{C}^{2}$: \index{$\chi_{h}^{\C}$ and $\chi_{v}^{\C}$}
$$\tilde \chi_{h} = \{(u_z,u_w) \in \mathbb{C}^{2}:|u_w| < \theta \cdot |u_z|\}\qand \tilde \chi_{v} = \{(u_z,u_w) \in \mathbb{C}^{2}:|u_z| < \theta\cdot |u_w|\}\; .$$
\end{definition}

\begin{definition} \label{defpiece}
A \emph{hyperbolic transformation} $(Y, F)$ is the data of a box $Y$ and a $C^2$-diffeomorphism from $Y$ onto its image in $Y^\se$ such that either $Y=Y^\se$ and $F$ is the identity  or:
\begin{enumerate}
\item $F( Y)\subset Y^\se$ and $F(\partial^s Y)\subset \partial^s Y^\se$ whereas  $Y\cap \partial^s Y^\se=\emptyset\qand F(Y) \cap \partial^u Y^\se =\emptyset \, .$
\item For every $z\in Y$,
every non-zero vector in the complement of $\chi_v$  is sent into $\chi_h$ by $D_zF$ and  has its first coordinate which is more than $\lambda$-expanded by $D_{z}F$.
\item For every $z\in F(Y)$, every non-zero vector in the complement of $\chi_h$ is sent into $\chi_v$ by $D_zF^{-1}$ and has its second coordinate which is more than $\lambda$-expanded by $D_{z}F^{-1}$.
\end{enumerate}
\end{definition}
\begin{figure}[h]
\centering
\includegraphics[width=8.5cm]{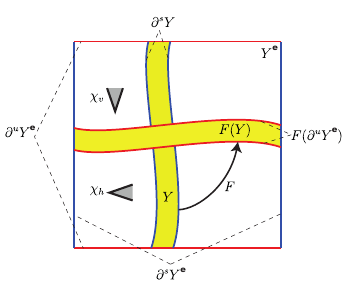}
\caption {A hyperbolic transformation.}\end{figure}
Here is the complex counterpart of the latter notion.

\begin{definition} \label{defcomplexpiece} A \emph{$C^\omega_\rho$-hyperbolic transformation} $(\tilde Y, F)$ is the data of a $C^\omega_\rho$-box $\tilde Y$ and a biholomorphism $F$ from $\tilde Y$ onto its image in $\tilde Y^\se$ such that 
either $(\tilde Y, F)=(\tilde Y^\se,id)$  or:
\begin{enumerate} 
\item $F(\tilde  Y)\subset \tilde Y^\se$ and $F(\partial^s  \tilde Y)\subset \partial^s \tilde Y^\se$ whereas  $\tilde Y\cap \partial^s \tilde Y^\se=\emptyset$ and $F(\tilde Y) \cap \partial^u \tilde Y^\se =\emptyset \, .$
\item For every $z\in \tilde Y$,
every non-zero vector in the complement of $\tilde \chi_v$ is sent into $\tilde \chi_h$ by $D_zF$ and has its first coordinate's modulus which is more than $\lambda$-expanded by $D_z F$.
\item For every $z\in F(\tilde Y)$, every non-zero vector in the complement of $\tilde \chi_h$ is sent into $\tilde \chi_v$ by $ D_z F ^{-1}$ and has its second coordinate  modulus which is more than $\lambda$-expanded by $  D_z F  ^{-1}$.
\item  The  restriction of  $F$ to  the box $Y:=\tilde Y\cap \R^2$  is a real analytic  hyperbolic transformation.
\end{enumerate}
We denote also by $F$ its restriction $F|Y$.
The pair $(\tilde Y, F)$ is called the \emph{complex extension or $C^\omega_\rho$-extension} of~$(Y,F)$. \index{Complex extension, $C^\omega_\rho$-extension}
\end{definition}

We define the following operation on hyperbolic transformations. 
\begin{definition}[$\star$-product]\label{defstarhenon}
Given two  hyperbolic transformations $ (Y,F)$ and $(Y',F')$, we define:
\[(Y,F)\star (Y',F'):=(Y \cap F^{-1}(Y'), F'\circ F)\; .\]
\label{defstarhenonc}  Given two  $C^\omega_\rho$-hyperbolic transformations
$ (\tilde Y,F)$ and $(\tilde Y',F')$, we define:
\[(\tilde Y,F)\star (\tilde Y',F'):=(\tilde Y \cap F^{-1}(\tilde Y'), F'\circ F)\; .\]
\end{definition}
The $\star$-product is a binary operation on the set of  hyperbolic transformations:
\begin{proposition}\label{prop_compo_starc}\label{prop_compo_star} \label{rema prop_compo_star}
The $\star$-product $(Y,F)\star (Y',F')$ of any pair of hyperbolic transformations is a  hyperbolic transformation.
The $\star$-product $(\tilde Y,F)\star (\tilde Y',F')$ of any pair of $C^\omega_\rho$-hyperbolic transformations is a $C^\omega_\rho$-hyperbolic transformation, moreover $(\tilde Y,F)\star (\tilde Y',F')$ is a complex extension of $(Y,F)\star (Y',F')$  if $(\tilde Y,F) $ and  $(\tilde Y',F')$ are respective complex extensions of $(Y,F)$ and $(Y',F')$.
\end{proposition}
\begin{proof} In \cite[\textsection 3.2.1]{PY01} it is noticed that the proof is immediate in the real case. The proof for the complex case is postponed to  \cref{proof prop_compo_starc}.
\end{proof}

The $\star$-product is associative. In particular, it will not be necessary to use brackets while taking $\star$-products of more than two factors. Also for every hyperbolic transformation $(Y,F)$ it holds:
\[(Y,F)\star (Y^\se, F^\se)= (Y^\se, F^\se)\star (Y,F)=(Y,F)\; ,\]
where $F^\se$ is the identity of $Y^\se$. The same occurs for $C^\omega_\rho$-hyperbolic transformation. This implies:
\begin{proposition} The set of (resp. $C^\omega_\rho$-) hyperbolic transformations endowed with the operation $\star$ is a monoid with neutral element $(Y^{\se}, F^\se)$ (resp. $(\tilde Y^{\se}, F^\se))$.\end{proposition}

Let $(\sV, \sA)$ be an oriented graph where $\sV$ is a finite set of vertexes and $\sA$ a finite set of arrows. The maps $\si : \sA\to \sV$ and $\st : \sA \to \sV$ 
 associate to each arrow $\sa \in \sA$, its origin vertex $\si(\sa)$ and its target vertex $\st(\sa)\in \sV$.  Let $\sB\subset \sA^2$ be the set of pairs $(\sa, \sa')$ of arrows such that $ \st(\sa)=\si(\sa')$. Then we will always assume that the  subshift  of finite type  defined by $(\sA, \sB)$ is transitive and  has positive entropy. This is equivalent to say that the graph is  strongly connected and there are two different paths with the same length and the same origin and target vertexes. \index{$\sB$}
 
The following is inspired from the notion of regular Cantor set \cite{GY10} and of the Markovian structure of \cite{PY09}. 
\begin{definition}\label{hyp type A}\index{Hyperbolic map of type $\sA$}\index{$C^\omega_\rho$-hyperbolic map of type $\sA$}\index{$F^\sA$} A \emph{hyperbolic map   of type $\sA$}  is a  map of the form:
\[F^\sA: (z,\si(\sa))\in \bigsqcup_{\sa\in \sA}  Y^{\sa}\times 
\{\si(\sa)\}\mapsto (F^\sa(z),\st(\sa))\in   Y^\se\times \sV\]
where  $(Y^\sa, F^\sa)_{\sa \in \sA}$ is a finite family of hyperbolic transformations $(Y^\sa, F^\sa)$ satisfying that $Y^\sa$ is disjoint from $Y^{\sa'}$ for every $\sa\neq \sa'\in \sA$ such that $\si(\sa)=\si(\sa')$.
 
The map $F^\sA$ is \emph{$C^\omega_\rho$-hyperbolic}  if 
there exists a family $(\tilde Y^\sa\times \{\si(\sa)\})_{\sa\in \sA}$ of complex extensions $(\tilde Y^\sa, F^\sa)$ of   $(Y^\sa, F^\sa)$ satisfying that $\tilde Y^\sa$ is disjoint from $\tilde Y^{\sa'}$ for every $\sa\neq \sa'\in \sA$ such that $\si(\sa)=\si(\sa')$. We put:
 \[D(\sA):= \bigsqcup_{\sa\in \sA}  Y^{\sa}\times 
\{\si(\sa)\}\qand \text{in the $C^\omega_\rho$- case  }\tilde D(\sA):= \bigsqcup_{\sa\in \sA}  \tilde Y^{\sa}\times 
\{\si(\sa)\}\; .\]
 \end{definition}
We notice that a hyperbolic map of type $\sA$ is a local $C^2$-diffeomorphism.
\begin{example}\label{simple example of hyperbolic map of type A}  Let $(\sV, \sA)$ be the graph with  a single vertex $\sV=\{\so\}$ and $N\ge 1$ distinct arrows $\sA=\{\sa_1,\cdots, \sa_N\}$ with both origin and target vertexes equal to $\so$.   Let $\delta>0$ be small compared to $1/N$. For every $1\le j\le N$, put:
\[ I_j:= \left[\frac 2N (j-1)-1+\delta^2, \frac 2N j-1-\delta^2\right]\; .\]
For $1 \le j \le N$, let $C_j$ be the affine, orientation preserving map from $I$ onto $I_j$, let:
\[Y^{\sa_j}:=  I_j \times I\qand F^{\sa_j}:= (x,y)\in  Y^{\sa_j}\mapsto (C_j^{-1}(x), \sqrt{\delta} \cdot C_j(y))\; .\]
We notice that $(Y^{\sa_j}, F^{\sa_j})$ is a hyperbolic transformation for every $1\le j\le N$.  The domains $(Y^{\sa_j})_{1\le j\le N}$ are disjoint since the intervals $(I_j)_j$ are disjoint. Also the ranges of $(F^{\sa_j})_j$ are  $(I\times \sqrt{\delta} \cdot I_j)_j$  which are disjoint from the one other. Thus \emph{the following diffeomorphism is hyperbolic of type~$\sA$}:
\[F^\sA: (z,\so) \in \bigsqcup_{\sa\in \sA} Y^\sa\times \{\so\} \mapsto (F^\sa(z),\so)\text{  if }z\in Y^\sa\; .\]  
This hyperbolic map with $N=6$ is depicted on \cref{example geo model}.
\end{example}
    
We are going to associate geometric and dynamical objects to the following symbolic objects:
\begin{definition}\index{$\sA$ and $\sA^*$}  \index{$\si$ and $\st$}
Let $\sA^*\subset \sA^{(\N)}$ be the set of \emph{admissible finite words} $\sc:=\sa_1 \cdots \sa_k\in \sA^k$:  for each $j< k$ we have  $(\sa_j,\sa_{j+1})\in \sB$. 
We denote by $\se$ the empty word in $\sA^*$.  The \emph{origin and target vertexes} of $\sc$ are defined by the maps:
\[\si: \sa_1\cdots \sa_k\in \sA^*\setminus \{\se\}\mapsto  \si(\sa_1)\in \sV
\qand   
\st: \sa_1\cdots \sa_k\in \sA^*\setminus \{\se\}\mapsto  \st(\sa_k)\in \sV\]
If the pair $(\sc, \sc')\in (\sA^*)^2$ is admissible $(\st(\sc)=\si(\sc'))$, we denote by  $\sc\cdot \sc'\in \sA^*$ the concatenation of these two words.  The \emph{length} $|\sc|$ of $\sc\in \sA^*$ is its number of letter.  \index{Length}
\end{definition} 
There are three ways of taking limits of these admissible words: by both sides, the right side and the left side. This defines three sets.
\begin{definition}\index{$\arr \sA$, $\avv \sA$ and $\overleftrightarrow \sA$}
Let $\overleftrightarrow \sA$ be the set of sequences $(\sa_i)_{i\in \Z}\in \sA^{\Z}$ which are admissible: $(\sa_j,\sa_{j+1})\in \sB$ for every $j$; it is the space of the orbits of the aforementioned transitive subshift of finite type. 
 
Let $\avv \sA$ be the set of sequences $\ss:=(\sa_i)_{i\ge 0}\in \sA^{\N}$ such that   $(\sa_j,\sa_{j+1})\in \sB$ for every $j\ge 0$.

   Let $\arr \sA$ be the set of sequences $\su:=(\sa_i)_{i<0}\in \sA^{\Z_{-}^{*}}$ such that $(\sa_{-j-1},\sa_{-j})\in \sB$ for every $j>0$. 
\end{definition}
We endow the sets  $\arr \sA$,  $\avv \sA$ and   $\overleftrightarrow \sA$ with the topologies induced by the  product  topologies of $\sA^{\Z_-^*}$, $\sA^\N$ and  $\sA^\Z$    which are compact by Tychonoff's theorem.\medskip 

The map $\si $ can be canonically extended from $\sA^*$ to $\avv \sA$ by $\si((\sa_i)_{i\ge 0}):=\si (\sa_0)$.  We extend similarly $\st $ to $\arr \sA$ by  $\st((\sa_i)_{i< 0}):=\st (\sa_{-1})$.  
We extend canonically the concatenation operation to the sets $\sB^*$, $\avv \sB$ and  $\arr \sB$ of  pairs $(\sd, \sd')$ in respectively $(\sA^*)^2$,  $ \sA^*\times \avv \sA$ and  $ \arr \sA\times \sA^*$ such that $\st(\sd)=\si(\sd')$.    \index{$\sB^*$, $\avv \sB$ and $\arr \sB$} 
 For $\sw\in \sA^*$, we denote:
 \begin{multline}
     \sA^*\cdot \sw:= \{\sd\cdot \sw\in \sA^*: (\sd, \sw)\in \sB^*\}   \qand  \sw \cdot \avv \sA := \{\sw \cdot \ss \in \avv \sA: ( \sw, \ss)\in \avv \sB\}    \\ \qand    \arr \sA\cdot \sw:= \{\su\cdot \sw\in \arr \sA: (\su, \sw)\in \arr \sB\} \; .
      \end{multline}
  
Given a hyperbolic map of type $\sA$, by associativity of the $\star$-product, we denote for every $\sc:= \sa_1 \cdots \sa_k\in \sA^*$:
\begin{equation}\label{pour homo}
(Y^{\se}, F^\se) := (Y^{\se}, id)  \qand (Y^{\sc}, F^\sc) := (Y^{\sa_1}, F^{\sa_1})\star\cdots \star (Y^{\sa_k}, F^{\sa_k})\; 
\end{equation}
which is a hyperbolic transformation by  \cref{prop_compo_star}.
If $F$ is $C^\omega_\rho$-hyperbolic of type $\sA$, we denote similarly $(\tilde Y^{\se}, F^\se) := (\tilde Y^{\se}, id)$ and $(\tilde Y^{\sc}, F^\sc)$ the $\star$-product of the $C^\omega_\rho$-hyperbolic transformations associated to the letters of $\sc$.  

We recall that a hyperbolic basic set  is a hyperbolic compact set which is transitive and locally maximal.\index{Hyperbolic basic set}
\begin{proposition} \label{horseshoeHp} \index{$\Lambda$} The maximal invariant set $\Lambda=\bigcap_{n\in \Z} (F^\sA)^n (Y^\se\times \sV)$  is a hyperbolic basic set.
\end{proposition}
\begin{proof} First note that $\Lambda$ is compact since it is an intersection of compact sets. Furthermore,  it is locally maximal since:
\[\Lambda\subset F^\sA (Y^\se\times \sV)\cap (F^\sA)^{-1}(Y^\se\times \sV)=
  \bigcup_{(\sa, \sa')\in \sA^2: \si (\sa)= \st(\sa')  } Y^\sa\cap F^{\sa'}(Y^{\sa'})\times\{\si(\sa)\} \]
which is strictly included in $Y^\se\times \sV$ by property $(1)$ of  hyperbolic transformation's \cref{defpiece}. Furthermore it is hyperbolic by the cone properties $(2)$ and $(3)$ of  \cref{defpiece}.
 
The transitivity of $\Lambda$ follows from the one of the shift $\overleftrightarrow \Lambda$ defined by $(\sA, \sB)$ and the next remark. \end{proof}
\begin{remark}\label{def hbar} \index{$h$}  If $F^\sA$ is a diffeomorphism, the dynamics of $\Lambda$ is conjugate to the shift on $\overleftrightarrow \sA$. In general, the dynamics of the inverse limit 
$\overleftrightarrow \Lambda:=\{(z_i, \sv_i)_{i\in \Z}\in \Lambda^{\Z}: (z_{i+1}, \sv_{i+1})=F^\sA(z_{i},\sv_{i})\}$ of $\Lambda$ is conjugate to the shift on $\overleftrightarrow \sA$ via the map
\begin{equation*}
h: (z_i, \sv_i)_{i\in \Z}\in \overleftrightarrow \Lambda \mapsto (a_i)_{i \in \Z} \in \overleftrightarrow \sA \quad \text{with } \sa_i \in \sA\text{ such that }  z_i \in Y^{\sa_i}\text{ and }  \sv_i= \si(\sa_i) \; .\end{equation*}
\end{remark}

Conversely, we will see in \cref{Horseshoe2model} \cpageref{Horseshoe2model} that modulo  iterations, any $C^2$-horseshoe is conjugate to a hyperbolic map $F^\sA$ of a certain type $\sA$.

\begin{definition}\label{def su} Given a hyperbolic map of type $\sA$, for $\ss=(\sa_i)_{i\ge 0}\in \avv \sA$ and $\su=(\sa_i)_{i<0}\in \arr \sA$, we  denote: \index{$W^\ss$ and $W^\su$} 
\[W^\ss:=\bigcap_{i\ge 0} Y^{\sa_0\cdots \sa_i}\qand W^\su:=\bigcap_{i> 0} F^{\sa_{-i}\cdots \sa_{-1}}(Y^{\sa_{-i}\cdots \sa_{-1}})\, .\]
\end{definition}

The following is immediate:
\begin{proposition} \label{stablemanifold} For every $\ss \in\avv \sA$, the set $W^\ss\times \{\si(\ss)\}$ is a local stable manifold of the hyperbolic set $\Lambda$ with tangent spaces in $\chi_v$ and endpoints in the two different components of $ \partial^{u} Y^{\se} $. 

For every $\su \in\arr \sA$, the set $W^\su\times \{\st(\su)\}$ is a local unstable manifold of $\Lambda$ with tangent spaces in $\chi_h$ and with endpoints in the two different components of $\partial^{s} Y^{\se}$. 
\end{proposition}
By \cref{stablemanifold}, for every $\su\in \arr \sA$ and $\ss\in \avv \sA$, there are $\theta$-Lipschitz functions $ w^{\su}$ and $ w^{\ss}$ on $I$ whose graphs satisfy  $\mathrm{Graph}\, w^{\su}=W^\su$ and $^t
\mathrm{Graph}\, w^{\ss}=W^\ss$.   

\begin{proposition} \label{regularite C2 de w}\label{wcontinususswithoutparameter}
The maps $x \in I \mapsto w^{\su}(x) \in \mathbb{R}$ and $y \in  I \mapsto w^{\ss}(y) \in \mathbb{R}$ are of class $C^{2}$ and depend continuously on $\su$ and $\ss$ in the $C^2$-topology.
\end{proposition}
\begin{proof}This is the continuity of the local stable and unstable manifolds of the hyperbolic set $\Lambda$.  
\end{proof}

The following enables to consider quadratic homoclinic tangencies of the basic set  $\Lambda$. It is defined using  a finite set $\sC$ of (new) arrows $\boxdot$ between   vertexes of $\sV$.
 \begin{definition}\label{def folding map}\index{Folding map $F^\sC$of type $\sC$}\index{$\arr \sA_\boxdot$ and $\avv  \sA_{\boxdot}$} \index{$F^\boxdot$}
Given $\boxdot\in \sC$, a folding transformation $(Y^\boxdot, F^\boxdot)$ is  the data of a filled square $Y^\boxdot=I^\boxdot\times J^\boxdot$ included in  $Y^\se\setminus \partial^s  Y^\se$ and a $C^2$-diffeomorphism $F^\boxdot$ from $Y^\boxdot$ onto its image in  $Y^\se\setminus \partial^u  Y^\se$, which satisfies:
\begin{enumerate}
\item The following subset  is nonempty and clopen:
\[\arr \sA_{\boxdot}:= \{\su \in \arr \sA: \st(\su)= \si(\boxdot) , 
W^\su\cap Y^{\boxdot}\neq \emptyset \;  \&\;  
W^\su\cap \partial^u Y^{\boxdot}= \emptyset\}\quad \text{ with }\partial^uY^{\boxdot}:=  I^{\boxdot}\times \partial J^{\boxdot} .\]
\item    With $\avv  \sA_{\boxdot}:= \{\ss \in \avv \sA: \si(\ss)= \st(\boxdot)\}$,  
for any $\su \in \arr \sA_\boxdot$ and $\ss\in \avv \sA_{ \boxdot }$, if the curves $W^\su$ and $(F^\boxdot)^{-1}(W^\ss)$ are tangent, then the tangency is unique, quadratic and lies in $\mathrm{int}\;  Y_\boxdot$.
\end{enumerate}
The folding transformation is of class $C^\omega_\rho$ if  $F^\boxdot$ extends to a biholomorphism from $\tilde Y^\boxdot:= Y^\boxdot +[-i\cdot \rho, i\cdot \rho]^2$ onto its image in $\C^2$.\index{$\tilde Y^\boxdot$} \index{Folding map of class $C^\omega_\rho$}

\end{definition}
Here is the main dynamical object of this subsection.
\begin{definition} \label{def systeme sC}  A system  of type $(\sA, \sC)$ is a $C^2$-local diffeomorphism $F$ of the form:
\[ F : (x,\sv)\in D(\sA)\sqcup D(\sC)  \mapsto \left\{\begin{array}{cl}
 F^\sA(x, \sv)& \text{if } (x,\sv)\in  D(\sA),\\
 F^\sC(x, \sv)& \text{if } (x,\sv)\in  D(\sC)
\end{array}\right.\] 
\index{System of type  $(\sA, \sC)$}
 for  a hyperbolic map $F^\sA$  of type $\sA$ and the map $F^\sC $ is the disjoint union of  a finite family of folding transformations $(Y^\boxdot, F^\boxdot)_{\boxdot \in \sC}$ for $F^\sA$ by 
\[F^\sC: (z,\si(\boxdot))\in D(\sC) :=\bigsqcup_{\boxdot\in \sC}  Y^{\boxdot}\times 
\{\si(\boxdot)\}\mapsto (F^\boxdot(z),\st(\boxdot))\in   Y^\se\times \sV\]
and such that for every $\boxdot\neq \boxdot'\in \sC$, the sets   $\arr \sA_{\boxdot}$ and  $\arr \sA_{\boxdot'}$ are disjoint.
\end{definition}

{
\noindent\fbox{%
    \parbox{\textwidth}{%
    
\begin{center} \textbf{ Summary of symbolic notations} \end{center}

Any object in the category of symbols will be denoted in this \textsf{\textbf{style}}. 
Moreover the following letters will be used canonically:\\

$\sv$ denotes a \emph{vertex}; these form a set $\sV$  of vertices.\\
$\sa$ denotes an arrow  associated to a \emph{hyperbolic transformation}; these form a set $\sA$ of arrows.\\
$\boxdot$ denotes an arrow associated to a \emph{folding transformation}; these form a set $\sC$  of arrows. \\
$\sc$ denotes an admissible \emph{finite}  word in an alphabet of arrows. \\
$\ss$ denotes an admissible \emph{forward infinite}  word in an alphabet of arrows. \\
$\su$ denotes an admissible \emph{backward infinite}  word in an alphabet of arrows. \\
$\si( \cdot)$ and $\st(\cdot)$ are maps assigning to an arrow or  word its \emph{origin} and \emph{target} vertices, when defined. \\
Sets with a star, like $\sA^*$, are formed by \emph{all} the admissible \emph{finite}  words in the given alphabet. \\ 
$\sW$ denotes a \emph{subset} of  admissible words satisfying a certain  \emph{extra property}. \\
Sets with a right arrow, like $\avv \sA$, are formed by admissible \emph{forward infinite} words.\\
Sets with a left arrow,  like $\arr \sA$, are formed by admissible  \emph{backward infinite} words. \\
Sets with a left-right arrow, like $\overleftrightarrow \sA$, are formed by   admissible \emph{bi-infinite} words. 
    }%
}

}

\begin{example}\label{simple example of hyperbolic map of type A2}\label{simple example of hyperbolic map of type A3}
   Let $F^\sA$ be the example of hyperbolic map of type $\sA$ of  \cref{simple example of hyperbolic map of type A} with $N$ even and $\delta>0$ small. We recall that $(\sV, \sA)$ is a graph with a unique vertex $\so$ and $N$ arrows $\sa_j$.  Let $\sC=\{\boxdot_j: 1\le j \le N-1\}$ be a set of $N-1$ other arrows satisfying $\si(\boxdot_j)=\st(\boxdot_j)=\so$. For $p$ in:
\[ \cP:=\left\{(p_i)_i \in\R^{N-1}: \forall 1\le i\le N-1,\quad p_i\in I_{i+1}\text{ and } p_i\text{ is $\sqrt[3]{\delta}$-distant from }\partial I_{i+1}\right\}, \]
 we define a folding transformation $F^{\boxdot_j}_p$   for $F^\sA$ by:
\[F^{\boxdot_j}_p:z=(x,y)\in Y^{\boxdot_j} \mapsto \left(x^2+ y + p_j, - \sqrt{\delta}\cdot  x\right) \]
\[\text{with}\quad Y^{\boxdot_j} =I^\boxdot\times J^{\boxdot_j}\; ,\quad 
I^\boxdot= \left[- {\delta^2}/{4}, {\delta^2}/{4}\right]\qand 
J^{\boxdot_j}:=\sqrt{\delta}\cdot I_j\, .\]
 Note that $F^{\boxdot_j}_p$ depends on a parameter $p\in \cP$ but its  definition domain $Y^{\boxdot_j}$ is independent of $p$, see \cref{example geo model}. 
\begin{fact} 
The map  $F^{\boxdot_j}_p$ is a folding transformation at $\so$ adapted to  $F^\sA$ with $ \arr \sA_{\boxdot_j}:=   \arr \sA\cdot {\sa_j}$.
\end{fact}
\begin{proof} 
First note that  the segment $I^\boxdot$  is disjoint from $\bigsqcup_k I_k= \bigsqcup_k \left[\frac 2N (k-1)-1+\delta^2, \frac 2N k-1-\delta^2\right]$. Thus $Y^{\boxdot_j}$ is disjoint from $\bigcup_{k=1}^{N} Y^{\sa_k}$. Also the image of $F^{\boxdot_j}_p$ is included in $(-1,1) \times[-\delta^2/4, \delta^2/4]  $, which is in the interior of $Y^\se$. We notice that for every $\su \in \arr \sA$, there exists $y^\su$ such that $W^\su=I\times \{y^\su\}$. Note also that $y^\su$ belongs to the interior of $\bigsqcup_k \sqrt{\delta}\cdot I_k$  and so $W^\su$ is disjoint from $I\times \bigsqcup_k \sqrt{\delta}\cdot \partial I_k\supset \partial^u Y^{\boxdot_j}$. 
Also $W^\su$ intersects $Y^{\boxdot_j}$ iff $\su \in \arr \sA_{\boxdot_j}:=  \arr \sA\cdot {\sa_j}$  which is indeed a clopen set.   Also, for $\su \in \arr \sA_{\boxdot_j}$,  the image by $F_p^{\boxdot_j}$ is:
\begin{equation}\label{form pli} F^{\boxdot_j}_{p}(W^{\su} \cap Y^{\boxdot_j}) = \left\{(x^2+y^\su+ p_j, -\sqrt{\delta}\cdot x) : -\frac{\delta^{2}}{4} \le x \le \frac{\delta^{2}}{4} \right\}. \end{equation}
As every stable manifold $W^\ss$, for $\ss\in \avv \sA$, is of the form $\{x_\ss\}\times I$, if $F^{\boxdot_j}_{p}(W^{\su} \cap Y^{\boxdot_j})  $ is tangent to $W^\ss$ then the tangency is quadratic.  \end{proof}
As the boxes $(Y^{\boxdot_j})_{1\le j\le N-1}$ are disjoint, we can define:
\[ F_p : (x,\so)\in   \bigsqcup_{j=1}^N  Y^{\sa_j} \sqcup  \bigsqcup_{j=1}^{N-1} Y^{\boxdot_j}\times \{\so\} \mapsto    \left\{\begin{array}{cl}
(F^\sa( x), \so)& \text{if } x\in Y^\sa,\;  \sa\in \sA\\
(F_p^\boxdot(x), \so)& \text{if } x \in    Y^\boxdot, \; \boxdot\in \sC.
\end{array}\right.\]
\begin{fact}   The map $F_p$ is a system of   type $(\sA, \sC)$ which is a diffeomorphism for every $p\in \cP$.\end{fact}
\begin{proof}   We already saw that $F^\sA$ is a diffeomorphism. Looking at the definition of $\cP$ and $F_p^\sC $, we see that the folding transformations $F^\boxdot_p$ have disjoint images from the one other. These images are included in $(-1,1) \times[-\delta^2/4, \delta^2/4] $  which is disjoint from the range of $F^\sA$. \end{proof}
\end{example}
\begin{figure}[h]
\centering
\includegraphics[width=8.5cm]{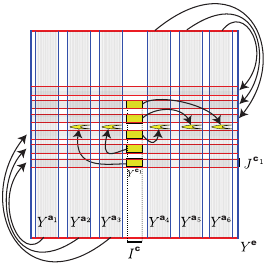}
\caption {An example  of a system of type $(\sA, \sC)$. }\label{example geo model}\end{figure}

 In order to do the parameter selection in \cref{parameterselection}, we will need the following object enabling to follow the tangency point of the homoclinic tangencies in a neat way:\index{ $\mathbb P(\chi_v)$}
\begin{definition}\label{def adapted proj}\index{Adapted projection}
 An adapted   projection to a map $F$  of type $(\sA, \sC)$ is a  $C^1$-submersion:
 \[\pi: (z, \sv)\in \R\times I\times \sV\mapsto (\pi^\sv(z), \sv)\in  \R\times \sV\]
  such that with $\mathbb P(\chi_v)\subset \mathbb P(\R)$ the set of lines spanned by vectors in $\chi_v$:
\begin{enumerate}[(0)]
\item For every $(z,\sd)\in \bigcup_{\sd\in \sA \cup \sC} Y^\sd\times \{\si(\sd)\}$, the $\pi^{\si(\sd)}$-fiber of $z$ is included in $Y^\se$.  
\end{enumerate}
\begin{enumerate}[(1)]
\item The map $(z,\sv)\in Y^\se\times \sV \mapsto \ker D_z \pi^\sv $ is of class $C^1$ and takes its values in $\mathbb P(\chi_v)$. 
\item For any $\sv\in \sV$, for every $x\in \R$,  it holds  $\pi^\sv(x,0)=x$.
\item 
For all $\sa\in \sA$ and  $z, z'\in Y^\sa$,  if $\pi^{\si(\sa)}(z)=\pi^{\si(\sa)}(z')$ then  $\pi^{\st(\sa)}\circ F^\sa(z)=\pi^{\st(\sa)}\circ F^\sa(z')$.
\item   For every $\su \in \arr \sA_\boxdot$ with $\boxdot \in \sC$, the curve $W^\su$ is tangent to a unique fiber of $\pi^{\st(\boxdot)}\circ F^\boxdot$ at a unique point $\zeta^\su\in \mathrm{int\, } Y^\boxdot$, and the tangency is quadratic.
\end{enumerate}
\end{definition}
  \begin{figure}[h]
\centering
 \includegraphics[height=6.5cm]{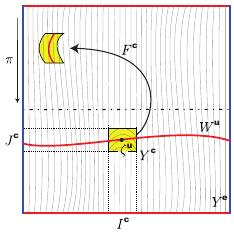}
\caption{Definition of $\zeta^\su$.}\label{fig zeta su}
\end{figure}
Let us comment this definition. Item $(0)$ says that for any $\sa \in \sA \cup \sC$, any $\pi^{\si(\sa)}$-fiber intersecting $Y^\sa$ is fully included in $Y^\se$. Item $(1)$  implies that the fibers of $\pi^\sv$ are $C^2$-curves with tangent space in $\chi_v$ and with two endpoints, one in each of the components of $\R\times \partial I$. Item $(2)$ implies that $\pi$ is a retraction.  Item $(3)$ implies, as we will prove in \cref{def bps} that each curve $W^\ss$ is a fiber of $\pi^{\si(\ss)}$ for $\ss$. Thus $\pi$ defines an extension of the stable lamination which is invariant by the dynamics $F^\sA$.   The last item is equivalent to say that for every $\su \in \arr \sA_\boxdot$, the map $x\mapsto \pi^{\st(\boxdot)}\circ F^\boxdot\circ (x,w^\su(x))$ has a unique critical point which is non-degenerate. It enables to describe how far an unstable curve $W^\su$, $\su\in \arr \sA$ is to be folded tangent to a  stable curve. The definition of $\zeta^\su$ is depicted in \cref{fig zeta su}.

\begin{example}
\label{simple example of hyperbolic map of type A4} The projection 
$\pi:(x,y, \so)\mapsto (x,\so)$ is adapted to the map $F_p$ of \cref{simple example of hyperbolic map of type A3}.
\end{example}
\begin{proof}  The  three first items of \cref{def adapted proj} are obviously  satisfied. Item $(3)$ as well since $F^\sA$ leaves invariant the vertical foliation. Finally we observe that by \cref{form pli}, for every $1\le j\le N-1$, for every $\su\in \arr \sA_{\boxdot_j}$ the curve $F^{\boxdot_j}_{p}(W^{\su} \cap Y^{\boxdot_j})  $ has a unique quadratic tangency with a  vertical line. This implies the last item.
 \end{proof}

More generally, we will see in \cref{Horseshoe2model} that any horseshoe of a surface $C^2$-diffeomorphism can be modeled by  a hyperbolic map of a certain type endowed with a map $\pi$ satifying items  (0-1-2-3-4) of \cref{def adapted proj}. As promised let us show:
\begin{proposition}  \label{def bps}
If $\pi$ is adapted to $F$ then for every $\ss\in  \avv \sA$,  the curve  $W^\ss$ is a fiber of $\pi^{\si(\ss)}$.   
\end{proposition}
\begin{proof}   Let $\ss=\sa_0\cdots \sa_i\cdots \in  \avv \sA$, $\sv=\si(\ss)$, 
 $z\in   W^\ss$. As $W^\ss =\bigcap_{n\ge 0} Y^{\sa_0\cdots \sa_n}$, it follows that $(F^\sA)^n(z, v)$ is in {$Y^{\sa_n} \times \{\si(\sa_n)\}$}.  Thus by \cref{def adapted proj}-(0), the $\pi$-fiber of $(F^\sA)^n(z, v)$ is included in $Y^\se\times \sV$ for every $n$. By \cref{def adapted proj} (3), the $\pi$-fiber of $(z, v)$ is sent into the $\pi$-fiber of $(F^\sA)^n(z, v)$ by $(F^\sA)^n$. The tangent spaces of the fibers of $\pi$ are in $\chi_v$ by 
\cref{def adapted proj}-(1), and every vector in  $\chi_v$  has its second coordinate which is more than $\lambda$-expanded by { $D_z (F^\sa)^{ -1 }$ for any $\sa \in \sA$ and $z \in F^\sa(Y^\sa)$} by \cref{defpiece} (3). 
Thus {the $\pi$-fiber of $(z, v)$ } is a local stable manifold of $(z, \sv)$. By \cref{stablemanifold} and uniqueness of the local stable manifold of the hyperbolic set $\Lambda$, it follows  that $W^\ss\times \{\sv\}$ is equal to the $\pi$-fiber of $(z,\sv)$.
\end{proof} 

We are going to  give sufficient conditions for a system $F$  of type $(\sA, \sC)$ to display a  wandering stable component. We will use the following:

\begin{definition}  \label{defHs}
For every $\sc\in \sA^*$, let $H^\sc:= F^\sc(I\times\{0\}\cap Y^\sc)$.
 \end{definition}

\begin{definition} \label{defYnorme}
For any $\sc\in \sA^*$, let  $|Y^\sc|$ be the maximum of the lengths of $I\times\{y\}\cap Y^\sc$ among $y \in I$ and 
 $|F^\sc(Y^\sc)|$ be the maximum of the lengths of $\{x \}\times I\cap F^\sc(Y^\sc)$ among $x \in I$.
 \end{definition}
The following enables to extend the tangency point $\zeta^\su$ for some finite words $\sc\in \sA^*$.
\begin{proposition}[definition of $\sA^*_\boxdot$ and $\zeta^\sc$]\index{$\sA^*_\boxdot$}\index{$\zeta^\sc$} \label{defzetafixep} \label{factzetajnew}
For every $\boxdot\in \sC$, there exists $n_\boxdot> 0$ minimal such that for  any $\sc$ in:
\[  \sA^*_\boxdot := \{ \sc \in \sA^* :  |\sc| \ge n_\boxdot \text{ and } 
 \sc \text{ equals the } |\sc|   \text{   last letters of a certain } \su \in  \arr \sA_\boxdot\}   \]
 the curve $H^\sc$ intersects $Y^\boxdot$ but not $\partial^u Y^\boxdot$, the curve $H^\sc$  is tangent to a unique fiber of $\pi^{\st(\boxdot)} \circ F^\boxdot$ at a unique point $\zeta^\sc\in \mathrm{int\, } Y^\boxdot$, and the tangency is quadratic.
\end{proposition}
\begin{proof}
It is sufficient to prove that for every sequence $(\sc_i)_{\ge 0}$ of words $\sc_i \in \sA^*$ such that $|\sc_i|\to \infty$ and  such that there exists $\su_i\in \arr \sA_\boxdot$    finishing by $\sc_i$,  the following properties hold true when $i$ is large enough: the curve $H^{\sc_i}$ intersects $Y^\boxdot$ but not $\partial^u Y^\boxdot$, the curve $H^{\sc_i}$  is tangent to a unique fiber of $\pi^{\st(\boxdot)} \circ F^\boxdot$ at a unique point in $\mathrm{int\, } Y^\boxdot$, and the tangency is quadratic. By compactness of {$\arr \sA_\boxdot$}, we can assume that $(\su_i)_i$ converges to $\su\in  \arr \sA_\boxdot$. By continuity of $\su\mapsto W^\su$, the curve $W^{\su_i}$ is $C^2$-close to $W^\su$ when $i$ is large. 

We notice that each of the $H^{\sc_i}$ is a component of the image by $(F^\sA)^{|\sc_i|}$ of the curve $I\times \{0\}\times \sV$  which is uniformly transverse to the stable lamination of $\Lambda$.  When $i$ is large,  by the Inclination Lemma, the curve $H^{\sc_i}$ is $C^2$-close to $W^{\su_i}$ and so to $W^\su$.

By definition of $\arr \sA_\boxdot$, the curve $W^\su$ intersects $Y^\boxdot$ but not $\partial^u Y^\boxdot$. This is equivalent to say that 
$W^\su$ intersects  the interior of $Y^\boxdot$ but not $\partial^u Y^\boxdot$ which is an open condition on curves. Thus for $i$ sufficiently large $H^{\sc_i}$ intersects  the interior of $Y^\boxdot$ but not $\partial^u Y^\boxdot$. 

Also $W^{\su}$ has a tangency with the foliation defined by $\pi^{\st(\boxdot)} \circ F^\boxdot$, which is unique, in int $Y^\boxdot$  and quadratic. This is equivalent to say that the curve $TW^{\su} \subset Y^\se\times \mathbb P(\R)$ is transverse to the surface $\{(z,\ker D  \pi^{\st(\boxdot)} \circ D_z F^\boxdot ): z\in   Y^{\boxdot} \}\subset Y^\se\times \mathbb P(\R)$  and intersects it at a unique point in $  \mathrm{int} \; Y^\boxdot\times \mathbb P(\R)$.   This is an open condition on $C^1$-curves in $Y^\se\times \mathbb P(\R)$. 
Thus $H^{\sc_i}$ satisfies the same property when $i$ is large.\end{proof}
\begin{remark}\label{limit sets}
We have shown that the set of $C^2$-curves $\{H^\sc: \sc\in \sA^*_\boxdot\}\cup \{W^\su: \su\in \arr \sA_\boxdot\}$ is compact. 
Also the subset $\{\zeta^\sc: \sc \in  \sA^*_\boxdot\cup   \arr \sA_\boxdot\}$ of $\R^2$ is compact. 
\end{remark}

When $F$ is real analytic, we define   $\rho>0$ and the sets $(\tilde Y^\sa)_{\sa\in \sA}$ and $\tilde Y^\boxdot$  by the following:
 \begin{proposition}[and definition of $\rho>0$]\label{def rho}
If $F$ is real analytic, then there exists $\rho>0$ such that for every   $\sc \in \sA^*$, the hyperbolic transformation $(Y^{\sc},  F^{\sc}) $   extends to a  $C^\omega_\rho$-hyperbolic transformation  $(\tilde Y^{\sc}, F^{\sc}) $ and the map $F^\boxdot$ is of class $C^\omega_\rho$ on $\tilde Y^\boxdot$ for every $\boxdot\in \sC$.
\end{proposition}
\begin{proof}
The first condition on $F^\sA$  is proved in \cref{rho uniform} and the second condition on $F^\boxdot$ is obvious by \cref{def folding map}. 
\end{proof}

\begin{definition}   \label{defYnormec}
 If $F$ is analytic, for any $\sc\in \sA^*$, let
 $|F^\sc(\tilde Y^\sc)|$ be the maximum of the diameters of $\{z \}\times \tilde I \cap F^\sc(\tilde Y^\sc)$ among $z \in \tilde I$.
 \end{definition}

Here is the proposition which gives sufficient conditions to ensure the existence of a wandering stable component. It is a variation of Theorem D of \cite{berger2018zoology} and its  Remark 3.5.
\begin{theorem}  \label{propreel} Let $F$ be a system of type $(\sA, \sC)$ endowed with an adapted projection $\pi$.
 Let $(\boxdot_j)_{j\ge 1}\in \sC^\N$ and  let $(\sc_j)_{j \ge 1}$ be a sequence of words $\sc_j  \in  \sA^*_{\boxdot_j}$ such that $\si(\sc_j)=\st(\boxdot_{j-1})$ and satisfying the following assumptions when $j$ is large:
 \begin{enumerate}[$(i)$]
\item the limit exists and is non-zero $\breve \gamma_j:=    |Y^{\sc_{j+1}}|^{1/2}\cdot |Y^{\sc_{j+2}}|^{1/2^2}\cdots |Y^{\sc_{j+k}}|^{1/2^{k}}\cdots  \, ,$ \index{$\breve \gamma_j$}
\item  the distance between $\pi^{\st(\sc_{j+1})}$-fibers of  $ F^{\sc_{j+1}}\circ F^{\boxdot_j}(\zeta^{\sc_j})$ and   $ \zeta^{\sc_{j+1}}$ is small compared to $  \breve \gamma_{j+1}$,
\item   the width  $|Y^{\sc_{j}}|$ is small  and the height  $ |F^{\sc_j}(Y^{\sc_{j}})|$  is small compared to $ \breve \gamma_j^2$:
\[|Y^{\sc_{j}}| = o(1)\qand  |F^{\sc_j}(Y^{\sc_{j}})| = o( \breve \gamma_j^2).\]   
\end{enumerate}
Then there exist $J\ge 0$ and  a sequence of nonempty open subsets $(B_j)_{j\ge J}$ of $ Y^\se$ such that:
\[  B _j \subset Y^{\boxdot_j}, \quad F^{\boxdot_j}(B_j)\subset Y^{\sc_{j+1}},
\quad F^{\sc_{j+1}}\circ  F^{\boxdot_j}(B_j)\subset B_{j+1}   \qand     \lim_{k\to \infty}  \diam F ^k( B_j \times  \{\si(\boxdot_j)\})=0\; .\]
\end{theorem} 

\begin{remark} \label{def Vj equiv2}
We can show that $(i)$ is equivalent to ask that {$\log |Y^{\sc_j} |\asymp |\sc_j|$} satisfies that $\sum_j  2^{-j} |\sc_j|<\infty$.   
We will see, using distortion bound, that assumption $(ii)$ is equivalent to say that the distance between the $\pi^{\si(\sc_{j+1})}$-fibers of the points $  F^{\boxdot_j}(\zeta^{\sc_j})$ and   $(F^{\sc_{j+1}})^{-1}( \zeta^{\sc_{j+1}})$ is small compared to $  \breve \gamma_{j}^2$.
The first part of assumption {$(iii)$} implies that $|\sc_j| $ is large when $j$ is large. 
\end{remark}
\begin{theorem}  \label{propcomplex}Under the assumptions of \cref{propreel},  if moreover  $F$ is analytic  and if
\begin{enumerate}[$(\widetilde{iii})$]
\item    $\breve \gamma_j$ is small compared to $ |Y^{\sc_{j+1}}|$ and   $ |F^{\sc_j}(\tilde Y^{\sc_{j}})| $ is small compared to $\breve \gamma_j^2$ when $j$ is large:
 
\[\breve \gamma_j= o(|Y^{\sc_{j+1}}|)\qand  |F^{\sc_j}(\tilde Y^{\sc_{j}})| = o( \breve \gamma_j^2).\] 
\end{enumerate} then each $B_j$ is the real trace of an open subset $\tilde B_j\subset \C^2$ such that:
 \[ \tilde B _j \subset \tilde Y^{\boxdot_j}, \quad F^{\boxdot_j}(\tilde B_j)\subset \tilde Y^{\sc_{j+1}},
\quad F^{\sc_{j+1}}\circ  F^{\boxdot_j}(\tilde B_j)\subset \tilde B_{j+1}   \qand     \lim_{k\to \infty}  \diam F^k(\tilde  B_j \times  \{\si(\boxdot_j)\})=0\; ,\]
where $\tilde Y^{\boxdot_j}$ and $\tilde Y^{\sc_{j+1}}$ are defined by  \cref{def rho}.
\end{theorem} 
We notice that $B_j\times  \{\si(\boxdot_j)\}$ and $\tilde B_j\times  \{\si(\boxdot_j)\}$ are  {stable domains} for the dynamics $F$ which are  in wandering stable components (for their  combinatorics are wandering).  
The proof of  \cref{propreel,propcomplex} will occupy all \cref{sectiontildeBj} and will need a development of the techniques of implicit represention of \cref{sect: Implicit representations and initial bounds} and a normal form defined in \cref{sectiondefBj}.

\subsection{Unfolding of wild type}\label{sec Unfolding of wild type}
In this section we are going to state natural conditions on families of systems of type $(\sA, \sC)$ so that at a dense set of parameters the map will satisfy the assumptions of \cref{propreel,propcomplex}.

We recall that a  closed subset is \emph{regular} if it is   nonempty and  equal to the closure of its interior.\index{Regular closed set}  Let $\cP$ be a nonempty, regular compact subset of $\R^d$.
 For the  proofs of our main theorems, $\cP$ will be a hypercube of $\R^5$.\index{$\cP$} Let $n, m\ge 1$ and let $(K_p)_p$ be a family of compact sets of $\R^n$. For $r\ge 0$,  a family $(f_p)_{p\in  \cP}$ of maps $f_p: K_p\to   \R^m$  is \emph{of class $C^r$} if 
$p\in \cP \mapsto K_p$ is continuous for the Hausdorff topology
and  the map $(p, z)\in \hat K:=\bigcup_{p\in \cP} \{p\}\times K_p\mapsto f_p(z)\in \R^m$ is of class $C^r$ in the sense we defined \cpageref{def cr}.

We are now going to consider an unfolding of the stable and unstable laminations of a $C^2$-family of maps $(F_p)_{p\in \cP}$ of type $(\sA,\sC)$.  
 
\begin{definition}\label{def family adapted}
A family $(F_p)_{p\in \cP}$ of systems of type $(\sA, \sC)$ is  \emph{regular} 
if each $F_p$ is a  system of type $(\sA, \sC)$ with a same set $\arr \sA_\boxdot$ for every $p \in \cP$ at each fixed $\boxdot\in \sC$ and so that  the family $(F_p)_{p\in \cP}$ is of class $C^2$.  
Note that for every $p\in \cP$, the map $F_p$ is of the form:
\[ F_p : (x,\sv)\in   \bigsqcup_{\sa\in \sA} Y^\sa_p\times \{\si(\sa)\} \sqcup  \bigsqcup_{\boxdot \in \sC} Y^\boxdot_p\times \{\si(\boxdot)\} \mapsto \left\{\begin{array}{cl}
F_p^\sA( x, \sv)& \text{if } (x, \sv)\in \bigsqcup_{\sa\in \sA} Y^\sa_p\times \{\si(\sa)\},\\
F_p^\sC(x, \sv)& \text{if } (x,\sv)\in    \bigsqcup_{\boxdot \in \sC} Y^\boxdot_p\times \{\si(\boxdot)\}.
\end{array}\right.\]
A family of projections $(\pi_p)_p$ is adapted to $(F_p)_p$  if each $\pi_p$ is a projection adapted to $F_p$, the family $(\pi_p)_p$ is of class $C^1$ and the family $(\ker D_z \pi_p)_{p\in \cP}$ is of class $C^1$.
\end{definition}
\begin{example}\label{simple example of adapted} For instance the family $(F_p)_p$ of systems of type $(\sA, \sC)$ of \cref{simple example of hyperbolic map of type A3} is regular and the trivial family of projections $(\pi)_{p\in \cP}$ of 
\cref{simple example of hyperbolic map of type A4} is adapted to it. 
\end{example}

Let us fix a regular family $(F_p)_p$ of type $(\sA, \sC)$ endowed with an adapted  family  $(\pi_p)_p$ of projections.  As in \cref{def family adapted}, for every $p\in \cP$, we add an index $p$ to all the geometric objects defined in the last subsection. For instance, we denote also by $\Lambda_p$ the maximal invariant set of $F^\sA_p$ defined in \cref{horseshoeHp}. \index{$\Lambda_p$}  For every $\sc=\sa_1 \cdots \sa_k \in \sA^*$, we denote $(Y^\sc_p, F^\sc_p):=(Y^{\sa_1}_p, F^{\sa_1}_p)\star \cdots (Y^{\sa_k}_p, F^{\sa_k}_p) $. For $\ss\in \avv \sA$ and  
$\su\in \arr \sA$, we denote by $W^\ss_p$ and $W^\su_p$ the corresponding local stable and unstable manifolds of $F^\sA_p$. 

The following is immediate:
\begin{proposition}\label{conti combi}\label{ AL C2}
For every $\sc \in \sA^*$, the family $((Y^\sc_p, F_p^\sc))_{p\in \cP}$ is of class $C^2$.
\end{proposition}
 
By \cref{stablemanifold}, for every $\su\in \arr \sA$ and $\ss\in \avv \sA$,  for every $p\in \cP$, there are $\theta$-Lipschitz functions $ w^{\su}_{p}$ and $ w^{\ss}_{p}$ on $I$ whose graphs satisfy: \index{$ w^{\su}_{p}$ and $ w^{\ss}_{p}$ }
\begin{equation}\label{def wu et ws} \mathrm{Graph}\, w^{\su}_{p}=W^\su_p\qand ^t
\mathrm{Graph}\, w^{\ss}_{p}=W^\ss_p\; .\end{equation}

The following states that the stable and unstable laminations depend $C^2$ on the parameter:
\begin{proposition} \label{regularite C2 de w}\label{wcontinususs}
The maps $(p,x) \in \cP \times I \mapsto w^{\su}_{p}(x) \in \mathbb{R}$ and $(p,y) \in \cP \times I \mapsto w^{\ss}_{p}(y) \in \mathbb{R}$ are of class $C^{2}$.
Moreover,  the two following maps are continuous:
$$\su\in \arr \sA \mapsto (w^{\su}_{p})_{p \in \cP} \in C^{2}(\cP \times I,I) \text{ and } \ss\in \avv \sA \mapsto (w^{\ss}_{p})_{p \in \cP} \in C^{2}(\cP \times I,I)\; .$$
\end{proposition}
\begin{proof}This is an immediate consequence of \cref{cartecool4cpct} of the Appendix applied to the family of hyperbolic sets $(\Lambda_p)_p$.\end{proof}

\begin{definition}\label{deficalVus} 
For every $\boxdot\in \sC$, we define:
\begin{itemize}
\item 
The map $ p\in \cP\mapsto a^\su_p:=\pi^{\st(\boxdot)}_p\circ F^\boxdot_p(\zeta^\su_p)$, for every $\su \in \arr \sA_\boxdot$.\index{$a^\su_p$}
\item The map $ p\in \cP\mapsto b^\ss_p:=w^s_p(0)$,  for every $\ss\in \avv \sA$.\index{$b^\ss_p$}
\item The map $ p\in \cP\mapsto\cb  \cal V( \su, \ss,p) = a^\su_p -b_p^\ss$, for all  $\su \in \arr \sA_\boxdot$ and  $\ss\in \avv \sA_{\boxdot }$.\index{$\mathcal V(\su, \ss, p)$}
\end{itemize}
\end{definition}  
As $\pi_p^\sv(x,0)=x$ for every $(x, \sv) \in I\times \sV$,  it holds $\pi^{\si(\ss)} _p(W^\ss_p)=\{b_p^\ss\}$.
We recall that by \cref{def folding map}, $\zeta_p^\su$ is the unique critical point of $\pi^{\st(\boxdot)}_p\circ F^\boxdot_p$. Thus  $a^\su_p$ is the critical value of $\pi^{\st(\boxdot)}_p\circ F^\boxdot_p\circ w^\su_p$ and $\cb  \mathcal V( \su, \ss, p)$ quantifies how far  $F^\boxdot_p(W^\su_p)$ is to be tangent to $W^\ss_p=(\pi^{\st(\boxdot)})^{-1}(\{\sb_p^\ss\})$.  In particular we have:
\begin{remark}\label{rema geo int cal V}
The curve $F^\boxdot_p(W^\su_p \cap  Y^\boxdot_p)$ is tangent to $W^\ss_p$ if and only if $\cb \cal V(\su, \ss, p)=0$.
\end{remark}
\begin{figure}[h]
\centering
\includegraphics[height=6.5cm]{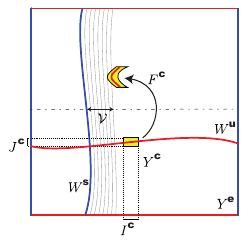}
\caption{ Definition of $\cb \cal V$.}
\end{figure}

\begin{proposition}\label{def calV}For every $\boxdot\in \sC$, the following are $C^1$-functions depending continuously on $\su\in  \arr \sA_\boxdot$, $\ss \in \avv \sA_{ \boxdot }$:
$$p\mapsto \zeta^{\su}_{p}\in \mathrm{int} \,  Y_p^\boxdot\; , 
\quad p\in \cP\mapsto a^\su_p\in I
\; ,\quad p\in \cP\mapsto b^\ss_p\in I
 \qand p\mapsto \cal V{\cb (\su, \ss,p)}\in \R\; .$$ \end{proposition}
\begin{proof} We recall that $(p,x) \mapsto w^\su_p(x)$ and $(p,y) \mapsto w^\ss_p(y)$ are $C^2$-functions depending continuously on $\su\in \arr \sA$ and $\ss\in \avv \sA$ by \cref{wcontinususs}. Thus the map $p\mapsto b_p^\ss=w^\ss_p(0)$ is of class $C^1$ and depends continuously on $\ss$. By \cref{def adapted proj} (4), the curve $W^\su_p\cap Y^\boxdot_p$ displays a unique tangency with a fiber of $\pi^{\st(\boxdot)}_p\circ F^\boxdot_p$ and this tangency is quadratic. This means that the curve 
$TW^\su_p\cap Y^\boxdot_p \times \mathbb P(\R)$ of $Y^\se\times \mathbb P(\R)$ intersects the surface $\{(z,\ker D  \pi^{\st(\boxdot)} \circ D_z F^\boxdot ): z\in   Y^{\boxdot} \}$ of $Y^\se\times \mathbb P(\R)$ at a unique point $(\zeta^\su_p, T_{\zeta^\su_p} W^\su)$ 
and this intersection is transverse. As this curve and this surface vary $C^1$ with $p$ and  continuously with $\su$, by transversality we obtain the sought regularity of  $p\in \cP\mapsto  \zeta^\su_p$.
Then the sought regularities of $p\in \cP\mapsto  a^\su_p=\pi^{\st(\boxdot)}_p\circ F^\boxdot_p(\zeta^\su_p)$ and $p\mapsto {\cb \cal V( \su, \ss,p)} = a^\su_p -b_p^\ss$ follow.
\end{proof}

The next definition  
 regards a regular family $(F_p)_{p\in \cP}$ of systems of type $(\sA, \sC)$ endowed with an adapted family $(\pi_p)_p$ of projections. 
\begin{definition}\label{nice unfolding}\index{Unfolding  of wild type $(\sA, \sC)$}
 The family  $(F_p, \pi_p)_{p\in \cP}$ is an \emph{unfolding  of wild type $(\sA, \sC)$} if:
 \begin{enumerate}[$(\mathbf H_1)$]\label{H0bis}\index{Assumption $\mathbf{(\mathbf H_1-\mathbf H_2)}$}
\item For every $p\in \cP$,  for every $\boxdot \in \sC$, there exist $ \su\in  \arr \sA_{\boxdot} $ and $\ss\in  \avv \sA_{ \boxdot }$  satisfying that
$F^{\boxdot}_{p}(W^{\su}_{p}\cap { Y^{\boxdot}_p})$ has a quadratic tangency with $W^{\ss}_{p}$. \label{H1} \label{H0'} 
\item
 For all $ (\su_\boxdot , \ss_\boxdot )_\boxdot \in \prod_{\boxdot \in \sC} \arr \sA_{\boxdot}\times \avv \sA_{ \boxdot }$, the following is a $C^1$-diffeomorphism onto its image:
\[p\in \mathrm{int}\, \cP \mapsto ({\cb \cal V(\su_\boxdot , \ss_\boxdot ,p))_{\boxdot \in \sC}}\in  \R^{\Card \sC}  \; \] 
\end{enumerate}
\end{definition}

Let us comment this definition.  Property $(\mathbf H_1)$ asks for $\Card \sC$ simultaneous homoclinic tangencies, each of which being associated to a different $\boxdot\in \sC$.   By \cref{def calV}, we already now that the map involved in $(\mathbf H_2)$ is of class $C^1$. What is required is its invertibility. It means roughly speaking that the unfoldings of the {tangencies} are independent and non-degenerated along the parameter space  $\mathrm{int}\,  \cP$ {whose dimension is equal to $\Card \sC$}.

We introduced the notion of unfolding of wild  type $(\sA, \sC)$ to show the existence of a dense set of parameter $p$ for which the the system $F_p$ has a    wandering stable domain.  To this end, a form of dissipation is necessary:
\begin{definition} A hyperbolic map $F^\sA$ of type $\sA$  is  \index{$\epsilon=0.1$}\index{Moderately dissipative system}\label{H3}  \emph{moderately dissipative} if it holds: 
\[\| DF^\sA\| \cdot  \|\mathrm{det}\, DF^\sA\|^\epsilon < 1\quad \text{with } \epsilon =\frac1{10}.\]
An unfolding $(F_p, \pi_p)_p$ of wild type is moderately dissipative if $F^\sA_p$ is moderately dissipative for every $p\in \cP$.
\end{definition} 
Here is the main (abstract) result of this manuscript:
\begin{theorem}[Main]\label{main them Fp}
Let $(F_p, \pi_p)_{p\in \cP}$ be a moderately dissipative unfolding of wild  type $(\sA, \sC)$ with $\Card \sC=5$. Then there exists a dense set of parameters $p\in \cP$ for which $(F_p, \pi_p)$   satisfies the conclusions of \cref{propreel}  (and of \cref{propcomplex} if $F_p$ is analytic) with subsets 
 $(B_j\times \{\si(\boxdot_j)\})_{j\ge J}$ such that every point $(z, \sv)\in B_J\times \{\si(\boxdot_J)\}$ satisfies:
\begin{enumerate}[$(\blacklozenge)$]\item 
There exist $0<t<1$ and $\mu$  in the set of invariant probability measures $\cal M_p(\Lambda_p)$ of $F_p|\Lambda_p$, such that  the limit set of the sequence of empirical measures $\mathscr e_n:= \frac{1}{n} \sum_{i=0}^{n-1} \delta_{F^i_p(z,\sv)}$ contains $t \cdot \mu+(1-t)\cdot \cal M_p(\Lambda_p)$.\end{enumerate}
 \end{theorem}

The proof of this theorem will occupy the whole \cref{parameterselection}. In this proof we will  first show that the assumptions of \cref{propreel} are satisfied at a dense set of parameter. Finally in \cref{section emergence}, we will show  conclusion $(\blacklozenge)$ by selecting the combinatorics of  sequences of words $(\sc_j)_j$. {\cb We recall that Condition $(\blacklozenge)$ is useful to show  using \cref{horseshoe emerge} that the  selected dynamics have emergence of positive order in Corollaries \ref{coro A} and \ref{coro B}.}

\cref{main them Fp} being admitted for now, let us show how it implies the main theorems. To this end, we define the following  geometric model which applies to surface maps families.
\begin{definition}[$C^r$-Geometric model] \label{geo model} \index{$C^r$-Geometric model} A $C^r$-family $(f_p)_{p\in \cP}$ of surface diffeomorphisms  $f_p\in \Diff^2(M)$ displays the \emph{geometric model} if there is a  moderately dissipative   unfolding $(F_p, \pi_p)_{p\in \cP}$ of wild type $(\sA, \sC)$  with $\Card \sC=5$ so that $(F_p)_p$ is embedded into $M$ via a $C^2$-family $(H_p)_p$ of $C^r$-embeddings $H_p:  Y^\se \times \sV\hookrightarrow M$: 
\[f_p\circ H_p|\bigsqcup_{\sd\in \sA\sqcup \sC}  Y^{\sd}_p\times \{\si(\sd)\}=   H_p \circ F_p, \quad \forall  p\in \cP.\] 
\end{definition}

In  \cref{example for the model} we will prove that the families of maps defined in \cref{simple example of adapted} define a moderately dissipative unfolding of wild type $(\sA,\sC)$ 
for any $\delta>0$ small enough depending on an even $N\ge 2$, and so the conclusions of \cref{main them Fp} {when moreover $N =6$}. 

 In \cref{section Density of wandering domain in the Newhouse domain},  we will state a general \cref{premain4application} implying that  the geometric model is displayed in any non-degenerate  unfolding of five homoclinic tangencies  of a same dissipative  saddle periodic point. This will imply Theorem \ref{Thepremain4application}. 
Then we will show that this implies the two following theorems. 
 The first one  states that the geometric model appears locally densely  among generalized Hénon maps:
\begin{theorem} \label{Newhouse model2}
For every $b\neq 0$ small, there is  a regular compact set $ \cP_b \subset \R^5$ such that with: 
\[f_p: (x,y)\in \R^2\mapsto  (x^6+ \sum_{i=0}^4 p_i \cdot x^i-  y, b\cdot x)\text{ for every }p=(p_i)_{0\le i\le 4} \in \cP\; , \]
the family $(f_p)_{p\in \cP}$ displays the $C^\omega$-geometric model \ref{geo model}.

Moreover there is a neighborhood $V_b$ of $b$ such that $\bigcap_{b'\in V_b}\cP_{b'}$ has nonempty interior.    
\end{theorem}
\begin{remark}\label{rk main wandering2}
The same  holds true for the family of  maps:
\[f_p: [z,w, t]\in \mathbb P^2(\C) \mapsto  [ z^6+ \sum_{i=0}^4 p_i \cdot z^i\cdot t^{6-i} - w^6, b\cdot z^6, t^6] \; .\]
\end{remark}

The second states that the geometric model appears openly and densely in the space of unfolding of a dissipative homoclinic tangency. We recall that a saddle periodic point displays a \emph{homoclinic tangency} if its unstable manifold is tangent to its stable manifold. The homoclinic tangency is \emph{dissipative} if the determinant of the differential at  the periodic cycle has modulus less than 1. 
\begin{theorem} \label{Newhouse model}
For every $r \in [2,\infty] \cup \{\omega\}$,
if $f\in \Diff^r(M)$ displays a dissipative, quadratic homoclinic tangency, then there exist a $C^r$-perturbation $\tilde f$ of $f$, a $C^r$-open-dense  set of families  {$(f_p)_{p\in \R^5}$} satisfying  that $f_0=\tilde f$ such that for $n\ge 1$ and  a regular compact neighborhood  $\cP$ of $0\in  \R^5$, the family {$(f^n_p)_{p\in  \cP}$}  satisfies the $C^r$-geometric model \ref{geo model}. 
\end{theorem}
These theorems together with  \cref{consequence} imply the first proof of the existence of a wandering Fatou component of a polynomial automorphism of $\C^2$ and a proof of the last Takens' problem in any regularity:
\begin{proof}[Proof of Theorems \ref{main wandering} and \ref{theorem B}]
These theorem are direct consequences of respectively \cref{Newhouse model,Newhouse model2}  and the  following consequence of main \cref{main them Fp}.\end{proof}

\begin{corollary}[Main] \label{consequence}\label{Main theorem} For every $r \in [2,\infty] \cup \{\omega\}$, if a family $(f_p)_{p\in \cP} $ satisfies the $C^r$-geometric model \ref{geo model} with a family of embeddings  $(H_p)_{p\in \cP} $, then
there exists a dense subset of $\cP$ of parameters $p$ at which the map $f_p$ has the following properties:
\begin{enumerate}[(1)]
\item There exists a wandering stable component $\cal C$. If  $r=\omega$,  then the holomorphic extension of $f_p$ has a wandering Fatou component containing $\cal C$. 
\item For every $x\in \cal C$, the limit set of the orbit of $x$ contains  the horseshoe $K_p:=H_p(\Lambda_p)$.
\item  \label{Main coro emergent}  For every $x\in \cal C$, the sequence $(\mathscr e^f_n)_n$ of empirical measures $\mathscr e_n^f:= \frac{1}{n} \sum_{i=0}^{n-1} \delta_{f^i_p(x)}$ diverges.
Furthermore, there exist $0<t<1$ and $\mu'$  in the set of invariant probability measures $\cal M_p(K_p)$ of $f_p|K_p$ such that  the limit set of the sequence $(\mathscr e^f_n)_n$   contains $t \cdot \mu'+(1-t)\cdot \cal M_p(K_p)$.
\end{enumerate}
\end{corollary}
\begin{proof}  Let$(F_p, \pi_p)_{p\in \cP}$ be the  moderately dissipative unfolding  of a certain wild type $(\sA, \sC)$  with $\Card \sC=5$ which is embedded into $M$ via  $(H_p)_p$ and such that  $f_p\circ H_p=   H_p \circ F_p$.
 Let $p\in \cP$ be in the  dense set of parameters given by \cref{main them Fp}. Let $(B_j\times \{\si(\boxdot_j)\})_{j \ge J}$ and $(\sc_j)\in (\sA^*)^\N$ be the associated sequences of stable domains and words for the fixed parameter $p$ satisfying both the conclusions of \cref{propreel} (resp. \cref{propcomplex} if $F_p$ is analytic) and condition $(\blacklozenge)$.    Put $D_j:= H_p(B_j\times \{\si(\boxdot_j)\})$ which is the real trace of  $\tilde D_j:= H_p(\tilde B_j\times \{\si(\boxdot_j)\})$ in the analytic case. 
By the conclusions of \cref{propreel}, it holds for every $j\ge J$:
 \[  D _j \subset  H_p( {Y^{\boxdot_j}_p} \times \si({\boxdot_j})), \quad f_p(D_j)\subset H_p(Y^{\sc_{j+1}}\times \{\si(\sc_{j+1})\}) \qand f_p^{|\sc_{j+1}|+1}(D_j)\subset D_{j+1}\; \]
and respectively in the analytic case by the conclusions of \cref{propcomplex}, it holds for every $j\ge J$: 
 \[  \tilde D _j \subset  H_p( {\tilde Y^{\boxdot_j}_p} \times \si({\boxdot_j}) 
), \quad f _p(\tilde D_j)\subset H_p(\tilde Y^{\sc_{j+1}}\times \{\si(\sc_{j+1})\}) \qand f_p^{|\sc_{j+1}|+1}(\tilde D_j)\subset \tilde D_{j+1}\; .\]
The last conclusion of \cref{propreel} (resp. \ref{propcomplex}) implies that $D_j$ (resp.  $\tilde D_j$) is a stable domain:
\[\lim_{k\to \infty} \diam f^k_p (D_j)=0\qand \text{resp.}\quad  \lim_{k\to \infty} \diam f^k_p (\tilde D_j)=0\; .\] 
  Thus $D_j$ (resp. $\tilde D_j$) is included in a  stable component $\cal C$ (resp. $\tilde {\cal C}$) .   We now prove  $(1)$-$(2)$-$(3)$ for this fixed parameter $p$.
 
\emph{Proof of (1).}  Assume for the sake of contradiction that ${\cal C}$ is preperiodic.
  By taking $j$ larger, we can assume it $q$ periodic for some $q\ge 1$: $f_{p}^q({\cal C})= {\cal C}\; .$ 
Let $D_j'\Subset D_j$ be with nonempty interior. We have $D'_j\Subset {\cal C}$ and so
\[A:= D'_j\cup f_{p}^q(D'_j)\Subset {\cal C}\]  
Since ${\cal C}$ is a stable component, it holds $\lim_{n\to \infty} \diam f^n_{p}(A)=0$.

We recall that $D_p(\sA):= \bigsqcup_{\sa\in \sA}  Y_p^{\sa}\times 
\{\si(\sa)\}$ and $D_p(\sC):= \bigsqcup_{\boxdot \in \sC}  Y_p^{\boxdot }\times 
\{\si(\boxdot)\}$.  We observe that for $k\ge j+1$, with  $N_k:= 1+ |\sc_{j+1}|+\cdots + 1 +|\sc_{k}|$, it holds that:
  \[\left\{ \begin{array}{cl} 
   f_p^{l}(D_j)\subset D_k\subset H_p(D_p(\sC) )&\text{ if }l=N_k\\
f_p^{l}(D_j)\subset H_p(D_p(\sA) )&\text{ if }  l\in \{N_k+1, \dots, N_{k+1}-1\}.\end{array}\right.\] 
  The sets $H_p(D_p(\sA))$ and $ H_p(D_p(\sC))$ are disjoint and compact. 
As $N_k\to \infty$ and $N_k-N_{k-1}\to \infty$,   for $k$ large enough, it holds:
\begin{enumerate}[(i)]
\item the diameter of $f_p^{N_k-q}(A)$ is smaller than the distance between  
  $ H_p(D_p(\sA)) $ and $H_p(D_p(\sC))$. 
\item   the set   $f_p^{N_k-q}(A)$ contains  $f_p^{N_k-q}(f^{q}_p( D_j'))=f_p^{N_k}( D_j')$ and so   intersects $H_p(D_p(\sC))$. 
\item  the set $f_p^{N_k-q}(A)$ contains $ f_p^{N_k-q}( D_j')$
and so intersects $H_p(D_p(\sA))$ since $N_k-q\in \{N_{k-1}+1, \dots, N_k-1\}$ 
for $k$ sufficiently large. 
  \end{enumerate}
Conclusions (i) and (ii) contradict  (iii).  Thus ${\cal C}$ is a  wandering stable component.

 Literally the same argument (by replacing $\cal C$ by $\tilde {\cal C}$)  proves that the open set $\tilde {\cal C}\subset \tilde M$ is wandering.  We observe that $\tilde {\cal C}$ is included in a Fatou component $\cal F$.  Let $\cal F'\Subset \cal F$  be a precompact open subset of $\cal F$  which  intersects the interior of $\tilde B_j$. Then any cluster value for $(f^n_{p}|\cal F')_n$ is a holomorphic map constant on $\tilde B_j\cap \cal F'$ and so is constant. Thus any cluster value of  $(\diam f^n_{p}(\cal F'))_n$ is zero and so $\cal F'$ is included in $\tilde {\cal C}$. Since $\cal F$ can be written as an union of such open sets $\cal F'$ it follows $\cal F=\tilde {\cal C}$ and so $\cal F$ is a wandering Fatou component.

 \emph{Proof of (3).} Let $j=J$. By $(1)$ the stable component $\cal C$  containing $D_J$ is wandering.   Let $(z,\sv)$ be a point of $B_J  \times \{\si(\boxdot_J)\} $ and  $x := H_p(z,\sv)  \in  D_J \subset \cal C$. By \cref{main them Fp} $(\blacklozenge)$, there exist $0<t<1$ and $\mu \in \cal M_p(\Lambda_p)$ such that  the limit set of the sequence of empirical measures $\mathscr e_n:= \frac{1}{n} \sum_{i=0}^{n-1} \delta_{F^i_p(z,\sv)}$ contains $t \cdot \mu+(1-t)\cdot \cal M_p(\Lambda_p)$. We notice that the pushforward of $\mathscr e_n$  by $H_p$ is equal to $\mathscr e_n^f:= \frac{1}{n} \sum_{i=0}^{n-1} \delta_{f^i_p(x)}$ since $f_p\circ H_p=   H_p \circ F_p$. Also the pushforward by $H_p$ of $\cal M_p(\Lambda_p)$ is equal to $\cal M_p(K_p)$ since $K_p$ is the image of  $\Lambda_p$ by the embedding $H_p$.   As the pushforward operation on probability measures by an embedding is an embedding in the spaces of probability measures,   the limit set of  $(\mathscr e_n^f)_n$ contains $t \cdot  \mu' +(1-t)\cdot \cal M_p(K_p)$ with $\mu':=  H_{p \, *}  \mu \in  \cal M_p(K_p)$. In particular  $(\mathscr e_n^f)_n$  diverges. Finally, for every $x'$ in $\cal C$,  the limit set of $(\frac{1}{n} \sum_{i=0}^{n-1} \delta_{f^i_p(x')})_n$ is equal to the limit set of $(\mathscr e_n^f)_{n\ge 0}$ and so contains   $t \cdot  \mu' +(1-t)\cdot \cal M_p(K_p)$, since $\cal C$ is a stable component.   
 
 \emph{Proof of (2).}
We now recall that any open set $U$ of $K_p$ has positive mass for a  measure $\mu_U \in  \cal M_p(K_p) $, and so for the measure $ t \cdot  \mu' +(1-t)\cdot \mu_U  $. Hence by \emph{(3)}   the limit set of any point in $\cal C $ intersects $U$ and so is dense in $K_p$. As the limit set is compact, it must contain $K_p$. 
\end{proof}

\section{Examples of families displaying the geometric model}\label{sec: examples}

In this Section, we give examples satisfying the geometric model (\cref{geo model}).

 First we show that the family $(F_p, \pi_p)_{p\in \cP}$  of \cref{simple example of adapted} is an unfolding of wild type $(\sA, \sC)$ and we give a sketch of proof of how to deduce simply from this an example of wandering Fatou component for a real polynomial automorphism of $\C^2$.

{The main result of this Section is \cref{premain4application} from which we will deduce \cref{Newhouse model2} (and thus examples of wandering Fatou components for real polynomial automorphisms of $\C^2$) and \cref{Newhouse model}. Namely, 
\cref{premain4application} states that an unfolding of wild type is embedded into some iterate of any $d$-parameter family which displays a non-degenerate unfolding of $d$-quadratic homoclinic tangencies. When $d=5$ and up to a dissipativeness assumption, the geometric model is satisfied. 
To prove both \cref{premain4application,Newhouse model} we will need to recall classical results about Newhouse phenomenon (\cref{classical,Ne79}). Let us point out that the proof of \cref{premain4application} is technical and could be omitted in a first reading. }

\subsection{A simple example of unfolding of wild type $(\sA, \sC)$} \label{example for the model} 
We are going to prove that for any even $N\ge 2$ and $\delta>0$ small enough, the family $(F_p, \pi_p)_{p\in \cP}$ of \cref{simple example of adapted} is a {moderately dissipative } unfolding of wild type $(\sA, \sC)$. To this end we must show that it satisfies (in particular) Property $(\mathbf {H_1})$ of \cref{nice unfolding}, which by \cref{rema geo int cal V} is equivalent to say that the Cantor sets
$\{a^\su_p : \su\in \arr \sA_{\boxdot_j}\}$  and   $\{b^\ss _p : \ss\in \avv \sA_{\boxdot_j}\}$ intersect for all $p\in \cP$ and {$1 \le j \le N-1$}. This robust intersection will be obtained using the following well-known notion:
\begin{definition}[Thickness of a Cantor set] Given a Cantor set $K\subset \R$, a \emph{gap} of $K$ is a connected component of $\R\setminus K$. Given a bounded gap $G$ of $K$ and $x$ in the boundary of $G$, the \emph{bridge} $B$ of $K$ at $x$ is the connected component of $x$ in the complement of the union of the gaps larger or equal  than $|G|$. The thickness of $K$ at $x$ is
$\tau(K,u)=|B|/|G|$. The \emph{thickness} of $K$, denoted by $\tau(K)$ is the infimum of these $\tau(K,u)$ among all
boundary points $x$ of bounded gaps.
\end{definition}
The following is the celebrated Newhouse gap Lemma.
\begin{lemma}[{\cite[Lemma 4]{Ne79}}]\label{GapLemma} Let $K_1,K_2\subset \R$ be Cantor sets with thicknesses $\tau_1$ and $\tau_2$. If $\tau_1\cdot \tau_2>1$, then one of
the three following possibilities occurs: $K_1$ is contained in a component of $\R\setminus K_2$, $K_2$ is contained in a component of $\R\setminus K_1$ or
$K_1\cap K_2\not= \emptyset$.
\end{lemma}
We are now ready to prove:
\begin{proposition}\label{exam est AC}
For every even integer $N\ge 2$ and $\delta  >0$ smaller than a positive function of  $N$,  the family $(F_p, \pi_p)_{p\in \cP}$ of  \cref{simple example of adapted} is an unfolding of wild type $(\sA, \sC)$ which is moderately dissipative.
\end{proposition}
\begin{proof} We recall that $\sA= \{\sa_j: 1\le j\le N\}$ and the hyperbolic transformations are 
\[Y^{\sa_j}:=  I_j \times I\qand F^{\sa_j}:= (x,y)\in  Y^{\sa_j}\mapsto (C_j^{-1}(x), \sqrt{\delta} \cdot C_j(y)) \]
with $ I_j:= \left[\frac 2N (j-1)-1+\delta^2, \frac 2N j-1-\delta^2\right]$ and  $C_j$ the affine, orientation preserving map from $I$ onto $I_j$.
We notice that with $\epsilon=0.1$, if $\delta$ is sufficiently small, then for every $\sa\in \sA$, $$\max_{p\in \cP ,\, z\in Y^{\sa}} \|D_zF^\sa\| \cdot   |\det \, D_zF^\sa|^\epsilon=
(\frac1 N-\delta^2)^{-1} \delta ^{\epsilon/2}<1\; . $$
Thus  $F^\sA$ is moderately dissipative.  Also the family of systems $(F_p, \pi_p)_{p\in \cP}$  of  type $(\sA,\sC)$ satisfies property $(\mathbf H_2)$. Indeed, for $ (\su_j, \ss_j)_j\in \prod_{i=1}^{N-1} \arr \sA_{\boxdot_j}\times \avv \sA_{\boxdot_j}$, by \cref{form pli} of \cref{simple example of hyperbolic map of type A2}, we have for every $ 1 \le  j \le N-1 $:
\begin{equation} \label{calVexmple} 
{  a^{\su_j}= y^{\su_j}+ p_j \; ,\quad b^{\ss_j}_p =  x^{\ss_{j}} } \qand 
{\cb \cal V(\su_j,\ss_{j+1},p)}= y^{\su_j}-x^{\ss_{j+1}}+ p_j \; .
\end{equation}
Thus  the map $ p\in \cP \mapsto {\cb (\mathcal{V}(\su_{1},\ss_{2},p), \cdots,\mathcal{V}(\su_{N-1},\ss_{N},p))}\in \R^{N-1}$ is an affine map of linear part $\mathrm{Id}$, and so a diffeomorphism.

Let us show that the regular family $(F_p, \pi_p)_{p\in \cP}$ of systems satisfies $(\mathbf H_1)$. To this end, we shall study the following Cantor sets:
\[K_s:= \{x^\ss: \ss\in \avv \sA\}\qand 
K_u:= \{y^\su: \su\in \arr \sA\}\]

The Cantor set $K_s$ is the limit set of the IFS given by the contractions $(C_j)_{j=1}^{N}$. The Cantor set $K_u$ is   the limit set of the IFS given by the contractions $(\sqrt{\delta}\cdot C_j)_{j=1}^{N}$. It suffices to show that for every $p=(p_j)_j\in \cP$ and $1\le k\le N-1$, the Cantor set $\sqrt{\delta}\cdot C_k(K_u)+p_k$ intersects 
$K_s$. To show this we compute the thicknesses of these Cantor sets and their diameters. To this end, it is worth to recall:
\begin{fact}\label{rema4thickness}The intervals $(I_j)_j$ are disjoint and display the same length $2/N-2\delta^2$. The complement of their union in $[-1+\delta^2,1-\delta^2]$ is made by intervals of length $2\delta^2$.
The left endpoint of $I_1$ is $-1+\delta^2$ and the right endpoint of $I_N$ is $1-\delta^2$. The component of $0$ in $I\setminus \bigcup_{j=1}^{N} I_j$ is $V^\boxdot :=(- \delta^2,\delta^2)$.
\end{fact} 
We are ready to show:
\begin{lemma} \label{thicknesss} 
The Cantor set $K_s$ is at distance $ \frac{N\delta^2}{N- 1 +N\delta^2}\sim \frac{\delta^2 N}{N- 1}$ from the boundary of $I$ and its thickness is $\tau(K_s):=\frac{N- 1 +N\delta^2}{\delta^2 N^2}-1\sim \frac{N-1}{\delta^2 N^2}$ when $\delta\to 0$.
\end{lemma}
\begin{proof} 
Indeed, the Cantor set $K_s$ is the limit set of the IFS given by the contractions $(C_j)_{j=1}^{N}$ and each $C_{j}$ is affine, so we only need to compute the ratio $|B|/|G|$ for the $(N-1)$ gaps between $C_{j}(I)=I_{j}$ and $C_{j+1}(I)=I_{j+1}$ for $1 \le j \le N-1$.
By Fact \ref{rema4thickness}, each of these gaps has length:
\begin{equation} \label{gap size} |G|:=2\delta^2 \sum_{j\ge 0} (\frac1N -\delta^2)^j = \frac{2\delta^2}{1- \frac1N +\delta^2}=\frac{2\delta^2 N}{N- 1 +N\delta^2} =2d(K_s, I^c) \; .\end{equation}
The bridges between these gaps have length $|B|:=\frac2N -|G|$. Then the proposition follows from the quotient of these two equalities:
$$\tau(K_s):= \frac2{N|G|} -1=\frac{N- 1 +N\delta^2}{\delta^2 N^2}-1\; .
$$
\end{proof}
\begin{remark}\label{taille troue}  By \cref{gap size}, for every $j$, the lengths of gaps of $K_s^j:= K_s\cap I_j$ are $\le \frac{2\delta^2 N}{N- 1}\cdot (\frac1N -\delta^2)\sim \frac{2\delta^2}{N-1}$. Also $d(I_j^c, K_s^j)\sim (\frac1N -\delta^2) \cdot \frac{\delta^2 N}{N- 1}\sim \frac{\delta^2 }{N- 1}$.
\end{remark}

\begin{lemma} \label{thicknessu} 
The Cantor set $K_u$ has diameter equivalent to $2
\sqrt{\delta}\frac{N-1}{N}$ and
its thickness is equivalent to $\sqrt{\delta}\frac{N-1}{N}$ when $\delta\to 0$.
\end{lemma}
\begin{proof} The Cantor set $K_u$ is the limit set of the IFS given by the contractions $(\sqrt{\delta} \cdot C_j)_{j=1}^{N}$ and each $\sqrt{\delta} \cdot C_{j}$ is affine. Let $[-a,a]$ be the convex hull of $K_u$. We observe that $a$ is the unique fixed point of $\sqrt{\delta} \cdot C_N$, whereas $-a$ is the fixed point of $\sqrt{\delta}\cdot C_1= x\mapsto -\sqrt{\delta}\cdot C_N(-x)$. Since by definition $C_{N}$ is the orientation preserving affine map sending $I$ to $I_{N}$, this yields that $\sqrt{\delta} \cdot C_N(x)=\sqrt{\delta}
((\frac1N-\delta^2)(x-1)+1-\delta^2)$ for $x \in I$. Thus the unique fixed point of $\sqrt{\delta} \cdot C_N$ is:
$$a:=\sqrt{\delta}
(1-\frac1N)\cdot (1-
\frac{\sqrt{\delta}}N+\delta^{5/2})^{-1} \sim
\sqrt{\delta}\frac{N-1}{N}\quad \text{when }\delta\to 0\; .$$
This implies the estimate on the diameter of $K_u$. To compute the thickness, we only need to compute the ratio $|B|/|G|$ for the $(N-1)$ gaps between $\sqrt{\delta} \cdot C_{j}([-a,a])$ and $\sqrt{\delta} \cdot C_{j+1}([-a,a])$ for $1 \le j \le N-1$.
We denote $\tilde I_j = \sqrt{\delta} \cdot C_{j}([-a,a]) \subset [-a, a]$. We observe that each $\tilde I_j$ is of length $b:= 2a \cdot \sqrt{\delta} (\frac1N-\delta^2)$.
There are $N-1$ gaps between the segments $\tilde I_j $ which are equal. Thus their size is {$\frac{2a- N b}{N-1}$}. Thus the thickness is:
\[{ b\cdot (\frac{2a- N b}{N-1})^{-1}= (N-1)\frac{2a \cdot \sqrt{\delta} (\frac1N-\delta^2)}{2a- N \cdot 2a \cdot \sqrt{\delta} (\frac1N-\delta^2)} } \sim  \sqrt{\delta}\frac{N-1}N\; .
\]
\end{proof} 
For $1 \le j \le N$, let us denote:
\begin{equation}
\label{def Kj}
K_s^j:=K_s \cap I_j= C_j(K_s)\qand K_u^j:= \sqrt{\delta} \cdot C_j(K_u) \; .\end{equation}
The sets $K_s^j$ and $K_u^j$ are Cantor sets of the same thickness as respectively $K_s$ and $K_u$. Hence by \cref{thicknessu,thicknesss}, the product of the thicknesses of $K_s^j$ and $K_u^j$ is greater than 1.
By Remark \ref{taille troue}, all the gaps of $K_s^j$ are smaller than $\sim \frac{2\delta^2}{N-1}$ and $\sim \frac{\delta^2 }{N- 1}$ distant to the boundary of $I_j$. By \cref{thicknessu}, this is small compared to the diameter of $K_u^j$ which is equivalent to $\sim 2 \frac{\sqrt{\delta} (N-1)}{N} {\cdot \sqrt{\delta}
(\frac1N-\delta^2)}$.
As $ K_u^j + p_{j} $ is included in $I_{j+1}$, by the gap \cref{GapLemma}, it must intersect $K_s^{j+1}$ for every $p\in \cP$. This proves Property $(\mathbf H _ 1)$. 
\end{proof}
 
The following remark together with \cref{Main theorem} could have been used  to prove the existence of a polynomial automorphism of $\C^2$ with a wandering Fatou component using systems of type $(\sA, \sC)$ with simple combinatorics (full shift).   We just give a sketch of proof of the remark since in the next subsection we will prove the existence of a stronger example of family of automorphisms satisfying the assumption of \cref{Main theorem}, where we manage to bound the degree by 6. This other example will enable to prove Theorem~\ref{main wandering}. Nevertheless, this will require the full generality of $(\sA, \sC)$  systems.

\begin{remark}\label{for sketch}
The moderately dissipative unfolding $(F_p)_p$ of wild type $(\sA, \sC)$ given by  \cref{simple example of adapted} and \cref{exam est AC} with $\delta $ small and $N=6$, can be perturbed to be left invariant by a family of polynomial automorphisms. More precisely,  from \cref{simple example of adapted} and \cref{exam est AC} one can deduces the existence of a real  family of real polynomial automorphisms (of unknown degree) which satisfies  the $C^\omega$-geometric model \ref{geo model} and so the assumption of  \cref{Main theorem}. 
 \end{remark}
\begin{proof}[Sketch of proof of \cref{for sketch}]
 The map $F_p$ has constant Jacobian determinant for every $p$. Thus by   Dacarogna-Moser theorem \cite{DM90}, it is possible to extend the family  $(F_p)_p$ to a family $(f_p)_p$  of diffeomorphisms  $f_p$ of $\R^2$ such that $(f_p)_p$  satisfies the geometric model when $N=6$. Thus the map $f_p$ has a wandering stable component at a dense set of parameter $p$.   
Moreover, it is possible to show the existence of a neighborhood $\mathcal U$ of $(f_p)_{p\in \cP}$ formed by families displaying the geometric model (the proof is not complicated but rather boring since many items must be checked) via an embedding $H_p$ equal to the canonical inclusion: $Y^\se\times \{\so\}\hookrightarrow \R^2$. Then by noting that $(\delta^{-1/4}\cdot f_p)_{p\in \cal P}$ has constant Jacobian equal to $1$, by Turaev's Theorem \cite[Theorem 2 and Remark  1]{Turaev_2002}, this family can be approximated by a  smooth family of compositions of generalized Hénon-maps  {and  thus by a smooth family of compositions of polynomial generalized Hénon-maps}  $(Q_p)_{p\in \cP}$. 
These are polynomial automorphisms of $\R^2[X, Y]$. Remark that $\tilde f_p:=  \delta^{1/4} \cdot  Q_p$ is a polynomial automorphism whose restriction to $[-1,1]^2$ is $C^2$-close to $f_p$. Moreover $(\tilde f_p|[-1,1]^2)_p$ is $C^2$-close to $(f_p|[-1,1]^2)_p$ and so it displays the geometric model via a smooth family of  analytic embeddings. 
\end{proof}

\subsection{Natural examples satisfying the geometric model}
\label{section Density of wandering domain in the Newhouse domain}
In this subsection, we are going to state \cref{premain4application} from which we will deduce \cref{Newhouse model2,Newhouse model}.
\cref{premain4application} states that an unfolding of a wild type is embedded into some iterate of any $d$-parameter family which displays a non-degenerate unfolding of $d$-quadratic homoclinic tangencies:
\begin{definition}[Non-degenerate unfolding]\label{Non-degenerate unfolding}
Let $d\ge 1$, $\hat \cP$ an open set of $\R^d$, and $(f_p)_{p\in \hat \cP}$ be a $C^2$-family of surface diffeomorphisms leaving invariant a continuation\footnote{See \cref{Przy} \cpageref{Przy}.} $(K_p)_{p\in \hat \cP}$ of hyperbolic basic sets. 
Given $p_0\in \hat \cP$ and $d$ points of quadratic tangencies between the stable and unstable manifolds of $K_{p_0}$, their \emph{unfolding along $(f_p)_p$ is non-degenerated} if the matrix with the following coefficients $(a_{i,j})_{(i,j)\in\{1,\dots, d\}^2}$  is invertible: $a_{i,j}$ is the derivative following the  $j^{th}$ -coordinate of $\R^d=T_{p_0}\cP$ of the relative position of the local stable and unstable manifolds  associated to  the $i^{th}$-quadratic tangency point.\end{definition}

We recall that a periodic point $P$ of period $n$  is non-conservative if $|\det D_Pf^n|\neq 1$. 
Here is the general result enabling the aforementioned applications:
\begin{theorem}\label{premain4application}
Let $r \in [2,\infty] \cup \{\omega\}$, let $d\ge 1$, let  $\hat \cP$ be a nonempty open subset of $\R^d$ and let 
$(f_p)_{p\in \hat \cP}$ be a   $C^r$-family of surface diffeomorphisms.
Given $p_0\in \hat \cP$, assume that $f_{p_0}$ has a non-conservative\footnote{This assumption is actually not necessary using Duarte's theorem \cite{Du08}, but sufficient for all our applications needing the moderate dissipativeness assumption.}  periodic saddle point $P $ displaying $d$ different quadratic homoclinic tangencies.

 If these quadratic tangencies unfold non-degenerately with $(f_p)_p$, then there exist $k\ge 1$,  a regular compact subset $\cP\subset \hat \cP$ arbitrarily close to $p_0$, an unfolding $( F_p, \pi_p)_{p \in \cP}$ of a certain wild type $(\sA,\sC)$ with $\Card \sC=d$,  $(F_p)_p$  being of class $C^r$ and a $C^r$-family of embeddings $(H_p)_p$  such that
\[f^k_p\circ H_p = H_p\circ F_p\quad \forall p\in \cP .\]
Moreover $ F^\sA_p$ is moderately dissipative for every $p \in \cP$   if $P$ is dissipative. 

 Furthermore, for every $C^r$-perturbation $(\tilde f_p)_{p\in \hat {\cal P}}$  of $( f_p)_{p\in \hat {\cal P}}$ the same conclusion holds true for a same regular compact subset $\cP\subset \hat \cP$. 
\end{theorem}
\begin{remark}\label{main4application} Hence if $\Card \sC= 5$, the family  $(f_p)_{p\in \cP}$ satisfies the $C^r$-geometric model \ref{geo model}.
\end{remark}
The proof of this theorem is given in the next subsection \ref{section:premain4application}. Let us deduce now rather quickly \cref{Newhouse model2,Newhouse model} and Theorem \ref{Thepremain4application}.
\begin{proof}[Proof of Theorem \ref{Thepremain4application}] This is an immediate consequence of  \cref{consequence} and \cref{main4application}. 
 (Remind that \cref{consequence} follows from \cref{main them Fp} whose proof will be given in \cref{parameterselection}).
\end{proof}

  \begin{proof}[Proof of \cref{Newhouse model2} and \cref{rk main wandering2}]
 Let us consider the Tchebychev polynomial of degree 6:
 \[P_0(X)= 32\cdot X^6-48\cdot X^4+18\cdot X^2-1\]
  The critical points and values of this polynomial are respectively:
  \[(\zeta_{i\, 0})_{1\le i\le 5}= (-\sqrt{3}/2,  -1/2, 0, 1/2, \sqrt{3}/2)\qand (a_{i\, 0})_{1\le i\le 5}=(-1,1,-1,1,-1).\]
   They are all quadratic.  All critical values are sent to $\beta_0=1$ which is a repelling fixed point of $P_0$.  We notice that $[-\beta_0, \beta_0]$ is the complement of the basin of $\infty$. For $p=(p_0,\dots, p_4)\in \R^5$, we consider:
   \[P_p(X)= P_0(X)+\sum_{j=0}^4 p_j\cdot X^j\; .\] 
Let $\beta_p$, $(\zeta_{i\, p})_{1\le i\le 5}$ and $(a_{i\, p})_{1\le i\le 5}$ be  the respective  continuations of the fixed point $\beta_0$, of the critical points  and of the critical values for $P_p$. 
We have the following:
\begin{lemma} \label{tangences deployees}
The unfolding of the images $P_p(a_{i\, p})$ of the critical values $a_{i\, p}$ w.r.t.  the fixed point $\beta_p$ is non-degenerated:
\[ \det\, [\partial_{p_j} (P_p(a_{i\, p}))-\partial_{p_j} \beta_p ]_{1 \le i \le 5,   0 \le j \le 4} \neq 0  \quad \text{ at } p=0   \, . \] 
\end{lemma}

\begin{proof}
One way to prove this lemma is to use Epstein's transversality (see \cite{epstein}, Theorem 1.1 and the subsequent Remark). Let us give a direct proof. We have $DP_0(\beta_0) = 36$ and $\partial_p P_p (X) =  \sum_{j=0}^4 X^j \cdot dp_j$. Thus we have $\big({  D P_0 (\beta_0) \cdot \partial_p \beta_p + \partial_p P_p(\beta_0) \big)_{|p=0} } = (\partial_p \beta_p)_{|p=0}$ and so:  \[(\partial_p \beta_p)_{|p=0}= -\frac{1}{35} \sum_{j=0}^4 dp_j \; .\]
Also $(\partial_p a_{i\, p})_{|p=0}=  (DP_0(\zeta_{i\, 0})\cdot \partial_p \zeta_{i\, p}+ \partial_p P_p(\zeta_{i\, 0}))_{|p=0}=(\partial_p P_p(\zeta_{i\, 0}))_{|p=0}$ implies:
\[ (\partial_p a_{i\, p})_{|p=0}=  \sum_{j=0}^4  (\zeta_{i\, 0})^j\cdot dp_j\qand     ( \partial_p  P_p( a_{i\, p}))_{|p=0}=DP_0 (a_{i\, 0}) \cdot (\partial_p a_{i\, p})_{|p=0} + ( \partial_p P_p( a_{i\, 0}))_{|p=0} \]
and so:
\[ ( \partial_p P_p( a_{i\, p}))_{|p=0} = (-1)^i \cdot  36 \cdot  \sum_{j=0}^4  (\zeta_{i\, 0})^j\cdot dp_j + { \sum_{j=0}^4 (-1)^{i\cdot j} \cdot  dp_j    }   \, . \] 
Then to show the lemma, it suffices to see the invertibility of the $5\times 5$-matrix whose $(i,j)$ entry is $[\partial_{p_j} (P_p(a_{i\, p}))-\partial_{p_j} \beta_p ]=
[(-1)^i \cdot  36 (\zeta_{i\, 0})^j+(-1)^{i\cdot j}+1/35]$.
\end{proof}

  For $b\ge 0$ small and $p\in \R^5$ small, we consider the Hénon map:
     \[f_{p\, b}: (x,y)\mapsto  (P_p(x)-y, b\cdot x)\]
For $b=0$, the family $(f_{p\, 0})_{p\in \R^5}$ is semi-conjugate to $(P_{p})_{p\in \R^5}$ via the maps $(x,y)\mapsto P_p(x)-y$, while for $b>0$, the maps  $f_{p\, b}$ are polynomial automorphisms. 
For $b=0$, the point $\beta_{0\, 0}:=(\beta_0,0)$ is fixed and  hyperbolic for $f_{0\, 0}$. Thus it persists  as a hyperbolic fixed point $\beta_{p\, b}$ of $f_{p\, b}$ for $b>0$.  We notice that $W^u_{loc} (\beta_{0\, 0}; f_{0 \, 0}):= (-0.9, 1.1)\times \{0\}$ is a  local unstable manifold of  $\beta_{0 \,  0}$: there exists an arbitrarily small neighborhood of $\beta_{0 \, 0}$ in $\R \times \{0\}$ which is sent diffeomorphically to $W^u_{loc}(\beta_{0 \, 0};f_{0 \, 0})$ by an iterate of $f_{0 \, 0}$. 
For $p$ and $b$ small, $W^u_{loc} (\beta_{0\, 0}; f_{0 \, 0})$ persists as a curve  $W^u_{loc} (\beta_{p\, b}; f_{p\, b})$  $C^2$-close to $W^u_{loc} (\beta_{0\, 0}; f_{0 \, 0})$. On the other hand, the set 
\[W^s_{loc} (\beta_{0\, 0}; f_{0\, 0}):= \{(x,y)\in [-2,2]^2: P_0(x)- y= \beta_{0} \}\]
is a local stable manifold of $f_{0\, 0}$. For $p$ and $b$ small, it persists as a curve $W^s_{loc} (\beta_{p\, b}; f_{p\, b})$  $C^2$-close to  $W^s_{loc} (\beta_{0\, 0}; f_{0\, 0})$.  As the critical points $(\zeta_{i\, 0},0)_{1\le i\le 5}$ belong to  $W^u_{loc} (\beta_{0\, 0}; f_{0\, 0})$ and are sent by  $f_{0\, 0}^2$ to quadratic tangency points between  $f_{0\, 0}^2(W^u_{loc} (\beta_{0\, 0}; f_{0\, 0})$ and $W^s_{loc} (\beta_{0\, 0}; f_{0\, 0})$ which unfold non-degerately by \cref{tangences deployees}, for $b$ small there is a proper compact subset $\hat \cP$ close to $0$ and $p_b\in \hat \cP$ close to $0$ such that 
$f^2_{p_b\, b}(W^u_{loc} (\beta_{p_b\, b}; f_{p_b\, b})$ and $W^s_{loc} (\beta_{p_b\, b}; f_{p_b\, b})$ have five points of quadratic tangency $(f^2_{p_b\, b}(\tilde \zeta_{i\, b}))_{1\le i\le 5}$ which unfold non-degenerately with $(f_{p\, b})_{p\in \hat \cP}$. Moreover 
$(\tilde \zeta_{i\, b})_{1\le i\le 5}$ are close to  $ (( \zeta_{i\, 0},0))_{1\le i\le 5}$. Hence we can apply \cref{main4application} of \cref{premain4application} with $P=\beta_{p\, b}$ to deduce \cref{Newhouse model2}.
The same argument proves \cref{rk main wandering2}.\end{proof}

We recall that that a point $x$ of a hyperbolic compact set $K$ displays a homoclinic tangency if its stable manifold is tangent to its unstable manifold.    \cref{premain4application} enables also to deduce rather quickly {\cref{Newhouse model}} on the Newhouse domain:
\begin{definition}\label{def newhouse domain}  Let $r \in [2,\infty] \cup \{\omega\}$.     A basic set $K$ for a surface diffeomorphism $f$ is $C^r$-\emph{wild} if for every $C^r$-perturbation $\tilde f$ of $f$, there is a point in the hyperbolic continuation of  $ K$ for $\tilde f$ displaying a quadratic homoclinic tangency. 
The \emph{dissipative $C^r$-Newhouse domain} $\cal N^r$ \index{$\cal N^r$} is the open set of   surface diffeomorphisms leaving invariant a hyperbolic basic set displaying a $C^r$-robust quadratic homoclinic tangency and an area contracting periodic point\footnote{a $q$-periodic point $P$ satisfying $|\det Df^q(P)|<1$.}.
\end{definition} 
The existence of wild hyperbolic basic sets was proved by Newhouse \cite{Newhouse} by extending the concept of thickness to hyperbolic basic sets: 
\begin{definition}[{Stable and unstable thicknesses}]The \emph{stable thickness} $\tau_s(K)$ of a horseshoe  $K$ for a dynamics $f$  is the infimum of the thicknesses of $K\cap W^u_{loc}(z)$ among $z\in K$. The \emph{unstable thickness} $\tau_u(K)$ of $K$ is is the infimum of the thicknesses of $K\cap W^s_{loc}(z)$ among $z\in K$.  

A point $z\in K$  \emph{does not bound an unstable gap} if $z$ does not bound a gap of the Cantor set  $W^s_{\loc}(z)\cap K$ of $W^s_{\loc}(z)$. A point $z\in K$  \emph{does not bound a stable gap} if $z$ does not bound a gap of the Cantor set  $W^u_{\loc}(z)\cap K$ of $W^u_{\loc}(z)$. \index{Stable and unstable thicknesses}\index{Stable and unstable gaps}
\end{definition}
Here is Newhouse's theorem implying the existence of wild basic sets:
\begin{theorem}[Newhouse \cite{Ne79}]
\label{classical}
Let $f$ be a surface $C^2$-dynamics having a hyperbolic basic set $\Lambda$ containing two basic sets  $\Lambda_1$ and $\Lambda_2$ satisfying the following properties:
\begin{enumerate}
\item $\tau_u(\Lambda_1)\cdot \tau_s(\Lambda_2)>1$.  
\item  There is a point  $P_1\in \Lambda_1$ which does not bound an unstable gap of $\Lambda_1$ and a point $P_2\in \Lambda_2$ which does not bound a stable gap of $\Lambda_2$,
such that a local unstable manifold  $W^u_{loc}(P_1)$ has a quadratic tangency with a local stable manifold $W^s_{loc}(P_2)$ at a point $T$, 
\end{enumerate}
then for any $C^2$-perturbation $\tilde f$  of $f$, the continuation of $\Lambda$ has local unstable and stable manifolds  $W^u_{loc}(\tilde P_1)$ and $W^s_{loc}(\tilde P_2)$   close to respectively  $W^u_{loc}(P_1)$ and $W^s_{loc}(P_2)$ and  displaying a quadratic tangency at a point $\tilde T$ close to $T$. 
\end{theorem}
\begin{proof} Let us briefly recall Newhouse's argument and see how we obtain that $\tilde T$ is close to $T$ (this is not explicitly stated in \cite{Ne79} but it is an easy consequence of his construction). Indeed, there exists a foliation $\cal F$ whose tangent space is of class $C^1$ and  which  extends the stable lamination of $\Lambda$  (see \cref{extension lam})  on a neighborhood $L$ of $\Lambda$. Up to looking at an iterate, we can assume $T\in L$. We use a local unstable manifold $W$ which is transverse to this foliation  and use it to define the holonomy   $\mathscr {hol}$ from a neighborhood of $T$ in $L$ into $W$   along $\cal F$. The image of the stable lamination of {$\Lambda_2$  is a Cantor set $K_2$ with thickness at least  $\tau_s(\Lambda_2)$. Nearby $T$, the tangency points between the leaves of $\cal F$ and the unstable lamination of $\Lambda_1$  is sent by $\mathscr {hol}$ onto a Cantor set $K_1$ of thickness locally at least close to $\tau_u(\Lambda_1)$.} 
By assumption $K_1$ intersects $K_2$ at $\mathscr {hol}(T)$  which does not bound a gap of $K_1$ nor of $K_2$. Thus any perturbation of $K_1$ and $K_2$ cannot be included in the gap of the one other. When we perturb $f$ the tangent space to the foliation $\cal F$ persists to one which is  $C^1$-close. Thus  the compact sets $K_1$ and $K_2$ persist. By \cite[Proposition 6]{Ne79} their thicknesses vary also continuously and so by the Gap \cref{GapLemma} we get that $\tilde T$ is close to $T$.  
\end{proof}
Newhouse showed also that wild hyperbolic sets appear nearby homoclinic tangencies:
\begin{theorem}[{Newhouse \cite[Lemma 7 and Theorem 3]{Ne79}}]\label{Ne79}
Let $(f_t)_{t \in (-1,1)}$ be a $C^r$-family of surface diffeomorphisms such that $f_0$ has a homoclinic quadratic tangency associated to a $n_0$-periodic point $P_0$, which unfolds non-degenerately and such that $|\det Df^{n_0}_0(P_0)|< 1$.
Then arbitrarily close to $0$, there exists a nonempty open interval $I$ and there is a continuation $(K_t)_{t\in I}$ of hyperbolic basic sets  such that for every $t\in I$, 
the basic set $K_t$ has a homoclinic tangency which is non-degenerately unfolded  and  $K_t$ contains the continuation $P_t$ of $P_0$.
\end{theorem} 
To show this,  Newhouse constructed a first horseshoe  $\Lambda_1(t)\ni  P_t$ with unstable thickness  $\tau_u>0$ (see  \cite[Proposition 6]{Ne79}) which persists for $t$ in an open interval $\hat I$ nearby $0$. Then for any $\tau>0$ he found for   a certain $t\in \hat I$  a second horseshoe $\Lambda_2(t)$ such that:
\begin{itemize}
\item the stable thickness of $\Lambda_2(t)$ is $>\tau$ (see \cite[$  \ell  \, 16$  Page 124]{Ne79}). 
\item A local unstable manifold of $\Lambda_1(t)$ has a quadratic homoclinic tangency with a local  stable manifold of $\Lambda_2(t)$ which is non-degenerately unfolded by $(f_t)_t$ (see \cite[Lemma 7]{Ne79}).
\item the basic set $\Lambda_1(t)$ and $\Lambda_2(t)$ are included in a basic set $K_t$ (see \cite[Lemmas 7-8]{Ne79}).
\end{itemize}
Then \cref{Ne79} is a consequence of this construction with $\tau>1/\tau_u$ and of \cref{classical}. Taking $\tau$ large enough so that both $\tau > \max(\tau_u,  (\tau_u^2+3\tau_u+1)/\tau_u^2)$ and $\tau_u > (2\tau+1)^2/\tau^3$ hold true, the main theorem of \cite{Kr92,HKY93} implies furthermore:

\begin{corollary}\label{Ne79bis} Under the assumptions of \cref{Ne79}, arbitrarily close to $0$,  there exists a nonempty open interval $I$ and there is a continuation of a hyperbolic basic set $(K_t)_{t\in I}$ containing two basic sets 
$\Lambda_1(t)\ni P_t$ and $\Lambda_2(t)$ such that 
\begin{enumerate}[(i)]
\item $\tau_u(\Lambda_1(t))\cdot \tau_s(\Lambda_2(t))>1$, moreover $K_t$ is a Cantor set,
\item  there is a Cantor set of quadratic homoclinic tangencies between the leaves of an unstable lamination $W^u(\Lambda_1(t)) $ and a stable lamination $W^s(\Lambda_2(t)) $ which unfold non-degenerately.
\end{enumerate}
\end{corollary}
\begin{remark}\label{NH4moderate dissipative} Using modern techniques (see \cite[Lemma 1]{GST08}, \cite[Example 2.3]{berger2018zoology}, and historically  \cite[Eqs. (3.2) and (3.3)]{TLY}),  the set $\Lambda_2(t)$ constructed by Newhouse is given by a renormalization of a horseshoe of a perturbation of $(x,y)\in [-4,4]^2 \mapsto (x^2-2, 0)$. Thus we can assume that $f_t$ is asymptotically moderately dissipative at  $\Lambda_2(t)$:   \[\limsup_{n\to \infty}  \|\det Df^n|\Lambda_2(t)\|^{1/\epsilon} \cdot \|Df^n| \Lambda_2(t)\|<1\;  .\]
\end{remark}

We are now ready to show that the geometric model appears typically in the Newhouse domain:
\begin{proof}[Proof of \cref{Newhouse model}]
Let $f$ be a $C^r$-surface diffeomorphism having a {dissipative} saddle periodic point $P$ displaying a quadratic homoclinic tangency at a point $T$.  
We first apply \cref{propotangences} given below with $d=5$ to perturb $f$ into $\tilde f$ so that  the map $\tilde f $ has 5 different quadratic tangencies between the stable and unstable  manifolds of a dissipative periodic saddle point $z$, which unfold non-degenerately with $(f_p)_{p\in \R^5}$ for any $(f_p)_{p\in \R^5}$ in a $C^r$-open and dense set of $C^r$-families containing $\tilde f$ : $f_0=\tilde f$. To conclude, we apply \cref{main4application} of \cref{premain4application} to this open and dense set of 5-dimensional unfoldings of $\tilde f$. This gives a regular compact set $\cP$ arbitrarily close to $0$ such that $(f_p)_{p\in \cP}$ satisfies the geometric model. Finally we reparametrize the family so that $0$ belongs to the interior of $ \cP$.
\end{proof} 
\begin{proposition}  \label{propotangences}  
Let $f$ be a $C^r$-surface diffeomorphism having a {dissipative} saddle periodic point $P$ displaying a quadratic homoclinic tangency at a point $T$. Then for any $d > 1$, there exists $\tilde f$  $C^r$-close to $f$  such  that $\tilde f $ has $d$ different quadratic homoclinic tangencies associated to some dissipative periodic saddle point $z$. These tangencies unfold non-degenerately with $(f_p)_{p\in \R^d}$ for any $(f_p)_{p\in \R^d}$ in a $C^r$-open and dense set of $C^r$-families containing $\tilde f$: $f_0=\tilde f$.
\end{proposition}
\begin{proof}[Proof of \cref{propotangences}]
{We begin by a lemma which allows to unfold the initial tangency: }
\begin{lemma}\label{lemme premier de unfolding}
There exists a $C^r$-flow $(\phi^t)_t$ such that the quadratic homoclinic tangency of $P$ is non-degenerately unfolded 
by  $(\phi^t\circ f)_t$.
\end{lemma}
\begin{proof} For the sake of simplicity we assume that $P$ is a fixed point; the case where $P$ is $n\ge 2$ periodic is left to the reader.  
 First  we choose a neighborhood $N$ of $ T$ sufficiently small to intersect 
$W^s_{loc}(P, f)\cup    W^u_{loc}(P, f)$ 
only at the union of  two small neighborhoods of $ T$ in  $W^u_{loc}( P, f)$ and $ W^s_{loc}( P, f)$.  
Now we consider a $C^\infty$-vector field $\chi$ supported by $N$ such that $\chi(T)$ is not in the tangent space of $W^u_{loc}(P, f)$  (or equivalently of  $W^s_{loc}(P, f)$) at $T$ . 
Let $(\breve \phi^t)_t$ be the flow of $\chi$.
We notice that   the family $(\breve \phi^{t }\circ f)_{t}$ unfolds non-degenerately the homoclinic tangency at $T$.  This gives the fact when $r\neq \omega$. Otherwise, let $\tilde \chi$ be an analytic vector field which is $C^\infty$-close to $\chi$. Let $( \phi^t)_{t}$ be its flow. Then  the family $(\phi^{t }\circ f)_{t}$ unfolds non-degenerately the homoclinic tangency at $T$. 
\end{proof}

Now we use \cref{Ne79bis} which asserts the existence of $t_0\in \R$ small such that  $ \tilde  f := \phi^{t_0}\circ f$ has a hyperbolic basic set $K_0$   containing two basic sets  $\Lambda_1 $ and $\Lambda_2 $ such that 
\begin{enumerate}[(i)]
\item $\tau_u(\Lambda_1 )\cdot \tau_s(\Lambda_2 )>1$,
\item  there is a Cantor set of quadratic  tangencies between the leaves of  $W^u(\Lambda_1) $ and   $W^s(\Lambda_2) $
\end{enumerate}
Note that the map {$ \tilde   f $} is $C^r$-close to $f$ since $t_0$ is small. As a Cantor set is not countable, there exist $d$ points $(P_i)_{1\le i \le  d}$ with local unstable and stable  manifolds {$W^u  (P_i,  \tilde    f )$ and $W^s  (P_i, \tilde    f )$} displaying a quadratic tangency at a point $T_i$, such that the   orbits of $T_i$ and $T_j$ are disjoint when $j\neq i$. 

\begin{fact}  \label{factdeploi}
The set of $C^r$-families  $(f_p)_{p\in \R^d}$ such that $f_0=\tilde    f $ and for which the quadratic tangencies at   {$(T_i)_{1 \le i \le d }$} unfold non-degenerately is $C^r$-open and dense.
\end{fact}
\begin{proof}
Clearly the set of non-degenerate unfoldings  is $C^r$-open.  Let us show that it is $C^r$-dense. Let $(f_p)_{p\in \R^d}$ be a $C^r$-family containing { $\tilde    f $}. 
Let us show that it can be $C^r$-approximated by one for which the $d$ homoclinic tangencies unfold non-degenerately. 

A proof like the one of {\cref{lemme premier de unfolding}} shows that there exists $d$ $C^r$-flows $(\phi^t_i)_{1\le i\le d}$  such that with:
\[t=(t_j)_j\in {\R^{d}}\mapsto \phi^{t}:= \phi_d^{t_d}\circ \cdots \circ  \phi_1^{t_1}\] 
the family $(\phi^{t}\circ {  \tilde    f })_{p\in \R^{d}}$ unfolds non-degenerately the $d$ homoclinic tangencies at $(T_i)_{1\le i\le d}$. 
  This proof assumes first that   $1\le r\le \infty$, and takes the support of each flow $\phi_j$ disjoint and  supported by a small neighborhood of $T_j$. For the case $r=\omega$, we approximate each flow by an analytic one.  
  
Recall that given a matrix $A\in \mathscr M_d(\R)$ and a matrix $B\in \mathscr {Gl}_d(\R)$ there is an open and dense set of $\eta\in \R$ such that $A+\eta B$ is
 invertible. Thus for an open and dense set of small $\eta\in \R$, 
the $C^r$-perturbation  $(\phi^{\eta\cdot p}\circ f_{p})_{p\in \R^{d}}$ of $(f_{p})_{p\in \R^{d}}$  unfolds non-degenerately the  homoclinic tangencies at {$(T_i)_{1 \le i \le  d }$}.

\end{proof}
 Using the continuity of the parametrized  stable and unstable laminations \cref{cartecool4cpct} \cpageref{cartecool4cpct}, and the density of stable and unstable manifolds of a given periodic point $z\in K_0$ in them, up to a small parameter translation, we can assume that $P_i=z$ for every $i$.  Again  by \cref{cartecool4cpct}, the unfolding is still non degenerate. 
     \end{proof} 
\subsection{Proof of  \cref{premain4application}}\label{section:premain4application}
Let $P$ be a non-conservative periodic point  of  $f_{p_0}$ such that $W^u(P; f_{p_0})$ and  $W^s(P; f_{p_0})$ have $d$   quadratic tangencies at points $T_1, \dots, T_d$, which unfold non-degenerately with $(f_p)_p$. Up to inversing the dynamics, we can assume that $P$ is dissipative. 

 The aim of this subsection is to show that there exist $k\ge 1$,  a regular compact subset $\cP\subset \hat \cP$ arbitrarily close to $p_0$, a moderately dissipative unfolding {$( F_p, \pi_p)_{p \in \cP}$} of a certain  wild type $(\sA,\sC)$  with $\Card \sC=d$, {with $(F_p)_p$  being of class $C^r$},  and a $C^r$-family of embeddings $(H_p)_p$ such that:
\[ f^k_p\circ H_p= H_p\circ F_p\; .\]
In order to do so, we first construct a wild basic set $K$  with $d$ homoclinic tangencies which unfold non-degenerately, then we construct a family of  hyperbolic maps $(F_p^\sA)_p$ and finally we will  construct a family of folding maps $(F_p^\sC)_p$.\medskip

\noindent {\bf Step 1: Construction of the wild basic set $K$.} The following uses Newhouse's \cref{Ne79bis}.
\begin{lemma}\label{etape initial pour newhouse}
There is $p_1\in \hat \cP$ close to $p_0$ such that:
 \begin{enumerate}
 \item  The map $f_{p_1}$ has a hyperbolic basic set $K$ containing the continuation of $P$ and two basic sets $\Lambda_1$ and $\Lambda_2$ satisfying $\tau_u(\Lambda_1)\cdot \tau_s(\Lambda_2)>1$.  Moreover $K$ is a Cantor set. 
 \item There are $d$ points of quadratic tangencies between local unstable and stable manifolds 
 of points of $\Lambda_1$ and $\Lambda_2$ which do not bound  unstable and stable gaps and which unfold non-degenerately with $(f_p)_p$.
 \item $f_{p_1}$ is asymptotically moderately dissiptative at $K$: \[\limsup_{n\to \infty}  \|\det Df^n_{p_1}|K\|^{1/\epsilon} \cdot \|Df_{p_1}^n| K\|<1\;  .\]
 \end{enumerate}
\end{lemma}
\begin{proof}
Up to a coordinate change of the parameter set $\hat \cP$, we can assume that $p_0=0$ and that along the one-parameter family  $(f_t)_{t\in J}$, where  $ J:=\R\times \{0\}\cap \hat \cP$, 
the first homoclinic tangency of $P$  unfolds non degenerately while the $d-1$-other homoclinic tangencies remain tangent.  We notice that  with $[a_{i,j}]_{1\le i,j\le d}$ the matrix of the unfolding (see \cref{Non-degenerate unfolding}), we have $a_{1,1}\neq 0$ and {$a_{i,1}=0$ for every $2\le i \le d$}. As the matrix $[a_{i,j}]_{1\le i,j\le d}$ is invertible, it holds that $[a_{i,j}]_{2\le i,j\le d}$ is invertible. 
 Now by  \cref{Ne79bis}, there exists $t_0\in J$ small and a hyperbolic basic set $K$ containing   two basic sets $\Lambda_1$ and $ \Lambda_2 $ such that:
\begin{enumerate}
\item the set $\Lambda_1$ contains the hyperbolic  continuation of $P$,
\item  
$\tau_u(\Lambda_1)\cdot \tau_s(\Lambda_2)>1$.    
\item the  local unstable and stable manifolds 
 of $\Lambda_1$ and $\Lambda_2$ have uncountably many  points $T'_i$ of quadratic tangency which unfold non-degenerately with $(f_t)_{t\in J}$.
\end{enumerate} 
As there are only countably many gaps in a Cantor set, we can assume that $T'_1$ satisfies the no-ending gap condition (2) of \cref{classical}.

At $(t_0, 0)\in \hat \cP$, the matrix of the unfolding of $(T_1', T_2, T_3, \dots , T_d)$ is  
of the form $[a'_{i,j}]_{1\le i,j\le d}$ with $[a'_{i,j}]_{2\le i,j\le d}\approx [a_{i,j}]_{2\le i,j\le d}$ and $a'_{i,1}\neq 0$ iff $i=1$. Thus $[a'_{i,j}]_{1\le i,j\le d}$ is invertible and the unfolding of these $d$-homoclinic tangencies is non-degenerated. 
By density of the local unstable and stable manifolds of points in $\Lambda_1$ and $\Lambda_2$  which do not bound gaps, there is a parameter  $p_1\in \hat \cP$ nearby  $(t_0, 0)$ which displays   $d$ homoclinic tangencies at points $(T_1', T_2', T_3', \dots , T_d')$ between local unstable and stable manifolds of points of $\Lambda_1$ and $\Lambda_2$ which satisfy the no-ending gap condition (2) of \cref{classical}. Moreover by continuation, their  unfolding is  non-degenerate.  By \cite[Proposition 6]{Ne79} the unstable and stable thicknesses of these basic sets vary continuously. Thus at the parameter $p_1$, we still have $\tau_u(\Lambda_1)\cdot \tau_s(\Lambda_2)>1$. Then the following achieves the proof:\end{proof}
\begin{fact} We can assume that $f_{p_1}$ is asymptotically moderately dissipative at $K$.\end{fact}
\begin{proof}  
Otherwise, as the unstable and stable manifolds of  a given periodic point $\tilde P\in \Lambda_2$ is dense in $W^u(K)$ and $W^s(K)$, using the non-degeneracy of the unfolding, there is a parameter close to $p_0$ such that $\tilde P$ has $d$ homoclinic tangencies which unfold non-degenerately. By \cref{NH4moderate dissipative} , the point $\tilde P$ is moderately dissipative. So we can redo the construction starting with this point $\tilde P$ instead of $P$, which is moderately dissipative and so is its dynamics nearby.
\end{proof}
\begin{remark} \label{wu different}By the proof, we can assume furthermore that the unstable manifolds containing the tangency points have disjoint orbits.
\end{remark}

\noindent {\bf Step 2: Construction of families of hyperbolic maps $(F^\sA_p)_p$ and conjugacy maps  $(H_p)_p$.} The following is a general (but somehow classical) \cref{Horseshoe2model} which might be useful  in other contexts than the proof of \cref{premain4application}. Hence it will be proved in \cref{proof Horseshoe2model}.

Let $r \in [2,\infty] \cup \{\omega\}$, let $f_0$ be a $C^r$-diffeomorphism of a surface $M$ leaving invariant a horseshoe\footnote{A horseshoe is both a hyperbolic basic set and a Cantor set.} $K_0$. Let  $d\ge 0$,  $\hat \cP$ be an open neighborhood of $0\in \R^d$ and let  $(f_p)_{p\in \hat \cP}$ be a $C^r$-family of diffeomorphisms of $M$ containing $f_0$. Let $(K_p)_{p }$ be the hyperbolic continuation of   $K_0$.  

\begin{proposition}\label{Horseshoe2model}  There  exist $k\ge 1$ arbitrarily large, a regular compact neighborhood $\cP$ of $0\in \hat \cP$, an oriented graph $(\sV, \sA)$  and a     $C^r$-family $(H_{ p})_{ p\in \cP}$ of embeddings  $  H_{ p}: Y^\se\times \sV\hookrightarrow  M$  such that: 
\begin{enumerate}
\item      There exists a $C^r$-family $(F^\sA_p)_p$ of  hyperbolic maps $F^\sA_p:D_p(\sA)\to Y^\se\times \sV$ of type $\sA$ such that $H_p\circ F^\sA_{p}:=  f_p^k\circ  H_{ p}|D_p(\sA)$,  
 and the set  $H_p(D_p(\sA))$ contains $K_p$, 
\item if  $\limsup_{n\to \infty}  \|\det Df^n_{0}|K_0\|^{1/\epsilon} \cdot \|Df_{0}^n| K_0\|<1\;  $,  then $(F_{ p}^\sA)_p$ is moderately dissipative,
\item there is a $C^1$-family $(\pi_p)_{p\in \cP}$ of maps $\pi_p: \R\times I\times \sV\to   \R \times \sV$ which satisfies items (0)-(1)-(2)-(3) of \cref{def adapted proj} and such that $(\ker D\pi_p)_p$ is of class $C^1$.
\end{enumerate}
\end{proposition}

A simple consequence of the proof will be:
\begin{corollary}\label{coro Horseshoe2model} If under the assumptions of \cref{Horseshoe2model}, the family $(f_p)_{p\in \hat \cP}$ is of class $C^{2+}$, then the same conclusion holds true and moreover $(\pi_p)_p$ is of class $C^{1+}$. 
\end{corollary}

\noindent {\bf Step 3: Construction of the folding transformations.}
Let us come back to the proof of  \cref{premain4application}.
Let $p_1\in \hat \cP$ and let $K\supset \Lambda_1\sqcup \Lambda_2$ be  the basic set of $f_{p_1}$ given by \cref{etape initial pour newhouse}. Let $\cP_1$ be a neighborhood of $p_1$ for which this hyperbolic set persists as $(K_p)_{p\in \cP_1}$. 
Let $(T_i)_{1\le i\le d}$ be the $d$ tangencies points between the local unstable and stable  manifolds of $K$ satisfying the no-ending gap condition (2) of \cref{classical}.  Also it holds  $\tau_u(\Lambda_1)\cdot \tau_s(\Lambda_2)>1$. Thus by \cref{classical}, we have:
\begin{fact}\label{pour def de zeta} For every $p$ close to $p_1$, there are   $d$ tangencies points   between the local unstable and stable  manifolds of $K_p$ which are close to $(T_i)_{1\le i\le d}$.\end{fact}

Let $(F^\sA_p)_{p\in \cP_1}$ be the family of hyperbolic maps of type $\sA$ given by \cref{Horseshoe2model} for the continuation $(K_p)_{p\in \cP_1}$. Let $k\ge 1$ and let $(H_p)_p$ be the family of $C^r$-embeddings such that:
\[ {f^{k}_p} \circ H_p= H_p\circ { F^\sA_p} \; .\]
Let $\Lambda_p$ be the maximal invariant set of $F^\sA_p$. Let $(\pi_p)_p$ be the family of projections given by \cref{Horseshoe2model} (3). 
  We put also:
\[W^s(\Lambda_p):=  \bigcup_{\ss\in \avv \sA}W^{\ss}_{p}\times \{\si(\ss)\}\qand 
W^u(\Lambda_p):=  \bigcup_{\su\in \arr \sA}W^{\su}_{p}\times \{\st(\su)\}\]

Let us construct the  folding transformations. As each $T_i$ is in the unstable lamination of $K_{p_1}$, up to changing $T_i$ by  $f_{p_1}^{-m}(T_i)$ for every $i$, we can assume that each $T_i$ belongs to $H_{p_1}( W^{u}_{p_1}(\Lambda_{p_1})\cap \mathrm{int\, } Y^\se { \times \sV}) $. Let $\su_i\in \arr \sA$ be such that $H_{p_1}^{-1}(T_i)$ belongs to 
$ W^{\su_i}_{p_1}\times \{\st(\su_i)\}$. 
By \cref{wu different}, we have $\su_i\neq \su_j$ for $i\neq j$.  Put $W^u_{loc}(T_i, f_p):= { H_{p}}(W^{\su_i}_{p}\times \{\st(\su_i)\})$. 

Also each $T_i$ is in the stable lamination of $K_{p_1}$. Thus for every  $n\ge 1$ large enough, each point $f_{p_1}^{n}(T_i)$ is in $H_{p_1}( W^{s}_{p_1}(\Lambda_{p_1})\cap \mathrm{int\, } Y^\se {\times \sV }) $. Moreover there exists {$\ss_i\in \avv \sA$} such that $ f_{p_1}^{n}(W^u_{loc}(T_i, f_{p_1}))$ has a quadratic tangency with $H_{p_1}( W^{\ss_i}_{p_1}\times \si(\ss_i)) $ at $f_{p_1}^{n}(T_i)$.
Also the unfolding of these $d$ tangencies induced by  $(f_p)_p$ is  non degenerated by \cref{etape initial pour newhouse} $(2)$. 
By  regularity of $\pi$, there is a small neighborhood of $p_1$ formed by parameters $p$ for which  the curve $ W^{u}_{loc}(T_i, f_p)$ has a quadratic tangency with a fiber of  $\pi^{\si(\ss_i)}_p\circ (H^{\si(\ss_i)}_p)^{-1}\circ  f_{p}^{n}$ at a point $T_i(p)$   nearby $T_i$. Let $\zeta^{\su_i}_p\in W_p^{\su_i}$ be the preimage by  $H_p^{\st(\su_i)}$ of $T_i(p)$.  As the unfolding is non degenerate, the following map has invertible differential at $p=p_1$:
\[p\mapsto \left(\pi^{\si(\ss_i)}_p\circ (H^{\si(\ss_i)}_p)^{-1}\circ f^{n}_{p} (T_i(p)) - b^{\ss_i}_p\right)_{1\le i\le d}\quad \text{with }\{b^{\ss_i}_p\}:={\pi^{\si(\ss_i)}_p(W^{\ss_i}_p)}\; .
\]

By regularity of the stable laminations, there is a clopen neighborhood  $\sW_i \subset \avv \sA$ of each $\ss_i$, such that  for any $(\tilde \ss_i)_{{ 1 \le  i \le d }}\in \prod_{1\le i\le d} \sW_i$, the following map has invertible differential at $p=p_1$:
\[p\mapsto \left(\pi^{\si(\ss_i)}_p\circ (H^{\si(\ss_i)}_p)^{-1}\circ f^{n}_{p} (T_i(p)) - b^{\tilde \ss_i}_p\right)_{1\le i\le d}\quad \text{with } {\{b^{\tilde \ss_i}_p\}:=\pi_p^{\si(\tilde \ss_i)}(W^{\tilde \ss_i}_p) }  \; .
\]
Up to reducing each $\sW_i$, we can assume it of the form $\sW_i= \sc_i\cdot \avv \sA$ for  admissible words $\sc_i \in \sA^*$  of the same length $|\sc|$. Thus up to replacing $n$ by  $ n+|\sc|$, we can assume that $\sc_i=\se$ and obtain: 
\begin{fact}\label{fact second de nondeg} The following maps have invertible differential at  $p=p_1 $:
\[p\mapsto \left(\pi^{\si(\ss_i)}_p\circ (H^{\si(\ss_i)}_p)^{-1}\circ f^{n}_{p} (T_i(p)) - b^{\tilde \ss_i}_p\right)_{1\le i\le d}\; , \quad \forall (\tilde \ss_i)_{1\le i\le d}\in \prod_{i=1}^d\{\ss \in \avv \sA: \si(\ss)=  \si(\ss_i)\}
.\]  
\end{fact}
{We recall that $k$ is a large integer such that $f_p^{k}\circ H_p= H_p\circ F_p^\sA $ by \cref{Horseshoe2model}. } 
As $n$ is any large number, we can assume that $n=m\cdot k$ for a certain $m\ge 1$. By replacing $\sA$ by $\sA^*\cap \sA^m$ and { $F^\sA_p$ by $(F^\sA_p)^m$}, we can assume that $k=n$.  

For every $1\le i\le d$, let $\boxdot_i$ be a new arrow from $\st(\su_i)$ to $\si(\ss_i)$. Let $\sC:= \{\boxdot_i: 1\le i\le d\}$. Put:
\[\cb \avv \sA_{\boxdot}:= \{\ss\in \avv \sA: { \si(\ss)=\st(\boxdot)\}} \qand F^{\boxdot}_p:= (H^{\st(\boxdot)}_p)^{-1}\circ f^{k}_{p}\circ H^{\si(\boxdot)}_p\quad ,\qquad \forall \boxdot \in \sC .\]
The latter map   is defined on a neighborhood of $\zeta^{\su_i}_{p_1}$ {\cb when $\boxdot =\boxdot _i$}. The following  rephrases Fact~\ref{fact second de nondeg}:
\begin{fact}\label{fact dernier de nondeg} For any $( \ss_\boxdot)_{\boxdot\in \sC}\in \prod_{\boxdot \in \sC} \avv \sA_{\boxdot}$
and with $\cb  (\su_\boxdot)_{\boxdot \in \sC}=(\su_i)_{\boxdot _i\in \sC}$ 
, the following map has invertible differential at  $p=p_1 $:
\[{\cb p\mapsto \left(\pi^{\st(\boxdot)}_p\circ F^{\boxdot}_p(\zeta^{\su_\boxdot}_p) - b^{\ss_\boxdot}_p\right)_{\boxdot\in \sC}\;. }\] 
\end{fact}
{\cb For every $i$, with $\boxdot=\boxdot_i$,} there is a quadratic tangency at  the point  $\zeta^{\su_i}_{p}$  between  $W^{\su_i}_p$ and  a fiber of  $\cb \pi^{\st(\boxdot)}_p \circ F_p^{\boxdot}$. Thus there is  a small clopen neighborhood $\cb \arr \sA_{\boxdot}$ of $\su_i$ such that for every $\cb \su_\boxdot  \in \arr \sA_{\boxdot}$ and $p$ close to $p_1$, there is a quadratic tangency between $\cb W^{\su_\boxdot}_p$  and a fiber of  $\cb \pi^{\st(\boxdot)}_p \circ F_p^{\boxdot}$ at a point $\zeta^{\su_\boxdot }_p$ close to  $\zeta^{\su_i}_{p}$. By regularity of the unstable lamination and compactness of $\cb \prod_{\boxdot \in \sC} \avv \sA_{\boxdot}$, by reducing each $\cb \arr \sA_{\boxdot}$, we obtain: 
\begin{fact} \label{factH2} For every $\cb (  \su_\boxdot ,  \ss_\boxdot )_{\boxdot \in \sC}\in \prod_{\boxdot  \in \sC} \arr \sA_{\boxdot }\times \avv \sA_{\boxdot}$, the following map has invertible differential {for every $p$ close to $p_1 $}:
\[ p\mapsto {\cb \left(\pi^{\st(\boxdot)}_p\circ F^{\boxdot }_p(\zeta^{ \su_\boxdot  }_p) - b^{  \ss_\boxdot  }_p\right)_{\boxdot \in \sC}.}\] 
\end{fact} 
Up to reducing each $\cb \arr \sA_{\boxdot}$, there is a segment $\cb I^{\boxdot}$ {(independent of $p$ close to $p_1$)} such that for every $\cb \su_\boxdot\in \arr \sA_{\boxdot}$ {and $p$ close to $p_1$}, the point $\cb \zeta^{\su_\boxdot }_p$ is the unique critical point of $\cb { \pi^{\st(\boxdot)}_{p}} \circ F^{\boxdot}_p|I^{\boxdot}\times I\cap { W^{\su_\boxdot}_{p}}$ and belongs to $\cb \mathrm{int}I^{\boxdot}\times I$. Moreover we can assume that the convex hull of $\cb I^{\boxdot}\times I\cap \bigcup_{\su_\boxdot \in \arr \sA_{\boxdot}} { W^{\su_\boxdot}_{p}}$ is sent into $\mathrm{int\, } Y^\se$ by $\cb F^{\boxdot}_p$. 

  For every $\cb \boxdot \in \sC$,  since every local unstable manifold has all its tangent vectors in the cone $\chi_h$, we can reduce $I^\boxdot$ if necessary and then find a segment $J^{\boxdot} \Subset I$ independent of $p$ close to $p_1$ such that, with $\cb { Y^{\boxdot}_p}=I^{\boxdot}\times J^{\boxdot}$ and $\cb {\partial^u Y^{\boxdot}_p}=I^{\boxdot}\times \partial J^{\boxdot}$, the following set is a small clopen neighborhood of $\su_i$ included in $\cb \arr \sA_{\boxdot}$: 
\[ \{\su \in \arr \sA: \st(\su)= \si({\boxdot}) , 
{W^\su_p \cap Y^{\boxdot}_p \neq \emptyset \;  \&\;  
W^\su_p \cap \partial^u Y^{\boxdot}_p}= \emptyset\} \]
Thus up to replacing $\cb \arr \sA_{\boxdot}$ by the latter set, we can suppose that it holds:
\[\cb \arr \sA_{\boxdot}= \{\su \in \arr \sA: \st(\su)= \si({\boxdot}) , 
{W^\su_p \cap Y^{\boxdot}_p \neq \emptyset \;  \&\;  
W^\su_p \cap \partial^u Y^{\boxdot}_p}= \emptyset\} \]
{for any $p$ close to $p_1$.   
Up to reducing each $\cb \arr \sA_{\boxdot }$ a new time and consequently $\cb J^{\boxdot }$ and then reducing $\cb I^{\boxdot }$, we can suppose that $\cb F^{\boxdot }_p$ sends $\cb Y^{\boxdot }_p$ into the interior of $Y^\se$.}

By Fact \ref{pour def de zeta}, for every $p$ close to $p_1$, there are   $\cb (\su_\boxdot , { \ss_\boxdot })_{\boxdot \in \sC}\in \prod_{\boxdot  \in \sC} \arr \sA_{\boxdot}\times \avv \sA_{\boxdot}$ close to $\cb (\su_i, \ss_{i})_{\boxdot _i\in \sC} $  such that each $\cb W^{\su_\boxdot }_p$ has a quadratic tangency  with the preimage of $\cb W^{\ss_\boxdot }_{p}   $ by $\cb F^{\boxdot}_p$ at a   point close to $\cb \zeta_{p}^{\su_\boxdot }$. By uniqueness such a point must be  $\cb \zeta_{p}^{\su_\boxdot }$.  This proves $(\bf H_1)$.

Property $(\bf H_2)$ is given by {Fact  \ref{factH2}} and the local inversion theorem together with the compactness of $\cb \prod_{\boxdot  \in \sC} \arr \sA_{\boxdot}\times \avv \sA_{\boxdot}$ to define the regular compact neighborhood $\cP$ of $p_1$.
Clearly, the families $(F^\sA)_{p \in \cP}$ and $(F^\sC)_{p \in \cP}$ define a regular family of systems of type $(\sA,\sC)$. 
\medskip

Now given a $C^r$-perturbation $(\tilde f_p)_p$ of $(f_p)_p$, by structural stability, the wild dissipative basic set $K$ for $f_{p_1}$ persists as a wild dissipative basic set $\tilde K$ for $\tilde f_{p_1}$ (by \cite[Proposition 6]{Ne79} the unstable and stable thicknesses of these basic sets vary continuously). So Step 1 remains for $(\tilde f_p)_p$ at the same parameter $p_1\in \cP$.  By structural stability, the Markov partition and so the encoding into a $\sA$-system persists as well for $(f_p)_{p\in \cP_1}$.  So Step 2 remains valid for $(\tilde f_p)_p$ for the same regular subset $\cP_1\subset \cP$. Likewise Step 3 remains valid since the conditions on the folding maps are open on the elements involved. This shows the last item of \cref{premain4application}.

\section{Sufficient conditions for a wandering stable domain}\label{section Implicit representations of hyperbolic transformation}
 
In this section, we consider a system of type $(\sA, \sC)$ endowed with an adapted projection $\pi$ which satisfies the assumptions of  \cref{propreel,propcomplex} and we prove the conclusions of these Theorems.  Before performing these proofs, we recall and develop  the formalism of implicit representations that will be used throughout this paper. 

 \subsection{Implicit representations and initial bounds}\label{sect: Implicit representations and initial bounds}
  In the following, we will use the formalism of implicit representations. It was introduced by Palis and Yoccoz in \cite{PY01} and \cite{PY09} in order to get bounds on the iterations of hyperbolic transformations\footnote{They call such maps affine-like maps.}. In fact, this formalism goes back to the generating functions and the Shilnikov cross coordinates, appeared in the work of Shilnikov \cite{SH67} and his school \cite{GST08}. 
\begin{proposition}\label{ALacpara} Let $(Y ,F )$ be a hyperbolic transformation. Then, there are functions $\cX_0$ and $\cY_1$ in $C^2(I^2,I)$ satisfying for every $(x_0,y_0)\in Y $ and $(x_1,y_1)\in F(Y)$:
\begin{equation}\tag{$\cal A$} F(x_0,y_0)= (x_1, y_1) \Leftrightarrow\left\{
\begin{array}{cc}
x_0 =&\cX_0(x_1, y_0)\\
y_{1}=& \cY_1(x_1, y_0)\\
\end{array}\right. \end{equation}
\end{proposition}
\begin{proof}
Let $ (x_0,y_0)\in Y $ and $(x_1,y_1)\in F(Y)$. By properties \emph{(1)} and \emph{(2)} of \cref{defpiece}, the image by $F$ of $\{y=y_0\}\cap Y$ is the graph of a function $x\in I \mapsto \cY_1(x,y_0)\in I$. Likewise, by properties \emph{(1)} and \emph{(3)} of \cref{defpiece}, the image by $F^{-1}$ of $\{x=x_1\}\cap F(Y)$ is the transpose of the graph of a function $y\in I \mapsto \cX_0(x_1,y)\in I$.
We notice that the functions $(x_1,y_0)\mapsto \cX_0(x_1,y_0)$ and $(x_1,y_0)\mapsto \cY_1(x_1,y_0)$ are of class $C^2$. Observe also that by construction, $F(x_0,y_0)= (x_1,y_1)$ iff the intersection point
$\{(x_1,\cY_1(x_1,y_0))\}= \mathrm{Graph}\, \cY_1(\cdot ,y_0)\cap \{x=x_1\}$ is the image by $F$ of
$\{(\cX_0(x_1,y_0),y_0)\}= \, ^t\! \mathrm{Graph}\, \cX_0(x_1,\cdot)\cap \{y=y_0\}$.
\end{proof}
\begin{definition} The pair $(\cX_0,\cY_1)$ is the \emph{implicit representation} of $(Y,F)$.\index{Implicit representation}
\end{definition}

Let us introduce the natural counterpart of implicit representation in the complex setting:
\begin{proposition}\label{ALacparacompl}
Let $(\tilde Y ,F )$ be a $C^\omega_\rho$-hyperbolic transformation.Then there are holomorphic functions $\mathcal{Z}_0$ and $\mathcal{W}_1$ defined on $\tilde I^{2}$ satisfying for every $(z_0,w_0)\in \tilde Y $ and $(z_1,w_1)\in F(\tilde Y)$: \index{$\mathcal{Z}_0$ and $\mathcal{W}_1$}
\[ F(z_0,w_0)= (z_1,w_1) \Leftrightarrow\left\{
\begin{array}{cc}
z_0 =& \mathcal{Z}_0(z_1, w_0)\\
w_{1}=& \mathcal{W}_1(z_1, w_0)\\
\end{array}\right. \; .\] 
\end{proposition}
\begin{proof}
Let $ (z_0,w_0)\in \tilde Y $ and $(z_1,w_1)\in F(\tilde Y)$. By properties \emph{(1)} and \emph{(2)} of \cref{defcomplexpiece}, the image by $F$ of $\{w=w_0\}\cap \tilde Y$ is the graph of a function $z\in \tilde I \mapsto \mathcal{W}_1(z,w_0)\in \tilde I$. Likewise, by properties \emph{(1)} and \emph{(3)} of \cref{defcomplexpiece}, the image by $F^{-1}$ of $\{z=z_1\}\cap F(\tilde Y)$ is the transpose of the graph of a function $w \in \tilde I \mapsto \mathcal{Z}_0(z_1,w)$. By transversality,  the functions $(z_1,w_0)\mapsto \mathcal{Z}_0(z_1,w_0)\in \tilde I$ and $(z_1,w_0)\mapsto \mathcal{W}_1(z_1,w_0)$ are holomorphic. Observe also that by construction, $F(z_0,w_0)= (z_1,w_1)$ iff
the intersection point
$\{(z_1, \mathcal{W}_1(z_1,w_0))\}= \mathrm{Graph}\, \mathcal{W}_1(\cdot ,w_0)\cap \{z=z_1\}$ is the image by $F$ of
$\{(\mathcal{Z}_0(z_1,w_0),w_0)\}= \, ^t\! \mathrm{Graph}\, \mathcal{Z}_0(z_1,\cdot)\cap \{w=w_0\}$.
\end{proof}
\begin{definition} The pair $(\mathcal{Z}_0,\mathcal{W}_1)$ is the \emph{implicit representation} of $(\tilde Y,F)$.
\end{definition}
\begin{remark} \label{remholimpl} 
If $(\tilde Y, F)$ is a complex extension of  a hyperbolic transformation  $(Y,F)$, with implicit representations $(\mathcal{Z}_0,\mathcal{W}_1)$ and $ (\mathcal{X}_0,\mathcal{Y}_1)$ respectively, it holds:
$(\mathcal{Z}_0,\mathcal{W}_1)|I^2= (\mathcal{X}_0,\mathcal{Y}_1)$.
\end{remark}

\begin{proposition}\label{rep_jac} If $(\cX_0,\cY_1)$ is the implicit representation of a hyperbolic transformation $(Y,F)$, for every $(x_0, y_0)\in Y$ sent to $(x_1,y_1)$ by $F$, it holds:
\[\det\, DF(x_0, y_0)\cdot \partial_{x_1}\cX_0(x_1,y_0)= \partial_{y_0}\cY_1(x_1, y_0)\; .\]
\end{proposition}
\begin{proof}
By $(\cal A)$, we have $F(\cX_0(x_1,y_0), y_0)=(x_1, \cY_1(x_1, y_0))$ and so
\[\det D_{x_0,y_0}F \cdot \det D_{x_1, y_0}(\cX_0(x_1,y_0), y_0)=\det D_{x_1, y_0}(x_1, \cY_1(x_1, y_0)).\]
Then the proposition follows from the following equalities:
$$\det D_{x_1, y_0}(\cX_0(x_1,y_0), y_0)=\partial_{x_1} \cX_0(x_1,y_0)\qand \det D_{x_1, y_0}(x_1, \cY_1(x_1, y_0))=\partial_{y_0} \cY_1(x_{1}, y_0)\; .$$
\end{proof}
An immediate consequence of the latter proof is: 
\begin{corollary}\label{Crep_jac} If $(\tilde Y, F)$ is a $C^\omega_\rho$-hyperbolic transformation  with implicit representation $(\cZ_0,\cW_1)$, then for every $(z_0, w_0)\in \tilde Y$ sent to $(z_1,w_1)$ by $F$, it holds:
\[\mathrm{det }\, DF(z_0, w_0)\cdot \partial_{z_1}\cZ_0(z_1,w_0)= \partial_{w_0}\cW_1(z_1, w_0)\; .\]
\end{corollary}

The following distortion was introduced in \cite[Page 19]{PY09}.
\begin{definition} \label{defdist}
\index{Distortion {$\mathscr{D}(Y,F)$}}
The \emph{distortion} {$\mathscr{D}(Y,F)$} of a hyperbolic transformation $(Y,F)$ with implicit representation $(\cX, \cY)$ is the maximum absolute value attained on $I^{2}$ by the six following functions: 
$$\partial_x \log |\partial_x \cX|,\quad \partial_y \log |\partial_x \cX|, \quad \partial_x \log |\partial_y \cY|,\quad \partial_y \log |\partial_y \cY|, \quad \partial^{2}_{y^{2}} \cX, \quad \partial^{2}_{x^{2}} \cY \; .$$
\end{definition}
Let us recall:
\begin{theorem}[{\cite[({\bf MP6}), Page 21]{PY09}}]\label{PY01 cor 3.4} 
There exists $C_1>0$ which depends only on $\theta$ and $\lambda$ such that for every $N\ge 1$, if $((Y_i,F_i))_{1\le i\le N}$ are hyperbolic transformations, then the distortion of $(Y,F)= (Y_1,F_1)\star \cdots \star (Y_N,F_N)$ satisfies:
\[ {\mathscr{D}(Y,f)} \le C_1\cdot \max_{1\le i\le N} {\mathscr{D}(Y_i,F_i)} \; .\] 
\end{theorem}

For every $\sc\in   \sA^*$, let $(\cX^\sc,\cY^\sc)$ denote the implicit representation of $(Y^{\sc},F^{\sc})$.

\begin{corollary}\label{PY01 cor 3.4coro} 
\index{Distortion  bounds {$\mathscr{D}$} and {$\mathscr{B}$} }
If $F^\sA$ is a hyperbolic map of type $\sA$, there exists {$\mathscr{D}>0 $} and   {${\mathscr{B}}:=2 \exp(2\sqrt 2\mathscr{D} )$}
 such that for any $\sc\in \sA^*$ and $(x,y) \in I^2$, it holds both  {$\mathscr{D}(Y^\sc, F^\sc)   \le \mathscr{D}$}  and the following:
\[     |Y^\sc|   \cdot  {\mathscr{B} }^{-1}  \le |\partial_x \cX^{\sc}(x,y)|\le |Y^\sc|\cdot   {\mathscr{B}  }   \qand |F^\sc(Y^\sc)|\cdot   {\mathscr{B}^{-1}} \le |\partial_y \cY^{\sc}(x,y)|\le |F^\sc(Y^\sc)|\cdot  {\mathscr{B} } . \]  
\end{corollary}
\begin{proof} We fix $\sc\in \sA^*$ and $(x,y) \in I^2$.  A consequence of \cref{PY01 cor 3.4} and of the finiteness of  $\sA$ is:
 \begin{equation} \label{inegcX}|\partial_x \cX^{\sc}(x,y)| \cdot (  {\mathscr{B} }/2)^{-1} \le |\partial_x \cX^{\sc}(x',y')| \le |\partial_x \cX^{\sc}(x,y)|  \cdot ( {\mathscr{B} }/2) \quad \forall  (x',y') \in I^2 \, , \end{equation}
 \[ \text{with }    \mathscr{D}   := C_1 \cdot \max_{\sa \in \sA} \mathscr{D} (Y^{\sa},F^{\sa}) \qand \mathscr{B} :=2\exp(\mathrm{diam}(I^2) \mathscr{D} )= 2\exp(2\sqrt 2 \mathscr{D} ) \, .\]  
 By \cref{defYnorme}, the number $|Y^\sc|$ is the maximum of the lengths of $I\times\{y'\}\cap Y^\sc$ among $y' \in I$. Thus there exists $y' \in I$ such that $|Y^\sc| = |\int_{I} \partial_x \cX^{\sc}(x',y') dx'|$. Then \cref{inegcX} together with $\mathrm{diam}(I)=2$ imply  the first bound. The proof of the second one is the same by symmetry. 
\end{proof}

\subsection{Normal form and definitions of the {sets $B_j$ and $\tilde B_j$}} \label{sectiondefBj}

In this subsection, we consider a system $F$ of type $(\sA, \sC)$ endowed with an adapted projection $\pi$.  Our aim is to prove a normal form useful for the proofs of \cref{propreel,propcomplex}. We start with general initial bounds.

We recall that for $\sc\in \sA^*$, the curve $H^\sc$ was defined in \cref{defHs}, and for $\boxdot\in \sC$, the set  $\sA^*_\boxdot$ and the critical points $(\zeta^\sc)_{\sc\in \sA^*_\boxdot}$ were defined in \cref{defzetafixep}.
\begin{proposition}[definition of {$\epsilon_\sC$}]\label{def epsilonboxdotpfixed}\index{$\epsilon_\sC$}\label{defzetafixep2} \label{factzetajnew2}
There exists {$\epsilon_\sC>0$} such that {for every $\boxdot \in \sC$}:
\begin{enumerate}
\item  $F^\boxdot(Y^\boxdot) \subset  Y^\se$ is {$2\epsilon_\sC$}-distant from {$\partial^u Y^\se$},  
\item   $\zeta^\sc \in Y^\boxdot$ is  {$2\epsilon_\sC$}-distant to $\partial Y^\boxdot$, for every $\sc\in \sA^*_\boxdot $, 
\item $H^\sc$ is {$2\epsilon_\sC$}-distant to $\partial^u Y^\boxdot$, for every $\sc\in \sA^*_\boxdot $.
\end{enumerate}\end{proposition}
\begin{proof}{By finiteness of $\sC$, it suffices to show that there exists such a constant for every $\boxdot \in \sC$.}

By  \cref{def folding map}, we have $F^\boxdot (Y^\boxdot) \subset Y^\se \setminus \partial^u Y^\se$ and so the first statement is easy.

By \cref{limit sets}, the set $\{\zeta^\sc: \sc \in \sA^*_\boxdot\sqcup \arr \sA_\boxdot\}$ is compact. It is included in the interior of $Y^\boxdot$ by \cref{def adapted proj} (4) and \cref{defzetafixep}. This proves the second item.

By \cref{limit sets}, the set $\bigcup_{\su\in \arr{\sA}_\boxdot} W^u\cup \bigcup_{\sc\in \sA^*_\boxdot} H^\sc  $ is compact. By 
\cref{def folding map} (1) and  \cref{defzetafixep}, it does not intersect $\partial^u Y^\boxdot$. Thus the last item of the proposition follows. 
\end{proof}

From now on, we work under the assumptions of \cref{propreel} (resp. \cref{propcomplex}), which defined {the sequences} $(\sc_j)$ (and so $(\zeta^{\sc_j})_j $) and {$(\boxdot_j)_j$}.
The idea is to define the domains  $B_j$ and $\tilde B_j$ using a normal form for the folding of the curve $H^{\sc_j}$ with respect to the $\pi$-fibers. 
 However it is rather complicated to work with {$\pi^{\si(\sc_{j+1})}$-fibers, for they are} a priori not analytic. Hence we will work with:
\begin{definition}   Let $\breve V^{j+1}$ be the pullback of $\{x_{j+1}\}\times I$ by $F^{\sc_{j+1}}$, with  $x_{j+1}$ the first coordinate of~$\zeta^{\sc_{j+1}}$. \index{$\breve V^{j+1}$ and $x_{j+1}$}
\end{definition}
  We note that: 
\begin{equation}\label{def vi hi}      \breve V^{j+1}= \{(\breve {\mathscr v}_{j+1}(y), y): y\in I\} \quad \text{with}\quad  \breve {\mathscr v}_{j+1}(\cdot ):= \cX^{\sc_{j+1}}( x_{j+1}, \cdot )
 \; , \quad \forall j\ge 1. \end{equation}  \index{$ \breve {\mathscr v}_{j+1}$ and $\mathscr h_j$}
 We will also use the following notation $\mathscr h_j$ whose expression is  similar to the one of  $\breve {\mathscr v}_{j+1}$:
 \begin{equation}\label{defhi}
 H^{\sc_j}= \{(x, \mathscr h_{j}(x)): x\in I\} \quad \text{with}\quad \mathscr h_j(\cdot ):=   \cY^{\sc_{j}}(\cdot ,0)  \; , \quad \forall j\ge 1 .
 \end{equation}
{Let $F^{\boxdot_j}_{1}$ and $F^{\boxdot_j}_{2}$ be the $x$-coordinate and the $y$-coordinate of $F^{\boxdot_j}$.  The $x$-projection $I^{\boxdot_j}\subset I$ of $Y^{\boxdot_j}$ is the definition domain of 
 $x\mapsto F^{\boxdot_j} \circ (x, \mathscr h_{j}(x))$.} For every $r>0$, let $\D(r):=\{z\in \C: |z|<r\}$.

 Here is the aforementioned normal form:

\begin{proposition} \label{normfor}
There exists $K>0$ such that for every $j$ large enough, there exist $\breve x_j\in I^{\boxdot_j}$ which is {$\epsilon_\sC$}-distant to the boundary of $ I^{\boxdot_j}$,   real numbers $b_j $, $ q_j$ and a   $C^2$-function $r_j $ such that for every $(x,y)\in {[-\epsilon_\sC, \epsilon_\sC]^2}$ the following is well-defined and holds true:
\begin{equation}\label{Defrp12}
{(F^{\boxdot_j}_{1}-  \breve{\mathscr v}_{j+1} \circ F^{\boxdot_j}_{2})} \circ \varphi_j(x,y) =   q_{j} \cdot  x^2+b_j\cdot y+r_{j}(x, y) \; 
\end{equation} 
with $\varphi_j(x,y) =( \breve x_j +x, \mathscr h_j( \breve x_j   +x)+y)  $, $1/K\le |q_j|,|b_j|\le K$  and $r_j$ satisfying:
\begin{equation}\label{Defrp12proprietyes}
0 =\partial_x r_{j}(0)=\partial_y r_{j}(0)=\partial_x^2 r_{j}(0)  \text{ and }  r_{j}(0)  = o (\breve \gamma_j^2) \; . 
\end{equation}
Also if $F$ is analytic, there exists {$\tilde \epsilon_\sC>0$} independent of $j$, such that  $\varphi_j$ and $r_j$ extend holomorphically to {$\D(\tilde \epsilon_\sC)^2$}.
\end{proposition} \index{$\tilde \epsilon_\sC$}
We prove this proposition below. This normal form enables us to define  rescaling coordinates and the domains  $B_j$ and $\tilde B_j$. We need just to define the rescaling factors:

\begin{definition}\label{def sigmaj et gammaj}\index{$\sigma_j$}\index{$\gamma_j$} Let $y_j$ be the $y$-coordinate of $F^{\boxdot_j}(\zeta^{\sc_j})$. We define:
\[\sigma_j:= \partial_x \cX^{\sc_j} (x_j, y_{j})\qand \gamma_{j}:=  |\frac{\sigma_{j+1}}{q_j}|^{1/2}\cdot |\frac{\sigma_{j+2}}{q_{j+1}}|^{1/2^2}\cdots |\frac{\sigma_{j+k}}{q_{j+k-1}}|^{1/2^{k}}\cdots\;\]\end{definition}
By the latter proposition,  $1/K\le |q_j|\le K$ and so $\prod_k  |1/q_k|^{1/2^{k}}$  {is well-defined} and in $[1/K, K]$. Also by \cref{PY01 cor 3.4coro}, it holds $|\sigma_j|/   {\mathscr{B} } \le |Y^{\sc_j}|\le  {\mathscr{B} } |\sigma_j|  $. Thus the limit $\gamma_j$ exists and satisfies:
\begin{equation}\label{def sigma et gamma}
\frac{\breve \gamma_j}{  \mathscr{B} \cdot   K} \le \gamma_{j}\le  \mathscr{B} \cdot  K\cdot \breve \gamma_j  \; . \end{equation} 

Therefore  $\gamma_j$ is well-defined, small and nonzero because $\breve \gamma_j$ is so by assumptions $(i)$ and $(iii)$ of \cref{propreel}. 
In particular, since the domain of $\varphi_j$ contains {$[-\epsilon_\sC , \epsilon_\sC]^2$}, the following is well-defined:
\begin{definition} \index{$\Phi_{j}$}\label{def XiYi} The rescaling coordinates $\Phi_j$ are defined for $j$ large by:
\begin{equation*}
\Phi_j: (X,Y )\in (-0.3,0.3)^2\mapsto {\varphi_j\left(\gamma_{j}\cdot X, \frac{q_j \cdot  \gamma_{j}^2}{2b_j} \cdot Y\right)= 
\left( \breve  x_j +\gamma_{j}\cdot X, \mathscr h_j( \breve  x_j+\gamma_{j}\cdot X)+\frac{q_j \cdot \gamma_{j}^2}{2b_j} \cdot Y\right) }
\; .
\end{equation*}
\end{definition}
Similarly, if $F$ is analytic,  the definition domain of $\varphi_j$ contains {$\D(\tilde \epsilon_\sC)^2$}. So it holds:
\begin{fact} If $F$ is analytic, the rescaling coordinates $\Phi_j$ extend  holomorphically to $\D(0.3)^2$. 
\end{fact}

\begin{definition}\label{def: Bj}
The real image of $\Phi_j$ is denoted by  $B_j:= \Phi_j ((-0.3,0.3)^2)$.  If $F$ is analytic, the complex image of $\Phi_j$ is denoted by  $\tilde B_j:= \Phi_j (\D(0.3)^2)$.
\end{definition}
 \begin{remark}\label{contraction of F}   Actually for $j$ large enough, the map 
$\varphi_j$ is also well defined on $[-3,3]^2$.  Using \cite[Theorem C and  Remark 3.5]{berger2018zoology}, one can show that when $j$ is large, the following map:
$$F_j:=(X,Y)\in (-0.3,0.3)^2\mapsto (\Phi_{j+1})^{-1}\circ F^{\sc_{j+1}}  {\circ F^{\boxdot_j} \circ} \Phi_{j}(X,Y)\; ,$$
is $C^2$-close to $(X,Y)\in (-0.3,0.3)^2\mapsto \mathrm{sgn\, }(\sigma_{j+1} q_j) (X^2+\frac12 Y,0)$, and obtain \cref{propreel} (1). However we will give a variation of this proof which works in our complex case.
\end{remark}

\begin{proof}[Proof of \cref{normfor}] 
To understand the position of the curve $\breve V^{j+1}= (F^{\sc_{j+1}})^{-1}(\{x_{j+1}\}\times I)$ w.r.t. $F^{\boxdot_j}(H^{\sc_j}\cap Y^{\boxdot_j})$, we 
 are going to use the two following curves:
\begin{itemize} 
\item The $\pi^{\si(\sc_{j+1})}$-fiber $\hat V^{j+1}$ of $(F^{\sc_{j+1}})^{-1}(\zeta^{\sc_{j+1}})$.
\item  The $\pi^{\st(\boxdot_j)}$-fiber $V^{j+1}$  of $F^{\boxdot_j}(\zeta^{\sc_j})$.
\end{itemize}\begin{figure}[h]
\centering
\includegraphics[width=13cm]{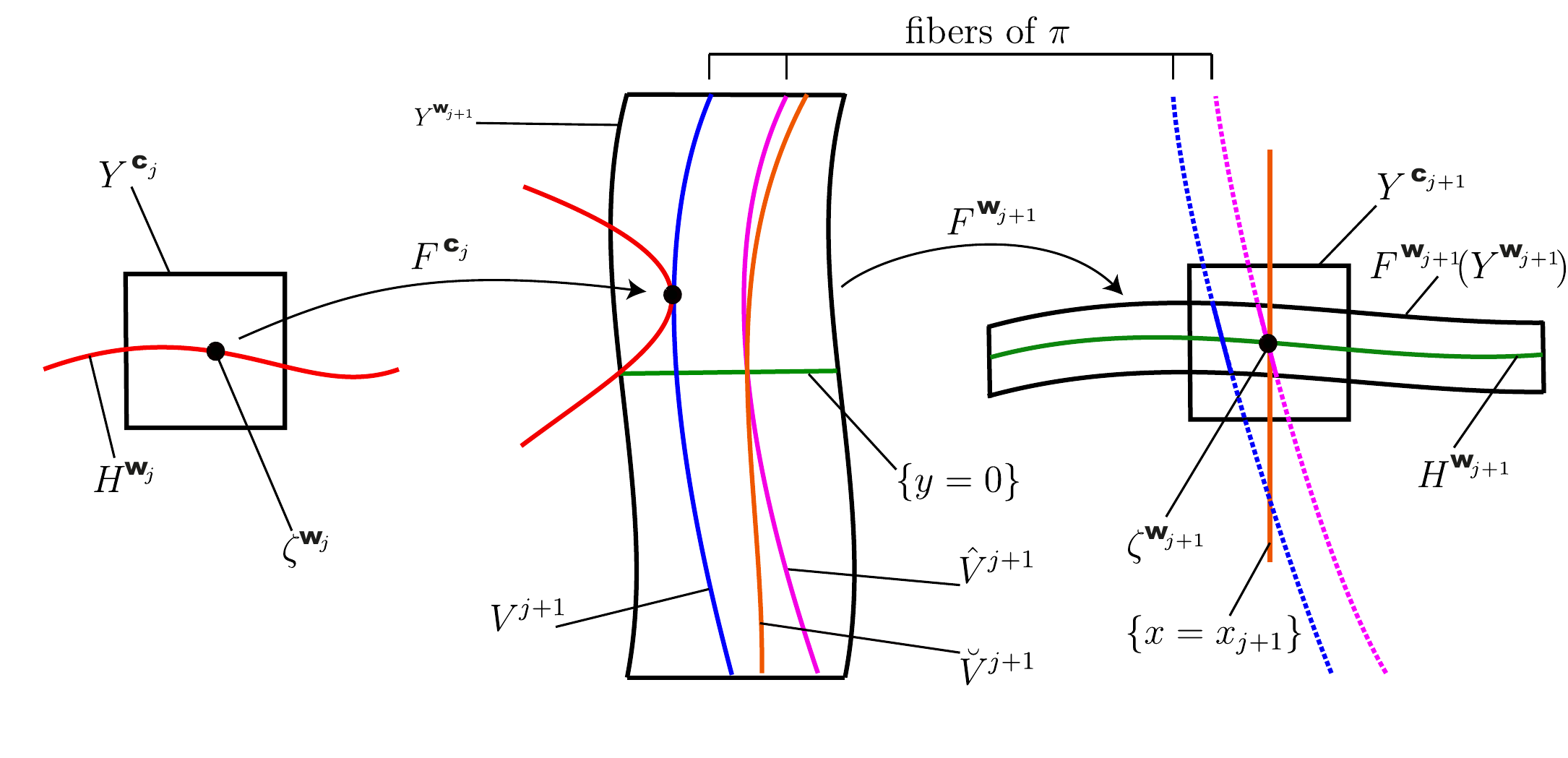}
\caption{Notations for the proof of \cref{normfor}}\end{figure}
The two next lemmas states that these curves are close to $\breve V^{j+1}$.

\begin{lemma}\label{def Vj equiv}
The two curves  $V^{j+1}$  and $\hat V^{j+1}$ are $C^2$-close and $o(\breve \gamma_{j}^2)$-$C^0$-close when $j$ is large.
\end{lemma}
\begin{proof}
By assumption (ii) of \cref{propreel},  the distance between the $\pi^{\st(\sc_{j+1})}$-fibers of the points $ F^{\sc_{j+1}}\circ F^{\boxdot_j}(\zeta^{\sc_j})$ and   $ \zeta^{\sc_{j+1}}$ is small compared to $  \breve \gamma_{j+1}$. 
As the horizontal direction is at  least $\asymp |Y^{\sc_{j+1}}|^{-1}$ expanded by $F^{\sc_{j+1}}$ by \cref{PY01 cor 3.4coro}, it follows that
the  pull back of these fibers by $F^{\sc_{j+1}}$ are at $C^0$-distance  small compared to $  \breve \gamma_{j+1}\cdot  |Y^{\sc_{j+1}}|=\breve \gamma_j^2$.
By the equivariance of the fibers given by \cref{def adapted proj}.(3), these two pullbacks are equal to respectively $V^{j+1}$ and $\hat V^{j+1}$ respectively. Since the $\pi$-fibers depend continuously on the based points in the  $C^2$-topology, these  two curves are also $C^2$-close.
\end{proof}

\begin{lemma} \label{lemmeC2close}
The two curves  $V^{j+1}$  and $\breve V^{j+1}$ are $C^2$-close and $o(\breve \gamma_{j}^2)$-$C^0$-close when $j$ is large.
\end{lemma}
\begin{proof}
By \cref{def Vj equiv}, it suffices to show that the two curves  $\hat V^{j+1}$  and $\breve V^{j+1}$ are $C^2$-close and $o(\breve \gamma_{j}^2)$-$C^0$-close when $j$ is large.
 The {$\pi^{\st(\sc_{j+1})}$}-fiber going through $\zeta^{\sc_{j+1}}$ is a curve {included in $Y^\se$ with endpoints in $\partial^u Y^\se$ and  tangent spaces in $\chi_v$ by \cref{def adapted proj} (0) and (1)}.  The same occurs obviously  for the curve $\{x_{j+1} \}\times I$. Hence the curves  $\hat V^{j+1}$ and $\breve V^{j+1}$ are the pullbacks by 
$F^{\sc_{j+1}}$ of two curves which are uniformly transverse to the unstable lamination of the maximal invariant set $\Lambda$ associated to the hyperbolic map $F^\sA$. Hence by the Inclination Lemma, the curves  $\hat V^{j+1}$ and $\breve V^{j+1}$ 
are $C^2$-close when $|\sc_{j+1}|$ is large and so when $j$ is large. 

 We now show that the Hausdorff distance between  $\hat V^{j+1}$ and $\breve V^{j+1}$ 
 is small compared to   $\breve \gamma_j^2$. Note that the images by $F^{\sc_{j+1}}$ of $\hat V^{j+1}$ and $\breve V^{j+1}$ intersects at $\zeta^{\sc_{j+1}}$. These curves have tangent spaces in $\chi_v$ and are included in $F^{\sc_{j+1}}(Y^{\sc_{j+1}})$, whose 
  vertical height is {smaller than  $|F^{\sc_{j+1}}(Y^{\sc_{j+1}})|$} and thus small compared to {$\breve \gamma_{j+1}^2$} by assumption $(iii)$ of   \cref{propreel}. { Thus the $C^0$-distance between the images by $F^{\sc_{j+1}}$ of $\hat V^{j+1}$ and $\breve V^{j+1}$ is small compared to  $\breve \gamma_{j+1}^2 \le \breve \gamma_{j+1}$.  As the horizontal direction is at  least $\asymp |Y^{\sc_{j+1}}|^{-1}$ expanded by $F^{\sc_{j+1}}$ by \cref{PY01 cor 3.4coro}, it follows that the $C^0$-distance between $\hat V^{j+1}$ and $\breve V^{j+1}$  is  small compared to $\breve \gamma_{j+1}\cdot  |Y^{\sc_{j+1}}|= \breve \gamma_{j}^2$.} \end{proof}

The following lemma is the second main step of the proof of \cref{normfor}, it defines the  new critical point $\breve x_j$ defined by $\breve V^{j+1}$, which is close to the first coordinate $x_j$ of $\zeta^{\sc_j}$.
\begin{lemma} \label{ptxj}  
For every $j$ large enough, the following map:
\[\breve \Delta_j: x\in I^{\boxdot_j}\mapsto { (F^{\boxdot_j}_{1}-\breve {\mathscr v}_{j+1}\circ F^{\boxdot_j}_{2})}  (x,\mathscr h_{j}(x))     
\]
has a unique critical point $\breve x_j$ in the $\breve \gamma_j$-neighborhood of $x_j$ in $I^{\boxdot_j}$ which is {$\epsilon_\sC$}-distant to $\partial I^{\boxdot_j}$. Moreover  $|\breve x_j- x_j|$ is small compared to $\breve \gamma_j$ and $|\breve \Delta_j(\breve x_j)|$ is small compared to $\breve \gamma_j^2$.
\end{lemma}

\begin{proof}
{By \cref{def adapted proj} (4), for every $\boxdot \in \sC$ and $\su \in \arr \sA_\boxdot$, the curve $W^\su$ is sent by $F^\boxdot$ tangent to a unique fiber of $\pi^{\st(\boxdot)}$ at a unique point $F^{\boxdot}(\zeta^\su)$, with $\zeta^\su \in \mathrm{int\, } Y^\boxdot$, and the tangency is quadratic. By  \cref{def adapted proj} $(1)$, there exists a $C^2$-function $\mathscr v_{\su}:  I\to \R$ such that this fiber is equal to $\{({\mathscr v}_{\su}(y), y): y\in I\}$. Thus there exists $r \in (0,\epsilon_\sC/3)$ such that $ [x_\su-3r,x_\su+3r] \Subset I^\boxdot$ and  the following map 
\[\Delta_{\su}: x\in I^\boxdot \mapsto   F_1^{\boxdot} (x, w^{\su}(x))-  \mathscr{v}_{\su} \circ  F^{\boxdot}_2  (x, w^{\su}(x))  \; \]
has second derivative of absolute value  $> 3r$ on $[x_\su-3r,x_\su+3r]$ and a unique critical point in $(x_\su-r/3 ,x_\su+r/3)$ equal to $x_\su$, with $x_\su$ the first coordinate of $\zeta^\su$. Note that $\Delta_{\su}$ is indeed well-defined on $I^\boxdot$ because $\su \in \arr \sA_\boxdot$. Since $w^\su$  depends continuously on $\su$ in the $C^2$-topology and $\bigcup_{\boxdot \in \sC} \arr \sA_\boxdot$ is compact, we can take a uniform value of $r$ among $\boxdot \in \sC$ and $\su \in \arr \sA_\boxdot$.

On the other hand, by \cref{defzetafixep}, the curve $H^{\sc_j}$ is tangent to a unique fiber $V^{j+1}$ of $\pi^{\st(\boxdot_j)}\circ F^{\boxdot_j}$ at a unique point $\zeta^{\sc_j} \in \mathrm{int\, } Y^{\boxdot_j}$, and the tangency is quadratic. 
Still by  \cref{def adapted proj} $(1)$ there exists a $C^2$-function $\mathscr v_{j+1}: { I\to \R}$ such that $V^{j+1}  = \{({\mathscr v}_{j+1}(y), y): y\in I\}\; .$
 
 When $j$ is large we recall that $|\sc_j|$ is large.
Still by \cref{defzetafixep}, the word $\sc_j$ is made by the last letters of some $\su_j\in { \arr  \sA_{\boxdot_j}} $.
By the Inclination Lemma (as in the proof of \cref{defzetafixep}), the map  $\mathscr h_j$, defined on $I^{\boxdot_j}$, is $C^2$-close to $w^{\su_j}$ (also defined on $I^{\boxdot_j}$) and $\zeta^{\sc_j}$ is close to $\zeta^{\su_j}$ when $j$ is large. Thus $x_j$ is close to $x_{\su_j}$ (in particular $[x_j-2r,x_j+2r] \Subset [x_{\su_j}-3r,x_{\su_j}+3r]$), and the maps $\mathscr{v}_{\su_j}$ and $\mathscr{v}_{j+1}$ are $C^2$-close. Then the following map is defined on $I^{\boxdot_j}$ and  $C^2$-close to $\Delta_{\su_j}$ on $[x_j-2r,x_j+2r]$:
\[\Delta_j: x\in I^{\boxdot_j} \mapsto   F_1^{\boxdot_j} (x, \mathscr{h}_j(x))-  \mathscr{v}_{j+1} \circ  F^{\boxdot_j}_2  (x, \mathscr{h}_j(x))\; . \]
Thus the  map $\Delta_j$ has second derivative of absolute value  $> 2r$  in $[x_j-2r,x_j+2r]$ and has a unique critical point in $(x_j-r/2,x_j+r/2)$ which is equal to $x_j$ by \cref{defzetafixep}.} Note that $ \Delta_j(x_j)=0$ is a local  extremum of $ \Delta_j$.
By  \cref{lemmeC2close}, the maps $\mathscr v_{j+1}$ and $\breve {\mathscr v}_{j+1}$ are both $C^2$-close and $o(\breve \gamma_{j}^2)$-$C^0$-close when $j$ is large. Then for $j$ large enough the following function:
\[\breve \Delta_j: x\in I^{\boxdot_j}\mapsto { (F^{\boxdot_j}_{1}-\breve {\mathscr v}_{j+1}\circ F^{\boxdot_j}_{2})  (x,\mathscr h_{j}(x)) }
\]
has a unique critical point $\breve x_j$ in {$(x_j-r,x_j + r)$} which is close to $x_j$. Thus by  \cref{def epsilonboxdotpfixed} (2), $\breve x_j$ is {$\epsilon_\sC$}-distant to $\partial I^{\boxdot_j}$ for every $j$ large.  
Furthermore, the  second derivative of $\breve \Delta_j$ has absolute value {$>r$} on $[-r+  \breve x_j, r+ \breve x_j]$. In particular, the $C^0$-distance between the extremal values    $ \Delta_j(  x_j)$ and $\breve \Delta_j( \breve x_j)$  of $  \Delta_j$ and $ \breve \Delta_j$ is small compared to $\breve \gamma_{j}^2$.   Since $\Delta_j(x_j) = 0$, we then have $\breve \Delta_j( \breve  x_j) = o(\breve \gamma_{j}^2)$. 
Let us estimate the distance between $\breve x_j$ and $ x_j$:
\[ |  \breve \Delta_j(x_j)  -  \breve \Delta_j( \breve  x_j) | = \left| \int_{ \breve x_j}^{x_j}   D \breve  \Delta_j(t)dt \right| \ge\left| \int_{  \breve  x_j}^{ x_j}   \int_{ \breve x_j}^t D^2 \breve \Delta_j(s)dsdt \right| \ge \frac{r}{2} ( x_j -   \breve x_j)^2   \,  .\]
As $| \breve  \Delta_j( \breve x_j)  -  \Delta_j(  x_j) | = o(\breve \gamma_{j}^2)$ and by \cref{lemmeC2close} 
$| \breve  \Delta_j(   x_j)  -  \Delta_j(  x_j) | = o(\breve \gamma_{j}^2)$
, it holds $|\breve x_j -   x_j| = o(\breve \gamma_{j})$. \end{proof}

 The third step of the proof of \cref{normfor} will be to obtain the normal form of \cref{Defrp12}. It is done below in \cref{Defrp13}. 
 We recall that by \cref{ptxj}, $\breve x_j$ 
is $\epsilon_\sC$-distant from $\partial I^{\boxdot_j}$. On the other hand, the curve $H^{\sc_j}$ intersects  $Y^{\boxdot_j}$ since $\sc_j \in \sA^*_{\boxdot_j}$ and is $\epsilon_\sC$-distant to 
$\partial^u Y^{\boxdot_j}$ by \cref{def epsilonboxdotpfixed} $(3)$. 
Thus it follows:
\begin{fact} \label{Fact sur epsilon} The function $\varphi_j(x,y)=( \breve x_j +x, \mathscr h_j( \breve x_j   +x)+y)$ is well defined on {$ (-\epsilon_\sC, \epsilon_\sC)^2$} and takes its values in  {$Y^{\boxdot_j}$}.
\end{fact}
Similarly it holds:
\begin{fact} \label{extensionrj}
If $F$ is analytic, there exists {$\tilde \epsilon_\sC>0$} such that for every $j$ large enough, the map $\varphi_j$ is well-defined on {$\D(\tilde \epsilon_\sC)^2$}   and satisfies both {$\varphi_j(\D(\tilde \epsilon_\sC)^2) \subset \tilde Y^{\boxdot_j}$ and $F^{\boxdot_j} \circ \varphi_j(\D(\tilde \epsilon_\sC)^2) \subset \tilde Y^\se$}.
\end{fact}

\begin{proof}
We recall that  $\breve x_j \in I^{\boxdot_j}$ is  {$\epsilon_\sC$}-distant from $\partial I^{\boxdot_j}$ and so {$\min(\rho, \epsilon_\sC)$}-distant to $\partial \tilde I$. Since the maps $\mathscr{h}_j$ extend to $\tilde I$ and are uniformly $C^1$-bounded on $\tilde I$ by \cref{defcomplexpiece} $(2)$, the maps $\varphi_j$ are well defined on {$\D(\tilde \epsilon_\sC)^2$ for $\tilde \epsilon_\sC<\min(\rho, \epsilon_\sC)$} and uniformly $C^1$-bounded. Thus for $\tilde \epsilon_\sC$ small, the set {$ \varphi_j(\D(\tilde \epsilon_\sC)^2)$} has   a uniformly small diameter and has real trace {$ \varphi_j((-\epsilon_\sC, \epsilon_\sC)^2) \subset  \varphi_j(\D(\tilde \epsilon_\sC)^2) $}. So to get the two last inequalities,
it suffices to use Fact \ref{Fact sur epsilon} and remark that the set {$\tilde Y^{\boxdot_j} \cap (F^{\boxdot_j})^{-1}(\tilde Y^\se)$} contains a uniformly large complex neighborhood of  {$Y^{\boxdot_j}$}.
\end{proof}

 Also by \cref{ptxj} there exist  real numbers $b_j$, $q_j$ and a $C^2$-function $r_j$ such that for every sufficiently  large $j$ and  {$(x,y)\in (-\epsilon_\sC, \epsilon_\sC)^2$}:
\begin{equation}\label{Defrp13}
{(F^{\boxdot_j}_{1}-  \breve{\mathscr v}_{j+1} \circ F^{\boxdot_j}_{2})} \circ \varphi_j(x,y) =   q_{j} \cdot  x^2+b_j\cdot y+r_{j}(x, y) \; ,
\; 
\end{equation} \index{$b_j$, $q_j$ and $r_j$}
\[\text{with}\quad 0 =\partial_x r_{j}(0)=\partial_y r_{j}(0,0)=\partial_x^2 r_{j}(0)  \text{ and }  r_{j}(0)  = o (\breve \gamma_j^2) \; .\]
If $F$ is analytic, {the map $\breve{\mathscr v}_{j+1}$ extends holomorphically to $\tilde I$ and then,} by Fact \ref{extensionrj}, the function $r_j$ extends  holomorphically  to {$\D(\tilde \epsilon_\sC)^2$}. The next lemma achieves the proof of the proposition.
\end{proof} 

\begin{lemma}\label{minoration qi}\label{complex normality}  The families $(\mathscr \varphi_j)_j$ and   $(r_j)_j$ are  normal in the $C^2$-topology, and if $F$ is analytic, they are also normal on {$\D(\tilde \epsilon_\sC)^2$}.   Moreover, there exists $K\ge 1$ such that for $j$ large enough: \[\quad K^{-1}<|q_{j}|<K \qand K^{-1}<|b_{j}|< K \, . \]
\end{lemma}
\begin{proof}By {finiteness of $\sC$, by} compactness of $\avv \sA$ and $\arr \sA_\boxdot$  and by the Inclination Lemma, the following are compact subsets of $C^2(I,I)$:  
\[\{w^\ss : \ss\in \avv \sA\}\cup \{ \cX^\sc(x,\cdot)  : \sc \in \sA^*   , \, x \in I   \} \, , \]
\[\{ w^\su : { \boxdot \in \sC,}  \su\in \arr \sA_\boxdot\}\cup \{  \cY^\sc(\cdot,0) : {\boxdot \in \sC,} \sc \in \sA^*_\boxdot  \} \, . \]
We recall that $ \breve {\mathscr v}_{j+1} $  belongs to the first compact set and $ \mathscr h_j $  belongs to the second one. Thus the sequence of $C^2$-diffeomorphisms {$(F^{\boxdot_j}\circ \varphi_j)_j$} is normal and {$((x,y)\in I^2\mapsto (x-\breve{\mathscr v}_{j+1}(y), y))_j$} as well, and so also the family of their compositions.  Thus   $(r_j)_j$ is  normal and $(|b_j|)_j$, $(|q_j|)_j$, $(\|r_j\|_{C^2})_j$ are bounded. If $F$ is analytic we do the same proof using this time the   Inclination Lemma in $\C^2$. Note also that: {
\[(\det D F^{\boxdot_j})\circ \varphi_j(0)=\det D( F^{\boxdot_j} \circ \varphi_j)(0) =\det D((F^{\boxdot_j}_{1}-  \breve{\mathscr v}_{j+1} \circ F^{\boxdot_j}_{2},   F_2^{\boxdot_j})\circ \varphi_j)(0) = - b_j \cdot  \partial_x ( F_2^{\boxdot_j} \circ \varphi_j)(0)\; ,\] }
{ where the first inequality follows from the fact that $\varphi_j$ is conservative and the second equality follows from the fact that the determinant of  $(x,y) \mapsto (x-{\breve{\mathscr v}_{j+1}(y)}, y )$ is equal to 1. }
As {$|\det D F^{\boxdot_j}|$} is positive on its compact domain and {$(  F_2^{\boxdot_j} \circ \varphi_j )_j$} is normal, there  exists a positive bound from below  for $|b_{j}|$. Finally the existence of the bound from below for $|q_j|$ follows from the uniform  bound $r$ at the end of the proof of \cref{ptxj}.
\end{proof}

\begin{remark}\label{remark de seb} The set $B_j$  contains both $\breve \zeta_j:=\Phi_j(0)$ and  $\zeta^{\sc_j}$ when $j$ is large,  since  $|\breve x_j -   x_j|$ is small compared to $\breve \gamma_j$ by \cref{ptxj} and thus compared to $\gamma_j$. \end{remark}
 
This Remark with \cref{minoration qi} implying $|q_j /b_j| \le K^2$ and the normality of $(\varphi_j)_j$ give:
\begin{fact}\label{bound de Bj} For $j$ large enough, the set $B_j$ contains both $\zeta^{\sc_j}$ and $\breve \zeta _j=\Phi_j(0)$, and has its diameter dominated by  $\gamma_j$. Likewise the set $\tilde B_j\supset B_j$ has its diameter dominated by  $\gamma_j$. 
\end{fact}

\subsection{Proof of \cref{propreel,propcomplex}} \label{sectiontildeBj} We start by proving one by one each of  the six inclusions and then we will prove the two limits of \cref{propreel,propcomplex}. 
 
\begin{claim}\label{two first inclusion of prop intermediare}
For $j$ large, the domain $B_j$ is included in $Y^{\boxdot_j}$ and if  $F$ is analytic, the domain  $\tilde B_j$ is included in $\tilde Y^{\boxdot_j}$. Moreover they are uniformly distant   to respectively $\partial Y^{\boxdot_j}$ and $\partial \tilde Y^{\boxdot_j}$.
\end{claim}
\begin{proof} By Fact \ref{bound de Bj} , the sets $B_j\subset  \tilde B_j$ contain  $\breve \zeta_j$ and have small diameters. As the point $\breve \zeta_j$  is uniformly distant to $\partial Y^{\boxdot_j}$ and $\partial \tilde Y^{\boxdot_j}$  by  \cref{normfor} for its $x$-coordinate $\breve x_j$ and by \cref{def epsilonboxdotpfixed} (3) for its $y$-coordinate, the claim follows.
\end{proof}

Let us show that $F^{\boxdot_j}$ sends $B_j$ into $ Y^{\sc_{j+1}}$, and if   $F$ is analytic, that  $F^{\boxdot_j}$ sends $\tilde B_j$ into $\tilde Y^{\sc_{j+1}}$, as claimed in the  second   inclusions of  the two conclusions of \cref{propreel,propcomplex}. We are going to show even more, that they are respectively included in  the following sets:
\[ Y^{\sc_{j+1}}_{\cal B}:= \{(\cal X_{j+1} (x,y),y): (x,y)\in \cal B\}\, \quad \text{with } \cal B:=\mathrm{pr}_1(B_{j+1})\times F_2^{\boxdot_j}(B_j)
\; ,\]
with $\mathrm{pr}_1$ the first coordinate projection and $F^{\boxdot_j}_2$ the second coordinate of $F^{\boxdot_j}$, and respectively:
\[\tilde Y^{\sc_{j+1}}_{ { \cal B}}:= \{(\cal Z_{j+1} (z,w), w): (z,w)\in \tilde{ \cal B}\}\, \quad \text{with }  \tilde{ \cal B}:=\mathrm{pr}_1(\tilde B_{j+1})\times  F_2^{\boxdot_j} (\tilde B_j)
\; .\]

\begin{proposition} \label{claimwanderingcomplex} For every $j \ge 1$ sufficiently large, the map $F^{\boxdot_j}$ sends $ B_j$ into $Y^{\sc_{j+1}}_{\cal B}\subset  Y^{\sc_{j+1}}$. If moreover  $F$ is analytic, the map $F^{\boxdot_j}$ sends $\tilde B_j$ into $\tilde Y^{\sc_{j+1}}_{ \cal B}  \subset \tilde  Y^{\sc_{j+1}}$.

\end{proposition}
\begin{proof}
We define $\psi_j(x,y):=(x-{\breve{\mathscr v}_{j+1}(y)}, y )$. To prove this proposition, it suffices to show:
\[\psi_j \circ  F^{\boxdot_j}(B_j)  \subset  \psi_j(Y^{\sc_{j+1}}_{\cal B})\qand \text{resp. } \psi_j \circ  F^{\boxdot_j} (\tilde B_j )  \subset  \psi_j(\tilde Y^{\sc_{j+1}}_{\cal B}).\]
Thus it suffices to prove the two following inclusions:
\begin{equation} \label{deux inclusions} \psi_j \circ { F^{\boxdot_j}(B_j) } \quad \subset\quad  \left(-\frac 1 4 |\gamma_{j+1}\sigma_{j+1}|,\frac 1 4  |\gamma_{j+1}\sigma_{j+1}| \right) \times { F^{\boxdot_j}_2(B_j) } \quad \subset\quad  \psi_j(Y^{\sc_{j+1}}_{\cal B}) \, ,\end{equation}
\begin{equation} \label{deux inclusionsC} \text{resp. } \psi_j \circ { F^{\boxdot_j}}(\tilde B_j)\quad  \subset\quad  \D \left(\frac 1 4 |\gamma_{j+1}\sigma_{j+1}| \right) \times  F^{\boxdot_j}_2(\tilde B_j)   \quad \subset\quad  \psi_j(\tilde Y^{\sc_{j+1}}_{\cal B}) \, .\end{equation}

\noindent \emph{Proof of the first inclusions of \cref{deux inclusions} and  \cref{deux inclusionsC}.} By \cref{def: Bj}, we have $B_j= \Phi_j( (-0.3,0.3)^2)$. Thus it suffices to show that  for every $(X,Y) \in (-0.3,0.3)^2$ the first coordinate $x$ of the following point is in $\left(-\frac 1 4 |\gamma_{j+1}\sigma_{j+1}|,\frac 1 4  |\gamma_{j+1}\sigma_{j+1}| \right)$:
\begin{equation}\label{def xy reel} (x,y):= 
\psi_j \circ  F^{\boxdot_j}  \circ \Phi_j(X,Y)=\psi_j \circ  F^{\boxdot_j}  \circ \varphi_j (\gamma_{j}\cdot X, \frac{q_j \cdot \gamma_{j}^2}{2b_j} \cdot Y) \, . \end{equation}
Note that  $|x|$ is the horizontal distance from ${F^{\boxdot_j} } \circ \Phi_j(X,Y)$ to   $\breve{V}^{j+1}$.  By \cref{Defrp12} it holds:
\begin{equation} \label{x'y'i+1R}
x=
q_{j} \cdot\gamma_j^2\cdot (X^{2}+\frac  Y 2+  R_j(X,Y))\quad \text{ with } 
 R_j(X,Y):=\frac1{q_j\gamma_{j}^2} r_{j}(\gamma_{j} X, \frac{q_j}{2b_j}\gamma_{j}^2 Y)
\, . \end{equation}
On the other hand, if $F$ is analytic, we have $\tilde B_j= \Phi_j( \D(0.3)^2)$ by  \cref{def: Bj}. Thus it suffices to show that  for every $(Z,W) \in \D(0.3)^2$ the first coordinate $z$ of the following point is in $\D(\frac 1 4 |\gamma_{j+1}\sigma_{j+1}|)$:  
\begin{equation}\label{def zw complex} (z, w):= 
\psi_j \circ F^{\boxdot_j}  \circ \Phi_j(Z,W)=\psi_j \circ  F^{\boxdot_j}  \circ \varphi_j (\gamma_{j}\cdot Z, \frac{q_j \cdot \gamma_{j}^2}{2b_j} \cdot W) \, . \end{equation}
Again $|z|$ is the distance from $  F^{\boxdot_j}  \circ \Phi_j(Z,W)$ to the curve  $\breve{V}^{j+1}$.  By  \cref{Defrp13} it holds:
\begin{equation} \label{z'w'i+1C}
z= 
q_{j} \cdot\gamma_j^2\cdot (Z^{2}+\frac W 2+R_j(Z,W))
\, 
\, . \end{equation}
To conclude, we need the following, whose proof is done below:
\begin{lemma}\label{minoration qi complex}  $(i)$ For $j$ large,   the map
$ (X,Y)\in (-0.3,0.3)^2\mapsto R_j( X,  Y)$ is $C^0$-small,

$(ii)$   if moreover  $F$ is analytic, the map
$ (Z,W)\in \D(0.3)^2 \mapsto R_j( Z, W)$ is $C^0$-small.
\end{lemma}

Since $|X|^2+\frac12 |Y|\le (0.3)^2+ \frac 1 2 \cdot 0.3 < 1/4$ with $|q_{j} \cdot\gamma_j^2|= |\gamma_{j+1}\cdot \sigma_{j+1}|$, the latter lemma implies immediately that for $j$ large enough and $(X,Y) \in (-0.3,0.3)^2$, the number $x$ is in  $\left(-\frac 1 4 |\gamma_{j+1}\sigma_{j+1}|,\frac 1 4  |\gamma_{j+1}\sigma_{j+1}| \right)$, which concludes the proof of the first inclusion of \cref{deux inclusions}. 

If $F$ is analytic,  since $|Z|^2+\frac12 |W|\le (0.3)^2+ \frac 1 2 \cdot 0.3 < 1/4$ for every  $(Z,W)\in \D(0.3)^2$, the same argument gives that  $z$ is in $\D(|\gamma_{j+1}\sigma_{j+1}|/4)$ and so gives the first inclusion of \cref{deux inclusionsC}.

\begin{proof}[Proof of \cref{minoration qi complex}] Let us first prove $(i)$. We recall that  $R_j(X,Y):=\frac1{q_j\gamma_{j}^2} r_{j}(\gamma_{j} X, \frac{q_j}{2b_j}\gamma_{j}^2 Y)$ for every $ (X,Y)\in (-0.3,0.3)^2$. We  recall that $\gamma_j$ is small for $j$ large. Moreover by \cref{Defrp12proprietyes} it holds:
\[0=\partial_x R_{j}(0)=\partial_y R_{j}(0)=\partial_x^2 R_{j}(0)\; .\]
 By \cref{Defrp12proprietyes}, $r_j(0)$ is small compared to $\breve \gamma_j^2$ and thus compared to  $\gamma_j^2$. Thus  $R_{j}(0)$ is small and so the  $C^0$-norm of $R_{j}$ is dominated by a	 small constant plus the $C^0$-norm of its second derivative.  We recall that $(r_j)_j$ is a   normal family in the $C^2$-topology by \cref{minoration qi}. Let $\mu$ be a uniform modulus of continuity of their second derivative. 
We observe that $(\partial_Y^2 R_{j})_j$ and $(\partial_X \partial_Y R_{j})_j$ converge uniformly to $0$ in the $C^0$ topology. On the other hand $\partial_X^2 R_{j }(0)=0$ and so $\partial_X^2 R_{j}(X,Y)=  \frac1{q_j} \partial_x^2 r_{j}( \gamma_{j} X, \frac{q_j}{2b_j}\gamma_{j}^2 Y)$ is bounded by $ \frac1{|q_j|} \mu( \diam (-0.3,0.3)^2 \cdot  K^2\cdot  \gamma_j)$ which converges uniformly to $0$ since $\gamma_j\to 0$ and $|q_{j}|^{-1} \le K$ (by \cref{minoration qi}). Thus  $R_{j}$ is $C^0$-small for $j$ large. 

Let us now prove $(ii)$. Let us notice that {$(r_j|\D(\tilde \epsilon_\sC)^2)_j$} is normal by \cref{complex normality}. Then the proof of $(ii)$ is the same as the proof of $(i)$, replacing $(-0.3,0.3)^2$ by $\D(0.3)^2$.   
\end{proof}

\noindent \emph{Proof of the second inclusions of \cref{deux inclusions} and  \cref{deux inclusionsC}.}   We will use the following\footnote{This lemma is a consequence of \cite[Lemma 2.30]{berger2018zoology} in the real case.}:
 \begin{lemma} \label{lemcomplexriRi}
The following map is $C^1$-close to the first coordinate projection when $j$ is large: 
\[\Upsilon : (x,y) \in  \cal B \mapsto \frac{\cal X_{j+1}(x,y)-\cal X_{j+1}( x_{j+1},y)}{\sigma_{j+1}} + x_{j+1}\, .\]
If $F$ is analytic, the following map is $C^1$-close to the first coordinate projection when $j$ is large: 
\[ \tilde \Upsilon : (z,w) \in  \tilde {\cal B} \mapsto \frac{\cal Z_{j+1}( z,w)-    \cal Z_{j+1}(x_{j+1},w)}{\sigma_{j+1}} + x_{j+1}\, .\]
\end{lemma}
\begin{proof}
To show this lemma it suffices to use \cref{distortion}, whose assumptions must be fulfilled.  The first is that when $j$ is large, the diameter of $\cal B$ must be small and when $F$ is analytic, that the diameter of $\tilde {\cal B}$ must be small compared to $\sigma_{j+1}$. Indeed  by \cref{def Vj equiv2}, it holds $\breve \gamma_j =o( |Y^{\sc_{j+1}}|)$ and so $\breve \gamma_j=o(|\sigma_{j+1}|)$ by \cref{PY01 cor 3.4coro}. Thus by \cref{def sigma et gamma}, we have $\gamma_j=o(|\sigma_{j+1}|)$. 
The second assumption in the real case is that the distortion of $(\cX_j)_j$ is bounded; this is the case by   \cref{PY01 cor 3.4}. The second assumption in the complex case is that $\tilde{\cal B}$ is uniformly distant to the boundary of $\tilde I^2$. 
This is indeed the case since $\tilde B_j \ni \zeta^{\sc_j}$ has a small diameter by Fact \ref{bound de Bj}, while { the first coordinate of $\zeta^{\sc_j}\in Y^{\boxdot_j}$ and the second coordinate of  $F^{\boxdot_j} (\zeta^{\sc_j})\in F^{\boxdot_j} (Y^{\boxdot_j})$ are uniformly distant from $\partial \tilde I$,} 
by \cref{def epsilonboxdotpfixed} (1) and (2).\end{proof}

 Let us now use this lemma  to prove  the second inclusions of \cref{deux inclusions} and \cref{deux inclusionsC}. We notice that it suffices to show:
 \[\{ \cal X_{j+1}(x,y)-\breve {\mathscr v}_{j+1}(y) : x \in \mathrm{pr}_1(B_{j+1})    \}\supset   (- |\gamma_{j+1}\sigma_{j+1}|/4, |\gamma_{j+1}\sigma_{j+1}|/4)  \,  \quad \forall y \in \  F_2^{\boxdot_j}( B_j) \, .\]
\[  \text{ resp. }  \{   \cal Z_{j+1}(z,w)-\breve {\mathscr v}_{j+1}(w) : z\in \mathrm{pr}_1(\tilde B_{j+1})    \}\supset \D( |\gamma_{j+1}\sigma_{j+1}|/4)\,  \quad \forall w \in \   F_2^{\boxdot_j}(\tilde B_j) \, .\]
Now we   recall that  by \cref{def vi hi} we have   $\breve {\mathscr v}_{j+1}(y) = \cal X_{j+1}(x_{j+1},y)$ (resp.  by the discussion following \cref{Defrp13} we have $\breve {\mathscr v}_{j+1}(w) = \cal Z_{j+1}(x_{j+1},w)$). So it suffices to show that:
\[\{\Upsilon(x,y)-x_{j+1}: x \in  \mathrm{pr}_1(B_{j+1})   \}\supset      (- |\gamma_{j+1}|/4, |\gamma_{j+1}|/4)      \, \quad \forall y \in \  { F_2^{\boxdot_j}( B_j)} \, .   \]
\[   \text{ resp. } \{\tilde \Upsilon(z,w)- x_{j+1}: z\in  \mathrm{pr}_1(\tilde B_{j+1})   \}\supset \D( |\gamma_{j+1}|/4)\, \quad \forall w \in \  { F_2^{\boxdot_j}(\tilde B_j)} \, .   \]
Now we   notice that  $|\breve x_{j+1} -   x_{j+1}|$ is small compared to  $\breve \gamma_{j+1}$ by \cref{ptxj} and so compared to $\gamma_{j+1}$ by \cref{def sigma et gamma}.
As  $\mathrm{pr}_1( B_{j+1})= \breve x_{j+1}  +( - 0.3\cdot \gamma_{j+1}  , 0.3\cdot \gamma_{j+1} )$ (resp. $\mathrm{pr}_1(\tilde B_{j+1})= \breve x_{j+1}+\D(0.3\cdot \gamma_{j+1})$), the result follows from \cref{lemcomplexriRi}  and $1/4<0.3$.
\end{proof}

The following (with \cref{claimwanderingcomplex})  shows  the two remaining  inclusions of \cref{propreel,propcomplex}:
\begin{proposition}\label{claimwanderingcomplex2}
For every $j \ge 1$ large enough, the map 
$F^{{\sc_{j+1}}}$ sends $Y^{\sc_{j+1}}_{\cal B}$ into $B_{j+1}$.
If moreover  $F$ is analytic, the map 
$F^{{\sc_{j+1}}}$ sends $\tilde Y^{\sc_{j+1}}_{\cal B}$ into $\tilde B_{j+1}$.
\end{proposition}
\begin{proof}
Since   $Y^{\sc_{j+1}}_{\cal B}=\{(\cal X_{j+1}(x,y),y): (x,y)\in \cal B \}$ and $\tilde Y^{\sc_{j+1}}_{ \cal B }=\{(\cal Z_{j+1}(z,w), w): (z,w)\in \tilde { \cal B} \}$,  by definition of the implicit representation, it holds:
\begin{fact}\label{keyprop 2.2}
The map $F^{\sc_{j+1}}$ sends $ Y^{\sc_{j+1}}_{\cal B}$  to $\{(x,\cal Y_{j+1}(x,y)): (x,y)\in \cal B\}$. If moreover  $F$ is analytic, the  map $F^{\sc_{j+1}}$ sends $\tilde Y^{\sc_{j+1}}_{\cal B }$  to $\{(z,\cal W_{j+1}(z,w)): (z,w)\in  \tilde { \cal B}   \}$.
\end{fact}

Thus we shall prove that  $\{(x,\cal Y_{j+1}(x,y)): (x,y)\in \cal B\}$  is included in $B_{j+1}$ and 
respectively that  $\{(z,\cal W_{j+1}(z,w)): (z,w)\in  \tilde { \cal B}   \}$) is included in   $\tilde B_{j+1}$. By the skew product form of the sets $B_{j+1}$ and $\tilde B_{j+1}$ (see \cref{def XiYi,def: Bj}),    it suffices to show that:
\begin{multline*} | \cY_{j+1}(x,y)-\mathscr h_{j+1}(x)|=|\cY_{j+1}(x,y)-\cY_{j+1}(x,0)|\\ \qand \text{resp.}\quad   | \cW_{j+1}(z,w)-\mathscr h_{j+1}(z)|=|\cW_{j+1}(z,w)-\cW_{j+1}(z,0)|\end{multline*}
are  small compared to $\gamma_{j+1}^2$ for any $(x,y) \in \cal B$  and resp.  $(z,w) \in \tilde { \cal B}$.   They are respectively smaller than $|F^{\sc_{j+1}}(Y^{\sc_{j+1}})|$ and $|F^{\sc_{j+1}}(\tilde Y^{\sc_{j+1}})|$. By \cref{propreel} $(iii)$ and \cref{propcomplex} $(\widetilde {iii})$, we have $ |F^{\sc_{j+1}}(Y^{\sc_{j+1}})|  =o(\breve \gamma_{j+1}^2)$, and moreover $ |F^{\sc_{j+1}}( \tilde Y^{\sc_{j+1}})|  =o(\breve \gamma_{j+1}^2)$ if $F$ is analytic. Since $\breve \gamma_{j+1}^2 = O(\gamma_{j+1}^2)$ by \cref{def sigma et gamma}, the result follows.
 \end{proof}
 We are now in position to prove:
\begin{proof}[Proof of  \cref{propreel,propcomplex}]
The six inclusions of \cref{propreel,propcomplex} are given by  Claim \ref{two first inclusion of prop intermediare}, \cref{claimwanderingcomplex} and then  \cref{claimwanderingcomplex2}.
Hence it remains only to show:
\begin{equation} \label{lecon a bieblers}\lim_{k \to \infty} \diam   F^k(B_j \times \{\si(\boxdot_j)\}   )=0 \qand      \lim_{k \to \infty} \diam  {  F^k(  \tilde   B_j \times \{\si(\boxdot_j)\} } )=0 \quad \text{ if $F$ is analytic}    ,\end{equation}
 when $j$ is large.  To this end, it suffices to   show that for every $\eta>0$, for every $j$ sufficiently large, the $|\sc_{j+1}|$ first iterates of  $F^{\boxdot_j}(B_j)$ (resp. $F^{\boxdot_j}(  \tilde   B_j)$) have diameters smaller than $\eta$. To this end we use:
 \begin{lemma} \label{point} When $N$ is large, 
 $\max_{ (\sp ,\sq) \in \sB^*:   |\sp|,|\sq|\ge N} \diam Y^{ \sp} \cap F^{\sq}(Y^\sq)$ is small.
\end{lemma}
\begin{proof}
For every $s, t \in I$, let $\cal H_s$ be the image of $\{y=s\} \cap Y^\sq$ under $F^\sq$ and $\cal V_t$ be the image of $\{x =t\} \cap F^\sp (Y^\sp)$ under $(F^{\sp})^{-1}$.  We notice that $\cal H_s$ and $\cal V_t$ intersects at a unique point  $\phi(s,t)\in Y^\sp \cap F^{\sq} (Y^\sq)$, since by the cone properties of \cref{defpiece}, $\cal H_s$ is a graph over $x \in I$ with tangent spaces in $\chi_h$ and $\cal V_t$ is a graph over $y \in I$ with tangent spaces in $\chi_v$. It is straightforward that $\phi$ sends $I^2$ onto $Y^\sp \cap F^{\sq} (Y^\sq)$. Also by transversality, $\phi$ is differentiable. For every $(s,t)\in I^2$,  the vector $\partial_t \phi(s,t)$ is tangent to $\cal H_s$ and so in $\chi_h$ and the norm of its first coordinate projection is at most $\lambda^{-|\sp|}$ by property (2) of \cref{defpiece}. Thus the norm of $\partial_t \phi(s,t)$  is at most $\sqrt{1+\theta^2}\cdot \lambda^{-|\sp|}$.  Likewise the norm of $\partial_s \phi(s,t)$ is at most $\sqrt{1+\theta^2}\cdot \lambda^{-|\sq|}$. Thus the norm of $D\phi$ is small when $|\sp|$ and $|\sq|$ are large  and so the diameter of $\phi(I^2)=Y^\sp \cap F^{\sq} (Y^\sq)$ is small.\end{proof}

By \cref{point}, there exists $N$ independent of $j$ such that the $k^{th}$ iterate of { $F^{\boxdot_j}(B_j)$} has diameter smaller than $\eta$ for every $N \le k \le |\sc_{j+1}|-N$. By Fact \ref{bound de Bj},  the diameter of $B_j$ is dominated by $\gamma_j$ which is small when $j$ is large and likewise for { $F ^{ \sc_{j+1} }\circ F ^{ \boxdot_j}(B_j)\subset B_{j+1}$}. Thus their $N$ first iterates and preimages have also small diameter when $j$ is large. This implies the first limit in \cref{lecon a bieblers}.  

If $F$ is analytic, we prove the second  limit in \cref{lecon a bieblers} exactly the same way,   replacing \cref{point} by \cref{pointC} below. 
\end{proof}

\begin{lemma} \label{pointC}   When $N$ is large, 
 $\max_{ (\sp ,\sq) \in  \sB^*:     |\sp|,|\sq|\ge N} \diam \tilde Y^{ \sp} \cap F^{\sq}(\tilde Y^\sq)$  is small.
\end{lemma}
\begin{proof} The proof is exactly the same as for \cref{point}, up to changing $I$ to $\tilde I$ and $Y$ to $\tilde Y$. 
\end{proof}

\section{Parameter selection} \label{parameterselection}
 In this Section, we will work with an unfolding $(F_p,\pi_p)_{p\in \cP}$ of wild type $(\sA, \sC)$. Under extra assumptions ($\#\sC=5$ and moderate dissipativeness),  given any $p_0\in \mathrm{int\, } \cP $,  we will define in \cref{Infinite chain of nearly heteroclinic tangencies} a sequence of parameters $(p_i)_{i\ge 0}$ which converges to a certain $p_\infty$ close to $p_0$ such that at each $p_i$, the map $F_{p_i}$ will satisfy tangency conditions. Then  we prove  \cref{main them Fp} in \cref{sectionverifie}. To this aim, we first show that $F_{p_\infty}$ satisfies the assumptions of \cref{propreel}, and those of \cref{propcomplex} if $F_{p_\infty}$ is analytic, and finally we will show property $(\blacklozenge)$.

\subsection{Estimates for parameter families} \label{estimees}

We set up here some analytic constants useful for the proof.  
For every $\sc\in   \sA^*$ and $p\in \cP$, let $(\cX^\sc_p,\cY^\sc_p)$ denote the implicit representation of $(Y^{\sc}_p,F_p^{\sc})$. We recall that by  
\cref{conti combi},  the family  $(\cX_p^\sc, \cY_p^\sc)_{p\in \cP}$ is of class~$C^2$. Here is a consequence of \cref{PY01 cor 3.4}:
\begin{corollary} \label{lemdist}
\index{Distortion bounds  {$\overline{  \mathscr{D}  }$  } and    $\bar {\mathscr  B}$  }
There exists  {$\overline{\mathscr{D} }>0$} such that for all $\sc\in  \sA^*\text{ and } p\in \cP$: $  {\mathscr{D}(Y^\sc_p, F^\sc_p) \le   \overline{ \mathscr{D} } }$.
\end{corollary}
\begin{proof}As the domain of each $F_p$ is compact, and the set $\cP$ is compact, the $C^2$-norm of $F_p$ and its inverse are bounded independently of $p\in \cP$. Thus for every $\sa\in  \sA$, the distortion of each  hyperbolic transformation $(Y^\sa_p, F^\sa_p)$ is bounded independently of $p\in  \cP$. By finiteness of $\sA$, these distortions are bounded independently of $p\in \cP$ and $\sa\in  \sA$. We conclude using \cref{PY01 cor 3.4}.
\end{proof}
An immediate consequence of this bound is:
\begin{corollary}\label{coro bound}
For every $\sc\in \sA^*$, for every $(x,y), (x',y')\in I^2$, it holds for every $p\in \cP$: 
 {\[
\left | \frac{\partial_x \cX^\sc_p(x,y)}{\partial_x \cX^\sc_p(x',y')}\right | \le
\exp( \bar {\mathscr{D} }\cdot\|(x-x', y-y')\|) \le \bar { \mathscr{B} }
, \quad \text{with } \bar {\mathscr{B} }:= \exp( 2\sqrt{2}\cdot \bar {\mathscr{D} })<\infty\; .    \]   }
\end{corollary}

\begin{definition} \label{defwoverlin} Given a nonempty compact subset $\cP_0\subset \cP$, we define the \emph{maximal and minimal widths of  $\sc\in \sA^*$ [among $p\in \cP_0$]} by respectively:
\[\underline \sfw(\sc, \cP_0):= \min_{p\in \cP_0, \; I^2} |\partial_{x_1} \cX_p^\sc|
\qand
\overline \sfw(\sc, \cP_0):= \max_{p\in \cP_0, \; I^2} |\partial_{x_1} \cX_p^\sc|\; .\]
The \emph{maximal and minimal vertical heights} of $\sc$  are respectively:
\[\underline \sfh(\sc, \cP_0):= \min_{p\in \cP_0, \; I^2} |\partial_{y_0} \cY_p^\sc|
\qand
\overline \sfh(\sc, \cP_0):= \max_{p\in \cP_0, \; I^2} |\partial_{y_0} \cY_p^\sc|\; .\] \index{$\underline \sfw(\sc, \cP_0)$ and $\overline \sfw(\sc, \cP_0)$ } \index{$\underline \sfh(\sc, \cP_0)$ and $\overline \sfh(\sc, \cP_0)$ }
\end{definition}
\begin{remark}\label{rem sur w et h}
For any $p \in \cP_0$ and $\sc\in \sA^*$, we notice that any line $\R\times \{y\}$, $y\in I$, intersects $Y^\sc_p$ at a segment of length in $[2\underline \sfw(\sc, \cP_0),2\overline \sfw(\sc, \cP_0)]$ ; whereas any line $\{x\}\times \R$, $x\in I$, intersects $F^\sc_p(Y^\sc_p)$ at a segment of length in $[2\underline \sfh(\sc, \cP_0),2\overline \sfh(\sc, \cP_0)]$. In particular, with the notations of \cref{defYnorme},  is holds:
\[2\underline \sfw(\sc, \cP_0)\le |Y^\sc_p|\le 2\overline \sfw(\sc, \cP_0) 
\qand 2\underline \sfh(\sc, \cP_0)\le |F^\sc_p(Y^\sc_p)|\le 2 \overline \sfh(\sc, \cP_0) \; .\]
\end{remark}
 
\begin{center}
{\bf From now on, we will suppose that $(F_p,\pi_p)_p$ is moderately dissipative and we fix  a   parameter $p_0\in \mathrm{int \, }\cP$.}
\end{center}

The following will be used many times and proved in \cref{Proof of def P0}:
\begin{proposition}[and definition of $\eta_0$ and $\cP_0$]\index{$\cP_0$}\label{def P0}\label{utilisation epsilon}   There exists $\eta_0>0$ and a compact neighborhood $\cP_0\Subset \cP$ of $p_0\in \cP$  satisfying the following properties for every $\sc\in \sA^*$ s.t. $ \overline \sfw(\sc, \cP_0)<\eta_0$:
\begin{enumerate}
\item $\underline \sfw(\sc)\le \overline \sfw(\sc)< \underline \sfw(\sc)^{1-\epsilon^2}
,\quad $ with $\underline \sfw(\sc):=\underline \sfw(\sc, \cP_0)$ and $\overline \sfw(\sc):=\overline \sfw(\sc, \cP_0)$,
\item $ \overline \sfh(\sc)^\epsilon\le\underline \sfw(\sc), \quad $ with $\overline \sfh(\sc):= \overline \sfh(\sc, \cP_0)$.
\end{enumerate}
If at some $ \tilde p \in \cP_0$ the map $F_{\tilde p}$ is analytic, then for $\rho$ sufficiently small,
for every $\sc\in \sA^*$, the hyperbolic transformation $(Y^\sc_{\tilde p}, F^\sc_{\tilde p})$  extends to a  $C^\omega_\rho$-hyperbolic transformation whose implicit representation $(\cZ_{\tilde p}^\sc, \cW_{\tilde p}^\sc)$ satisfies:
\begin{enumerate}
    \item[(2')] $\max_{ \tilde I^2} |\partial_{w_0} \cW_{ \tilde p}^\sc|^\epsilon \le   \min_{\tilde I^2} |\partial_{z_1} \cZ_{\tilde p}^\sc|  \quad $ if $\quad  \overline \sfw(\sc, \cP_0)<\eta_0$. \end{enumerate}
\end{proposition}

We first notice the following monotonicity property for widths and heights: 

\begin{lemma} \label{monotonicitywh}
For every finite words $\sw \in \sA^*$ and $\sw' \in \sA^*$ such that $\sw \cdot \sw'$ is admissible, it holds:
\[  \overline \sfw(\sw \cdot \sw')  \le \overline \sfw(\sw)      \qand   \overline \sfh(\sw \cdot \sw')  \le \overline \sfh(\sw')        \]
\end{lemma}

\begin{proof} Let us prove that $ \overline \sfw(\sw \cdot \sw')  \le \overline \sfw(\sw)$, the proof of $\overline \sfh(\sw \cdot \sw')  \le \overline \sfh(\sw')      $ is similar. The case where $\sw$ or $\sw'$ is the empty word being trivial, let us suppose that $\min(|\sw|, |\sw'|) \ge 1$. By \cref{compoal} of the appendix, we have:
\[
 \overline \sfw(\sw \cdot \sw')  \le \frac1{1-\theta^2} \overline  \sfw(\sw)\cdot  \overline  \sfw(\sw').
 \]
According to the cone condition (2) of \cref{defpiece}, we have $|\partial_x \cX^{\sw'}_{p}| \le {\lambda^{-|\sw'|}} \le \lambda^{-1}$ and so:
\[
 \overline \sfw(\sw \cdot \sw')  \le \frac1{1-\theta^2} \cdot  \overline  \sfw(\sw)\cdot {\lambda^{-|\sw'|}} < \overline  \sfw(\sw).
 \] 
\end{proof}

\begin{definition}[$\sp_m$] \index{$\sp_{m}$}
For every $m \in \N^*$, for every $\ss=(\sa_j)_{j\ge 0} \in \avv \sA$ and $\su=(\sa_j)_{j< 0}\in \arr \sA$,
 let $\sp_m(\ss):=\sa_0\cdots \sa_{m-1}$ be the concatenation of the $m$ first terms of $\ss$ and let $\sp_m(\su)=\sa_{-m}\cdots \sa_{-1} $ be the concatenation of the $m$ last terms of $\su$.
\end{definition}

The following will be used many times.
\begin{lemma}[and definition of $\eta_1$]  \index{$\eta_1$} \label{ def sc} There exists $\eta_1\in (0, \eta_0)$ such that for every $\ss\in \avv \sA$, $\su\in \arr \sA$, for every $\delta\in (0, \eta_1)$, if $m\ge 1$ and $m'\ge 1$ are maximal such that:
\[\overline \sfw(\sp_m(\ss))> \delta\qand \overline \sfh(\sp_{m'} (\su))> \delta\; ,\]
 then for every $\sd\in \sA^*$ with $\underline {\sfw}(\sd) \ge \delta^{\epsilon^4} $ and such that  $\sc:= \sp_{m+1}(\ss) \cdot \sd \cdot \sp_{m'+1}(\su)\in \sA^*$, it holds:
 \[\delta\ge \overline \sfw(\sc)\ge \underline \sfw(\sc) \ge \delta^{1+\epsilon+2\epsilon^2}    \qand \delta\ge \overline \sfh(\sc)\ge \underline \sfh(\sc)  \; .\]
Moreover such $\sd\in \sA^*$ always exist.
\end{lemma}
\begin{proof}
By maximality of $m$ and $m'$, the maximal widths of $\sp_{m} (\ss)$ and $\sp_{m'} (\su)$ are small when $\eta_1$ is small, and sufficiently  to apply \cref{utilisation epsilon} to these words.
Also by maximality, it holds $\delta\ge \overline \sfw(\sp_{m+1}(\ss))$ and $\delta\ge \overline \sfh(\sp_{m'+1}(\su))$. 
Thus  by \cref{monotonicitywh} it holds:
\begin{equation*}\label{ineq involv delta}\delta\ge \overline \sfw(\sc)\ge \underline \sfw(\sc)\qand \delta\ge \overline \sfh(\sc)\ge \underline \sfh(\sc)\; .\end{equation*}
It remains to prove that $\underline \sfw(\sc)\ge \delta^{1+\epsilon+2\epsilon^2}$. By Lemma \ref{compoal} we have:
\[\underline \sfw(\sc)\ge 
\underline \sfw(\sp_m(\ss))\cdot \min_\sa \underline \sfw(\sa) \cdot \underline \sfw(\sd)\cdot \min_\sa \underline \sfw(\sa)\cdot \underline \sfw(\sp_{m'}(\su))  \cdot \frac1{(1+\theta^2)^{4}}\; . \]
We now use \cref{utilisation epsilon} (1) with $\underline \sfw(\sp_m(\ss))^{1-\epsilon^2} >\overline \sfw(\sp_m(\ss))> \delta$ to obtain:
\[\underline \sfw(\sc) > \delta^{(1-\epsilon^2)^{-1}} \cdot (\min_\sa \underline \sfw(\sa) )^{2} \cdot \underline \sfw(\sp_{m'}(\su))\cdot \underline \sfw(\sd) \cdot \frac1{(1+\theta^2)^{4}} \; .\]
By \cref{utilisation epsilon} (2), $\underline \sfw(\sp_{m'}(\su)) \ge \overline \sfh(\sp_{m'}(\su))^\epsilon$.  Also $ \underline \sfw(\sd)\ge   \delta^{\epsilon^4}$. Thus:
\[\underline \sfw(\sc) > \delta^{(1-\epsilon^2)^{-1}} \cdot (\min_\sa \underline \sfw(\sa) )^{2} \cdot \delta^\epsilon (1+\theta^2)^{-4}\cdot \delta^{\epsilon^4}\]
We now use the logarithm to compare this with the sought inequality:
\begin{multline} \log ( \underline \sfw(\sc)\delta^{-1-\epsilon-2\epsilon^2} )> 
({(1-\epsilon^2)^{-1}}+\epsilon+\epsilon^4-1-\epsilon-2\epsilon^2) \log \delta +2\log \min_\sa \underline \sfw(\sa) 
-4 \log  (1+\theta^2)\\ =
( \epsilon^{4} (1-\epsilon^2)^{-1}+\epsilon^4-\epsilon^2) \log \delta+2\log \min_\sa \underline \sfw(\sa) 
-4 \log  (1+\theta^2)
\end{multline}
Note that $({ \epsilon^{4}}(1-\epsilon^2)^{-1}+\epsilon^4-\epsilon^2)<0$, hence for  $\eta_1>\delta$ sufficiently  small, the right most sought inequality is obtained. 
Also, by connectedness of the graph, when $\delta$ is small, given any two vertexes  in $\sV$, there exists a word $\sd$ which joints them and so that $\underline \sfw (\sd)>\delta^{\epsilon^4}$. 
\end{proof}
\begin{remark}\label{ def sc2}
  Let $0<\mathscr W \le 1$  and $\cal L$ be numbers such that every pair of vertexes in $\sV$ can be connected by a path  $\sd^r$ such that ${\underline{\sfw}}(\sd^r)>\mathscr W$ and $|\sd^r|\le \cal L$. Thus, under the setting of  \cref{ def sc} fixing $m,m'$, $\ss$ and $\su$,  for every  word $\sd^o\in \sA^*$ satisfying 
\begin{equation*} \underline {\sfw} (\sd^o ) > {  \delta^{\epsilon^4} \cdot \frac{ 1+\theta^2}{ \mathscr W } \qand 
\si(\sd^o)=\st(\sp_{m+1}(\ss))}  \; ,
\end{equation*}
there exists   $\sd^r {  \in \sA^*}$ starting at $\st(\sd^o)$, ending at $\si ( \sp_{m'+1}(\su))$ and such that  ${\underline{\sfw}}(\sd^r)>\mathscr W$,  $|\sd^r|\le \cal L$. Thus in \cref{ def sc}, we can take  $\sd=  \sd^o\cdot \sd^r$.
\end{remark}

The following provides some uniform bounds on the unfolding.
\begin{proposition}[and definition of $\vartheta$ and $L_0$]\label{claim defining eta0 and L0}\index{$\vartheta$ and $L_0$} There exist  $\vartheta>0$ and $L_0>0$ such that for every
$p'\in \cP_0$, $\sm:=(\su_\boxdot,\ss_\boxdot)_{\boxdot\in\sC} \in \prod_{\boxdot \in \sC} \arr{\sA}_\boxdot \times \avv{\sA}_\boxdot$, 
the map  $\Psi_{\sm}:p\in \cP\mapsto ({\cb \cal V(\su_\boxdot, \ss_{\boxdot},p)})_{\boxdot\in \sC}$ is $\frac{L_0}2$-biLipschitz from a neighborhood of $p'$ in $\cP$ onto the $2\vartheta$-neighborhood of  $\Psi_{\sm}(p')$.
\end{proposition}
\begin{proof}
Let $\alpha>0$ be such that for every $p'\in \cP_0$, the euclidean closed ball $B(p', \alpha)$ is included in $\mathrm{int} \, \cP$. Let 
$\sm= (\su_\boxdot,\ss_\boxdot)_{\boxdot\in \sC} \in \prod_{\boxdot \in \sC} \arr{\sA}_\boxdot \times \avv{\sA}_\boxdot$. By  $(\mathbf H_2)$,  the restriction  $\Psi_{\sm}|B(p', \alpha)$  is  a diffeomorphism onto its image. By compactness of $B(p', \alpha)$, the restriction   $\Psi_{\sm}|B(p', \alpha)$  is biLipschitz; let  $L(p', \sm)/2>0$  be its biLipschitz constant. Also let  $\rho(p',\sm)>0$  be maximal such that
  $\Psi_{\sm}(B(p', \alpha))$  contains the  $2\rho(p', \sm)$-neighborhood of  $\Psi_{\sm}(p')$. By \cref{def calV}, the functions
 $\rho$ and $L$ are continuous on the compact set $\cP \times    \prod_{\boxdot \in \sC} \arr{\sA}_\boxdot \times \avv{\sA}_\boxdot $ and so bounded from below by positive constants $\vartheta$ and $L_0$.
\end{proof}

\subsection{Dynamics at the folding map} \label{sectionfolding}
In this subsection,   we  continue to  consider a moderately dissipative  unfolding  $(F_p, \pi_p)_{p\in \cP}$ of wild type $(\sA, \sC)$ and a fixed parameter $p_0\in \mathrm{int \, }\cP$. 
  The compact neighborhood $\cP_0\Subset \cP$ of $p_0$ was defined by \cref{def P0}, whereas  $\vartheta$, $L_0$ and $\eta_1$ were defined in \cref{claim defining eta0 and L0} and \cref{ def sc}.

To study the folding map $F_p^{\boxdot}| Y^\boxdot_p$  for any $\boxdot \in \sC$, we will use analogous functions to $\cal V({\cb \su, \ss,p})=a_p^{\su}-b_p^{\ss}$ defined in \cref{deficalVus}, when $\su$ and/or $\ss$ are replaced by finite words $\sc\in \sA^*$. The definitions of the counterparts of $a_p^{\su}$ and $b_p^{\ss}$ use the following counterparts of $W^{\su}_p$ or $W^\ss_p$:
\begin{definition}[$H^\sc_p$ and $V^\sc_p(\zeta)$]\label{def H and V}
\index{$H^\sc_p$ and $V^\sc_p(\zeta)$} For every $p\in \cP_0$, let $H^{\sc}_p$ be the push forward by $F^\sc_p$ of 
 $I\times \{0\}$ and  given a point $\zeta\in   Y^\se  $, 
let $V^{\sc}_p(\zeta)$ be the pullback by $F^\sc_p$ of the  $\pi^{\st(\sc)}_p$-fiber containing  $\zeta$: 
\[H^{\sc}_p=F_{p }^{{\sc}}(Y^{\sc} \cap I\times \{0\})\qand 
V^{\sc}_p(\zeta)= (\pi^{\st(\sc)}_p\circ F^\sc_p)^{-1}(\{\pi^{\st(\sc)}_p(\zeta)\})\; .
\]
\end{definition}

The following is an immediate consequence of \cref{def adapted proj}\emph{(3)}:
\begin{fact}\label{identitifcation feuille V}
  The curve  $V^{\sc}_p(\zeta)$ is the fiber  of $\pi^{\si(\sc)}_p$ containing $(F^{\sc}_p)^{-1}(\zeta)$.
\end{fact}

We now introduce the analogous of \cref{factzetajnew} for a parameter family. 
\begin{proposition}[and definition of $\cb \sW_\boxdot$, $\zeta^\sc_p$ and $\eta_2$] \label{def a b}\index{$\eta_2$}\index{$\zeta^\sc_p$}   \index{$\cb \sW_\boxdot$}
 There exists $\eta_2 \in (0, \eta_1)$ such that for every $\boxdot \in \sC$, $p \in \cP_0$ and $\sc \in \sA^*$ in the set:
 \[   {\cb \sW_\boxdot}:=\{\sc \in \sA^* : \sc \text{ coincides with the } |\sc| \text{ last letters of a certain } \su \in \arr{\sA}_\boxdot \qand  \overline \sfh(\sc)<\eta_2\},  \]
 the curve $H^\sc_p$ intersects { $Y^\boxdot_p$} but not {$\partial^u Y^\boxdot_p$},
 the curve $H^\sc_p$  is tangent to a unique fiber of $\pi^{\st(\boxdot)}_p \circ F^\boxdot_p$ at a unique point $\zeta^\sc_p \in \mathrm{int\, } { Y^\boxdot_p }$, and the tangency is quadratic.
 
\end{proposition}
\begin{proof} 
  First observe that when $\eta_2$ is small, the minimal length of the words in $\cb \sW_\boxdot$ is large. Thus, the proof of this Proposition is the same as the one of \cref{factzetajnew} but for a parameter family this time. Indeed,  note that $\arr{\sA}_\boxdot$ does not depend on $p$ for every $\boxdot \in \sC$ (by \cref{def family adapted}), so it suffices to  replace the Inclination Lemma and the condition on  $|\sc|$ by respectively the parametric Inclination \cref{inclination} \cpageref{inclination} and the condition on $\overline \sfh(\sc)$. By finiteness of $\sC$, one can take the same value $\eta_2$ for every $\boxdot \in \sC$. 
\end{proof}
\begin{remark} 
Up to reducing $\eta_2$, we can assume that $  \sA^*_\boxdot(F_p)\supset {\cb     \sW}_\boxdot$ for any $p \in \cP_0$, where  $  \sA^*_\boxdot(F_p)$ was defined in \cref{defzetafixep} as  $  \sA^*_\boxdot$.
\end{remark}
\begin{definition}\index{$a_p^\sc$ and $b_p^\sc$}\label{def a b vrai} 
For every $\boxdot \in \sC$ and  $\sc\in   {\cb     \sW}_{\boxdot} $, we denote:
\[a^\sc_p:=  \pi_p^{\st(\boxdot)} \circ F_p^\boxdot(\zeta^\sc_p)\in I\qand b_p^{\sc}:= \pi^{\si(\sc)}_p   \circ (F^{\sc}_p)^{-1}(\zeta^{\sc}_p)\; .\]
\end{definition}

The following compares $a_p^{\sc}$ and $b_p^{\sc}$ to the limit case $a_p^{\su}$ and $b_p^{\ss}$:
\begin{proposition} \label{C0 bounds}
For every $\boxdot \in \sC$, $\sc \in  {\cb     \sW}_{\boxdot}$ and $p\in \cP_0$, it holds:
\begin{enumerate}
\item For every $\ss\in \avv \sA$ starting with $\sc$, it holds  $|b_p^{\sc}-b_p^\ss|\le 2\cdot   \overline \sfw(\sc) $.
\item For every $\su\in \arr \sA$ ending with $\sc$, it holds  $|a_p^{\sc}-a_p^{\su}|\le   \|D( \pi^{\st(\boxdot)}_p\circ F_p^\boxdot) \|_{C^0}\cdot 2\cdot \overline \sfh(\sc) $.
\end{enumerate}
\end{proposition}
\begin{proof}
We first prove $(1)$. Let $\ss'\in \avv \sA$ be such that $\ss= \sc\cdot \ss'$.
We recall that $\zeta^\sc_p\in {Y^\boxdot_p}$ for every $p$. Thus by \cref{def adapted proj} (0) and (1), the $\pi^{\st(\sc)}_p$-fiber of $\zeta^\sc_p$ is a curve with tangent space in $\chi_v$, included in $Y^\se$ and with endpoint in  and endpoints in  $I \times \partial I$. 
These properties also hold true for  $W^{\ss'}_p$ (by \cref{def bps} for $W^{\ss'}_p$).   By \cref{defpiece} (3), their respective pullbacks $V^\sc_p(\zeta^\sc_p)$ and $W^\ss_p$ by $F^\sc_p$  are then also graphs over $y \in I$ with direction in $\chi_v$ and  endpoints in $I \times \partial I$, which are moreover  included in $Y^\sc_p$. They are then $|Y^\sc_p|$-$C^0$-close and so $2\overline \sfw (\sc)$-$C^0$-close by \cref{rem sur w et h}.
By \cref{def adapted proj} (3), they are also fibers of $\pi^{\si(\sc)}_p$.  By \cref{def adapted proj} (2),  their images by $\pi^{\si(\sc)}_p$ (equal to $b_p^{\sc}$ and $b_p^\ss$) are  the $x$-coordinates of the  intersection points of these fibers with $I\times \{0\}$.  This implies $|b_p^{\sc}-b_p^\ss|\le 2 \cdot \overline \sfw(\sc) $.

Let us  prove the second item. Similarly to the case of vertical curves, we can show that  $H^\sc_p \cap Y^\boxdot_p \subset H^\sc_p$ and $W^\su_p \cap Y^\boxdot_p \subset W^\su_p$ are $2\overline \sfh(\sc)$-$C^0$-close. Their respective images by $\pi^{\st(\boxdot)}_p\circ F_p^\boxdot$ are consequently  $ \|D(\pi^{\st(\boxdot)}_p\circ F_p^\boxdot )\|_{C^0}\cdot 2 \overline \sfh(\sc) $-close. Thus the endpoints of  $\pi^{\st(\boxdot)}_p\circ F_p^\boxdot(H^\sc_p \cap {Y^\boxdot_p})$ and  $\pi^{\st(\boxdot)}_p\circ F_p^\boxdot(W^\su_p\cap {Y^\boxdot_p})$ are $2\cdot \|D(\pi^{\st(\boxdot)}_p\circ F_p^\boxdot )\|_{C^0}\cdot \overline \sfh(\sc) $-close. We conclude by noting that $a^\sc_p$ and $a^\su_p$ are  endpoints of these segments.
\end{proof}
\begin{definition}\index{$\cb \mathcal V(\sc, \sc',p)$ and $\cb \mathcal V(\sc, \ss,p)$}  \label{def cal V new}
For all  $\boxdot, \boxdot' \in \sC$, 
$\ss \in \avv \sA_\boxdot $, $\sc \in     {\cb     \sW}_{\boxdot}$ and  $\sc' \in {\cb     \sW}_{\boxdot'}$ such that $\st(\boxdot)=\si(\sc')$, 
 for every $p\in \cP_0$, we put:
\[\mathcal V\cb (\sc, \sc',p) := a_p^{\sc}-b_p^{\sc'}  \qand
\mathcal V(\cb \sc, \ss,p)  := a_p^{\sc}-b_p^{\ss}\; .\]
\end{definition}
\begin{remark}\label{relation tangence cal V} 
We notice that $\cal V{\cb (\sc, \ss, p)}= 0$ if and only if $F_p^\boxdot(H^{\sc}_p \cap  {Y^\boxdot_p} )$ is tangent to $W^\ss_p$. Similarly,  $\cal V{\cb (\sc, \sc',p)}= 0$ if and only if $F_p^\boxdot(H^{\sc}_p \cap Y^\boxdot_p )$ is tangent to $V^{\sc'}_p (\zeta^{\sc'}_p)$. We notice that the tangency holds at $F_p^\boxdot(\zeta^\sc_p)$. 
\end{remark}
 \begin{center}\bf
From now on we work under the assumptions of \cref{main them Fp}.  In other words, the unfolding 
$(F_p, \pi_p)_p$ of wild type $(\sA, \sC)$ is moderately dissipative and  $\# \sC=5$. 
 \end{center}
 
Now we are ready to define the counterpart of the function involved in  $(\mathbf H_2 )$. We set:
\[\sC=:\{\boxdot_1,\boxdot_2,\boxdot_3,\boxdot_4,\boxdot_5\} \, .\]
Given $i\in \Z$, let  $\boxdot_i:=\boxdot_{[i]}$ where $[i]$ is the equivalence  class of $i$ in $\Z/5\Z$. In particular $\boxdot_{i+5}=\boxdot_i$ for every $i\in \Z$. 

\begin{definition}\label{def Psi}
For all $i_0\ge 0$ and  $(\sc_i)_{ i_0\le i\le  i_0+4}\in \prod_{i=i_0}^{ i_0+4}  {\cb     \sW}_{\boxdot_i}$ and $\ss \in  \avv \sA_{\boxdot_{i_0 +4}}$ such that $\st(\boxdot_i) = \si(\sc_{i+1})$ for every $i_0 \le i \le i_0 +3$, we put:
\[\Psi_{(\sc_i)_i, \ss }: p\in \cP_0\mapsto \Big(\big(\cal V({\cb \sc_i, \sc_{i+1},p}) \big)_{i_0\le i\le i_0+3}, \cal V({\cb \sc_{i_0+4}, \ss, p}) \Big)\; .\]
\end{definition}
Here is the counterpart of \cref{claim defining eta0 and L0}:
\begin{proposition}[and definition of $\eta_3$] \index{$\eta_3$}\label{implicite proposition}
There exists $\eta_3\in (0, \eta_2)$ such that for every  $(\sc_i)_{i_0\le i\le i_0+4}\in  \prod_{i=i_0}^{ i_0+4} {\cb     \sW}_{\boxdot_i}$ such that $\st(\boxdot_i) = \si(\sc_{i+1})$ for every $i_0 \le i \le i_0 +3$ and  $\sc_i$   satisfies $\overline \sfw(\sc_i)<\eta_3$ for every  $i_0 \le i \le i_0 + 4$, the following property holds true.

For every $p_1\in \cP_0$ and $\ss_{i_0+5}\in  \avv \sA_{\boxdot_{i_0 +4}} $, the restriction of $\Psi_{(\sc_i)_i, \ss_{i_0+5} }$ to a neighborhood of $p_1$ is a ${L_0}$-biLipschitz map onto the $\vartheta$-neighborhood of $\Psi_{(\sc_i)_i, \ss_{i_0+5} }(p_1)$.
\end{proposition}
\begin{proof}
 Let us take  for each $i_0 \le i \le i_0 +4$ a sequence $\su_i \in  \arr{\sA}_{\boxdot_i}$ such that the last letters of $\su_i$ are equal to $\sc_i$.    By \cref{utilisation epsilon}, we have $ \overline \sfh(\sc_i)<\overline \sfh(\sc_i)^\epsilon\le\underline \sfw(\sc_i)<\eta_3<\eta_2<\eta_0$. Let $K\ge 1$ be greater than the supremum of the $C^1$-norms of $p\in \cP_0\mapsto F_p^\boxdot(\zeta^{\su}_p)$ among  $\boxdot \in \sC$ and  $\su \in  \arr \sA_{\boxdot}$. By finiteness of $\sC$ and compactness of $\cP_0$ and   $\arr \sA_\boxdot$, the constant $K$ is finite. By  the parametric Inclination \cref{inclination} \cpageref{inclination}     when $\eta_3$ is small, the following pairs of maps are $C^1$-close:
\[p\in \cP_0 \mapsto \zeta^{\sc_i}_p  \qand p\in \cP_0 \mapsto \zeta^{\su_i}_p\quad \forall i_0\le i\le i_0+4 \; .\]
 By \cref{def calV} and compactness of $\bigcup_{\boxdot \in \sC} \arr \sA_\boxdot$, the maps $p\in \cP_0 \mapsto \zeta^{\su_i}_p$ are $C^1$-uniformly bounded among all tuples $(\su_i)_{i_0 \le i \le i_0 +4}$ of sequences $\su_i \in \arr{\sA}_{\boxdot_i}$  and so also among all $(\su_i)_{i_0 \le i \le i_0 +4}$ defined as before. Then the maps $p\in \cP_0 \mapsto \zeta^{\sc_i}_p$ are $C^1$-uniformly bounded among all tuples $(\sc_i)_i$ satisfying the hypotheses of the proposition. 
Also by looking at their images by $\pi^{\st(\boxdot_i)}_p \circ F_p^{\boxdot_i}$, the following pairs of maps are $C^1$-close, when $\eta_3>0$ is small:
\[p\in \cP_0 \mapsto a^{\sc_i}_p \qand p\in \cP_0 \mapsto a^{\su_i}_p\quad \forall i_0\le i\le i_0+4 \; .\]  
 Let us take  for each $i_0 \le i \le i_0 +4$ a sequence $\ss_i \in  \avv{\sA}$ such that the first letters of $\ss_i$ are equal to $\sc_i$.   
As each map $p\mapsto \zeta^{\sc_i}_p$ is uniformly $C^1$-bounded, 
the family of    the $\pi^{\st(\sc_i)}_p$-fibers going through $ \zeta^{\sc_i}_p$ indexed by $p \in \cP_0$ is a 
 $C^1$-bounded family, uniformly transverse to the unstable lamination of the horseshoe $\Lambda_p$. Applying a last time the parametric Inclination \cref{inclination} \cpageref{inclination}, we obtain that its pull back by $(F^{\sc_i}_p)_p$  is $C^1$-close to $(W^{\ss_i}_p)_{p\in \cP_0}$, and is equal by  definition to  $(V^{\sc_i}_p  (\zeta^{\sc_i}_p))_p$.
  By Fact \ref{identitifcation feuille V}, the image by $(\pi^{\si(\sc_i)}_p)_p$ of this family of curves is  $( \pi_p^{\si(\sc_i)}\circ (F^{\sc_i})^{-1} ( \zeta^{\sc_i}_p))_p=:(b^{\sc_i}_p)_p$ which is therefore uniformly $C^1$-close to $( b^{\ss_i}_p)_p$.
  In other words, the following pairs of maps are $C^1$-close, when $\eta_3>0$ is small:
\[p\in \cP_0 \mapsto b^{\sc_i}_p \qand p\in \cP_0 \mapsto b^{\ss_i}_p\quad \forall i_0\le i\le i_0+4 \; .\]
Consequently when $\eta_3>0$ is small, the map $\Psi_{(\sc_i)_i, \ss_{i_0+5} }$ is uniformly $C^1$-close to  $p\in \cP_0\mapsto (\cal V({\cb \su_i, \ss_{i+1}, p}))_{i}$ among all the $(\sc_i)_i$. Thus \cref{claim defining eta0 and L0} implies the proposition by the inverse local theorem and compactness of the spaces involved.
\end{proof}

\subsection{Infinite chain of nearly heteroclinic tangencies}
\label{Infinite chain of nearly heteroclinic tangencies}
In this subsection,  we work under the assumptions of \cref{main them Fp}.  We consider a moderately dissipative unfolding  $(F_p, \pi_p)_p$ of wild type $(\sA, \sC)$ such that    $\sC$ is formed by five elements $\{\boxdot_1, \boxdot_2, \boxdot_3, \boxdot_4, \boxdot_5 \}$. We   fix $p_0\in \mathrm{int\, }\cP$ as in the previous subsection. 

The parameter selection is given by \cref{keyprop} below and proved along this subsection.
Its statement involves the following constants which are positive and finite:
\begin{equation}\label{def constant C1}   C_1:=2\sqrt 5\cdot L_0\cdot \sup_{p\in \cP_0 } (\|D(\pi _p\circ F_p^\sC)\|+1)  \qand
\eta_4:=\min(L_0 \vartheta/C_1, \eta_3)
\; ,\end{equation} \index{$\eta_4$}\index{$C_1$} where
 the compact subset $\cP_0\subset \cP$ was defined by \cref{def P0}, the constant $\eta_3$ was defined in \cref{implicite proposition}, whereas the constants $\vartheta$ and $L_0$ were defined in \cref{claim defining eta0 and L0}.

 We are going to define a sequence of parameters $(p_i)_{i\ge 1}$ which converges to a certain $p_\infty$ such that each $p_i$ is close to $p_{i-1}$ and at $p_i$ the map $F_{p_i}$ satisfy some tangency's conditions. 
 To this end we will ask that the distance between $p_i$ and $p_{i+1}$ is at most of the order $\delta^{\beta^i}$ for a certain $\delta$ small and a fixed $\beta>1$. Hence the following immediate fact will be useful to ensure that all the $p_i$ remain in $\cP_0$ (at which all our bounds were done): 
\begin{fact}\label{somme gamma} For every $\beta>1$, there exists $M_\beta>0$ independent of $1/2>\delta>0$ such that $\sum_{j\ge 0} \delta^{\beta^j}\le M_\beta\cdot \delta$ for every $1/2>\delta>0$ .\end{fact}\index{$M_\beta$}
This fact enables us to fix the following constant:
\begin{definition}[$\beta$ and $(\delta_j)_j$] \label{defkappadelta} Let $\beta =3/2$ and
let $1/2>\delta_0>0$ be any number smaller than $\eta_4$ such that the ball centered at $p_0$ with radius $C_1\cdot M_\beta \cdot \delta_0$ is included in $\cP_0$.
We put \index{$\beta$ and $(\delta_j)_j$}
\begin{equation} \label{def deltai}\delta_j:= \delta^{\beta^j}_0\quad \text{with } \beta=\frac32 .\end{equation}
\end{definition}

We notice that the sequence $(\delta_j)_j$ decreases super exponentially fast to $0$. Also it satisfies the following for every $i,m\ge 0$:
\begin{equation}\label{ineq:sum delta}
\delta_{i+m}= (\delta_i)^{\beta^m}\qand \sum_{j\ge i} \delta_j\le M_\beta\cdot \delta_i\; .
\end{equation}

We recall that for every $i\in \Z$, we put  $\boxdot_i=\boxdot_{[i]}$, where $[i]$ is the equivalent class of $i$ in  $\Z/5\Z$. The following involves the curves $H^\sc_p$ and $V^\sc_p(\zeta)$ introduced in \cref{def H and V}.
\begin{proposition}[Main] \label{keyprop} There exist a sequence of words $(\sc_{i})_{i\ge 1} \in \prod_{i\ge 1}  {\cb     \sW}_{\boxdot_i}$ and  a sequence of parameters $(p_{i})_{i\ge 1}\in \cP_0^\N$, such that  it holds  $\st(\boxdot_i) = \si(\sc_{i+1})$ and the following   for any $i\ge 1$:
\begin{enumerate}[$(\mathbf C_1)$]
\item $\delta_{i}^{ 1+ \epsilon+2\epsilon^2}\le \underline \sfw(\sc_{i})\le \overline \sfw(\sc_{i}) \le \delta_{i}$,
\item the map  $F^{\boxdot_i}_{p_i}$ sends $H^{\sc_i}_{p_{i}}\cap  { Y^{\boxdot_i}_{p_i}}$ tangent to $V^{\sc_{i+1}}_{p_{i}}(\zeta^{\sc_{i+1}}_{p_i})$ at the point  $F^{\boxdot_i}_{p_i} (\zeta^{\sc_{i}}_{p_i})$,
\item $\| p_0 - p_{1}\|<C_1 {\delta_{1}}$ and $\|p_{i}- p_{i-1}\|<C_1 {\delta_{i+4}}$ whenever $i\ge 2$.
\end{enumerate}
\end{proposition}

We will prove this proposition below.
From this we deduce the following result involving the constants $L_0$, $M_\beta$ and $C_1$  defined respectively in \cref{claim defining eta0 and L0}, Fact \ref{somme gamma} and \cref{def constant C1} \cpageref{def constant C1}:
\begin{corollary}\label{keycoro}  There exist $p_\infty\in \cP_0$ which is  $C_1 M_\beta\delta_0$-close to $p_0$ and a sequence of words $(\sc_{i})_{i\ge 1} \in \prod_{i\ge 1} { {\cb     \sW}_{\boxdot_i}} $   such that for any $i \ge 1$ it holds  $\st(\boxdot_i) = \si(\sc_{i+1})$  and:
\begin{enumerate}[$(\mathbf C_1)$]
\item $\qquad \delta_{i}^{ 1+ \epsilon+2\epsilon^2}\le \underline \sfw(\sc_{i})\le \overline \sfw(\sc_{i}) \le \delta_{i}$,\end{enumerate}
\begin{enumerate}[$(\mathbf C_2')$]
\item     $\qquad |\cal V({\cb \sc_i,\sc_{i+1}, p_\infty)}|\le  C_1 L_0 M_\beta {\delta_{i+5}}$.
\end{enumerate}
\end{corollary}
\begin{proof} 
We take the sequence of words $(\sc_{i})_{i\ge 1}$ and the sequence of parameters $(p_{i})_{i\ge 1}$ given by \cref{keyprop}. In particular, they satisfy  $(\mathbf C_1)$. By Conclusion $(\mathbf C_3)$ of \cref{keyprop}, we have $\|p_{i+1}- p_{i}\|<C_1 {\delta_{i+5}}$ whenever $i\ge 1$ with $(\delta_i)_{i \ge 1}$ decreasing super exponentially fast  to 0. Thus the sequence $(p_i)_{i \ge 1}$ converges to a parameter $p_\infty \in \R^5$. By \cref{ineq:sum delta},  it holds:
\[\|p_0-p_\infty\|\le \|p_{1}-p_0\|+ \sum_{i\ge 1} \|p_{i+1}-p_i\| < C_1 \delta_1 +  C_1 \sum_{i\ge 1} \delta_{i+5}   \le C_1  M_\beta \delta_0   \; .
\]
By \cref{defkappadelta}, this implies that $p_\infty$ is in $\cal P_0$. Still using  Conclusion $(\mathbf C_3)$ of \cref{keyprop} and then \cref{ineq:sum delta}, we have $\|p_i-p_\infty\|\le C_1  M_\beta \delta_{i+5}$ for $i \ge 1$. We notice that $\cal V(\cb \sc_i,\sc_{i+1},p_i)=0$ by Conclusion $(\mathbf C_2)$ of \cref{keyprop}. Then $(\mathbf C_2')$ follows from \cref{implicite proposition}. 
\end{proof}
In \cref{sectionverifie}, we will show that $F_{p_\infty}$ displays a stable wandering domain by showing that the conclusions of \cref{keycoro} imply the assumptions of \cref{propreel}. \cref{keyprop} will be proved by induction. The first step will be implied by the following:
\begin{lemma}\label{Inductive Lemmasimplie}
There exist $(\sc_i)_{1\le i\le 5}\in \prod_{i=1}^5 {{\cb     \sW}_{\boxdot_i}}$ {with $\st(\boxdot_i)=\si(\sc_{i+1})$ for any $1 \le i \le 4$},  $\ss_6\in {\avv \sA_{\boxdot_5}}$  and a parameter $p_1\in \cP_0$ which is $C_1\delta_1$-close to $p_0$ such that:
\begin{enumerate}[$(\mathbf C_1)$]
\item  For every $1\le i\le 5$, it holds  $ \delta_{i}^{ 1+ \epsilon+2\epsilon^2}\le \underline \sfw(\sc_{i})\le \overline \sfw(\sc_{i}) \le \delta_{i}$.\end{enumerate}
\begin{enumerate}[$(\mathbf C_2'')$]
\item For any $1\le i\le 4$, it holds $\cal V(\cb \sc_i, \sc_{i+1}, p_1)=0$ and
$\cal V(\cb \sc_{5}, \ss_{6}, p_1) =0$.
\end{enumerate}
\end{lemma}
\begin{proof}
By {$(\mathbf {H_1})$} of \cref{nice unfolding}, there exist {
\begin{multline*}
\su_{1}\in \arr \sA_{\boxdot_1}, \; \su_{2}\in \arr \sA_{\boxdot_2},\; \su_{3}\in \arr \sA_{\boxdot_3},\;
\su_{4}\in \arr \sA_{\boxdot_4}, \; \su_5\in \arr \sA_{\boxdot_5}, \\
\qand \ss_{2} \in \avv \sA_{\boxdot_1},\ss_3 \in  \avv{\sA}_{\boxdot_2},\ss_4 \in  \avv{\sA}_{\boxdot_3},\ss_5 \in \avv{\sA}_{\boxdot_4},\ss_6 \in  \avv{\sA}_{\boxdot_5} \,  
\end{multline*}}
such that:
\begin{itemize}
\item[$(T')$] for every $1 \le j\le 5$, the curve $F_{p_0}^{\boxdot_j} (W^{\su_j}_{p_0} \cap { Y^{\boxdot_j}_{p_0} })$ is tangent to $W^{\ss_{j+1}} _{p_0}$.\medskip
\end{itemize}{We also pick any $\ss_{1} \in \avv \sA_{\boxdot_0}$.}
For every $1\le i\le 5$, let $m_{i}\ge 1$ and $m'_i\ge 1$ be maximal such that:
$$\overline \sfw(\sp_{m_{i}}(\ss_{i}))> \delta_i\qand \overline \sfh(\sp_{m'_i} (\su_i))> \delta_i.$$
 As $\delta_i <\eta_4 <\eta_1$, by \cref{ def sc}, for $1\le i\le 5$,  there exists $\sd_i\in \sA^*$ such that $\sc_{i}:= \sp_{1+m_i}(\ss_i) \cdot \sd_i\cdot  \sp_{1+m'_i}(\su_i)\in   \sA^*$ is well-defined and satisfies:
 \begin{equation} \delta_i\ge \overline \sfh(\sc_{i})\ge \underline \sfh(\sc_{i})\qand \delta_i\ge \overline \sfw(\sc_{i})\ge \underline \sfw(\sc_{i})\ge \delta_i^{ 1+\epsilon+2\epsilon^2}\; .\end{equation}
 {Moreover $\sc_{i}  \in {\cb     \sW}_{\boxdot_i} $ since $ \su_{i} \in \arr \sA_{\boxdot_i} $} and since $\overline{\sfh}(\sc_i) \le \delta_i <\eta_4 <\eta_2$. 
 {We remark that $\st(\boxdot_i)=\si(\ss_{i+1})=\si(\sc_{i+1})$.}
This proves $(\mathbf C_1)$. 

Let us prove $(\mathbf C''_2)$.  In \cref{def Psi}, we introduced:
\[\Psi_{(\sc_i)_i, \ss_6}:
p\in \cP_0\mapsto \Big(\big(\cal V{\cb (\sc_i, \sc_{i+1}, p)}  \big)_{1\le i\le 4}, \cal V({\cb \sc_{5}, \ss_{6},p })\Big)\; .\]
 By \cref{C0 bounds}, the latter function is {$2{\sqrt 5\cdot \sup_{p\in \cP_0} (\|D(\pi _p\circ F_p^\sC)\|+1)\cdot \delta_1 } = C_1\delta_1/L_0$}-$C^0$-close to the following:
\[\Psi_{(\su_i, \ss_{i+1})_i}:p\in \cP_0\mapsto ( \mathcal V(\cb 
\su_i,\ss_{i+1},p))_{1\le i\le 5}\]
where $\mathcal V\cb (\su_i,\ss_j,p)$ was defined in  \cref{deficalVus}.
By definition of $(\su_i, \ss_{i+1})_i$, this function vanishes at $p=p_0$. Thus it holds:
\[\|\Psi_{(\sc_i)_i, \ss_6}(p_0)\|\le C_1\delta_1/L_0\; .\]
By \cref{implicite proposition}, a restriction of $\Psi_{(\sc_i)_i, \ss_6}$ to a neighborhood of $p_0$ is $L_0$-biLipschitz onto the $\vartheta$-ball about $\Psi_{(\sc_i)_i, \ss_6}(p_0)$. As $C_1\delta_1/L_0 \le  C_1 \eta_4/L_0  \le \vartheta $, this ball contains $0$. Its preimage $p_1$ is $C_1\delta_1$-small. Thus Conclusion $(\mathbf C''_2)$ holds true.
\end{proof}
Here is the proposition which will give the step $i_0\to i_0+1$ of the aforementioned induction.
\begin{lemma}\label{Inductive Lemmasimplie3}
If $p_{i_0}\in \cP_0$, $(\sc_i)_{i_0\le i\le i_0+4} \in \prod_{i=i_0}^{i_0+4} {\cb     \sW}_{\boxdot_i}$  with $\st(\boxdot_i)=\si(\sc_{i+1})$ when $ i\le i_0+3$ and {$\ss_{i_0+5}\in \avv \sA_{\boxdot_{i_0+4}}$} satisfy:
\begin{enumerate}[$(h_1)$]
\item  $\overline \sfw(\sc_i)$ is smaller than $\eta_4$ for any $i_0\le i\le i_0+4$.
\item For every $i_0\le i\le i_0+3$, it holds $\cal V(\cb \sc_i, \sc_{i+1}, p_{i_0})=0$ and
$\cal V(\cb \sc_{i_0+4}, \ss_{{i_0+5}}, p_{i_0})=0$.
\end{enumerate}
Then there exist $\sc_{i_0+5}\in   {\cb     \sW}_{\boxdot_{i_0 +5}}$,  $\ss_{i_0+6}\in \avv \sA_{\boxdot_{i_0 +5}}$ with $\st(\boxdot_{i_0 +4})=\si(\sc_{i_0+5})$ and a parameter $p_{i_0+1}\in \cP_0$ which is $C_1\delta_{i_0+5}$-close to ${p_{i_0}}$ such that:
\begin{enumerate}[$(\mathbf C_1)$]
\item    $ \delta^{1+\epsilon+2\epsilon^2}_{i_0+5}\le \underline \sfw(\sc_{i_0+5})\le \overline \sfw(\sc_{i_0+5}) \le \delta_{i_0+5}$.\end{enumerate}\begin{enumerate}[$(\mathbf C_2'')$]
\item It holds $\cal V(\cb \sc_i, \sc_{i+1}, p_{i_0+1}) =0$ for every $i_0 + 1 \le i\le i_0+4$ and
$\cal V(\cb \sc_{i_0+5}, \ss_{{i_0+6}}, p_{i_0+1}) =0$.

\end{enumerate}
\end{lemma}
{
\begin{proof} The proof is similar to the one of \cref{Inductive Lemmasimplie}. By  $(\mathbf {H_1})$ of \cref{nice unfolding}, there exist  $\su_{i_0+5} \in \arr \sA_{\boxdot_{i_0 +5}}$ and
$\ss_{i_0+6} \in \avv \sA_{\boxdot_{i_0 +5}}$ such that the curve   $W_{{p_{i_0}}}^{\su_{i_0+5}}\cap  {Y^{\boxdot_{i_0 +5}}_{p_{i_0}}} $  is sent tangent to
$W_{{p_{i_0}}}^{\ss_{{i_0+6}}}$ by   $F^{\boxdot_{i_0 +5}}_{p_{i_0}}$.
Let $m_{i_0+5}\ge 1$ and $m'_{i_0+5}\ge 1$ be maximal such that:
\[\overline \sfw(\sp_{m_{i_0+5}}(\ss_{i_0+5}))> \delta_{i_0+5}\qand \overline \sfh(\sp_{m'_{i_0+5}} (\su_{i_0+5}))> \delta_{i_0+5} .\]
 As $\delta_{i_0+5} <\eta_4 <\eta_1$, by \cref{ def sc} there exists $\sd_{i_0+5} \in \sA^*$ such that the word 
$\sc_{i_0+5}:= \sp_{1+m_{i_0+5}}(\ss_{i_0+5}) \cdot \sd_{i_0+5} \cdot  \sp_{1+m'_{i_0+5}}(\su_{i_0+5}    )\in   \sA^*$ is well-defined and   satisfies $\st(\boxdot_{i_0 +4})=\si(\ss_{i_0+5})=\si(\sc_{i_0+5})$ with:
\begin{equation} \delta_{i_0+5}\ge \overline \sfh(\sc_{i_0+5})\ge \underline \sfh(\sc_{i_0+5})\qand \delta_{i_0+5} \ge \overline \sfw(\sc_{i_0+5})\ge \underline \sfw(\sc_{i_0+5})\ge \delta_{i_0+5}^{ 1+\epsilon+2\epsilon^2 }\; .\end{equation}
{Moreover $\sc_{i_0+5} \in   {\cb     \sW}_{\boxdot_{i_0 +5}}$ since $\su_{i_0+5} \in \arr \sA_{\boxdot_{i_0 +5}}$}  and since $\overline{\sfh}(\sc_{i_0 +5}) \le \delta_{i_0 +5} <\eta_4 <\eta_2$.
This proves $(\mathbf C_1)$. Let us prove $(\mathbf C''_2)$. We recall:
\[\Psi_{(\sc_i)_i, \ss_{i_0+6}}:
p\in \cP_0\mapsto \Big(\big(\cal V(\cb \sc_i, \sc_{i+1}, p)\big)_{i_0+1\le i\le i_0+4}, \cal V(\sc_{i_0+5}, \ss_{{i_0+6}},p)\Big)\; .\]
 By \cref{C0 bounds}, the latter function is $C_1\delta_{i_0+5}/L_0$-$C^0$-close to the following:
\[\mathring{\Psi}:p\in \cP_0\mapsto
\Big(\big(\cal V(\cb \sc_i, \sc_{i+1}, p) \big)_{i_0+1\le i\le i_0+3}, \cal V(\sc_{i_0+4}, \ss_{i_0+5},p) ,\cal V(\su_{i_0+5}, \ss_{{i_0+6}},p)\Big)\; .
\]
By $(h_2)$ and definition of $(u_i, \ss_{i+1})_i$, this function vanishes at $p=p_{i_0}$. Thus it holds:
\[\|\Psi_{(\sc_i)_i, \ss_{i_0+6} }(p_{i_0})\|\le C_1\delta_{i_0+5}/L_0\; .\]
By \cref{implicite proposition}, a restriction of $\Psi_{(\sc_i)_i, \ss_{i_0+6} }$ to a neighborhood of $p_{i_0}$ is $L_0$-biLipschitz onto the $\vartheta$-ball about $\Psi_{(\sc_i)_i, \ss_{i_0+6} }(p_1)$. As $C_1 \delta_{i_0+5}/L_0 \le C_1 \eta_4/L_0  \le \vartheta$, this ball contains $0$. Its preimage $p_{i_0+1}$ is $C_1\delta_{i_0+5}$-close to $p_{i_0}$. Thus conclusion $(\mathbf C''_2)$ holds true.
\end{proof}}

We are now ready to:
\begin{proof}[Proof of \cref{keyprop}]  We prove by induction on $  i_0\ge 1$ the existence of parameters $(p_i)_{1\le i\le i_0}$, words  $(\sc_{i})_{i \le i_0+4} \in \prod_{i=1}^{i_0+4} {\cb     \sW}_{\boxdot_i}$ with $\st(\boxdot_i)=\si(\sc_{i+1})$ for $i\le i_0+3$ and of  $\ss_{i_0+5}\in \avv \sA_{\boxdot_{i_0+4}}$ satisfying:
\begin{enumerate}[$(\mathbf C_1)$]
\item  $\delta_{i}^{ 1+ \epsilon+2\epsilon^2}\le \underline \sfw(\sc_{i})\le \overline \sfw(\sc_{i}) \le \delta_{i}$ for every $i\le i_0+ 4$.\end{enumerate}
\begin{enumerate}[$(\mathbf C_2'')$]
\item 
$\cb \cal V(\sc_i, \sc_{i+1},p_{i_0}) =0$, for every $i_0\le i\le i_0+ 3$ and  $\cb \cal V(\sc_{i_0+4}, \ss_{i_0+5},p_{i_0}) =0$. 
\end{enumerate}
\begin{enumerate}[$(\mathbf C_3)$]
\item The parameters $p_{i_0-1}$ and $p_{i_0}$ are $<C_1 {\delta_{i_0+4}}$ distant whenever $i_0\ge 2$. The parameters $p_0$ and $p_{1}$ are $C_1 {\delta_{1}}$ distant.
\end{enumerate}
Observe that  by \cref{relation tangence cal V}, this statement implies \cref{keyprop}.

The proof of this statement is done by induction on $i_0\ge 1$ using \cref{Inductive Lemmasimplie,Inductive Lemmasimplie3}. We notice that step $i_0= 1$ is an immediate consequence of \cref{Inductive Lemmasimplie}. Let $i_0\ge 1$ and assume $(\sc_{i})_{i\le i_0+4} \in (\sA^*)^{i_0+4}$ constructed such that $(\mathbf C_1-\mathbf C''_2-\mathbf C_3)$ are satisfied at every $i\le i_0$.
By Fact \ref{somme gamma}, the distance between the parameter $p_{i_0}$ and $p_0$ is at most $C_{1} \sum_{i=1}^{i_0+4} \delta_i<C_{1} \cdot M_\beta \cdot \delta_0$, and so by \cref{defkappadelta} the parameter $p_{i_0}$ belongs to $\cP_0$. Thus we can use \cref{Inductive Lemmasimplie3} which implies $(\mathbf C_1-\mathbf C''_2-\mathbf C_3)$ at step $ i_0+1$.
\end{proof}
{Here is a consequence of the above proof which will be useful to obtain  Conclusions $(\blacklozenge)$ of \cref{main them Fp}. We recall that   $(\sA\sqcup \sC)^*$ denote the set of words $\sg= \sa_1\cdots \sa_m$  formed by letters $\sa_i\in \sA\sqcup \sC$ such that $\st(\sa_i)= \si(\sa_{i+1})$ for every $i<m$. 
\begin{corollary}\index{$\nu$}\label{keypropE}  Under the assumptions of \cref{keycoro}, there exists $\nu>0$  such that given any function $\mathsf {Cod}:  (\sA\sqcup \sC)^*\mapsto \sA^*$ satisfying  
\[ |\mathsf {Cod}(\sc)| \le \nu |\sc|\qand 
\st(\sc)= \si(\mathsf {Cod}(\sc)) \quad { \forall \sc \in (\sA\sqcup \sC)^* }  ,\]
then Conclusions $(\mathbf C_1)$ and $(\mathbf C'_2)$ of \cref{keycoro} hold true with words $(\sc_i)_i$ satisfying:
\begin{enumerate}[$(\mathbf C_0)$]
\item the sequence $\sf :=\sc_1\cdot \boxdot_1\cdot \sc_2\cdot \boxdot_2  \cdots \sc_i\cdot \boxdot_i \cdots$ 
 has infinitely many $m\ge  m'$ such that:
\[  \sp_m(\sf)= \sp_{m'}(\sf) \cdot \mathsf {Cod}(\sp_{m'}(\sf))\; .\]
\end{enumerate}
\end{corollary}
\begin{proof} In the proof of \cref{Inductive Lemmasimplie3},  
the word $\sc_{i_0+5}$ is chosen in function of $(\su_{i_0+5},\ss_{i_0+5})$  and of the form:   
\[\sc_{i_0+5}= \sp_{1+m_{i_0+5}} (\ss_{i_0+5}) \cdot \sd_{i_0+5}\cdot  \sp_{1+m'_{i_0+5}}(\su_{i_0+5}).\]
By using \cref{ def sc2} instead of \cref{ def sc},  the word $\sd_{i_0+5}$  can be chosen of the form 
\[\sd_{i_0+5} =   \sd_{i_0+5}^o\cdot \sd_{i_0+5}^r ,\]
 for any  word $\sd_{i_0+5}^o$  in $\sA^*$ such that:
\begin{equation}\label{liberte}    \underline {\sfw} (\sd^o_{i_0+5}) > \delta_{i_0+5}^{\epsilon^4}  \cdot \frac{ 1+\theta^2} { \mathscr  W }\qand 
\si(\sd_{i_0+5}^o)=\st(\sp_{1+m_{i_0+5}}(\ss_{ i_0+5}))\; ,
\end{equation}
and where $\sd_{i_0+5}^r$ is a word in $\sA^*$ such that $\underline {\sfw} (\sd^r_{i_0+5}) > \mathscr  W$ and such that $\sc_{i_0+5}$ is in $\sA^*$. By definition of $\mathscr  W$, the  word $\sd_{i_0+5}^r$ always exists whatever are $\sd_{i_0+5}^o$ and 
$\sp_{1+m'_{i_0+5}}(\su_{i_0+5})$.  Now we recall that $\sp_{1+m_{i_0+5}}(\ss_{i_0+5})$ depends only on $p_{i_0}$ and $i_0$ which do not depend on $\sd^o_{i_0+5}$. 
So really $\sd^o_{i_0+5}$ can be any word satisfying \cref{liberte}. 
For $j\ge 1$, let 
\[\sg_j:= \sc_1\cdot \boxdot_1\cdot \sc_2 \cdots \sc_{j-1}\cdot \boxdot_{j-1} \cdot  \sp_{1+m_{j}}({  \ss_{j}})
\]
With $j=i_0+5$, \cref{liberte}, is equivalent to:
\begin{equation}\label{liberte2} \log \underline {\sfw} (\sd^o_{j}) >{\epsilon^4} \log   {\delta_j} { + \log ( (1+\theta^2) / \mathscr  W ) \qand 
\si(\sd_{j}^o)=\st(\sp_{1+m_{j}}(\ss_{j}))} \; ,
\end{equation}
Put:
\[\mu :=-\inf_{\sd\in \sA^*}\frac{\log \underline {\sfw}(\sd)}{|\sd|}\]
We observe that $0<\mu <\infty$  since $\sA$ is finite. 
Then \cref{liberte2} is implied by:
\begin{equation}\label{liberte3} |\sd^o_{j}|  <-\frac{{\epsilon^4}}{\mu} \log   {\delta_j} -{ \frac{\log ((1+\theta^2)/\mathscr  W)}{\mu} \qand 
\si(\sd_{j}^o)=\st(\sp_{1+m_{j}}(\ss_{j}))}\; ,
\end{equation}
Now let us replace the latter inequality by a stronger one involving $|\sg_j|$ instead of $\delta_j$. 
It holds:
\[| \sg_j| = j+|\sc_1|+\cdots +|\sc_{j-1}|+m_j\le 2(|\sc_1|+\cdots +|\sc_{j}|)\; .\]
For every $i$ we have by $(\mathbf C_1)$:
\[\delta_i^{1+\epsilon+2\epsilon^2} \le \overline \sfw(\sc_i) \le \lambda^{-|\sc_i|}\Rightarrow |\sc_i|\log \lambda \le -({1+\epsilon+2\epsilon^2})  \log \delta_i =-({1+\epsilon+2\epsilon^2}) \beta^{i-j} \log \delta_j \; . \]
Thus:
\[|\sg_j|\le -2({1+\epsilon+2\epsilon^2})\sum_{i=1}^j\beta^{i-j} \log_\lambda \delta_j \le  -2({1+\epsilon+2\epsilon^2}) \frac{\beta}{\beta-1}\log_\lambda \delta_j \le -  {10\cdot  } \log  \delta_j   
\; .\]
Thus \cref{liberte3} and so \cref{liberte} is implied by 
\begin{equation}\label{liberte4}    |\sd^o_j|  <
\frac{\epsilon^4}{10\cdot  \mu}    |\sg_j| -\frac{\log ((1+\theta^2)/  \mathscr  W) }{\mu}  {   \qand 
\si(\sd_{j}^o)=\st(\sp_{1+m_j}(\ss_j))} \; ,
\end{equation}

We fix $\nu:= \frac{\epsilon^4}{20\cdot  \mu} $. 
Since $(|\sg_j|)_j$ tends to $+\infty$, if  $  |\sd^o_j|  \le \nu \cdot  |\sg_j|$ for any $j \ge 1$, then the left inequality of   \cref{liberte4} will be satisfied when $j$ is large. 

In particular, given any map 
$\mathsf {Cod}:  (\sA\sqcup \sC)^*\mapsto \sA^*$ satisfying $|\mathsf {Cod}(\sc)| \le \nu |\sc|$ and 
$\st(\sc)= \si(\mathsf {Cod}(\sc))$ for any $\sc \in (\sA\sqcup \sC)^*$, it is therefore possible to pick a sequence of  words  $\sw_j$ satisfying both the conclusions $(\mathbf C_1)$ and $(\mathbf C'_2)$ of \cref{keycoro} and $\sd^o_j = \mathsf {Cod}(\sg_j)$ for any $j \ge 1$. The proof is complete. 
\end{proof}}

\subsection{Proof of \cref{main them Fp} }  \label{sectionverifie}
{ 
To show the first part of   \cref{main them Fp} regarding the   existence of a family of sets $B_j \times \{o(\boxdot_j)\}$ satisfying the conclusions of \cref{propreel} (resp. \cref{propcomplex}), it suffices to show that the assumptions  of these theorems are verified at a dense set of parameters.

\begin{proof} [Proof of the first part of \cref{main them Fp}]
Let $(F_p, \pi_p)_{p\in \cP}$ be a moderately dissipative unfolding of wild type $(\sA, \sC)$ with $\#\sC=5$. In \cref{Infinite chain of nearly heteroclinic tangencies}, we showed that for every $p_0 \in \cP$ and every sufficiently small $\delta_0$, we can find $p_\infty \in \cP$  which is  $C_1 M_\beta\delta_0$-close to $p_0$, a sequence $(\boxdot_j)_{j \ge 1} \in \sC^{\N^*}$ and a  sequence of words $(\sc_{j})_{j \ge 1} \in \prod_{j\ge 1}  {\cb     \sW}_{\boxdot_j} $ such that $\st (\boxdot_j)=\si(\sc_{j+1})$ for every $j \ge 1$ and satisfying  the conclusions $(\bf C_1)$ and $(\bf C_2')$ of \cref{keycoro}. We are going to apply \cref{propreel}, {or respectively  \cref{propcomplex} if $F_{p_\infty}$ is analytic}.

\begin{proposition} \label{propree1help}
The system $F_{p_\infty}$ of type $(\sA, \sC)$  endowed with the adapted projection $\pi_{p_\infty}$ and the sequences $(\boxdot_j)_{j \ge 1}$, $(\sc_j)_{j \ge 1}$  satisfies the assumptions of \cref{propreel}, {and moreover those of \cref{propcomplex} if $F_{p_\infty}$ is analytic}.
\end{proposition}

The proof of \cref{propree1help} is given below. Together with \cref{propreel} (resp. \cref{propcomplex}), it  implies the existence of families of nonempty subsets $B_j$ of $Y^\se$  satisfying the conclusions of \cref{propreel} (resp. \cref{propcomplex}).
\end{proof}

The rest of this subsection is devoted to the proof of \cref{propree1help}.
In the following, we fix a parameter $p_\infty$ as in \cref{keycoro}. We are going to show that  $F_{p_\infty}$ endowed with  $\pi_{p_\infty}$, $(\boxdot_j)_{j \ge 1}$ and  $(\sc_j)_{j \ge 1}$  satisfies the assumptions of \cref{propreel}, by checking these assumptions one by one. 

For simplicity and since we work at fixed parameter, we will omit to note $p_\infty$ in index in the following. We begin by proving the following crucial estimates between the coefficients $\delta_j$ (defined in \cref{defkappadelta}), $|Y^{\sc_j}|$ (defined in \cref{defYnorme}) and $\breve \gamma_j$ (defined in \cref{propreel}):

\begin{lemma} \label{bound gamma}\label{gamma2togamma} \label{ineq sigma gamma delta}
For every $j\ge 1$, the coefficient $\breve \gamma_j$ is non zero and it holds:
\[2 \delta_j^{1+\epsilon+2\epsilon^2} \le   |Y^{\sc_j}| \le 2\delta_j  \qand  2 \delta_{j}^{3(1+\epsilon+2\epsilon^2)} \le  \breve \gamma_{j} \le 2  \delta_j^3  \, . \]

\end{lemma}
\begin{proof} To prove the first inequalities, we recall  that by \cref{rem sur w et h} we have $  2\underline \sfw(\sc_j) \le  |Y^{\sc_j}| \le 2\overline \sfw(\sc_j) $ for every $j \ge 1$. Thus the inequalities follow immediately from  
 \cref{keycoro} $(\mathbf C_1)$.
 
 We now prove the second ones. We first recall that: $$\breve \gamma_j=|Y^{\sc_{j+1}}|^{1/2}\cdot |Y^{\sc_{j+2}}|^{1/2^2}\cdots |Y^{\sc_{j+k}}|^{1/2^{k}}\cdots \; $$  if it is well-defined. 
 To show that $\breve \gamma_j$ is indeed well-defined and positive, we compute the logarithm of the latter infinite  product (with $\log 0 = - \infty$ by convention) and we show that it is finite:
$$\log \breve \gamma_j = \sum_{k>j} \log |Y^{\sc_k}|^{1/2^{k-j}}= \sum_{k>j} \log |Y^{\sc_k}|^{2^{j-k}} = \sum_{k>j} 2^{j-k}\log |Y^{\sc_k}|= \sum_{k>0} 2^{-k}\log |Y^{\sc_{j+k}}|
\; .$$
We just proved that  $2 \delta_{k} ^{1+\epsilon+2\epsilon^2} \le  |Y^{\sc_k}| \le 2  \delta_k$ for every $k > 0$. Thus we have:
\[  \log(2) +  (1+\epsilon+2\epsilon^2)\sum_{k>0} 2^{-k}\log \delta_{k+j}    \le  
\log \breve \gamma_j \le  \log(2) +  \sum_{k>0} 2^{-k} \log \delta_{k+j} \; .
\]
We recall that $\delta_{k+j}=\delta_{j}^{\beta^{k}}$ by \cref{ineq:sum delta}. Thus it holds:
\[
   \sum_{k>0} 2^{-k}\log \delta_{k+j}   =  \sum_{k>0} \left(\frac{\beta}2\right)^{k}\log \delta_{j}=\left(\frac{\beta}{2-\beta}\right)\log \delta_{j} \; .\]
Then  $ \log(2) + \left(\frac{\beta}{2-\beta}\right)  (1+\epsilon+2\epsilon^2)  \log \delta_{j}  \le 
\log \breve \gamma_j   \le  \log(2) +  \left(\frac{\beta}{2-\beta}\right) \log \delta_{j}$. In particular, $\breve \gamma_j$ is  positive and: 
\begin{equation}\label{pour choix beta}   2 \delta_{j}^{ \frac{\beta}{2-\beta} (1+\epsilon+2\epsilon^2)}   \le  \breve \gamma_j   \le  2 \delta_{j}^{\frac{\beta}{2-\beta}}    \end{equation}
which concludes the proof of the second inequalities since $\frac{\beta}{2-\beta}=3$.
\end{proof}

Here is an important consequence of the latter lemma:}
\begin{fact}\label{deltaj5}When $j$ is large,  $\delta_{j+5}$  is small compared to $\breve \gamma_j^2$.\end{fact}

\begin{proof}
By \cref{ineq:sum delta}, we have $\delta_{j+5}=\delta_{j}^{\beta^5} =  \delta_{j}^{7.59375}\, $.  By \cref{pour choix beta}, we have $\breve \gamma_{j} \ge 2 \delta_{j}^{ \frac{\beta}{2-\beta} (1+\epsilon+2\epsilon^2)}$.  Since $ \frac{2\beta}{2-\beta} (1+\epsilon+2\epsilon^2) = 6(1+\epsilon+2\epsilon^2)<7.59375$, the fact follows.
\end{proof}

\begin{remark}The proof  of Fact \ref{deltaj5} enables to see why we needed $5$ parameters in  the proof of the main theorems of this article. Indeed, for a $k$-parameter family, we would have to find a $2>\beta>1$ which statisfies:
 \[\frac{\beta}{2-\beta} (1+\epsilon+2\epsilon^2)<\beta^k\]
By taking $\epsilon$ small this gives the inequality:
 \[\frac{\beta}{2-\beta} <\beta^k\]
which does not have any solution for $2>\beta>1$ and $k<5$.
\end{remark}

\begin{proof}[Proof of \cref{propree1help}]
Let us show that $F$  endowed with $\pi$, $(\boxdot_j)_j$, $(\sc_j)_j$ satisfies the assumptions of \cref{propreel}.  First, by \cref{bound gamma}, the coefficient $\breve \gamma_j$  is  positive. This shows $(i)$. 

We now show $(ii)$. We remark that $\st(\boxdot_j)=\si(\sc_{j+1})$ by \cref{keycoro}. 
We want to prove that the $\pi^{\st(\boxdot_j)}$-fiber of $F^{\boxdot_j}(\zeta^{\sc_j})$ is $o(\breve \gamma_{j}^2)$-$C^0$-close to  the $\pi^{\si(\sc_{j+1})}$-fiber of $(F^{\sc_{j+1}})^{-1}(\zeta^{\sc_{j+1}})$.  
By \cref{keycoro} $(\mathbf C_2')$, we have $|\cal V(\cb \sc_j,\sc_{j+1})| = O( \delta_{j+5})$, where, by \cref{def a b vrai,def cal V new}, it holds:
\[ \cal V({\cb \sc_j,\sc_{j+1}}) =   a^{\sc_j} - b^{\sc_{j+1}} =  \pi^{\st(\boxdot_j)}  \circ F^{\boxdot_j}(\zeta^{\sc_{j}})   -    \pi^{\si(\sc_{j+1})} \circ (F^{\sc_{j+1}})^{-1}(\zeta^{\sc_{j+1}})  \, . \]
By Fact \ref{deltaj5}, we have  $\delta_{j+5} = o(\breve \gamma_j^2)$. Since $F^{\sc_{j+1}}$ expands the horizontal distance    by a factor dominated by $|Y^{\sc_{j+1}}|^{-1}$ and since $\breve \gamma_j^2=  |Y^{\sc_{j+1}}|\cdot \breve \gamma_{j+1}$ we obtain $(ii)$. 

By \cref{bound gamma},  $|Y^{\sc_{j}}|  \le 2 \delta_{j}$  and so $|Y^{\sc_{j}}|$ is small when $j$ is large. 

This shows the first statement of $(iii)$.  Finally, by \cref{rem sur w et h} and \cref{def P0} (2), we have:
\[|F^{\sc_j}(Y^{\sc_{j}})| \le 2 \overline \sfh(\sc_j) \le 2 \underline \sfw(\sc_j)^{1/\epsilon}  \le 2 |Y^{\sc_{j}}|^{10}   \, .   \]
By \cref{ineq sigma gamma delta}, we have $|Y^{\sc_{j}}|^{6(1+\epsilon+2\epsilon^2)}=O(\breve \gamma_{j}^2 )$. As  $6(1+\epsilon+2\epsilon^2) = 6 \times 1.12 < 10$, we obtain $|F^{\sc_j}(Y^{\sc_{j}})| =o(\breve \gamma_j^2)$ and the proof of $(iii)$ is complete.  

If moreover $F$ is analytic, by \cref{def P0} $(2')$, we have  $\max_{ \tilde I^2} |\partial_{w_0} \cW^{\sc_j}|^\epsilon \le   \min_{\tilde I^2} |\partial_{z_1} \cZ^{\sc_j}|$  when $j$ is large, after reducing $\rho$ if necessary, and we conclude the same. This shows assumption $(\widetilde{iii})$ of \cref{propcomplex}. 
\end{proof}

We are now going to prove the second part of \cref{main them Fp} regarding  Property $(\blacklozenge)$. This will be done  using a new symbolic aforemeithm. 
\label{section emergence}
We recall that  $\avv {\sA}\subset \avv {\sA\sqcup\sC}$ denote the sets of infinite words $\ss=\ss_1\cdots \ss_m\cdots $  in respectively the  alphabets $\sA$ and $\sA\sqcup \sC$ which are admissible: $\st (\sa_i)=\si(\sa_{i+1})$, for every $i$. The space  $\avv {\sA\sqcup\sC}$ is compact endowed with the following distance:
\[
d(\ss, \ss') = 2^{-\min(i:\ss_i \neq \ss'_i)}\qquad \text{for }\ss=\ss_1 \ss_2 \cdots,\ss'=\ss'_1 \ss'_2 \cdots\in \avv {\sA\sqcup\sC}\; .\]
Let $g$ and $g^\sA$  be the shift dynamics on $\avv {\sA\sqcup\sC}$ and $\avv {\sA}$:
\[ g:\ss_1\cdots \ss_m\cdots \in   \avv {\sA\sqcup\sC}\mapsto \ss_2\cdots \ss_{m+1}\cdots  \in \avv {\sA\sqcup\sC}\qand g^\sA :=g|  \avv {\sA}\; .\]
Let $\cal M(g)$ be the space of $g$-invariant measure on $\avv {\sA\sqcup\sC}$ and let $\cal M(g^\sA)$ be the space of $g$-invariant measures supported by $\avv {\sA}$.
 We endow $\cal M(g)\supset \cal M(g^\sA)$  with the Wasserstein distance  $d_W$ introduced in the introduction on \cpageref{definitiondelemergenceP4}. 
Given $\sf\in \avv {\sA\sqcup\sC}$, let $\mathsf e_n(\sf)$ be the probability measure on $\avv {\sA\sqcup\sC}$ which is equidistributed on the $n$ first $g$-iterates of $\sf$.  For every  $\sc\in {\sA\sqcup\sC}^*$, let us chose $\tilde \sc\in \avv {\sA\sqcup\sC}$ such that $\sp_{|\sc|}(\tilde \sc)=\sc$  and put $\mathsf e_n(\sc):= \mathsf e_n(\tilde \sc)$ for every $n\le |\sc|$. Here is the algorithm, it uses  $\nu>0 $ defined in \cref{keypropE}.
\begin{definition}\label{Emergence universelle}
Let $\mathsf {Cod}: (\sA\sqcup \sC)^* \to \sA^* $ be a map {sending $\sc\in  (\sA\sqcup \sC)^*$} to a word 
$\sd$ such that:
\begin{enumerate}
\item $\st(\sc)= \si(\sd)$ and $|\sd|\le \nu |\sc|$.
\item for every  word $\sd'\in \sA^*$ such that  $\st(\sc)= \si(\sd')$ and $|\sd'|\le \nu |\sc|$, it holds:
\[ d_W(\mathscr e_{|\sc\cdot \sd'|}(\sc\cdot \sd') , \{\mathscr e_{n}(\sc): {1\le n\le |\sc|} \})
\le d_W(\mathscr e_{|\sc\cdot \sd|}(\sc\cdot \sd), \{\mathscr e_{n}(\sc): {1\le n\le |\sc|} \})\; .\]
\end{enumerate}
\end{definition}
\begin{corollary}\label{keycoro3}    Under the assumptions of \cref{keycoro},  {there is a} parameter $p_\infty \in \cP_0$ which is $C_1M_\beta\delta_0$-close to $p_0$ and a sequence of words $(\sc_i)_i$ satisfying Conclusion $(\mathbf C_1)$ and $(\mathbf C_2')$ of \cref{keycoro} and moreover the point $\sf :=\sc_1\cdot \boxdot_1\cdot \sc_2\cdot \boxdot_2  \cdots \sc_i\cdot \boxdot_i \cdots\in \avv{\sA\sqcup \sC }$   satisfies:
\begin{enumerate}[$(\mathbf C_0)$]
\item   the limit set $L$  of $(\mathscr e_{n}(\sf))_{n\ge 0}$ contains $\frac{1}{1+\nu}\cdot \check \mu+\frac{\nu}{1+\nu}\cdot \cal M(g^\sA)$ for a certain $\check \mu  \in  \cal M(g^\sA)$.
\end{enumerate}
\end{corollary}
\begin{proof}We recall that for every $m\ge 0$ and $\ss = \ss_0\cdots \ss_j\cdots\in \avv {\sA\sqcup \sC}$, 
we defined $\sp_m(\ss)= \ss_0\cdots \ss_{m-1}$.
We apply  \cref{keypropE} with the function $\mathsf {Cod}$ of \cref{Emergence universelle}. 
It gives a parameter $p_\infty \in \cP_0$ which is $C_1M_\beta\delta_0$-close to $p_0$ and a sequence of words $(\sc_i)_i$ satisfying Conclusions $(\mathbf C_1)$ and $(\mathbf C_2')$ of \cref{keycoro} and moreover $\sf :=\sc_1\cdot \boxdot_1\cdot \sc_2\cdot \boxdot_2  \cdots \sc_i\cdot \boxdot_i \cdots$  satisfies: 
\begin{enumerate}[$({\bf C}_0'$)]
\item there are infinitely many $m\ge 0$ such that with $K_m:= \{\mathscr e_{n} ( \sp_n(\sf)): {1\le n\le m} \}$ and $m':=m+|\mathsf {Cod}(\sp_m(\sf))|$,  it holds:
\[  d_W \left(\mathscr e_{m'}  (\sp_{m'}(\sf) , K_m\right)
\ge \max_{ \{\sd \in \sA^*: \st(\sp_m(\sf))= \si(\sd)\text{ and }|\sd|\le \nu m \} }
d_W \left(\mathscr e_{m+|\sd|}(\sp_m(\sf)\cdot \sd), K_m \right)\; .\]
\end{enumerate}
Now we use:
\begin{lemma}
For every $\ss\in \avv {\sA\sqcup \sC}$, the measures 
$ \mathscr e_n(\ss)$ and $\mathscr e_{n}(\sp_n(\ss))$ are $2/n$-close.
\end{lemma}
\begin{proof}The measures $\mathscr e_n(\ss)$ and $\mathscr  e_{n}(\sp_n(\ss))$ are equidistributed on $n$ atoms which pairwise $2^{-n+k}$-close for $0\le k\le n$. Thus the transport cost is at most $\frac1n \sum 2^{-n+k}= 2/n$.
 \end{proof}
Thus $\mathscr e_{m'} (\sp_{m'}(\sf) )$ is close to $\mathscr e_{m'} (\sf) $ when $m$ is large. 
Also 
$\mathscr e_{m+|\sd|}(\sp_m(\sf)\cdot \sd)$ is close to 
$\frac{m}{m+ |\sd|} \mathscr e_{m}(\sf)
 + \frac{|\sd|}{m+ |\sd|}\mathscr e_{|\sd|}(\sd)$ when $m$ is large. 
   Consequently when $m$ is large:
  \[  d_W \left(\mathscr e_{m'}(\sf) , K_m\right)
\ge \max_{ \{\sd \in \sA^*: \st(\sp_m(\sf))= \si(\sd)\text{ and }|\sd|\le \nu m\} }
d_W \left(\frac{m}{m+ |\sd|}  \mathscr e_{m}(\sf)
 + \frac{|\sd|}{m+ |\sd|}\mathscr e_{|\sd|}(\sd), K_m \right)+o(1)\; .\]  
As $(\sV,\sA)$ is a  finite and connected graph, we do not need to assume that $\sp_m(\sf)\cdot \sd$ is admissible in the above inequality, also we can restrict this inequality to the case $|\sd|=[\nu \cdot m]$:
     \[  d_W \left(\mathscr e_{m'}(\sf) , K_m\right)
\ge \max_{ \{\sd \in \sA^*:  |\sd|=[\nu \cdot m]\} }
d_W \left(\frac{m}{m+ |\sd|}\mathscr e_{m}(\sf)
 + \frac{|\sd|}{m+ |\sd|}\mathscr e_{|\sd|}(\sd), K_m \right)+o(1)\; .\]  
Thus with $K_\infty = cl(\bigcup_{m\ge 0} K_m)$ the limit of  $(K_m)_m$ in the Hausdorff topology, we have:
     \[ 0= d_W \left(\mathscr e_{m'}(\sf) , K_\infty\right)   + o(1) 
\ge \max_{ \{\sd \in \sA^*:  |\sd|=[\nu \cdot m]\} }
d_W \left(\frac{1}{1+\nu}\mathscr e_{m}(\sf)
 + \frac{\nu}{1+\nu}\mathscr e_{|\sd|}(\sd), K_\infty \right)+o(1)\; .\]  
 Hence for any accumulation point $\check \mu$ of such $\mathscr  e_{m}(\sf)$, as $\mathscr e_{|\sd|}(\sd)$ can approximate any measure $\mu'$ of $\cal M(g^\sA)$, it holds:
\[
\max_{ \mu'\in \cal M(g^\sA) }
d_W \left(\frac{1}{1+\nu}\check \mu 
 + \frac{\nu}{1+\nu}\mu' , K_\infty \right)
=0.\]
Hence we obtained that the connected set  { 
$\frac{1}{1+\nu}\cdot \check \mu+\frac{\nu}{1+\nu}\cdot \cal M(g^\sA)$}  is included in $K_\infty$.  Finally we notice that the proportion of letters in $\sC$ is asymptotically small by $({\bf C}_1)$. Thus $\check \mu$ belongs to $\cal M(g^\sA)$.  As $K_\infty$ is equal to the union of the limit set $L$ of $(\mathscr e_m(\sf))$ with the discrete set $\bigcup_m K_m$, the limit set $L$   contains $\frac{1}{1+\nu}\cdot \check \mu+\frac{\nu}{1+\nu}\cdot \cal M(g^\sA)$.
\end{proof}

\begin{proof}[Proof of property $(\blacklozenge)$ of \cref{main them Fp}]
  We still consider a  moderately dissipative unfolding  $(F_p, \pi_p)_{p\in \cP}$ of wild type $(\sA, \sC)$, such that $\sC$ is formed by five elements $\{\boxdot_1, \boxdot_2, \boxdot_3, \boxdot_4, \boxdot_5 \}$, and a fixed parameter $p_0\in \mathrm{int\, }\cP$. We define $\cP_0$ as in \cref{def P0}. We apply \cref{keycoro3} which gives the existence of   a parameter $p_\infty \in \cP_0$ which is $C_1M_\beta\delta_0$-close to $p_0$ and a sequence of words $(\sc_i)_i$ satisfying Conclusions $(\mathbf C_1)$ and $(\mathbf C'_2)$ of \cref{keycoro} and moreover:
\begin{enumerate}[$(\mathbf C_0)$]
\item  the point $\sf :=\sc_1\cdot \boxdot_1\cdot \sc_2\cdot \boxdot_2  \cdots \sc_i\cdot \boxdot_i \cdots\in \avv{\sA\sqcup \sC }$   satisfies that 
the limit set $L$  of $(\mathscr e_{n}(\sf))_{n\ge 0}$ contains $\frac{1}{1+\nu}\cdot \check \mu+\frac{\nu}{1+\nu}\cdot \cal M(g^\sA)$ for a certain $\check \mu  \in  \cal M(g^\sA)$.

\end{enumerate}
Let   $B:=\bigcap_{n\ge 0} F^{-n}_{p_\infty}(Y^\se\times \sV)$. The map  $F_{p_\infty}$ sends $B$ into itself, and the restriction  $F_{p_\infty}|B$ is semi-conjugate to $g$ via the following map:
\[\sh:(z,\sv)\in B\mapsto (\sd_i)_{i\ge 0}\quad \text{with } \sd_i\in \sA\sqcup \sC \text{ such that } F_{p_\infty}^i(z, \sv)\in {  Y^{\sd_i}_{p_\infty}}  \times \{\si(\sd_i)\} \; . \]
We are going to use the following proved below:
\begin{lemma}\label{lifting} The   set of invariant probability measures $ \cal M(F_{p^\infty}^\sA)$  of $F_{p^\infty}^\sA$ is homeomorphic to 
$\cal M(g^\sA)$ via the pushforward $\sh_*$ induced by  $\sh$.  
\end{lemma}
By \cref{propree1help} and \cref{propreel}, there exists $J\ge 0$ and a stable domain   $B_J\subset B$  which is sent by $\sh$ to $\{g^{J}(\sf)\}$. By $(\mathbf C_0)$ and \cref{lifting}, the limit set of the empirical measures of points in $B_J$ contains:
\[\sh_*^{-1}(\frac{1}{1+\nu}\cdot \check \mu+\frac{\nu}{1+\nu}\cdot \cal M(g^\sA))=\frac{1}{1+\nu}\cdot \sh_*^{-1}\check \mu+\frac{\nu}{1+\nu}\cdot \sh_*^{-1}(\cal M(g^\sA))=\frac1{1+\nu}\cdot  \mu + \frac\nu{1+\nu}\cdot \cal M(F^\sA_{p^\infty})\; ,\]
with $\mu =\sh_*^{-1} \check \mu \in \cal M(F^\sA_{p^\infty})$. 
\end{proof}
\begin{proof}[Proof of \cref{lifting}] 
The map $\mu'\in \cal M(F_{p^\infty}^\sA)\mapsto \sh_*\mu'\in \cal M(g^\sA)$ is continuous on its domain which is   compact. Thus it suffices to show that it is a bijection. 
In order to show this, it suffices to recall that the space of invariant measures of a continuous map are homeomorphic to those of its inverse limit, and that the dynamics on the inverse limit $\overleftrightarrow \sA$ of $\avv \sA$ is conjugate to 
the dynamics on the inverse limit $\overleftrightarrow \Lambda$ of $\Lambda$ by \cref{def hbar}. 
\end{proof}

\begin{appendix}

\section{Some results on $C^\omega_\rho$-hyperbolic transformations}

\begin{proof}[Proof of \cref{prop_compo_starc}]\label{proof prop_compo_starc}
Let us just do the proof for $C^\omega_\rho$-hyperbolic transformations, the one for (real) hyperbolic transformations is similar. We assume that $(\tilde Y, F)\neq (\tilde Y^\se, id)\neq (\tilde Y',F')$, otherwise the proof is obvious.  Put: 
\[(\tilde Y'', F''):=(\tilde Y,F)\star (\tilde Y',F')\; .\]
  The proposition is the consequence of the two following  \cref{lemma A1,lemma A2}.\end{proof}
\begin{lemma}\label{lemma A1} The set $\tilde Y''$ is a  $C^\omega_\rho$-box. \end{lemma}
\begin{proof} Note first that $\tilde Y'' \subset \tilde Y\subset  \tilde Y^\se $. Let $\zeta : (z,w)\in \tilde I^{2} \rightarrow (\cal Z_1(z, w), w)\in \tilde Y$ and $\zeta' : (z,w)\in \tilde I^{2} \rightarrow (\cal Z_2(z, w), w)\in \tilde Y'$ be two biholomorphisms given by \cref{defcomplexbox} for the $C^\omega_\rho$-boxes $\tilde Y$ and $\tilde Y'$.
By the cone property $(2)$ of \cref{defcomplexpiece}, at $w\in \tilde I$ fixed, the (complex) curve $\D_{1\, w}:=\{(\cal Z_{1}(z, w), w): z\in \tilde I\}$ is sent by $F$ to a curve with tangent spaces in $\tilde \chi_h$. By property $(1)$ of \cref{defcomplexpiece}, the curve $ F(\D_{1\, w})$ is disjoint from $\partial^u \tilde Y^\se$ and with boundary in $\partial^s\tilde Y^\se$.
Still by property $(1)$ of \cref{defcomplexpiece}, for every $z_2\in \tilde I$, the curve $\{(\cal Z_2(z_2, w),w): w\in \tilde I\}$ is disjoint from $\partial^s \tilde Y^\se$. By \cref{defcomplexbox}, it has its tangent spaces in $\tilde \chi_v$ and its boundary is in {$\partial^u \tilde Y^\se$}. Thus $\{(\cal Z_2(z_2, w),w): w\in \tilde I\}$ intersects $F(\D_{1\, w})$ transversally at a unique point. We denote by $\zeta''(z_2, w)=(\cal Z_{12}(z_2,w), w)$ its preimage by $F$.
By transversality, $\zeta''$ is a holomorphic map from $\tilde I^{2}$ onto $\tilde Y''$. Since $|\partial_w \cal Z_{2}|<\theta$ and by the cone property $(3)$ of \cref{defcomplexpiece}, it holds $|\partial_w \cal Z_{12}|<\theta$. 
Also  $(\{(\cal Z_{2}(z_2,w),w): w\in \tilde I \})_{z_2\in \tilde I}$ are the leaves of a foliation. Each leaf  intersects transversally  the curve $F(\D_{1\, w})$ at the unique point {$F\circ \zeta''( z_2,w)$}, so the derivative of $F\circ \zeta''$ w.r.t. $z_2$ is non-zero. Thus, the derivative of  
$\zeta''(z_2,w)$ w.r.t. $z_2$ is non-zero. By definition, $\zeta''(z_2,w)$ belongs to $\tilde I\times\{w\}$ and so $\partial_{z_2} \zeta''$ is horizontal and non zero. Thus $D\zeta''$ is invertible and so $\zeta''$ is a biholomorphism.

Let us denote by $Y$, $Y'$, $Y''$ the real traces $\tilde Y \cap Y^\se$, $\tilde Y' \cap Y^\se$, $\tilde Y'' \cap Y^\se$ of  $\tilde Y$, $\tilde Y'$, $\tilde Y''$. It remains to show that $Y''$ is a (real) box. 
It is enough to show that $Y'' = \breve Y''$ where we set $(\breve Y'', F''):=(Y,F)\star (Y',F')$. Indeed both $Y$ and  $Y'$ are boxes by \cref{defcomplexbox} and so by \cite[\textsection 3.2.1]{PY01} it is also the case for $\breve Y''$. Then one just has to remark that:
\[   Y'' =  \tilde Y \cap F^{-1}(\tilde Y') \cap Y^\se  = Y \cap F^{-1}(\tilde Y') = Y \cap F^{-1}( Y' ) = \breve Y''  \, . \]
 where the first equality is true by definition of $ \tilde Y'' $, the second equality is true by definition of $Y$, the fourth one is given by definition of the real $\star$-product and the third one uses that $F| Y$ is a bijection  onto $F(Y)$ whose inverse is conjugate to a  hyperbolic transformation via the involution $(z,w)\mapsto (w,z)$. This shows that $\tilde Y''$ is a $C^\omega_\rho$-box.
\end{proof}
\begin{lemma}\label{lemma A2} The pair $(\tilde Y'', F'')$ is a $C^\omega_\rho$-hyperbolic transformation.  \end{lemma}
\begin{proof}The map $F''=F' \circ F$ is a biholomorphism on  $ \tilde Y''$ as a composition of two biholomorphisms. 
We now continue with the same notations  as in the previous lemma. 
Let us prove  $(1)$ of \cref{defcomplexpiece}.
Note that when $z_2\in \partial \tilde I$, then $\{(\cal Z_{12}(z_2,w),w): w\in \tilde I \}$ is sent by $F$ into $\{(\cal Z_{2}(z_2,w), w): w\in \tilde I\}\subset \partial^s \tilde Y'$. Thus by this  property $(1)$ for $(\tilde Y', F)$, the set $\{(\cal Z_{12}(z_2, w),w): w\in \tilde I \}$ is sent by $F''$ into $\partial^s \tilde Y^\se$: $F'' (\partial^s \tilde Y'')\subset \partial^s\tilde Y^\se$.
 By definition, the set $\tilde Y''$ is included in the definition domain of $F''$ which  sends it into $F' ( \tilde Y') \subset \tilde Y^\se$. As $(\tilde Y, F)\neq (\tilde Y^\se, id)\neq (\tilde Y', F')$, it holds $\tilde Y'' \cap \partial^s \tilde Y^\se \subset \tilde Y \cap \partial^s \tilde Y^\se=\emptyset$ and $ F'' ( \tilde Y'' ) \cap \partial^u \tilde Y^\se \subset F'(\tilde Y') \cap \partial^u \tilde Y^\se =\emptyset.$ Properties $(2)$, $(3)$ and $(4)$ of \cref{defcomplexpiece} for $(\tilde Y'', F'')$ come from the same properties for $(\tilde Y,F)$ and $(\tilde Y',F')$.
 \end{proof}

\label{extension proof}
{ 
The remaining of this  section is devoted to show the following given a finite graph $(\sV, \sA)$:
\begin{proposition} \label{rho uniform}
If  $F^{\sA}$ is a hyperbolic map of type $\sA$ and if $F^{\sA}$ is real analytic, then there exists $\rho>0$ such that for every $\sc\in \sA^*$, there exists a   $C^\omega_\rho$-box $\tilde Y^\sc$ such that $ Y^\sc$ is the real trace $\tilde Y^\sc \cap \R^2$ of $\tilde Y^\sc$, $F^\sc$ extends to a biholomorphism from $\tilde Y^\sc$ and $(\tilde Y^\sc,F^\sc)$ is a $C^\omega_\rho$-hyperbolic transformation. 
\end{proposition}
\begin{proof}We prove below the following:
\begin{lemma}\label{extension}
If $(Y, F)$ is a hyperbolic transformation and $F$ is real analytic, then for every $\rho>0$ sufficiently small, there exists a $C^\omega_\rho$-box $\tilde Y$ such that $Y$ is the real trace $\tilde  Y \cap \R^2$ of $\tilde Y$, $F$ extends to a biholomorphism from $\tilde Y$ and $(\tilde Y,F)$ is a $C^\omega_\rho$-hyperbolic transformation.  
\end{lemma}
Thus by finiteness of $\sA$ and by \cref{extension}, there exist $\rho>0$ such that each hyperbolic transformation $(Y^\sa, F^\sa)$ with $\sa \in \sA$ admits an extension as  a $C^\omega_\rho$-hyperbolic transformation $(\tilde Y^\sa, F^\sa)$. Then by \cref{prop_compo_star}, for every $\sc\in \sA^*$, we can define the $C^\omega_\rho$-hyperbolic transformation:
\[(\tilde Y^{\sc}, F^\sc) := (\tilde Y^{\sa_1}, F^{\sa_1})\star\cdots \star (\tilde Y^{\sa_k}, F^{\sa_k})\;, \quad \forall \sc= \sa_1\cdots \sa_k\in \sA^* \; .\]
By  \cref{rema prop_compo_star}, the pair $(\tilde Y^{\sc}, F^\sc) $ is a $C^\omega_\rho$-hyperbolic transformation extending $(Y^{\sc}, F^\sc) $.
\end{proof}
}
\begin{proof}[Proof of \cref{extension}]  Let us assume that $F$ is not the identity (otherwise we take the holomorphic  extension $(\tilde Y^\se, id)$).

As $F|Y$ is a real analytic diffeomorphism onto its image, it extends to  a biholomorphism from a neighborhood $\hat Y$ of $Y$ in $\C^2$ onto its image $F(\hat Y)$.  If $\hat Y$ is small enough, from 
cone properties (2) and (3) of \cref{defpiece},  at every point in $ \hat Y$, every non-zero vector in the complement of $\tilde \chi_v$   is sent into $\tilde \chi_h$ by $D F$ and has its first coordinate's modulus which is more than $\lambda$-expanded by $D F$. 
Also at every point in $ F(\hat Y)$, every non-zero vector in the complement of $\tilde \chi_h$ is sent into $\tilde \chi_v$ by $ D  F ^{-1}$ and has its second coordinate  modulus which is more than $\lambda$-expanded by $  D F  ^{-1}$. 

Let $\cX: I^2\to \mathrm {int}\,  I$ and $\cY:I^2\to \mathrm {int}\, I$ be the implicit representation of the hyperbolic transformation $(F,Y)$ as defined in \cref{ALacpara}. By the transversality used in the proof of this proposition, the maps $\cX$ and $\cY$ are real analytic. 
Thus for $\rho$-small enough they extend to holomorphic functions $\cZ$ and $\cW$ on  $\tilde I^2$, with $\tilde I= I+i[-\rho, \rho]$. 
By \cite[Lemma  3.2]{PY01}, it holds:
\[|\partial_x \cX|+|\partial_y \cX|< \theta\qand 
|\partial_x \cY|+|\partial_y \cY|< \theta\, , \quad  \text{on } I^2\, .\]
Thus for $\rho$ sufficiently small:
\[|\partial_z \cZ|+|\partial_w \cZ|< \theta\qand 
|\partial_z \cW|+ |\partial_w \cW|<\theta\, , \quad  \text{on } \tilde I^2\, .\]
In particular $\cZ$ and $\cW$ are contracting. As they map $I^2$ into the interior of $I$, for $\rho$-sufficiently small $\cZ$ and $\cW$  map $\tilde I^2$ into the interior of $\tilde I$.
Hence the following set is included in $\tilde Y ^\se\setminus \partial^s \tilde Y^\se$: 
\[\tilde Y:= \{(\cZ(z_1,w_0), w_0):  (z_1,w_0)\in \tilde I^2\}\]
Also $|\partial_w  \cZ|< \theta$.  Since $\partial_x \cal X \neq 0$, reducing $\rho$ if necessary, we also have $\partial_z \cal Z \neq 0$ and so the map  $(z,w) \in \tilde I^2 \mapsto (\cal Z(z,w),w) \in \tilde Y$ is a biholomorphism. Thus $\tilde Y$ is $C^\omega_\rho$-box. Note that for $\rho$ small, $\tilde Y$ is close to $Y$ and so included in $\hat Y$. Thus cone properties (2) and (3) of \cref{defcomplexpiece} are satisfied. 
 Put:
 \[\partial^ u \tilde Y:= \zeta(\tilde I\times \partial \tilde I)\qand \partial^s \tilde Y:= \zeta( (\partial \tilde I)\times \tilde I)\; .\]  
As for every $(x,y)\in I^2$, it holds $F(\cX(x,y), y)=(x, \cY(x,y))$ by  \cref{ALacpara} and by analyticity, we have:
\[F(\cZ(z,w), w)=(z, \cW(z,w))\quad \forall  (z,w)\in \tilde I^2\; ,\]
from which property (1)  of \cref{defcomplexpiece} is immediate. 
Finally note that if for $w_0\in I$ and $z_1\in \tilde I$, the number 
$\cZ(z_1,w_0)$ is real, then its  image $(z_1, \cW(z_1,w_0))$ by $ F$ is real. Thus   $z_1$ is real and so in $I$. Thus
\[\tilde Y\cap \R^2= \{(\cZ(z_1,w_0), w_0):  (z_1,w_0)\in  I^2\}=Y\, .\]
This achieves the proof of  \cref{defcomplexpiece} (4).
\end{proof}

\section{Bounds on implicit representations and distortion results}
\label{Computational proof of uniform bound on implicit representation}
Let us recall:
\begin{lemma}[{\cite[Ineq. (A.8) p.195]{PY09}}] \label{compoal}
 Let $(Y,F)$ and $(Y',F')$ be hyperbolic transformations with implicit representations $( \cX_0, \cY_1)$ and $( \cX_1, \cY_2)$. Then
the implicit representation $(\tilde \cX_0,\tilde \cY_2)$ of $(Y,F) \star (Y',F')$ satisfies:
\[ \frac1{1+\theta^2}\le \frac{ | \partial_{x_{2}} \tilde \cX_0 (x_{2},y_0)| }{ | \partial_{x_{1}} \cX_0 \big( x_{1},y_0 \big)| \cdot |\partial_{x_{2}} \cX_{1} \big( x_{2},y_{1}\big)|} \le \frac1{1-\theta^2}
\; ,\]
where $F$ sends $(x_0, y_0)\in Y$ to $(x_1,y_1)\in Y'$ and $(x_1,y_1)$ is sent to $(x_2,y_2)$ by $F'$.
\end{lemma}
{
We will also need the following complex analogous of the latter result:

\begin{lemma} \label{compoalC}
 Let $(\tilde Y,F)$ and $(\tilde Y',F ')$ be $C^\omega_\rho$-hyperbolic transformations with implicit representations $( \cZ_0, \cW_1)$ and $( \cZ_1, \cW_2)$. Then the implicit representation $(\tilde \cZ_0,\tilde \cW_2)$ of $(\tilde Y,F) \star (\tilde Y',F')$ satisfies:
\[ \frac1{1+\theta^2}\le \frac{ | \partial_{z_{2}} \tilde \cZ_0 (z_{2},w_0)| }{ | \partial_{z_{1}} \cZ_0 \big( z_{1},w_0 \big)| \cdot |\partial_{z_{2}} \cZ_{1} \big( z_{2},w_{1}\big)|} \le \frac1{1-\theta^2}
\; ,\]
\[ \frac1{1+\theta^2}\le \frac{ | \partial_{w_{0}} \tilde \cW_2 (z_{2},w_0)| }{ | \partial_{w_{0}} \cW_1 \big( z_{1},w_0 \big)| \cdot |\partial_{w_{1}} \cW_{2} \big( z_{2},w_{1}\big)|} \le \frac1{1-\theta^2}
\; ,\]
where $F$ sends $(z_0, w_0)\in \tilde Y$ to $(z_1,w_1)\in \tilde Y'$ and $(z_1,w_1)$ is sent to $(z_2,w_2)$ by $F'$.
\end{lemma}

\begin{proof} We just prove the inequalities on $\tilde \cZ_0$, those on regarding $\tilde \cW_2$ are obtained similarly.
We consider the map $\Psi: (z_{1},w_{1},z_{2},w_{0}) \in \tilde I^4 \rightarrow (z_{1}-\cZ_{1}(z_{2},w_{1}), w_{1}-\cW_{1}(z_{1},w_{0}))$, which is holomorphic. 
For all $(z_0, w_0) \in \tilde Y$, $(z _1,w_1) \in \tilde  Y'$ and $(z_2,w_2) \in \tilde Y^\se$ such that $F$ sends $(z_0, w _0)$ to $(z _1,w _1)$ and $F'$ sends $(z _1,w_1)$ to $(z_2,w_2)$, it holds $\Psi (z _{1},w_{1},z _{2},w _{0}) = (0,0)$. By the cone properties of $C^\omega_\rho$-hyperbolic transformations, we have $| \partial_{w_{1}}  \cZ_{1} | < \theta$  and  $| \partial_{z_{1}}  \cW_{1} | < \theta $ . Then $\Delta:= |\mathrm{det} ( \partial_{(z_{1},w_{1})} \Psi  )|=|1  - \partial_{w_{1}}  \cZ_{1}\cdot \partial_{z_{1}}  \cW_{1} |$  satisfies:
\begin{equation}\label{bound Delta}  1+\theta^{2}>\Delta   >1-\theta^{2}>0\end{equation}
 and so $\partial_{(z_{1},w_{1})} \Psi$ is invertible. By the implicit function Theorem, there exist implicitly defined holomorphic maps $ \tilde \cZ_1$ and $ \tilde \cW_1$ such that:
\begin{equation} 
\Psi \big(z_1, w_1,z_{2},w_{0} \big) = 0
\Longleftrightarrow \left(z_1=\tilde \cZ_1(z_2, w_0)\text{ and } w_1=\tilde \cW_1(z_2, w_0)\right)\; .
 \end{equation}
Furthermore, this theorem gives:
$$
 \begin{pmatrix}   \partial_{z_{2}} \tilde \cZ_1 &\partial_{w_0} \tilde \cZ_1  \\  \partial_{z_{2}} \tilde \cW_1 & \partial_{ w_0} \tilde \cW_1  \end{pmatrix}  =- \partial_{(z_{1},w_{1})} \Psi^{-1} \cdot \partial_{  z_{2}, w_0 } \Psi 
 =
 - \begin{pmatrix}  1 & - \partial_{w_{1}} \cZ_1 \\ - \partial_{z_{1}} \cW_1 & 1\end{pmatrix}^{-1}\cdot \begin{pmatrix}  - \partial_{z_{2}} \cZ_1 & 0 \\  0 &  - \partial_{w_{0}} \cW_1 \end{pmatrix}   \; .$$
Thus,  it holds $
\partial_{z_{2}} \tilde \cZ_1 =
\Delta^{-1}   \cdot  \partial_{z_{2}} \cZ_1 $. As $\tilde \cZ_0(z_2, w_0)= \cZ_0(\tilde \cZ_1(z_2, w_0),w_0)$, we then obtain $\partial_{z_2} \tilde \cZ_0= \partial_{z_1} \cZ_0 \cdot \partial_{z_2}\tilde \cZ_1 = \Delta^{-1}   \cdot   \partial_{z_1} \cZ_0 \cdot  \partial_{z_{2}} \cZ_1$ and the sought inequalities follow from \cref{bound Delta}. \end{proof}}
We are now in position to prove:
\label{Proof of def P0}
\begin{proof}[Proof of \cref{def P0}] 
 {\emph{Proof of (1)}}.
The left inequality being obvious, let us show the right inequality. 
Let $\sc \in \sA^{*}$ and $(\cX^\sc_p, \cY^\sc_p)$ be the implicit representation of $(Y^\sc_p, F^\sc_p)$ for $p\in \cP$. According to the cone condition (2) of \cref{defpiece}, we have $|\partial_x \cX^\sc_{p}| \le {\lambda^{-|\sc|}}$ and so:
\[(1+\epsilon^2)\max_{I^2} \log |\partial_x \cX^\sc_{p}| \le  
 \max_{I^2} \log |\partial_x \cX^\sc_{p}|-\epsilon^2\cdot |\sc|\cdot   \log \lambda\]
Then using \cref{coro bound}, we obtain:
\[(1+\epsilon^2)\max_{I^2} \log |\partial_x \cX^\sc_{p}| \le  
 \min_{I^2} \log |\partial_x \cX^\sc_{p}|+   \log  {\overline{\mathscr{B} }}-\epsilon^2\cdot |\sc|\cdot   \log \lambda\]
Let $N_0>0$ be such that \begin{equation}
\label{defN0}\log  {\overline{\mathscr{B} }}-\epsilon^2\cdot N_0\cdot   \log \lambda < -\log(1+\theta^2) +(1+\epsilon^2)\log (1-\theta^2)\; .\end{equation} Then:
\[(1+\epsilon^2)\max_{I^2} \log |(1- \theta^2)^{-1} \cdot \partial_x \cX^\sc_{p}| <  
\min_{I^2} \log |(1+\theta^2)^{-1}\cdot  \partial_x \cX^\sc_{p}|  \quad \forall \sc\in \sA^* \text{ with } |\sc|\ge N_0. \]
 As the latter inequality is strict, given any $p_0\in \cP$,   there exists a compact neighborhood $\cP_0 \subset \cP$ of $p_0$ such that for every word $\sc\in \sA^* $ with at least $N_0$ and at most $2N_0$ letters, for every $p,p'\in \cP_0$, we obtain:
\begin{equation}\label{dist4}  (1-\theta^2)^{-1-\epsilon^2} \cdot \max_{I^2} |\partial_x \cX_p^\sc|^{1+\epsilon^2} <  (1+\theta^2)^{-1} \cdot \min_{I^2}|\partial_x \cX_{p'}^\sc |. \end{equation}
Consequently for every word  $\sc\in \sA^* $ with at least $N_0$ and at most $2N_0$ letters,
\begin{equation}\label{dist4bis} (1-\theta^2)^{-1-\epsilon^2}\cdot \overline \sfw(\sc)^{1+\epsilon^2} <  (1+\theta^2)^{-1}  \cdot \underline \sfw(\sc) \; . \end{equation}
Let  $\sc\in \sA^*$ be such that $|\sc|\ge N_0$. It can be written as a concatenation $\sc = \sc_{1} \cdots \sc_{m}$ where each $\sc_{i}$ has at least $N_0$ and at most $2N_0$ letters. According to Lemma \ref{compoal}, we have:
\begin{equation}\label{dist5bis}   \overline \sfw(\sc)  \le  (1-\theta^2)^{-m} \cdot\overline \sfw(\sc_1) \cdots \overline \sfw(\sc_m)
\qand 
(1+\theta^2)^{-m } \cdot\underline \sfw(\sc_1)  \cdots \underline \sfw(\sc_m) \le  \underline \sfw(\sc) 
 \; .\end{equation}
Together with \cref{dist4bis}, this yields:
\begin{equation}\label{dist6bis}
\overline \sfw(\sc)^{1+\epsilon^2}\le \underline \sfw(\sc)\; .
\end{equation}
Since $1-\epsilon^2<(1+\epsilon^2)^{-1}$, this proves $(1)$ for every word with more than $N_0$ letters. Thus it suffices to assume $\eta_0>0$ small enough such that every $\sc\in \sA^*$ such that $\overline \sfw(\sc)<\eta_0$ has at least $N_0$ letters.\\

\noindent {\emph{Proof of (2)}.}
Let $\sc\in \sA^*$. By \cref{rep_jac} and since $F_p^\sA $ is moderately dissipative for every $p\in \cP$, for every $z=(x_0, y_0)\in Y_p^\sc$ sent to $(x_1,y_1)$ by $F^{\sc}_p$, it holds:
\[| \partial_{y_0} \cY^\sc_p(x_1,y_0)|^\epsilon = |\det\, D_{(x_0, y_0)}F_p^\sc|^\epsilon\cdot |\partial_{x_1} \cX^\sc_p(x_1,y_0)|^\epsilon\le \min_{Y^{\sc}_p} \|DF_p^\sc\|^{-1}\cdot |\partial_{x_1} \cX^\sc_p(x_1,y_0)|^\epsilon \, .\]
As $\|DF_p^\sc\|^{-1} \le |\partial_{x_1} \cX^\sc_p | $, it follows:
\begin{equation}\label{pre2prime} 
\max_{Y^\sc_p} | \partial_{y_0} \cY^\sc_p |^\epsilon\le \min_{Y^{\sc}_p} |\partial_{x_1} \cX^\sc_p |\cdot \max_{Y^\sc_p} |\partial_{x_1} \cX^\sc_p |^\epsilon \le \max_{Y^\sc_p} |\partial_{x_1} \cX^\sc_p |^{1+\epsilon }
 \, .\end{equation}
Thus $\overline{\mathsf h}(\sc) ^\epsilon \le \overline \sfw(\sc)^{1+\epsilon}$. Also by \cref{dist6bis} and since $0<\epsilon<1$, if $\bar {\sfw}(\sc)\le \eta_0$ then   $\overline{\mathsf h}(\sc) ^\epsilon< \underline \sfw(\sc)\, $.\\

 \noindent\emph{Proof of (2') (and definition of $\rho$)}. Let $F_p$ be real analytic for a certain $p\in \cP_0$. By \cref{rho uniform}, let  $\rho>0$ be such that for every $\sc\in \sA^*$, the hyperbolic transformation $(Y^{\sc}_{p}, F_{p}^\sc) $ admits a $C^\omega_\rho$-extension $(\tilde Y^{\sc}_{p}, F_{p}^\sc)$.  Then by \cref{pre2prime},   reducing $\rho$ if necessary, we have:
\begin{equation}    \label{H1C}
\max_{\tilde Y^{\sc}_{ p}   } | \partial_{w_0} \cW^\sc_p |^\epsilon \le  \min_{\tilde Y^{\sc}_{ p}   } |\partial_{z_1} \cZ^\sc_p | \cdot \max_{\tilde Y^{\sc}_{  p}   } |\partial_{z_1} \cZ^\sc_p |^{\epsilon}\le \lambda^{-N_0\cdot \epsilon}\cdot  \min_{\tilde Y^{\sc}_{  p}   } |\partial_{z_1} \cZ^\sc_p | 
 \quad \text{if }  N_0\le  |\sc|\le 2N_0
\; .
\end{equation}
By \cref{defN0},  $ -\epsilon \cdot N_0\cdot   \log \lambda <   \epsilon \log(1-\theta^2) -\log (1+\theta^2)$ and so it holds:
\begin{equation}\label{dist4bisC} (1-\theta^2)^{-\epsilon} \max_{\tilde I^2} |\partial_{w_0} \cW_{p}^\sc|^{ \epsilon } < {(1+\theta^2)}^{-1}  \cdot \min_{\tilde I^2} |\partial_{z_1} \cZ_{p}^\sc| \; \quad \forall \sc\in \sA^* \text{ with }  N_0\le  |\sc|\le 2N_0\; . \end{equation}
Then, given any $\sc\in \sA^*$ with $|\sc|\ge N_0$, we split $\sc$ into words of lengths in $[N_0, 2N_0]$  and proceed as in the proof of \cref{dist5bis} and \cref{dist6bis}, using \cref{compoalC} instead of \cref{compoal}, to obtain:
\begin{equation} \label{complexequivw}
    \max_{\tilde I^2} |\partial_{z_1} \cW_{p}^\sc|^{ \epsilon } \le    \min_{\tilde I^2} |\partial_{z_1} \cZ_{p}^\sc| \, ,
\end{equation}
for every  $\sc\in \sA^* $  such that $ |\sc| \ge N_0$ and so every $\sc\in \sA^*$ such that $\bar {\sfw}(\sc)\le \eta_0$.\end{proof}

In \cref{sectiontildeBj}, we needed the following distortion result:
\begin{proposition}  \label{distortion}

 $(i)$  {Let $(\cX _j)_j$  be a sequence of $C^2$-functions $\cX_j: I^2\to \R$ such that $\partial_x \cX_j$ does not vanish.  } Let $(\cal B_j)_j$ be a sequence of subsets of $I^2$ whose diameters are small when $j$ is large and let $(x^0_j,y^0_j)_j$ be a sequence of points in $\cal B_j$.  If the distortion $(\|\partial_x \log|\partial_x \cX_j| \|_{C^0}, \|\partial_y \log|\partial_x \cX_j| \|_{C^0})_j$  is bounded, then the following map is $C^1$-close to the first coordinate projection when $j$ is large:
\[\Upsilon_j : (x,y) \in  \cal B_j \mapsto \frac{\cal X_{j}(x,y)-\cal X_{j}(x^0_j,y)}{\partial_x  \cal X_{j}(x^0_j,y^0_j)} +  x^0_j \, .\]
$(ii)$ {Let $(\cZ _j)_j$  be a sequence of  holomorphic functions $\cZ_j: \tilde I^2\to \tilde I$ such that  $0<|\partial_z \cZ_j|<1$. }
Let $(\cal B_j)_j$ be a sequence of subsets of $\tilde I^2$ and $(z_j,w_j)_j$ be a sequence of points such that $(z_j,w_j) \in \cal B_j$ for every $j$. We suppose that the
diameter of  $\cal B_j$ is small compared to $|  \partial_z \cZ_{j}(z_j,w_j)|$ when $j$ is large and that the points $(z_j,w_j)$ are uniformly distant to $\partial \tilde I^2$. Then the following map is $C^1$-close to the first coordinate projection when $j$ is large:
\[ \tilde \Upsilon_j : (z,w) \in  \cal B_j \mapsto \frac{\cal Z_{j}(z,w)-\cal Z_{j}(z_j,w)}{\partial_z  \cal Z_{j}(z_j,w_j)} +  z_j \, .\]
\end{proposition}

\begin{proof} [Proof of \cref{distortion}]
\emph{Let us first prove $(i)$.} Up to replacing $\cal B_j$ by $\mathrm{pr}_1(\cal B_j) \times \mathrm{pr}_2(\cal B_j)$ (whose diameter is still small when $j$ is large), we can suppose that $\cal B_j$ is a product. {Note that $  \Upsilon_j(x^0_j,y^0_j)= x_0^j$, $\partial_x \Upsilon_j(x^0_j,y^0_j)= 1$ and  $\partial_y \Upsilon_j(x^0_j,y)= 0$ for every  $y \in \mathrm{pr}_2(\cal B_j) $.
 Thus by the mean value inequality, it suffices to show that the map $\partial^2_x \Upsilon_j= \partial^2_{x} \cX/\partial_{x} \cX(x^0_j,y^0_j)$ and  $\partial_x\partial_y  \Upsilon_j= \partial_{y} \partial_x \cX/\partial_{x} \cX(x^0_j,y^0_j)$ are bounded since the diameter of  $\cal B_j$ is small.   This is a direct consequence of the   distortion bound:
\[\left|\frac{\partial^2_{x} \cX(x,y)}{\partial_{x} \cX(x^0_j,y^0_j)}\right|
\le 
 |\partial_x \log |\partial_x \cX_j(x,y)|| \cdot \left|\frac{\partial_{x} \cX(x,y)}{\partial_{x} \cX(x^0_j,y^0_j)}\right| 
\le    \|D \log |\partial_x \cX_j| \, \|\cdot e^{|\partial_x \log |\partial_x \cX_j(x,y)||\cdot \diam \cal B_j}\, ,
\]
\[\left|\frac{  \partial_{y} \partial_x \cX(x,y)}{\partial_{x} \cX(x^0_j,y^0_j)}\right|
\le  |\partial_y \log |\partial_x \cX_j(x,y)||\cdot \left|\frac{\partial_{x} \cX(x,y)}{\partial_{x} \cX(x^0_j,y^0_j)}\right|\le  \|D \log |\partial_x \cX_j| \, \|\cdot  e^{|\partial_y \log |\partial_x \cX_j(x,y)||\cdot  \diam \cal B_j}\, .
\]}

\emph{We now prove $(ii)$.}
{Let us denote $ \sigma_j  :=   \partial_z \cZ_{j}(z_j,w_j)$ for every $j \in \N$ and recall that $0< |\sigma_j|<1$. We notice that for any $(z,w) \in \cal B_j$, it holds:
\[    \partial_z  \tilde  \Upsilon_j (z,w) = \frac{1}{ \sigma_j} \partial_z  \cal Z_j (z,w) \qand     \partial_w  \tilde  \Upsilon_j (z,w) = \frac{1}{ \sigma_j} ( \partial_w  \cal Z_j (z,w)  -  \partial_w  \cal Z_j (z_j,w)  ) \, . \]
In particular, we have $\partial_z  \tilde  \Upsilon_j (z_j,w_j)=1$ and  $\partial_w  \tilde  \Upsilon_j (z_j,w_j)=0$. Since $\tilde  \Upsilon_j (z_j,w_j)=z_j$, the tangent map of  $\tilde  \Upsilon_j$ at  $(z_j,w_j)$ is simply the first coordinate projection. 

On the other hand,  the maps $\cZ_j : \tilde I^2 \rightarrow \tilde I$ are  holomorphic on the interior of $\tilde I^2$ and the subsets $\cal B_j$ of $\tilde I^2$ are  uniformly distant to $\partial \tilde I^2$ (since by assumption the points $(z_j,w_j) \in \cal B_j$ are uniformly distant to the boundary and $\mathrm{diam} (\cal B_j)$ is small). By the Cauchy  integral formula, the second  differential of $\cZ_j$ is then uniformly $C^0$-bounded on $\cal B_j$ among $j \in \N$ by a constant $C$. Thus the second  differential of $\tilde  \Upsilon_j$ is  $C^0$-bounded by $2C \cdot |\sigma_j|^{-1}$ on $\cal B_j$. Since $\mathrm{diam} ( \cal B_j ) \cdot 2C \cdot |\sigma_j|^{-1}$ is small when $j$ is large (by assumption) and $(z_j,w_j) \in \cal B_j $, the mean value inequality implies that $\tilde  \Upsilon_j$ is $C^1$-close  on $\cal B_j $ to its tangent map at $(z_j,w_j)$. }
\end{proof}

\section{Standard results from hyperbolic theory} \label{sec:Standard results from hyperbolic theory} 
Let $M$ be a manifold and $1\le r\le \infty$.  
We recall that a  compact subset $K\subset U$ is \emph{hyperbolic} for a $C^r$-endomorphism $f$ of $M$ if it is invariant ($K=f(K)$) and
there exists $n\ge 1$ and open cone fields $\chi_s$ and $\chi_u$   such that for every $z\in K$, it holds:
\begin{enumerate}
\item the union of $\chi_s(z)$ and $\chi_u(z)$ is $T_zM\setminus \{0\}$,
\item  $D_zf$ sends the closure of  $ \chi_u(z)$ into  $ \chi_u(f(z))\cup \{0\}$ and  the preimage  by $D_z f$ of the closure of $\chi_s(f(z))$ is in $\chi_s(z)\cup \{0\}$,
\item  every vector in $\chi_u(z)$ is expanded by $D_zf^n$, every vector in  $D_{f^n(z)} f^{-n}(\chi_s( f^n(z) ) )$ is contracted by $D_zf^n$. 
\end{enumerate}
We recall that the inverse limit $
\overleftrightarrow  K:= \{(x_i)_{i\in \Z}\in K^\Z: f(x_{i-1})=x_{i}\}$ is endowed with the topology induced by the product one. The map  $\overleftrightarrow  f:= 
(x_i)_{i\in \Z}\in \overleftrightarrow  K\mapsto (f(x_i))_{i\in \Z}\in \overleftrightarrow  K$ is semi-conjugate to $f|K$ via $h_0: (x_i)_{i\in \Z}\mapsto x_0$.   
Note that if $f|K$ is a homeomorphism then $h_0$ is a homeomorphism.
\begin{theorem}[Anosov, Quandt {\cite[Proposition 1]{Qua88}}]\label{Przy}  For every $f'$ $C^1$-close to $f$, there exists a unique continuous map $h_{f'}\colon \overleftrightarrow    K \to M$ which is $C^0$-close to $h_0$  and such that:
\begin{itemize}
\item $h_{f'}\circ \overleftrightarrow   f = f'\circ h_{f'}$,
\item $K_{f'}:=h_{f'}(\overleftrightarrow   K)$ is hyperbolic for $f'$, and called the hyperbolic continuation of $K$,
\item the map $f' \mapsto h_{f'}(\overleftrightarrow  k)$ is of class $C^\infty$ and depends continuously on $\overleftrightarrow  k\in \overleftrightarrow  K$ for the $C^1$-topology,
\item if $f'|K_{f'}$ is a homeomorphism, then $h_{f'}$ is a homeomorphism. 
\end{itemize}
\end{theorem}
\begin{remark} Quandt did not state that $f' \mapsto h_{f'}(\overleftrightarrow  k)$ is of class $C^\infty$ but he uses the fact that   this map  is the fixed point of a hyperbolic operator, and so the implicit function theorem  implies that $f' \mapsto h_{f'}(\overleftrightarrow  k)$ is of class $C^\infty$.
\end{remark}

\begin{remark}\label{bi holder} If $f$ is a diffeomorphism, then $h_{f'}$ is bi-H\"older with exponent close to $1$ when $f'$ is close to $f$.
\end{remark}
For a proof see for instance \cite[Theorem \textsection 2.6]{Yoccozintro}. For every $k\in K$  and $\overleftrightarrow  k\in \overleftrightarrow  K$, we define their respective stable and unstable manifolds by:
\begin{multline} 
W^s(k, f)=\{k'\in M:\; \lim_{i\to \infty} d({f}^{i}(k), {f}^{i}(k')) =  0\}\qand \\
W^u(\overleftrightarrow  k, f)=
\{k'_0\in M: \exists (k_i')_{i<0}, f(k_{i-1}')=k_i', \lim_{i\to -\infty} 
d(k_i, k_i')=0\}
 \end{multline}
 and for $\eta>0$ let:
 \begin{multline} 
W^s_\eta(k, f)=\{k'\in M:\; \eta> d({f}^{i}(k), {f}^{i}(k'))\stackrel{i\to \infty}\longrightarrow 0\}\qand \\
W^u_\eta(\overleftrightarrow  k, f)=
\{k'_0\in M: \exists (k_i')_{i<0}, f(k_{i-1}')=k_i', 
\eta> d(k_i, k'_i)\stackrel{i\to -\infty}\longrightarrow 0\}
 \end{multline}  
These sets are  properly embedded $C^r$ manifolds which depend  continuously on $k\in K$ and $\overleftrightarrow  k \in \overleftrightarrow  K$ for the $C^r$-topology (see for instance \cite[Theorem \textsection 3.6]{Yoccozintro} or \cite[Proposition 9.1]{Ber10}). For the sake of simplicity let us assume that all the unstable manifolds have the same dimension $d_u$ and all the stable manifolds have the same dimension $d_s$. By continuity and invariance, this is the case when $K$ is transitive. Let us recall the following that we are going to generalize and prove in \cref{inclination}:
\begin{lemma}[Inclination Lemma for endomorphism]\label{inclination sin para}
Let $\Gamma$ be a   $d_u$-submanifold  which is  transverse to $W^s_{\eta}( k, f)$ for $k\in K$ at a point $z$. Then for every $n$ sufficiently large,  
there exists a neighborhood $\Gamma'$ of $z$ in $ \Gamma$ such that  $f^n(\Gamma')$ is a submanifold which is uniformly $C^r$-close to $W^u_{\eta}( \overleftrightarrow  {k^n}, f)$, 
among any  $\overleftrightarrow  {k^n}\in \overleftrightarrow  K$ such that $h_0(\overleftrightarrow  {k^n})= f^n( k)$.
\end{lemma}

For $d\ge 0$ and let $\cP$ be a  regular compact subset of $\R^d$: it is nonempty and equal to the closure of its interior. We recall that a family $(f_p)_{p\in \cP}$ of maps $f_p\in C^r(M,M)$ is of class $C^r$ if the map $(p,x)\in \cP \times M \mapsto f_p(x)\in M$ extends to a $C^r$ map from a neighborhood $\hat \cP \times M$ of $\cP \times M$. By regularity of $\cP\times M$, the  $C^r$-jet of this map is uniquely defined at every point  in $\cP\times M$. 
We say that  two families are $C^r$-close if their $C^r$-jet functions are close for the $C^0$-compact-open topology.

Assume that $0\in \cP$ and a compact hyperbolic set $K_0$ for $f_{0}$  persists by Theorem \ref{Przy} for every $p\in \cP$.
We denote $h_p:= h_{f_p}$ the semi-conjugacy and
$K_p:=h_p(K_0)$ its hyperbolic continuation at the parameter $p\in \cP$. For $p\in \cP$ and $\overleftrightarrow  {k_0}\in \overleftrightarrow  K_0$, we put $ k_p:= h_{p}(\overleftrightarrow  {k_0})$ and $\overleftrightarrow  k_p:=(h_p\circ \overleftrightarrow  {f_0}^{n}(\overleftrightarrow  {k_0}))_{n\in \Z}$.

\begin{theorem}\label{cartecool4cpct}
The families $(W^s_\eta(k_p, f_p))_{p\in \cP}$ and $(W^u_\eta(\overleftrightarrow  {k_p}, f_p))_{p\in \cP}$ of $C^r$-submanifolds are of class $C^r$ and depend continuously on respectively $k_0\in K_0$ and $\overleftrightarrow  {k_0}\in \overleftrightarrow  K_0$.
\end{theorem}
We will prove this result below. The following is a generalization of \cref{inclination sin para}  for families. 
\begin{lemma}[Parametric Inclination Lemma {\cite[Lemma 1.7]{Be2016,Be2016C}}]\label{inclination}
Let $(\Gamma_p)_{p\in \cP}$ be a $C^r$-family of  $d_u$-submanifolds such that  for some $k_0\in K_0$, the manifold  $\Gamma_p $ intersects uniformly transversally $W^s_{\eta}( k_p, f_p)$  at a point $z_p$ depending continuously on $p\in \cP$. Then when  $n\ge 0$ is large,  
there exists a neighborhood $\Gamma'_p$ of $z_p$ in $ \Gamma_p$ such that  $f^n_p(\Gamma'_p)$ is a submanifold $C^r$-close to $W^u_{\eta}( \overleftrightarrow  {k^n_p}, f_p)$ and $(f^n_p(\Gamma'_p))_{p\in \cP}$ is $C^r$-close to $(W^u_{\eta}( \overleftrightarrow  {k_p^n}, f_p))_p$ among any  $ \overleftrightarrow  {k^n_0}\in \overleftrightarrow  K_0$ such that  $h_0(\overleftrightarrow  {k^n_0})= f^n_0( k_0)$.
\end{lemma}
\begin{proof}
The result being semi-local, we may assume that $f_p$ is a map  from a small neighborhood $U_p$ of $K_p$ onto its image for every $p$. Then we extend $f_p$ to a larger open set $U'_p\Supset U_p$ containing a hyperbolic fixed point $H_p$ such that $\Gamma_p$ is a slice of the unstable manifold $W^u(H_p, f_p)$ and such that $f^{-n}_p(\Gamma_p)$ is disjoint from $U_p$ for every $n\ge 0$.
Then we consider the compact set $\hat K_{p}:= K_{p}\cup \{f^n(z_p):\; n\in \Z\}\cup \{H_p\}$. We notice that $\hat K_{p}$ is a hyperbolic compact set. Also
$\Gamma_p^n= W^u_\eta (f^n_p(z_p), f_p)$. Thus the lemma follows from \cref{cartecool4cpct}.
\end{proof}
\begin{proof}[Proof of \cref{cartecool4cpct}]
Let $\hat \cP$ be an open set such that there exists a $C^r$-extension $(f_p)_{p\in \hat \cP}$ of $(f_p)_{p\in \cP}$ for which the compact hyperbolic set persists to $(K_p)_{p\in \hat \cP}$.  Take $p_0\in \cP$. Let $r>0$ be small such that the balls $B_r$ and $B_{2r}$ centered at $p_0$ and of radii $r$ and $2r$ respectively are included in $\hat \cP$. As the statement to be proved is local, it suffices to show it for the family $(f_p)_{p\in B_r}$.
Let 
\[\rho: (x_1, ..., x_{d+1})\in \S^d=\{(x_1, ..., x_{d+1})\in \R^{d+1}: \sum x_i^2=1\}\mapsto p_0+ 2r (x_1, ..., x_{d})\in B_{2r}\]
Note that $\rho^{-1}(B_r)$ is formed by two components which are diffeomorphically mapped to $B_r$ by $\rho$. 
 Thus it suffices to show the proposition for the family 
 $(f_{\rho(s)})_{s\in \S^d}$. For every $\overleftrightarrow  k\in \overleftrightarrow  K_{p_0}$,  let
 \[\cal L(\overleftrightarrow  k):= \{(s, h_{\rho(s)}(\overleftrightarrow  k)): s\in \S^d\}\subset \S^d\times M.\] 
 By \cref{Przy}, the sphere $\cal L(\overleftrightarrow  k)$ are of class $C^r$ and depends continuously on $\overleftrightarrow  k \in \overleftrightarrow  {K_{p_0}}$ for this topology. Hence  $\bigcup_{\overleftrightarrow  k\in \overleftrightarrow  K} \cal L(\overleftrightarrow  k)$ is the image of $\S^d\times \overleftrightarrow  K$ by a $C^r$-immersion of this lamination. Moreover this lamination is $r$-normally hyperbolic in the sense of \cite[Definition 2.1]{Ber10}. Consequently, we can apply  \cite[Proposition 9.1]{Ber10} which states that the strong stable and unstable local manifolds  of a leaf  form $C^r$-immersed laminations. This is equivalent to say that $(W^s_{\eta}(k_{\rho(s)}; f_{\rho(s)}))_{s\in \S^d}$ and $(W^u_{\eta}(\overleftrightarrow  k_{\rho(s)}; f_{\rho(s)}))_{s\in \S^d}$ are $C^r$-families which depends  continuously on $k_{p_0}\in K_{{p_0}}$ and $\overleftrightarrow  k_{p_0}\in \overleftrightarrow  K_{{p_0}}$. 
\end{proof}

The following enables to extend the stable lamination to a $C^1$-foliation, by integrating the $C^1$-line field $\ell(p, \cdot )$.
\begin{theorem}\label{extension lam}
If  $(f_p)_{p\in \cP}$ is a $C^2$-family of diffeomorphisms  of a surface $M$ leaving invariant the continuation  $(K_p)_p$ of a hyperbolic set,    there exists  a neighborhood $\hat V =  {\bigcup_{p\in \cP}  \{p\}\times V_p } $ of $\bigcup_{p\in \cP}  \{p\}\times K_p$ and a $C^1$-function $\ell $ from $\hat V$ into  $\mathbb P(\R^2)$ satisfying for every $p\in \cP$:
\[\ell(p, z)=T_zW^s_\eta(k)\quad \text{if } z\in W^s_\eta(k)\text { with }  k\in K_p\qand 
D_zf_p(\ell(p, z))=\ell(p, f_p(z)) \quad \text{if }
 z\in V_p\cap f_p^{-1}(V_p).
\]
\end{theorem}
\begin{proof} 
The proof is the parametric counterpart of \cite[Corollary 1.11]{bonatticrovisier}.
Let $\mathbb P(TM)$ be the bundle over $M$ whose fibers at $z\in M$ is the Grassmanian $\mathbb P(T_zM)$ of lines in $T_zM$. We consider:
\[\hat f: (p, z, \ell)\in \hat{\cal P}\times \mathbb P(TM)\mapsto (p, f_p(z), D_zf_p(\ell) )\in \hat {\cal P}\times  \mathbb P(TM),\]
where $\hat{\cal P}$ is the open neighborhood of $\cal P$ to which the family extends.  We remark that $\hat f$ is of class $C^1$ on 
$\hat M:=  \hat{\cal P}\times \mathbb P(TM)$. We denote $E^s_p(k):=T_kW^s(k, f_p)$ and 
$E^u_p(k):= T_kW^u(k, f_p)$ for every  $k\in K_p$ and $p\in \cP$. Let $\hat K:= \{(p, k, E^s_p(k)): p\in \cP\qand k\in K_p\}$.  It is a  partially hyperbolic compact set for $\hat f$ with strong unstable direction at $(p,k,E^s_p(k) )$ equal to $\{(0,0)\}\times T_{E^s_p(k)}\mathbb P(T_k M)$. The strong {unstable} manifold  of a point $(p, k, E^s_{p}(k))$ is $(p, k, \mathbb P(T_kM)\setminus \{E^u_{p}(k)\})$. It intersects $\hat K$ only at one point.  Thus we can apply the following:  
  \begin{theorem}[{\cite[Main result]{bonatticrovisier}}]
Let $\hat f$ be a $C^1$-diffeomorphism of a manifold $\hat M$ and $\hat K$ a partially hyperbolic compact invariant set such that $T\hat M|\hat K=E^c\oplus E^{uu}$. Then, the next two properties are equivalent.
\begin{enumerate}
\item   There exists a compact $C^1$-submanifold $ S$ with boundary which:
\begin{itemize}
\item contains $\hat K$ in its interior,
\item is tangent to $E^c$ at each point of $\hat K$ (i.e. $T_xS=E^c(x)$ for each $x\in \hat K$),
\item is locally invariant:  {$\hat f(S)$} contains a neighborhood of $\hat K$ in $S$,
\end{itemize}
\item The strong unstable manifold of any $x\in \hat K$ intersect $\hat K$ only at $x$ (i.e. $W^{uu}(x)\cap \hat K=\{x\})$.
\end{enumerate}
\end{theorem}
\begin{remark}\label{coro BB} 
By \cite[Corollary 1.5]{bonatticrovisier}, if $\hat f$ is of class $C^{1+}$, then $S$ can be chosen of class $C^{1+}$. 
\end{remark}
As the strong unstable direction is the last coordinate, 
up to restricting $S$, there exists a function from a neighborhood $ \hat V$ of $  \hat K $ and a $C^1$-function $\ell $ from $\hat V$ into   $\mathbb P(\R^2)$ whose graph is equal to $S$. 
Thus by invariance and since the action of $\hat f$ on the two first coordinates is $(p,z)\mapsto (p, f_p(z))$,  we have:
\[
\hat f(p, z, \ell(p, z))=(p, f_p(z), \ell(p, f_p(z))),\quad 
\forall p\in \cP\text{ and } z\in V_p\cap f_p^{-1}(V_p).
\] 
As by definition $\hat f(p, z, \ell(p, z))=(p, f_p(z), D_zf_p(\ell(p, z)))$, we obtain $D_zf_p(\ell(p, z))=\ell(p, f_p(z))$. 
Given $p\in \cP$, $k\in K_p$ and $z\in W^s_\eta(k,f_p)$, if $\ell(p,z)\neq T_zW^s_\eta(k,f_p)$, then by invariance $\ell(p,f_p^n(z))=D_zf_p^n(\ell(p,z))\neq D_zf^n_p(T_zW^s_\eta(k, f_p))=T_{f_p^n(z)}W^s_\eta(f^n_p(k), f_p)$ for every $n\ge 0$. Moreover, by hyperbolicity the angle between $D_zf_p^n(\ell(p,z))$ and $ D_zf^n_p(T_z W^s_\eta(k,f_p))$  is bounded from below. This contradicts that the angle between $\ell(p,f_p^n(z))$, $\ell(p,f_p^n(k))$ and $ T_{f_p^n(z)}W^s_\eta(f^n_p(k), f_p)$ goes to $0$ as $d(f_p^n(z), f_p^n(k))\to 0$.
\end{proof}
Using \cref{coro BB} we obtain:
\begin{corollary}\label{coro extension lam}
If under the assumptions of \cref{extension lam}, the family  $(f_p)_{p\in \cP}$ is of class  $C^{2+}$, then the function $\ell $ given by \cref{extension lam} can be assumed of class $C^{1+}$. 
\end{corollary}

\section{Modeling a horseshoe by a hyperbolic map of type $\sA$}\label{proof Horseshoe2model}
\begin{proof}[Proof of \cref{Horseshoe2model}] \noindent \emph{Construction of the maps   $H_p$ and verification of property $(1)$}.
Let $\cP$ be a closed ball centered at $0$ included in $\hat \cP$. 
Hyperbolicity implies the existence of a splitting $TM|K_p= E^s_p\oplus E^u_p$ of two $Df_p$-invariant directions which are respectively contracted and expanded by $Df_p$.  
Let  $\hat K=\bigcup_{p\in \cP} \{p\}\times K_p$. By  \cref{extension lam} \cpageref{extension lam}, we have:
\begin{theorem} \label{def es eu}
 There exist a neighborhood $\hat U=\bigcup_{p\in  \cP} \{p\}\times U_p$ of $\hat K $ and 
two $C^1$-families  $(\ell^s_p)_p$ and $(\ell^u_p)_p$ of   line fields $\ell^s_p$ and $\ell^u_p$   on $ U_p$ such that:
\begin{enumerate}
\item  for every $z\in K_p$, the restriction  $\ell^s_p| W^s_{loc}(z,f_p)\cap U_p$ coincides with    $TW^s_{loc}(z,f_p) \cap U_p$ and $\ell^u_p| {W^u_{loc}(z,f_p)} \cap U_p$ coincides with $TW^u_{loc}(z,f_p) \cap U_p$,
\item for every $z\in U_p\cap f_p^{-1}(U_p)$, the   line  $D_zf_p(\ell^s_p(z))$ is equal to  $\ell^s_p(f_p(z))$ and the
line   $D_zf_p(\ell^u_p(z))$ is equal  to $\ell^u_p(f_p(z))$. 
\end{enumerate}
\end{theorem}

For every   $p \in \cP$, let  $\cF^s_p$ and $\cF^u_p$ be the foliations on $U_p$ whose tangent spaces to the leaves are respectively $\ell^s_p$ and $\ell^u_p$. 
For $\delta>0$, recall that  the  $\delta$-$\cF^s_p$-plaque of $z\in U_p$ is the connected component of $z$ in $\cF^s_p(z) \cap B(z,\delta)$, where $\cF^s_p(z)$ is the leaf of $\cF^s_p$ containing $z$ and $B(z,\delta)$ the closed $\delta$-ball about $z$. The  $\delta$-$\cF^u_p$-plaque of $z\in U_p$ is defined similarly. 
 Note that the $\delta$-local stable manifold $W^s_\delta(z;f_p)$ of $z\in K_p$  is a   $\delta$- $\cF^s_p$-plaque.

Up to shrinking $\hat U=\bigcup_p \{p\}\times U_p$, by hyperbolicity of $K_p$, we can assume that any pair of leaves of  $\cF^s_p$ and $\cF^u_p$ are transverse. Observe that for $\delta>0$ sufficiently small, the intersection of every pair of $\delta$-plaques of $\cF^u_p$ and $\cF^s_p$ consists of at most one point. Then for every $\eta>0$ sufficiently small, the following property holds true:
for any $p\in \cal P$, $z\in K_p$,   $x\in W^u_{\eta}(z; f_p)$ and $y\in W^s_{\eta}(z; f_p)$, the $\delta$-$\cF^s_p$-plaque of $x$   intersects the $\delta$-$\cF^u_p$-plaque of $y$ at exactly one point denoted $[x,y]$.

By identifying both $W^u_{\eta}(z; f_p)$ and $W^s_{\eta}(z; f_p)$ to $(-\eta,\eta)$ using the relative distance to $z$, we obtain a map:
\[\phi_{p,z}: (x,y)\in (-\eta,\eta)^2\equiv W^u_{\eta}(z; f_p)\times W^s_{\eta}(z; f_p)\mapsto [x,y]\in M \] 
Note that $\phi_{p,z}$ is a local $C^1$-diffeomorphism by transversality. It is injective by the uniqueness of the intersection point between the $\delta$-plaques. Thus $\phi_{p,z}$ is a diffeomorphism onto its image which is an open neighborhood of $z$. 
As $(\ell^s_p)_p$ and $(\ell^u_p)_p$  are  $C^1$-families, by construction,  the family of maps $(\phi_{p,z})_{p\in \cP,z\in K_p}$ is $C^1$. So by compactness of $\cP $ and $K_p$,  there is $\delta'>0$, such that the image of  $\phi_{p,z}$ contains the $\delta' $-neighborhood of $z$ for any $p\in \cP$ and $z\in K_p$.

  By  \cite[Theorem 3.12]{Bo08},  there is a \emph{Markov partition} of $K_0$ by rectangles of $\eta$-small diameter. This is a finite partition $(R^\sv_0)_{\sv\in \sV}$ by clopen subsets $R^\sv_0\subset K_0$ satisfying: 
\begin{itemize} 
\item Each subset $R^\sv_0$ is a rectangle:  there exist  two segments $I^\sv_0, J^\sv_0\subset (-\eta, \eta)$ satisfying
\begin{equation}\label{markov1} R^\sv_0= \breve \phi_0^\sv( I^\sv_0\times J^\sv_0 )\cap K_0 
  \qquad \text{with}\quad    z_0^\sv \in  R^\sv_0  \text{ and }       \breve \phi_0^\sv := \phi_{p,z_0^\sv}     
\; .\end{equation}
\item For every  pair $(\sv, \sw)\in \sV^2$ such that 
$R^{\sv\sw}_0:=R^\sv_0\cap f_0^{-1}(R^\sw_0)\neq \emptyset$, 
there are  two\footnote{When the diameters of the rectangle are not small, one has to consider more segments.} segments   $ I^{\sv\sw}_0\subset I^\sv_0$ and 
  $  J^{\sv\sw}_0\subset J^\sw_0$ such that:
  \begin{equation}\label{markov2}
 R^{\sv\sw}_0 = \breve \phi_0^\sv(  I^{\sv\sw}_0\times J^\sv_0  )\cap { K_0}
  \qand  
  f_0\circ \breve \phi_0^\sv(  I^{\sv\sw}_0\times J^\sv_0  )   =  { \breve \phi^\sw_0(  I^\sw_0\times J^{\sv\sw}_0)}  \; .
\end{equation}
\end{itemize}
 We assume the segments $I^{\sv}_0$ and  $J^{\sv}_0$ minimal such that \cref{markov1} holds true. Then $I^{\sv\sw}_0 $ and $J^{\sv\sw}_0  $ are uniquely defined by \cref{markov2}. 
Let $(R^{\sv}_{p})_{p\in \cP}$,  $(R^{\sv  \sw}_{p})_{p\in \cP}$ and $(z^\sv_{p})_{p\in \cP}$ be the hyperbolic continuations of $R^{\sv }_{0} $,  $R^{\sv  \sw}_{0} $ and $z^\sv_{0}$  (see \cref{Przy} \cpageref{Przy}). For $\cP$ sufficiently small, for every $p\in \cP$, there are minimal  segments  $I^{\sv}_p$ and $J^{\sv}_p$ included in $(-\eta, \eta)$ satisfying
 \begin{equation}\label{markov1bis} R^\sv_p= \breve \phi_p^\sv( I^\sv_p\times J^\sv_p )\cap K_p 
  \qquad \text{with}\quad        \breve \phi_p^\sv := \phi_{p,z_p^\sv}      
\; \end{equation}
and there are segments   $I^{\sv\sw}_p\subset I^{\sv}_p $ and $J^{\sv\sw}_p \subset J^{\sw}_p$ such that:
  \begin{equation}\label{markov2bis}
 R^{\sv\sw}_p = \breve \phi_p^\sv(  I^{\sv\sw}_p\times J^\sv_p  )\cap { K_p}
  \qand  
  f_p\circ \breve \phi_p^\sv(  I^{\sv\sw}_p\times J^\sv_p  )   =  { \breve \phi^\sw_p(  I^\sw_p\times J^{\sv\sw}_p)}  \; .
\end{equation}

\begin{fact}  
The segments $I^{\sv}_p$, $J^{\sv}_p$,   $I^{\sv\sw}_p$ and $J^{\sv\sw}_p $ vary $C^1$ with $p\in \cP$, that is each of them has its two boundary points varying $C^1$ with $p\in \cP$.
\end{fact}
\begin{proof} 
 Let us prove the fact for $(I^{\sv}_p)_p$, the proof is the same in the other cases. By the local product structure of $K_p$, by hyperbolic continuation of $(K_p)_p$ and by definition of $I^{\sv}_p$ as the minimal segment satisfying \cref{markov1bis}, there exist $C^1$-families of points $k_p^- \in K_p$ and $k_p^+ \in K_p$ such that $\{k_p^-, k_p^+ \}=   \breve \phi_p^\sv(\partial I^{\sv}_p)$. As $(\breve \phi_p^\sv)_p = ( \phi_{p,z^\sv_p})_p$ is of class $C^1$, the endpoints $\partial I^{\sv}_p$ vary $C^1$ with $p\in \cP$.  The proof of the regularity of the other endpoints is similar. 
 \end{proof} 
 
 Let:
\[Q^\sv_p:= \breve \phi_p^\sv(  I^\sv_p\times J^\sv_p)\; ,\quad 
\partial^s Q^\sv_p:= \breve \phi_p^\sv( \partial (I^\sv_p)\times J^\sv_p) \qand 
\partial^u Q^\sv_p:= \breve \phi_p^\sv( I^\sv_p\times \partial (J^\sv_p))\, .\]
\begin{fact}\label{Q disjoint}
The compact sets $(Q^\sv_p)_\sv$ are  disjoint for every $p\in \cP$.
\end{fact}
\begin{proof}
For every $\sv\in \sV$, by minimality of $I_p^{\sv}$ and \cref{markov1bis}, each of the two components  of  
$\partial^s Q^\sv_p$  contains a point of $K_p$. Hence, by \cref{def es eu} (1), each of the two components  of  
$\partial^s Q^\sv_p$  is equal to a local stable manifold of it.
 Likewise by minimality of $J_p^{\sv}$, \cref{markov1bis} and \cref{def es eu} (1), each of the two components  of  
$\partial^u Q^\sv_p$  is equal to a local unstable manifold of a point of $K_p$. 
 Hence by local maximality of $K_p$, the corners $\partial^u Q^\sv_p\cap \partial^s Q^\sv_p$ of $Q^\sv_p$ are in
  $K_p$. 
Hence for $\sv,\sw\in \sV$, if $Q^\sv_p$ intersects   $Q^{\sw}_p$ then  the sets  $Q^\sv_p$ and  $Q^{\sw}_p$ are nested or a stable boundary of $Q^\sv_p$ (resp. $Q^\sw_p$)  intersects an unstable boundary of $Q^\sw_p$  (resp. $Q^\sv_p$). Thus if $Q^\sv_p$ intersects   $Q^{\sw}_p$, then $R^\sv_p=Q^\sv_p\cap K_p $ intersects  $R^\sw_p=Q^\sw_p\cap K_p $ and so $\sv=\sw$.
\end{proof}

Let $k\ge 1$ be a large   integer  so that $f_p^k|K_p$ is still transitive. Put:
\[\sA:= \left\{ (\sv_0,\dots, \sv_k)\in  \sV^{k+1}: \bigcap_{j=0}^k f_p^{-j}(R_p^{\sv_j}) \neq\emptyset\right\}\; , \quad \si(\sv_0,\dots, \sv_k)=\sv_0 \qand \st(\sv_0 ,\dots, \sv_k)=\sv_k\;  .
\] 
Put $R^\sa_p:= \bigcap_{j=0}^k f_p^{-j}(R_p^{\sv_j})$ for every $\sa\in \sA$ and $p\in \cP$.  We recall that each $\breve \phi_p^\sw$ has its range in $U_p\cap f_p^{-1}(U_p)$, and so each map  $(\breve \phi_p^\sw)^{-1}\circ  f_p\circ \breve \phi_p^\sv$ preserves the horizontal and vertical directions, and respectively expands  and contracts them. Thus by \cref{markov2bis}, for every $\sa\in \sA$, there exist unique  $I_p^\sa\subset I_p^{\si(\sa)}$    and $J_p^\sa\subset J_p^{\st(\sa)}$ such that:
  \begin{equation}\label{markov2ters}
 R^{\sa}_p = \breve \phi_p^\sv(  I^{\sa}_p\times J^{\si(\sa)}_p  )\cap { K_p}
  \qand  
  f_p^k\circ \breve \phi_p^{\si(\sa)}(  I^{\sa}_p\times J^{\si(\sa)}_p  )   =  { \breve \phi^{\st(\sa)}_p(  I^{\st(\sa)}_p\times J^{\sa}_p)}  \; .
\end{equation}
For every $\sa\in \sA$ and $p\in \cP$, the following  map preserves again the horizontal and vertical directions
\begin{equation}\label{def breve fpa} \breve f_p^\sa:=   (\breve \phi^{\st(\sa)}_p)^{-1}\circ  f_p^k\circ \breve \phi_p^{\si(\sa)}|  I^{\sa}_p\times J^{\si(\sa)}_p\, . \end{equation}
We assume $k$ large enough so that it  satisfies moreover:
\begin{equation}\label{def de C} \min |\partial_x \breve f_p^\sa|>  \frac{C\cdot\lambda}{\theta^2} \; , \quad  \max |\partial_y \breve f_p^\sa|< \frac{\theta^2}{C\cdot \lambda}\quad \text{where }
C:= { \max_{(\sv  , \sw) \in \sV^2}} \left\{\frac{|I_p^\sv|}{|I_p^\sw|},\frac{|J_p^\sv|}{|J_p^\sw|}\right\}\, .\end{equation}
Thus there are $\theta^2/(C\lambda)$-contracting maps 
$\breve x^\sa_p: I^{\st(\sa)}_p\to I^\sa_p$ and $\breve y^\sa_p: J^{\si(\sa)}_p\to J^\sa_p$ such that:
\begin{equation}\label{impli forma} \breve f_p^\sa(\breve x^\sa_p(x_1), y_0)=(x_1,\breve y^\sa_p(y_0))\; ,
\end{equation}
for any $ (x_1, y_0)\in I_p^{\st(\sa)}\times J_p^{\si(\sa)}$. 
Note that $\breve x^\sa_p $ and $\breve y^\sa_p$ are defined via \cref{def breve fpa,impli forma} on some $\eta'$-neighborhoods $V^{\st(\sa)}_p$ and $ W^{\si(\sa)}$ of $ I_p^{\st(\sa)}$ and $J_p^{\si(\sa)}$ in such a way that \cref{impli forma} remains valid. If $\eta'$ is sufficiently small, these maps are still  $\frac{\theta^2} {C\cdot \lambda}$-contracting by \cref{def de C}. Let $\eta'>0$ be small enough so that these properties are valid for every $\sa$ in the finite set $\sA$ and $p$ in the compact set $\cP$.   By contraction, it holds:
\begin{equation}\label{impli forma2}
\breve x^\sa_p(V_p^{\st(\sa)} )\Subset V_p^{\si(\sa)}\qand \breve y^\sa_p(W_p^{\si(\sa)}) \Subset W_p^{\st(\sa)}\; .
\end{equation}
For every $\sv \in \sV$, let $\Delta^\sv_p$ be the affine map with diagonal and positive linear part which sends $Y^\se $ to $V^\sv_p\times W_p^\sv$.  
We define the $C^1$-family $(\breve H_p)_{p\in \cP}$ by:
\[\breve H_p:(z, \sv)\in  Y^\se\times \sV\mapsto {  \breve H_p^\sv}(z)\in M\; , \quad \text{with }\breve H_p^\sv: z\in Y^\se\mapsto \breve \phi_p ^\sv\circ  \Delta^\sv_p(z)\; .\]
Let $\breve Y^\sa_p$ be the preimage by $\Delta^{\si(\sa)}_p$ of $\breve x^\sa_p(V_p^{\st(\sa)} ) \times W^{\si(\sa)}$.    By \cref{impli forma2}, $\breve Y^\sa_p$ is included in $Y^\se\setminus \partial^sY^\se$. 
Also,  $\breve Y^\sa_p$ is sent by 
$ \breve f^\sa_p \circ \Delta_p^{\si(\sa)}$ onto   $ V^{\st(\sa)}_p\times \breve y^\sa_p( W^{\si(\sa)})\subset \Delta_p^{\st(\sa)}(Y^\se\setminus \partial^uY^\se)$. Put:
\[\breve F^\sa_p:=(\Delta_p^{\st(\sa)})^{-1} \circ \breve f^\sa_p \circ \Delta_p^{\si(\sa)}: \breve Y^\sa_p
\to Y^\se\setminus \partial^uY^\se
\, .\]
Let $(\breve F^\sA_p)_{p\in \cP}$ be the $C^1$-family of maps 
$\breve F^\sA_p$ from $\breve D_p(\sA):= \bigsqcup_{\sa\in \sA} \breve Y^\sa_p\times \{\si(\sa)\}$ equal to $(\breve F^\sa_p, \st(\sa))$ at  $ \breve Y^\sa_p\times \{\si(\sa)\}$. Then \cref{Horseshoe2model} (0)  holds true since $\breve F^\sA_p$ leaves invariant the vertical and horizontal directions, and by \cref{def breve fpa}:
\begin{equation}\label{conju breve} f^k_p\circ \breve H_p^{\si(\sa)}|\breve Y^\sa_p=\breve \phi_p^{\st(\sa)} \circ \breve f^\sa_p \circ \Delta_p^{\si(\sa)}=\breve H^{\st(\sa)}_p\circ \breve F_p^\sa \; .  \end{equation}

Let  $(H_p)_{p\in \cP}$ be a $C^r$-family which is $C^1$-close to $(\breve H_{ p})_{ p\in \cP}$. For every $\sv\in \sV$ and $z\in Y^\se$, let $H^\sv_p(z):= H_p(z, \sv)$. 
\begin{fact}\label{prehyp trans} Every nonzero vector in the complement of $\chi_v$ is sent by $D\breve F^\sA_p$   into $\chi_h$ and $\lambda$-expands its horizontal component, also  $D(\breve F^\sA_p)^{-1}$ sends every nonzero vector in the complement of $\chi_h$ into $\chi_v$ and expands its vertical component.\end{fact}
\begin{proof} This is a consequence of   \cref{def de C} which implies  $|\partial_x \breve F^\sA_p|>\lambda/\theta^2$ and $|\partial_y \breve F^\sA_p|<\theta^2/\lambda$. \end{proof}
 Also by \cref{conju breve} and then  \cref{impli forma,impli forma2}, for every $(x_1, y_0)\in I^2$, there exist $\breve \cX_p^\sa(x_1)$ and $\breve \cY_p^\sa(y_0)$ in the interior of $I$ such that:
  \[f^k_p \circ \breve H_p^{\si(\sa)} (\breve \cX_p^\sa(x_1), y_0)=
  \breve H^{\st(\sa)}_p\circ \breve F^\sa_p(\breve \cX_p^\sa(x_1), y_0)=\breve H^{\st(\sa)}_p(x_1, \breve \cY_p^\sa( y_0))\; .\]
Thus, as  $H_p^{\si(\sa)}$ and $H^{\st(\sa)}_p$ are $C^1$-close to 
$\breve H_p^{\si(\sa)}$ and $\breve H^{\st(\sa)}_p$, by hyperbolicity of $\breve F^\sa_p$ and the implicit function theorem,  there exist unique points  $ \cX_p^\sa(x_1, y_0)$ and $  \cY_p^\sa(x_1,y_0)$ close to  $\breve \cX_p^\sa(x_1)$ and $\breve \cY_p^\sa(y_0)$ satisfying :
  \[f^k_p \circ   H_p^{\si(\sa)} (\cX_p^\sa(x_1, y_0), y_0)=  H^{\st(\sa)}_p(x_1, \cY_p^\sa( x_1, y_0))\; .\]
Still by the implicit function theorem, the map $(p, x_1, y_0)\mapsto (\cX_p^\sa, \cY_p^\sa)(x_1, y_0)$ is of class $C^r$ and is $C^1$-close to 
 $(p, x_1, y_0)\mapsto (\breve \cX_p^\sa(x_1), \cY_p^\sa(y_0))$. Hence the following set:
 \[Y^\sa_p:= \{ (\cX_p^\sa(x_1, y_0), y_0): (x_1, y_0)\in I^2\}\]
 is a box included in $Y^\se\setminus \partial^s Y^\se$. Moreover it is sent by 
 $F^\sa_p:=( H_p^{\st(\sa)} )^{-1}\circ  f_p^k \circ   H_p^{\si(\sa)}|Y^\sa_p$ into  
 $\{ (x_1,\cY_p^\sa(x_1, y_0)): (x_1, y_0)\in I^2\}\subset Y^\se\setminus \partial^uY^\se$. Hence by Fact \ref{prehyp trans}, $(F^\sa_p, Y^\sa_p)$ is a hyperbolic transformation. Also the $C^r$-family of maps 
$F^\sA_p$ from $ D_p(\sA):= \bigsqcup_{\sa\in \sA} Y^\sa_p\times \{\si(\sa)\}$ equal to $(F^\sa_p, \st(\sa))$ at  $   Y^\sa_p\times \{\si(\sa)\}$ satisfies  \cref{Horseshoe2model} (1).

\emph {Verification of property $(2)$}.
If  $\limsup_{n\to \infty}  \|\det Df^n_{0}|K_0\|^{1/\epsilon} \cdot \|Df_{0}^n| K_0\|<1\;  $, then for $k$ large enough, $\breve F^\sA_0$ and so $ F^\sA_0$ are moderately dissipative. So are any $ F^\sA_p$ for $p$ in a sufficiently small $\cP$. 

\emph {Verification of property $(3)$}.
 We notice that the pull back by $H_p$ of the vector field $e^s_p$ is $C^1$ close  to be tangent to $\{0\}\times \R$ and is defined on  $Y^\se\times \sV$. Thus we can extend it to a $C^1$-nonzero-vector field on $ \R\times I \times\sV$ with values in $\chi_v$. Then  the holonomy  $\pi_p $ of this vector field restricted to the transverse section $ \R \times \{0\}\times \sV $ defines  a family of projections $(\pi_p)_{p\in \cP}$ satisfying (0)-(1)-(2)-(3) of \cref{def adapted proj}, {which shows property $(3)$ of \cref{Horseshoe2model}  and concludes the proof.}
\end{proof}
\begin{proof} [Proof of \cref{coro Horseshoe2model}]
We use  \cref{coro extension lam} while using \cref{extension lam} in the above proof.\end{proof}

\section{A general result on Emergence}\label{sec:emergence}
The aim of this section is to prove \cref{horseshoe emerge} \cpageref{horseshoe emerge}. Let  $K$ is a horseshoe of a $C^{1+ }$-surface diffeomorphism $f$. We want to show that the covering number $\cal N(\epsilon )$ at scale $\epsilon$ of  $\cal M_f(K)$ for the Wasserstein distance (see \cref{def Wasserstein}) has  order at least   the unstable dimension  $d_u$ of $K$: 
\[\liminf_{\epsilon\to 0} \frac{\log \log \cal N(\epsilon)}{-\log \epsilon}\ge  d_u\; .\]

\begin{proof}[Proof of \cref{horseshoe emerge}]

First we can assume $f$ of class $C^\infty$. Indeed, we can consider a smooth approximation $\tilde f$ of $f$. Then by \cref{bi holder}, the restriction of $\tilde f$ to the continuation $\tilde K$ of $K$ is conjugated to $f|K$ via a bi-H\"older conjugacy with exponent close to $1$. Thus the unstable dimension of $\tilde K$ is close to the one of $K$ and the space of invariant probability measures $\cal M_{\tilde f}(\tilde K)$ of $\tilde f |K $ is in bijection to $\cal M_f(K)$  via a bi-H\"older map with exponent close to $1$. So the orders of their covering numbers are close.

Thus we assume that $f$ is of class $C^\infty$. Then by \cref{coro Horseshoe2model}  of \cref{Horseshoe2model} (for $d=0$), the map $f|K$ is $C^{1+}$-semi-conjugate (via $\pi$) to a $C^{1+}$-function $g$ on a disjoint union of segments of the real line, which leaves invariant an expanding Cantor set $\Lambda_g$ of Hausdorff 
dimension equal to the unstable dimension $d^u$ of $K$.  

 Then by \cite[Theorem A]{bergerbochi2019} the covering number $\cal N_g$ of the space of invariant measures of $g| \Lambda_g$ has order $d^u$:
\[\log\log  \cal N_g(\eta)\sim -d^u \cdot \log \eta , \quad \text{as } \eta\to 0\]
The push forward by the Lipschitz semi-conjugacy is a   Lipschitz map from the space of invariant measure of $F^\sA$ to those of $g$. By the same reasoning as in \cref{lifting}, this pushforward is a bijection between the spaces of invariant measures of $g$ and $F^\sA$.  Thus the covering number of the space of invariant measure of $F^\sA$ has the sought property. 
\end{proof}

\end{appendix}

\printindex
\bibliographystyle{alpha}
\bibliography{references}
\vskip 5pt

\begin{tabular}{l l l}
\emph{\normalsize Pierre Berger}
& \quad &
\emph{\normalsize S\'ebastien Biebler}
\medskip\\
\small \texttt{pierre.berger@imj-prg.fr}& &\small \texttt{sebastien.biebler@imj-prg.fr} \\
\end{tabular}

\begin{tabular}{c}
\small Sorbonne Universit\'e, Universit\'e de Paris, CNRS,  \\
Institut de Math\'ematiques de Jussieu-Paris Rive Gauche \\
 F-75005 Paris, France \\
\end{tabular}

\end{otherlanguage}
\end{document}